\documentclass[12pt]{amsart}
\usepackage{amsthm}
\usepackage{amssymb}
\usepackage{amsbsy}
\usepackage{amscd}

\usepackage[mathscr]{eucal}
\usepackage{nicefrac}

\DeclareSymbolFont{rsfs}{U}{rsfs}{m}{n}
\DeclareSymbolFontAlphabet{\mathrsfs}{rsfs}

\usepackage{a4wide}
\usepackage[all]{xy}
\usepackage{tikz}
\usetikzlibrary{arrows,decorations.pathmorphing,backgrounds,positioning,fit,petri}
\usetikzlibrary{decorations.markings}

\usepackage{verbatim}
\usepackage{version}
\usepackage{color}
\definecolor{darkspringgreen}{rgb}{0.09, 0.45, 0.27}
\definecolor{deepjunglegreen}{rgb}{0.0, 0.29, 0.29}
\newenvironment{NB}{
\color{red}{\bf NB}. \footnotesize
}{}
\newenvironment{NB2}{
\color{blue}{\bf NB2}. \footnotesize
}{}
\excludeversion{NB}
\excludeversion{NB2}
\makeatletter

\usepackage{url}
\usepackage{ifmtarg}

\usepackage{xr-hyper}
\usepackage[
bookmarks=true,
colorlinks=true,
citecolor=deepjunglegreen,
unicode,
pdfnewwindow=true]{hyperref}
\usepackage{cleveref}
\let\RR\relax
\let\bN\relax
\let\SL\relax
\let\cO\relax
[http://arxiv.org/pdf/1706.02112.pdf]
[http://arxiv.org/pdf/1604.03625.pdf]

\usepackage{xspace}

\AtBeginDocument{%
\makeatletter
\let\oldref=\ref
\renewcommand{\ref}[1]{%
  \def\@mystring{affine_pre-#1}%
  \@ifundefined{r@\@mystring}{%
    \def\@mystring{blowup_pre-#1}%
    \@ifundefined{r@\@mystring}{%
      \oldref{#1}}{%
      \cite[\namecref{blowup_pre-#1}\oldref{blowup_pre-#1}]{blowup}%
      \begin{NB} Check an external reference !\end{NB}%
    }%
  }%
  {\cite[\namecref{affine_pre-#1}\oldref{affine_pre-#1}]{affine}%
    \begin{NB} Check an external reference !\end{NB}%
  }%
}
\renewcommand{\eqref}[1]{%
  \def\@mystring{affine_pre-#1}%
  \@ifundefined{r@\@mystring}{%
    \def\@mystring{blowup_pre-#1}%
    \@ifundefined{r@\@mystring}{%
      \textup{\tagform@{\oldref{#1}}}}{%
      \text{\cite[(\oldref{blowup_pre-#1})]{blowup}}%
      \begin{NB} Check an external reference !\end{NB}%
    }%
  }%
  {\text{\cite[(\oldref{affine_pre-#1}]{affine}}%
  \begin{NB} Check an external reference !\end{NB}%
}%
}
\makeatother
}

\usepackage{stmaryrd}
\SetSymbolFont{stmry}{bold}{U}{stmry}{m}{n}
\usepackage{pigpen} 
\usepackage{enumitem}
\usepackage{calc}

\usepackage{textgreek}

\renewcommand{\thesubsection}{\thesection(\@roman\c@subsection)}
\newcounter{number}
\makeatother
\newtheorem{Theorem}[equation]{Theorem}
\newtheorem{Corollary}[equation]{Corollary}
\newtheorem{Lemma}[equation]{Lemma}
\newtheorem{Proposition}[equation]{Proposition}

\theoremstyle{definition}
\newtheorem{Definition}[equation]{Definition}
\newtheorem{Example}[equation]{Example}

\newtheorem*{Convention}{Convention}

\theoremstyle{remark}
\newtheorem{Remark}[equation]{Remark}
\newtheorem{Remarks}[equation]{Remarks}
\newtheorem*{Claim}{Claim}

\numberwithin{equation}{section}

\usepackage{cleveref}
\crefname{Theorem}{Theorem\xspace}{Theorems}
\crefname{section}{\S}{\S\S}
\crefformat{section}{\S#2#1#3} 
\crefmultiformat{section}{\S\S#2#1#3}{, #2#1#3}{, #2#1#3}{, #2#1#3}
\crefname{Lemma}{Lemma\xspace}{Lemmas\xspace}
\crefformat{Lemma}{Lemma~#2#1#3}
\crefname{Proposition}{Proposition\xspace}{Propositions\xspace}
\crefformat{Proposition}{Proposition~#2#1#3}
\crefname{Corollary}{Corollary\xspace}{Corollaries\xspace}
\crefformat{Corollary}{Corollary~#2#1#3}
\crefname{Definition}{Definition}{Definitions}
\crefformat{Definition}{Definition~#2#1#3}
\crefname{Remark}{Remark\xspace}{Remarks\xspace}
\crefformat{Remark}{Remark~#2#1#3}
\crefname{Remarks}{Remark\xspace}{Remarks\xspace}
\crefformat{Remarks}{Remark~#2#1#3}
\crefname{Conjecture}{Conjecture\xspace}{Conjectures\xspace}
\crefformat{Conjecture}{Conjecture~#2#1#3}
\crefname{figure}{Figure\xspace}{Figure\xspace}
\crefformat{figure}{Figure~#2#1#3}

\crefformat{appendix}{\S#2#1#3}
\crefformat{enumi}{#2#1#3}

\newcommand{\thmref}[1]{Theorem~\ref{#1}}
\newcommand{\secref}[1]{\S\ref{#1}}
\newcommand{\lemref}[1]{Lemma~\ref{#1}}
\newcommand{\propref}[1]{Proposition~\ref{#1}}
\newcommand{\corref}[1]{Corollary~\ref{#1}}
\newcommand{\subsecref}[1]{\S\ref{#1}}
\newcommand{\remref}[1]{Remark~\ref{#1}}

\newcommand{\defeq}{\overset{\operatorname{\scriptstyle def.}}{=}}
\newcommand{\CC}{{\mathbb C}}
\newcommand{\ZZ}{{\mathbb Z}}
\newcommand{\QQ}{{\mathbb Q}}
\newcommand{\RR}{{\mathbb R}}
\newcommand{\proj}{{\mathbb P}}
\newcommand{\CP}{\proj}

\newcommand{\SL}{\operatorname{\rm SL}}
\newcommand{\SU}{\operatorname{\rm SU}}
\newcommand{\GL}{\operatorname{GL}}
\newcommand{\PGL}{\operatorname{PGL}}
\providecommand{\U}{\operatorname{\rm U}}
\renewcommand{\U}{\operatorname{\rm U}}

\newcommand{\grpSp}{\operatorname{\rm Sp}}

\newcommand{\algsl}{\operatorname{\mathfrak{sl}}} 
\newcommand{\su}{\operatorname{\mathfrak{su}}}
\newcommand{\gl}{\operatorname{\mathfrak{gl}}}

\newcommand{\g}{{\mathfrak g}}

\newcommand{\Spec}{\operatorname{Spec}\nolimits}
\newcommand{\Proj}{\operatorname{Proj}\nolimits}
\newcommand{\End}{\operatorname{End}}
\newcommand{\Hom}{\operatorname{Hom}}

\newcommand{\Ker}{\operatorname{Ker}}
\newcommand{\Coker}{\operatorname{Coker}}
\newcommand{\Ima}{\operatorname{Im}}

\newcommand{\codim}{\mathop{\text{\rm codim}}\nolimits}
\newcommand{\rank}{\operatorname{rank}}

\newcommand{\tr}{\operatorname{tr}}

\newcommand{\id}{\operatorname{id}}

\renewcommand{\MR}[1]{}

\newcommand{\HH}{\mathbb H}
\newcommand{\dslash}{/\!\!/}
\newcommand{\bM}{\mathbf M}
\newcommand{\bN}{\mathbf N}

\newcommand{\shfO}{\mathcal O}

\newcommand{\tslabar}{\mathbin{
\setbox0=\hbox{/\!\!/\!\!/}\rule[0.4\ht0]{\wd0}{.3\dp0}\kern-\wd0\box0}}

\newcommand{\la}{\lambda}

\newcommand{\Gr}{\mathrm{Gr}}

\newcommand{\cR}{\mathcal R}
\newcommand{\cF}{\mathcal F}
\newcommand{\cT}{\mathcal T}
\newcommand{\cK}{\mathcal K}
\newcommand{\cO}{\mathcal O}

\newcommand{\scP}{\mathscr P}
\newcommand{\scA}{\mathscr A}

\newcommand{\St}{\mathrm{St}}

\newcommand{\Uh}[2][G]{{\mathcal U}_{#1}^{#2}}
\newcommand{\Bun}[2][G]{\operatorname{Bun}_{#1}^{#2}}

\newcommand{\br}{\mathbf r}
\newcommand{\loc}{{\ft^\circ}}

\makeatletter
\newcommand{\cA}[1][{}]{%
  \@ifmtarg{#1}%
  {\mathcal A}
  {\mathcal A(#1)}
}
\newcommand{\cAh}[1][{}]{%
  \@ifmtarg{#1}%
  {\mathcal A_\hbar}
  {\mathcal A_\hbar(#1)}
}
\makeatother

\newcommand{\leftmapsto}{\mapsfrom}
\newcommand{\DD}{\mathbb D}
\newcommand{\DC}{\boldsymbol\omega}
\newcommand{\Stab}{\operatorname{Stab}}
\newcommand{\Lie}{\operatorname{Lie}}
\newcommand{\ft}{\mathfrak t}
\newcommand{\gr}{\operatorname{gr}}

\newcommand{\bNT}{\bN_T}

\newcommand{\calR}{\cR}
\newcommand{\calO}{\cO}
\newcommand{\bfN}{\bN}
\newcommand{\Iw}{\tilde G_\cK^\cO}

\newcommand{\laF}{{\lambda_F}}

\newcommand{\cL}{\mathcal L}

\newcommand{\po}{\ar@{}[dr]|{\text{\pigpenfont R}}}
\newcommand{\pb}{\ar@{}[dr]|{\text{\pigpenfont J}}}
\newcommand{\pp}{\ar@{}[dr]|{\text{\pigpenfont P}}}

\newcommand{\intsys}{\varpi}

\newcommand{\bz}{\mathbf z}

\begin{document}

\title[Coulomb branches of $3d$ $N=4$ gauge theories, II] {Towards a
  mathematical definition of Coulomb branches of $3$-dimensional
  $\mathcal N=4$ gauge theories, II
}
\author[A.~Braverman]{Alexander Braverman}
\address{
Department of Mathematics,
University of Toronto and Perimeter Institute of Theoretical Physics,
Waterloo, Ontario, Canada, N2L 2Y5
}
\email{braval@math.toronto.edu}
\author[M.~Finkelberg]{Michael Finkelberg}
\address{National Research University Higher School of Economics,
Russian Federation,
Department of Mathematics, 6 Usacheva st, Moscow 119048;
Institute for Information Transmission Problems}
\email{fnklberg@gmail.com}
\author[H.~Nakajima]{Hiraku Nakajima}
\address{Research Institute for Mathematical Sciences,
Kyoto University, Kyoto 606-8502,
Japan}
\email{nakajima@kurims.kyoto-u.ac.jp}
\curraddr{Kavli Institute for the Physics and Mathematics of the Universe (WPI),
  The University of Tokyo,
  5-1-5 Kashiwanoha, Kashiwa, Chiba, 277-8583,
  Japan}
\email{hiraku.nakajima@ipmu.jp}

\subjclass[2000]{}
\begin{abstract}
    Consider the $3$-dimensional $\mathcal N=4$ supersymmetric gauge
    theory associated with a compact Lie group $G_c$ and its
    quaternionic representation $\bM$. Physicists study its Coulomb
    branch, which is a noncompact hyper-K\"ahler manifold with an
    $\SU(2)$-action, possibly with singularities.
    We give a mathematical definition of the Coulomb branch as an
    affine algebraic variety with $\CC^\times$-action when $\bM$ is of
    a form $\bN\oplus\bN^*$, as the second step of the proposal given
    in \cite{2015arXiv150303676N}.
%
\end{abstract}

\maketitle

\setcounter{tocdepth}{2}

\begin{NB}
\section*{Some macros}
\begin{verbatim}
\newcommand{\Gr}{\mathrm{Gr}}
\newcommand{\cR}{\mathcal R}
\newcommand{\cF}{\mathcal F}
\newcommand{\cT}{\mathcal T}
\newcommand{\cK}{\mathcal K}
\newcommand{\cO}{\mathcal O}

\newcommand{\scP}{\mathscr P}


\newcommand{\St}{\mathrm{St}}

\newcommand{\Uh}[2][G]{{\mathcal U}_{#1}^{#2}}
\newcommand{\Bun}[2][G]{\operatorname{Bun}_{#1}^{#2}}
\newcommand{\DD}{\mathbb D}
\newcommand{\DC}{\boldsymbol\omega}
\newcommand{\Stab}{\operatorname{Stab}}
\end{verbatim}

\verb+\cAh+ yields $\cAh$.

\verb+\cAh[G,\bN]+ yields $\cAh[G,\bN]$.

\verb+\cAh[G]+ yields $\cAh[G]$.

\verb+\cAh[]+ yields $\cAh[]$.

\end{NB}

\section{Introduction}

This paper is a sequel to \cite{2015arXiv150303676N}. It is written for mathematicians so that no knowledge on physics is required to read. It can be read even independently of \cite{2015arXiv150303676N} except the next background subsection. We by no means intend to ignore past research in physics. We try our best to give appropriate physics references either here or in \cite{2015arXiv150303676N} so that the reader could see how we understand them.
%
\begin{NB}
A reader in a hurry can safely ignore the motivation and skip the next
subsection.
\end{NB}%
There are two companion papers \cite{blowup,affine} where specific
topics arising from this paper are discussed.

\subsection{Background}

In \cite{2015arXiv150303676N} the third named author proposed an
approach to define the Coulomb branch $\mathcal M_C$ of a
$3$-dimensional $\mathcal N=4$ SUSY gauge theory in a mathematically
rigorous way.
The Coulomb branch $\mathcal M_C$ is a hyper-K\"ahler manifold with an
$\SU(2)$-action possibly with singularities. It has been studied by
physicists intensively over the years, nevertheless lacks a firm
mathematical footing as the physical definition involves quantum
corrections.

A key idea in \cite{2015arXiv150303676N} was to use a certain moduli
stack, motivated by reduction of the generalized Seiberg-Witten type
equation to $S^2 = \CP^1$, associated with a compact Lie group $G_c$ and
its representation $\bM$ over the quaternions $\mathbb H$,
corresponding to a given SUSY gauge theory.
Then the space of holomorphic functions on $\mathcal M_C$ is proposed
as the dual of the critical cohomology with compact support of the
moduli stack, associated with an analog of the complex Chern-Simons
functional.

The goal of this paper is to define a commutative
\emph{multiplication} on the dual cohomology group, under the
assumption that $\bM$ is of cotangent type, i.e., $\bM = \bN\oplus
\bN^*$ for a complex representation $\bN$ of $G$. 
The commutative ring of functions determines $\mathcal M_C$ as an
affine algebraic variety as its spectrum. Therefore the remaining step
is to find a hyper-K\"ahler metric, or equivalently the twistor space.

Under the cotangent type assumption, it was heuristically shown
(\cite{2015arXiv150303676N}) that the critical cohomology group can be
replaced by the ordinary cohomology group with compact support for a
\emph{smaller} moduli stack $\cR$ of pairs of holomorphic principal
$G$-bundles $\scP$ with holomorphic sections of the associated vector
bundle $\scP\times_G\bN$ over $\CP^1$. Here $G$ is a complexification of
$G_c$.
In fact, we further change the moduli stack of pairs over $\CP^1$ by a
stack modeled after the affine Grassmannian. In other words, we
consider the moduli stack of pairs over the non-separated scheme
$\tilde D$ obtained by gluing two copies of the formal disk
$D=\Spec\CC[[z]]$, along the punctured disk $D^*$.
This allows us to apply various techniques used in study of affine
Grassmannian.
(See below for some detail, and the main body for full detail.)
The graded dimension of the cohomology group does not change as we
will see later, and we conjecture that two cohomology groups are naturally isomorphic.

The original intuition for the multiplication in
\cite{2015arXiv150303676N} was a conjectural $3d$ topological quantum
field theory associated with the gauge theory. Namely the dual
cohomology group is regarded as the quantum Hilbert space $\mathcal
H_{S^2}$ associated with $S^2$, and the multiplication is given by the
vector corresponding to the $3$-ball with two balls removed from the
interior.
Intuitively our definition of the product uses the same picture, but
one of three boundary components is very small. Let us view it as
$\Sigma\times [0,1]$ with small $3$-ball $B^3$ removed from the
interior, where $\Sigma = S^2$.  Then it can be regarded as a movie of
$\Sigma$ on which a small $2$-ball appears and disappears, where
$[0,1]$ is the time direction. (See Figure~\ref{meteorite}.)
We have a multiplication $\mathcal H_{\Sigma}\otimes \mathcal
H_{\partial B^3}\to \mathcal H_{\Sigma}$ from the $3d$ point of
view. On the other hand, from a $2d$ observer on $\Sigma$, the
intersection $B^3\cap (\Sigma\times t)$ appears in a small
neighborhood of a point. Therefore the observer sees something happening
on the formal disk $D$. From this point of view, $\mathcal H_{\Sigma}$
is a module of a ring $\mathcal H_{\partial B^3}$, where the
multiplication on $\mathcal H_{\partial B^3}$ is given by considering
two formal disks. This $2d$ movie is nothing but the convolution
diagram for the affine Grassmannian, or more precisely
Beilinson-Drinfeld Grassmannian.

\begin{figure}[htbp]
\centering
\def\svgscale{.5}
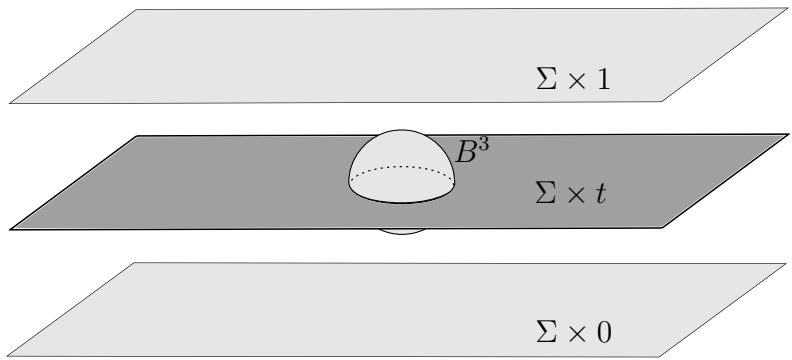
\caption{3$d$ picture vs 2$d$ movie.}
\label{meteorite}
\end{figure}

We do not expect our technique to be applied to more general
$3$-manifolds with boundaries, but we think that it is a good starting
point nevertheless.
\begin{NB}
    I am not sure that we can extend our result to general Riemann
    surface $\Sigma$. Therefore I delete the previous paragraph.
\end{NB}%

Let us also mention that the relevance of the affine Grassmannian (and
\cite{MR2135527}, see below) in $3d$ SUSY gauge theories with $\bN=0$ was pointed
earlier by Teleman \cite{2014arXiv1404.6305T}.

\subsection{Construction}

We believe that our definition of the affine algebraic variety
$\mathcal M_C$ is interesting in its own, even ignoring physical
motivation. For mathematicians who do not know physics, it can be
regarded as new construction of a class of algebraic varieties with
interesting structures.
\begin{NB}
    See \subsecref{subsec:new_properties} below for review of properties.
\end{NB}%
The class contains moduli spaces of based maps from $\CP^1$ to a flag
variety (also conjecturally instantons on multi Taub-NUT spaces), and
spaces studied in the context of geometric representation theory, such
as slices in affine Grassmannian. But we expect $\mathcal M_C$ is an
unknown space in most cases.

Let us give a little more detail on the definition. Let $G$ be a
complex reductive group, and $\bN$ be its representation.
We consider the moduli space $\cR$ of triples $(\scP,\varphi,s)$,
where $\scP$ is a $G$-bundle on the formal disk $D = \Spec \CC[[z]]$,
$\varphi$ is its trivialization over the punctured disk $D^* = \Spec
\CC((z))$ and $s$ is a section of the associated vector bundle
$\scP\times_G\bN$ such that it is sent to a regular section of a
trivial bundle under $\varphi$. We have a natural action of $G_\cO =
G[[z]]$, the group of $\cO$-valued points of $G$ by changing
$\varphi$. If we ignore $s$, we get the moduli space of pairs
$(\scP,\varphi)$, which is nothing but the affine Grassmannian
$\Gr_G$.
When $\bN = \g$, the adjoint representation, $\cR$ is 
the affine Grassmannian Steinberg variety.
For general $\bN$, we have a projection $\cR\to\Gr_G$ whose fibers are
infinite dimensional vector spaces.

It is a direct limit of inverse limits of schemes of finite type, but
its $G_\cO$-equivariant Borel-Moore homology group $H^{G_\cO}_*(\cR)$
is well-defined, as we will explain in the main text.

\begin{NB}
Geometric Satake correspondence says that the category of
$G_\cO$-equivariant perverse sheaves on $\Gr_G$ has a monoidal
structure so that it is equivalent to the tensor category of finite
dimensional representations of the Langlands dual group of $G$. This
is due to Lusztig, Ginzburg, Drinfeld, Mirkovic-Vilonen,
Beilinson-Drinfeld (see \cite{MV2} and the references therein).
The monoidal structure is given by a convolution diagram, which is
roughly regarded as a correspondence between products of $\Gr_G$ given
by modifications of $G$-bundles.
\end{NB}%

We then introduce a convolution diagram for $\cR$ as an analog
of one for $\Gr_G$, used in \cite{MV2} for geometric Satake correspondence.
It defines a convolution product $\ast$ on $H^{G_\cO}_*(\cR)$. This construction has originally appeared in \cite{MR2135527} for $\Gr_G$ itself (i.e.,
$\bN=0$) and $\bN=\g$. There is a closely related earlier work
\cite{MR2115259} for $\bN=\g$, and also \cite{MR3013034}.
We thus get a graded commutative ring $\cA$, and define the Coulomb
branch as its spectrum $\mathcal M_C = \Spec\cA$.

We can thus regard these earlier results as computation of examples of
$\mathcal M_C$: For $\bN=0$, $\mathcal M_C$ is the the algebraic variety ${\mathfrak Z}^{G^\vee}_{{\mathfrak g}^\vee}$ formed by the pairs $(g,x)$ such that $x$ lies in a (fixed) Kostant slice in ${\mathfrak g}^\vee$, and
$g\in G^\vee$ satisfies $\operatorname{Ad}_g(x)=x$.
\begin{NB}
    the symplectic reduction of the cotangent bundle $T^* G^\vee$ of
    the Langlands dual group $G^\vee$ with respect to a nondegenerate
    character of $N^\vee\times N^\vee$, where $N^\vee\subset G^\vee$
    is a unipotent subgroup.
\end{NB}%
See \subsecref{subsec:prev}\ref{subsub:pure} for more
detail.

For $\bN=\g$, the adjoint representation, $\mathcal M_C$ is
$(\ft\times T^\vee)/W$, where $\ft$ is the Lie algebra of a maximal
torus $T$, $T^\vee$ is the dual torus of $T$, and $W$ is the Weyl
group. See \subsecref{subsec:prev}\ref{subsub:adjoint} for more detail.

\begin{NB}
As a byproduct of our construction we get a $G_\cO$-equivariant
constructible complex $\scA$ on $\Gr_G$ defined by
$\pi_*\DC_{\cR}[-2\dim\bN_\cO]$, where $\DC_{\cR}$ is the dualizing
complex on $\cR$ and $\pi\colon \cR\to\Gr_G$ is the projection
mentioned above. See \ref{sec:sheav-affine-grassm}. We can recover
$H^{G_\cO}_*(\cR)$ as $H^*_{G_\cO}(\Gr_G,\scA)$. Moreover the
construction of the convolution product gives us a homomorphism
$\mathsf m\colon\scA\star\scA\to \scA$, where $\star$ is the
convolution product on $D_{G_\cO}(\Gr_G)$, the $G_\cO$-equivariant
derived category on $\Gr_G$. It is an associative multiplication on
$\scA$. We could not prove that it is commutative, but check a weaker
statement instead. Therefore the convolution product on
$H^{G_\cO}_*(\cR)$ can be also recovered from $\mathsf m$.

This $(\scA,\mathsf m)$ reminds us the perverse sheaf
$\cR_{\mathrm{ABG}}$ studied in \cite{MR2053952}. It corresponds to
the regular representation $\CC[G^\vee]$ of the Langlands dual group
$G^\vee$ under the geometric Satake correspondence, and hence is a
commutative ring object in $\operatorname{Perv}_{G_\cO}(\Gr_G)$. It
was used to construct the cotangent bundle of the flag variety for
$G$ as its cohomology.
We do not expect $\cR_{\mathrm{ABG}}$ arises as $(\scA,\mathsf m)$
except in type $ABCD$. It means we have more examples of commutative
ring objects in $D_{G_\cO}(\Gr_G)$ than our construction. It is
interesting to look for other constructions. We will propose a gluing
construction in \ref{subsec:glue} as an example.
We understand that this gluing construction has appeared in the
construction of Higgs branches of $3d$ Sicilian theories associated
with punctured Riemann surfaces by Ginzburg-Kazhdan \cite{GK} via
\cite{MR3366026}. See the last paragraph in \ref{subsec:glue}.
\end{NB}%

\subsection{Quantization of \texorpdfstring{$\mathcal M_C$}{M\textunderscore c}}\label{subsec:intro_quantum}

As a byproduct of our definition, we obtain a noncommutative
deformation (\emph{quantized Coulomb branch}) $\cAh$, flat over
$\CC[\hbar]$. This is defined as the equivariant homology group
$H^{G_\cO\rtimes\CC^\times}_*(\cR)$, where $\CC^\times$ is the group
of loop rotations, as in earlier works
\cite{MR2115259,MR3013034,MR2422266}.
Therefore the classical limit $\cA$ has a Poisson bracket. 
We show that it is symplectic on the regular locus of $\mathcal M_C$
(\propref{prop:symplectic}).
When $\bN=0$, $\cAh$ is a quantum hamiltonian reduction of the ring of
differential operators on $G^\vee$ (\cite[Th.~3]{MR2422266}). When
$\bN=\g$, $\cAh$ is expected to be the spherical subalgebra of the
graded Cherednik algebra, based on \cite{MR2115259,MR3013034}. (See
\subsecref{subsec:prev}\ref{subsub:adjoint} for a precise statement.)

The quantized Coulomb branch $\cAh$ contains
$H^*_{G\times\CC^\times}(\mathrm{pt})1$ as a commutative subalgebra
(\emph{Cartan subalgebra}). Correspondingly we have a morphism
$\intsys\colon\mathcal M_C\to \ft/W\cong \CC^\ell$ such that all
functions factoring through $\intsys$ are Poisson commuting. A generic
fiber of $\intsys$ is $T^\vee$
(Propositions~\ref{prop:comm},~\ref{prop:integrable}).
Thus $\intsys$ is an \emph{integrable system} in the sense of
Liouville.
In examples $\bN=0$ and $\g$, the morphism $\intsys$ is an obvious
projection to $\ft/W$.
\begin{NB}
    For $\bN=0$, we get the Toda system if we introduce a birational
    isomorphism $\mathcal M_C\approx T^*T^\vee$. But we do not know
    how to interpret $\mathcal M_C\approx T^*T^\vee$ in our
    framework. In particular, we do not have a recipe to give analog
    of $\mathcal M_C\approx T^*T^\vee$ for general $\bN$.

    For $\bN=\g$, we expect to get the Calogero-Moser system if we
    deform $\mathcal M_C$ by the flavor symmetry group $G_F=\CC^\times$.
\end{NB}%

\begin{Remarks}
    (1) A physical explanation of the `quantization'
    is given in
    \cite[\S3.4]{2015arXiv150304817B}. It is coming from the
    $\Omega$-background, i.e., the $\CC^\times$-action on
    $\RR^3$. This should have the same origin with our definition.

    (2) 
    In $4$-dimension, it is known that Coulomb branches are closely
    related to the Hitchin system (see e.g., \cite{MR1492533}, a
    review for mathematicians). But the integrable system
    above 
    is not a simple limit of the $4$-dimensional integrable system
    even for $(G,\bN) = (\SL(2),0)$: $\mathcal M_C$ is the
    Atiyah-Hitchin space, which is the complement of the infinity
    section in $4d$ Coulomb branch for $\RR^3\times S^1$, the total
    space of Seiberg-Witten curves $y^2 = x^3 - x^2 u + x$
    (\cite[(3.18)]{MR1490862}), but the hamiltonian is $u$ in $4d$
    while it is $v \defeq x-u$ in $3d$.
    Anyhow the integrable system is coming from monopole operators
    corresponding to the cocharacter $\la = 0$. And it is proposed
    that they form a Poisson commuting subalgebra in \cite[\S3.2,
    \S4.2]{2015arXiv150304817B}.
\end{Remarks}

\begin{NB}
\subsection{Properties of \texorpdfstring{$\mathcal M_C$}{M\textunderscore c}}\label{subsec:new_properties}

Various interesting properties of the Coulomb branch $\mathcal M_C$
have been discussed in the physics literature, and they are realized
in our definition. Let us not reproduce those which have been already
explained in \cite{2015arXiv150303676N}, and mention only properties
which are not appeared there.

\begin{enumerate}[leftmargin=0cm,
itemindent=2pt,
labelwidth=-\parindent,
itemsep=2pt,
label=(\alph*),
align=left
]

  \item\label{item:pure} \cite{MR2135527} For $\bN=0$, $\mathcal M_C$
is the symplectic reduction of the cotangent bundle $T^* G^\vee$ of
the Langlands dual group $G^\vee$ with respect to a nondegenerate
character of $N^\vee\times N^\vee$, where $N^\vee\subset G^\vee$ is a
unipotent subgroup. It can be also regarded as the moduli space of
solutions of Nahm's equation associated with the Langlands dual group
$G_c^\vee$ of a maximal compact subgroup $G_c$ of $G$. (See
\secref{sec:monop-moduli-spac}.)

  \item \cite{MR2115259,MR2135527,MR3013034} For $\bN=\g$, the adjoint
representation, $\mathcal M_C$ is $(\ft\times T^\vee)/W$, where $\ft$
is the Lie algebra of a maximal torus $T$, $T^\vee$ is the dual torus
of $T$, and $W$ is the Weyl group.

\begin{NB2}
  \item \cite[\S8]{2015arXiv150303676N} There is a combinatorial
formula (\emph{monopole formula}) of the graded dimension of $\cA$, in
other words the $\CC^\times$ character of $\cA = \CC[\mathcal
M_C]$. (See \propref{prop:monopole_formula} for the reproduced proof).
\end{NB2}%

  \item \label{item:fg} $\cA$ is integral, finitely generated and
normal (\corref{cor:integral},
Propositions~\ref{prop:fg},~\ref{prop:normality}).

  \item\label{item:quantization} $\cA$ has a noncommutative
deformation (\emph{quantized Coulomb branch}) $\cAh$, flat over
$\CC[\hbar]$. Therefore $\cA$ has a Poisson bracket. When $\bN=0$,
$\cAh$ is a quantum hamiltonian reduction of the ring of differential
operators on $G^\vee$ \cite{MR2422266}. When $\bN=\g$, $\cAh$ is
expected to be the spherical subalgebra of the graded Cherednik
algebra, based on \cite{MR2115259,MR3013034}. (See
\subsecref{subsec:prev}\ref{subsub:adjoint}.)
    
\begin{NB2}
  \item\label{item:action} \cite[\S4(iv)(c)]{2015arXiv150303676N} We
have an action of $\pi_1(G)^\vee$, the Pontryagin dual of $\pi_1(G)$,
on $\mathcal M_C$. (See \secref{sec:grading}).
\end{NB2}%

  \item\label{item:integrable_system} The quantized Coulomb branch
$\cAh$ contains $H^*_{G\times\CC^\times}(\mathrm{pt})1$ as a
commutative subalgebra (\emph{Cartan subalgebra}). Correspondingly we
have a morphism $\intsys\colon\mathcal M_C\to \ft/W\cong \CC^\ell$
such that all functions factoring through $\intsys$ are Poisson
commuting. A generic fiber of $\intsys$ is $T^\vee$
(Propositions~\ref{prop:comm},~\ref{prop:integrable}).

\begin{NB2}
      \item \label{item:flavor} \cite[\S5]{2015arXiv150303676N}
    Suppose $\bN$ is a representation of a larger group $\tilde G$
    containing $G$ as a normal subgroup. Then $\mathcal M_C$ has a
    deformation parametrized by $\Spec H^*_{G_F}(\mathrm{pt})$ where
    $G_F = \tilde G/G$.  (\propref{prop:deformation}).

      \item \cite[\S5]{2015arXiv150303676N} Let $T_F$ be a maximal
    torus of $G_F$ above. For a cocharacter $\la_F\colon \CC^\times\to
    T_F$, one can construct a $\ZZ_{\ge 0}$-graded algebra such that
    the coordinate ring of the above deformation of $\mathcal M_C$ is
    its degree $0$ piece.  It is expected to be similar to the
    Grothendieck-Springer resolution of $\g^*$. (See
    \subsecref{subsec:flav-symm-group2}.)

      \item\label{item:Hamiltonianreduction} Suppose $G_F = \tilde G/G$ is a
    torus in the situation of \ref{item:flavor}. Then $\mathcal
    M_C(G,\bN)$ is the Hamiltonian reduction of $\mathcal M_C(\tilde
    G,\bN)$ by the action of $G_F^\vee$ given by \ref{item:action}
    (\propref{prop:reduction}). Here we write groups and
    representations to distinguish two Coulomb branches.

      \item \cite[\S5(ii)]{2015arXiv150303676N} When $G$ is a torus,
    $\cA$ has a linear basis with explicit structure constants. (See
    \thmref{abel}.) When $\bN$ is coming from a short exact sequence
    $1\to G=T\to \tilde T \to T_F \to 1$ of tori, $\mathcal M_C$ is a
    toric hyper-K\"ahler manifold \cite{MR1792372}
    (\subsecref{sec:toric-hyper-kahler}).
\end{NB2}%

      \item\label{item:abelianization} Let $T^\vee$ be the dual torus
    of $T$ and $\bNT = \bN$ regarded as a $T$-module. We have dominant
    birational morphisms
\begin{equation*}
    \qquad
    \ft\times T^\vee/W \to \mathcal M_C(T,\bNT)/W
    \leftarrow \mathcal M_C(G,\bN), \qquad
    	\mathcal M_C(G,\bN) \leftarrow \mathcal M_C(G,0).
\end{equation*}
(See \secref{sec:abel}.) Here we write groups and representations to
distinguish two Coulomb branches.

  \item \label{item:recipe} \cite[\S4]{MR2135527} There is a recipe to
identify $\mathcal M_C$ with a known space: Suppose we have
$\varPi\colon\mathcal M \to \ft/W$ such that $\mathcal M$ is normal
and all the fibers of $\varPi$ have the same dimension. If $(\mathcal
M,\varPi)$ and $(\mathcal M_C,\intsys)$ coincide up to codimension
$2$, they are isomorphic. (See \thmref{prop:flat} for the precise
statement.)

  \item \label{item:degeneration} There is a filtration of $\cA$
induced by the stratification of $\Gr_G$ by $G_\cO$-orbits. The
associated graded ring $\gr\cA$ has a linear basis with explicit
structure constants. It means that $\mathcal M_C$ has a degeneration
to a variety with combinatorial flavor. (See
\secref{sec:degeneration}.)

  \item\label{item:PGL2} When $G=\PGL(2)$, $\mathcal M_C$ is a
hypersurface $\xi^2 - \delta\eta^2 = \delta^m$ in
$\CC^3$ for some $m\in\ZZ_{\ge 0}$. (See \lemref{lem:SL2}(2).)

  \item\label{item:monopole} When $(G,\bN)$ is coming from a quiver of
type $ADE$, $\mathcal M_C$ is the moduli space of based maps from
$\CP^1$ to the flag variety of the corresponding type $ADE$
(\ref{pestun}). More generally, if $\bN$ is coming from a framed
quiver of type $ADE$, $\mathcal M_C$ is a slice in the affine
Grassmannian of the corresponding type $ADE$ under a dominance
condition, and its generalization in general. (See \ref{QGT}.)
\end{enumerate}

Some are known in physics literature:

\begin{Remarks}
\begin{enumerate}[leftmargin=0cm,
itemindent=2pt,
labelwidth=-\parindent,
itemsep=2pt,
label=(\arabic*),
align=left
]

  \item \ref{item:pure} was known, at least for type $A$, in the
physics literature. See \cite{MR1490862} for $k=2$, \cite{MR1443803}
for arbitrary $k$.

  \item A physical explanation of the `quantization' in
\ref{item:quantization} is given in
\cite[\S3.4]{2015arXiv150304817B}. It is coming from the
$\Omega$-background, i.e., the $\CC^\times$-action on $\RR^3$. This
should have the same origin with our definition.

  \item 
In $4$-dimension, it is known that Coulomb branches are closely
related to the Hitchin system (see e.g., \cite{MR1492533}, a review
for mathematicians). But \ref{item:integrable_system} is not a simple
limit of the $4$-dimensional integrable system even for $(G,\bN) =
(\SL(2),0)$: $\mathcal M_C$ is the Atiyah-Hitchin space, which is the
complement of the infinity section in $4d$ Coulomb branch for
$\RR^3\times S^1$, the total space of Seiberg-Witten curves $y^2 = x^3
- x^2 u + x$ (\cite[(3.18)]{MR1490862}), but the hamiltonian is $u$ in
$4d$ while it is $v \defeq x-u$ in $3d$. Therefore it is not clear for
us whether our integrable system has been known in the physics
literature or not.
Anyway it is coming from monopole operators corresponding to the
cocharacter $\la = 0$. And it is proposed that they form a Poisson
commuting subalgebra in \cite[\S3.2, \S4.2]{2015arXiv150304817B}.

\begin{NB2}
  \item The property \ref{item:Hamiltonianreduction} is naturally
predicted from the monopole formula \eqref{eq:19}, as was observed in
\cite[\S5.1]{Cremonesi:2013lqa}.
\end{NB2}%

  \item In physics literature (see e.g., \cite[\S2]{MR1490862}), it is
usually said that the Coulomb branch is \emph{classically}
$(\RR^3\times S^1)^{\rank G}/W$. This is compatible with
\ref{item:abelianization}, as $\ft\times T^\vee$ is diffeomorphic to
$(\RR^3\times S^1)^{\rank G}$. 
In \cite[\S4]{2015arXiv150304817B} a relation between $\mathcal
M_C(T,\bNT)$ and $\mathcal M_C(G,\bN)$ is proposed. See
\remref{rem:BDG1} for more detail.

  \item A filtration \ref{item:degeneration} is not explicitly
mentioned in the physics literature, but a closely related claim is given 
 in \cite[\S4]{2015arXiv150304817B}. See \remref{rem:BDG2} for detail.

   \item \ref{item:monopole} are known in physics literature. More
 precise references will be given in \ref{blowup-intro}.

\begin{NB2}
The following should be moved to \cite{blowup}.

Physicists claim that $\mathcal M_C$ is the moduli space of
monopoles on $\RR^3$ in the first case of \ref{item:monopole}. (See
\cite{MR1451054} for type $A$, \cite{MR1677752} in general.) It is
likely that a known bijection between two moduli spaces (given in
\cite{MR709461,MR769355} for $A_1$, \cite{MR994495,MR987771} for
classical groups and \cite{MR1625475} for general groups) is an
isomorphism of affine algebraic varieties, but it is not clear to
authors whether the proofs give this stronger statement. For $A_1$,
one can check it by using \cite{MR1215288}, as it is easy to check
that the bijection between the moduli space of solutions of Nahm's
equation and the moduli space of based maps is an isomorphism.

In the second case of \ref{item:monopole}, physicists say that
$\mathcal M_C$ is the moduli space of \emph{singular} monopoles on
$\RR^3$ (see \cite{MR1636383}). It is natural to conjecture that this
moduli space and ours are isomorphic as affine algebraic varieties,
but it is not known even at the level of a bijection.
\end{NB2}%
\end{enumerate}
\end{Remarks}
\end{NB}

\subsection{The organization of the paper}
\secref{sec:triples} is devoted to the definition of $\cR$ and its
equivariant Borel-Moore homology group.
We give the definition of the convolution product in
\secref{sec:definition}. Some basic properties of $\cA$, except the
commutativity of the multiplication, are established also in
\secref{sec:definition}.
In \secref{sec:abelian} we determine $\cA$ when $G$ is a torus. We give a linear basis with explicit structure constants.
\secref{sec:abel} is a technical heart of the paper. We analyze
$\mathcal M_C$ using the localization theorem in equivariant homology
groups. The compatibility of convolution products with the
localization is studied. One of proofs of the commutativity of the multiplication is given. (Another proof will be given in \cite{affine}.)
\begin{NB}
The recipe~\ref{item:recipe} is given here.
\end{NB}%
We also give a recipe to identify $\mathcal M_C$ with a known space,
following \cite{MR2135527}: Suppose we have $\varPi\colon\mathcal M
\to \ft/W$ such that $\mathcal M$ is normal and all the fibers of
$\varPi$ have the same dimension. If $(\mathcal M,\varPi)$ and
$(\mathcal M_C,\intsys)$ coincide up to codimension $2$, they are
isomorphic. (See \thmref{prop:flat} for the precise statement.)
\begin{NB}
    \secref{sec:integrable-systems} is expository in nature, and will
    not appear in the final text.
\end{NB}%
In~\secref{sec:degeneration} we introduce a degeneration
\begin{NB}
 \ref{item:degeneration}   
\end{NB}%
of $\mathcal M_C$ to a variety with combinatorial flavor,
and prove that $\cA$ is finitely generated and normal
\begin{NB}
    \ref{item:fg}
\end{NB}%
as applications.
\begin{NB}
we introduce a filtration of $\cA$ whose
associated graded $\gr\cA$ has a linear basis with explicit structure
constants. Finite generation of $\cA$ and the normality of $\mathcal
M_C$ are proved as applications.
\end{NB}%
We also determine $\mathcal M_C$ when $G$ is $\PGL(2)$ or $\SL(2)$.
\begin{NB}
In \ref{sec:sheav-affine-grassm} we study constructible complex $\scA$ on $\Gr_G$ mentioned above.
\end{NB}%
\begin{NB}
\ref{QGT} is devoted to computation of $\mathcal M_C$ when
$(G,\bN)$ is coming from a (framed) quiver of type $ADE$ by using the
recipe~\ref{item:recipe}.
In \ref{sec:twist} we consider a folding of a quiver. This allows
us to generalize the result of the previous section to type $BCFG$.
\end{NB}%

There is one
\begin{NB}
    the number could be changed.
\end{NB}%
appendix.
\begin{NB}
In \ref{sec:commute}, we give the second proof of the commutativity
of the multiplication using Beilinson-Drinfeld Grassmannian.
\end{NB}%
\begin{NB}
In \ref{sec:associativeBD} we study the associativity of the
multiplication defined using Beilinson-Drinfeld Grassmannian.
\end{NB}%
In \secref{sec:monop-moduli-spac} we compare the answer for $\bN=0$ in
\cite{MR2422266} and one in the physics literature for type $A$. Both
are understood in a uniform way that $\mathcal M_C$ is the moduli
space of solutions of Nahm's equations for the Langlands dual group
$G_c^\vee$.

\subsection{Companion papers}

In \cite{blowup} we study Coulomb branches of quiver gauge theories.
When $(G,\bN)$ is coming from a quiver of type $ADE$, $\mathcal M_C$
is the moduli space of based maps from $\CP^1$ to the flag variety of
the corresponding type $ADE$ (\ref{pestun}).
More generally, if $\bN$ is coming from a framed quiver of type $ADE$,
$\mathcal M_C$ is a slice in the affine Grassmannian of the
corresponding type $ADE$ under a dominance condition, and its
generalization in general. (See \ref{QGT}.)

In \cite{affine} we study the following object:
As a byproduct of our construction we get a $G_\cO$-equivariant
constructible complex $\scA$ on $\Gr_G$ defined by
$\pi_*\DC_{\cR}[-2\dim\bN_\cO]$, where $\DC_{\cR}$ is the dualizing
complex on $\cR$ and $\pi\colon \cR\to\Gr_G$ is the projection.
We can recover $H^{G_\cO}_*(\cR)$ as $H^*_{G_\cO}(\Gr_G,\scA)$. Moreover the construction of the convolution product gives us a homomorphism
$\mathsf m\colon\scA\star\scA\to \scA$, where $\star$ is the
convolution product on $D_{G_\cO}(\Gr_G)$, the $G_\cO$-equivariant
derived category on $\Gr_G$. Therefore $(\scA,\mathsf m)$ is a ring object in $D_{G_\cO}(\Gr_G)$.


\subsection*{Notation}

\begin{enumerate}[leftmargin=0pt,
itemindent=2pt,
labelwidth=-\parindent,
itemsep=2pt,
align=left
]
      \item We basically follow the notation in Part I
    \cite{2015arXiv150303676N}. However we mainly use a complex
    reductive group instead of its maximal compact subgroup. Therefore
    we denote a reductive group by $G$, and its maximal compact by
    $G_c$. On the other hand, we use the notation $\cR$ for the
    variety of triples, though it was used for the corresponding space
    associated with $\proj^1$ in Part~I.

      \item Let us choose and fix a maximal torus $T$ of $G$. Let $W$
    be the Weyl group of $G$. Let $Y$ denote the coweight lattice of $G$.
    The Lie algebra of $G$ (resp.\ $T$) is denoted by $\g$ (resp.\ $\ft$).

      \item The constant sheaf on a space $X$ is denoted by
    $\CC_X$. The dualizing complex is denoted by $\DC_X$. The Verdier
    duality is denoted by $\DD$. (We take $\CC$ as the base ring.) 

      \item We will not use the usual homology group, and denote the
    Borel-Moore homology group with complex coefficients by
    $H_i(X)$. It is $H^{-i}(\DC_X)$, and the dual of cohomology group
    with compact support. When a group $G$ acts on $X$, the
    equivariant Borel-Moore homology group is denoted by $H^G_i(X)$.

    \begin{NB}
          \item The following should be fixed throughout the
        paper:
        \begin{equation*}
            \begin{split}
                & i\colon \cR\to \cT, \\
                & \pi\colon \cT\to\Gr_G, \\
                & z\colon \Gr_G\to \cT, \\
                & \Pi\colon \cT \to \bN_\cK,\\
                & \intsys\colon \mathcal M_C \to \Spec(H^*_G(\mathrm{pt})).
            \end{split}
        \end{equation*}
    \end{NB}
\end{enumerate}

\subsection*{Acknowledgments}

We thank
S.~Arkhipov,
R.~Bezrukavnikov,
T.~Braden,
M.~Bullimore,
S.~Cherkis,
T.~Dimofte,
P.~Etingof,
B.~Feigin,
D.~Gaiotto,
D.~Gaitsgory,
V.~Ginzburg,
A.~Hanany,
K.~Hori,
J.~Kamnitzer,
R.~Kodera,
A.~Kuznetsov,
A.~Oblomkov,
V.~Pestun,
N.~Proudfoot,
L.~Rybnikov,
Y.~Tachikawa,
C.~Teleman,
M.~Temkin
for the useful discussions.
We also thank J.~Hilburn and B.~Webster for pointing out mistakes in
an earlier version of the proof of \thmref{abel}, and the formula in
\eqref{q-relation} respectively.
This work was started in the Fall 2014, while H.N.\ was
visiting the Columbia University, and he wishes to thank its warm
hospitality.
H.N.\ also thanks
the Simons Center for Geometry and Physics,
the Institut Mittag-Leffler,
the Isaac Newton Institute for Mathematical Sciences, and
the Higher School of Economics, Independent University of
Moscow, where he continued the study.
All of us thank the Park City Math Institute,
where parts of this work were done.

A.B.\  was partially supported by the NSF grant DMS-1501047.
The research of M.F.\ was supported by the grant 
RSF 19-11-00056.
%
The research of H.N.\ was supported by in part by the World Premier
International Research Center Initiative (WPI Initiative), MEXT,
Japan, and JSPS Kakenhi Grant Numbers
23224002, 
23340005, 
24224001, 
25220701, 
16H06335. 

\section{A variety of triples and its equivariant Borel-Moore
homology group}\label{sec:triples}

Our construction of a space $\cR$, a direct limit of inverse limits
of schemes of finite type, and its equivariant Borel-Moore homology
group are based on similar well-known results for the affine
Grassmannian and the affine Grassmannian Steinberg variety. See
\cite[\S4.5]{Beilinson-Drinfeld}, \cite[\S7]{MR2135527} and the
references therein for the detail and proofs.

\subsection{A variety of triples}
\label{triples}

Let $G$ be a complex connected reductive group. 
The connectedness assumption is not essential. See
Remark~\ref{discrepancy}(3) below.
Let $\cO$ denote the formal power series ring $\CC[[z]]$ and $\cK$ its
fraction field $\CC((z))$.
Let $\Gr_G = G_\cK/G_\cO$ be the affine Grassmannian, where $G_\cK$,
$G_\cO$ are groups of $\cK$ and $\cO$-valued points of $G$. It is an
ind-scheme representing the functor from affine schemes to sets:
\begin{equation*}
    S \mapsto \{ (\scP,\varphi) \mid
    \text{a $G$-bundle $\scP$ on $D\times S$,
      a trivialization $\varphi\colon \scP|_{D^*\times S}
      \to G\times D^*\times S$}\},
\end{equation*}
where $D = \Spec(\cO)$ (resp.\ $D^* = \Spec(\cK)$) is the formal disk
(resp.\ punctured disk). We simply say $\Gr_G$ is the moduli space of
pairs $(\scP,\varphi)$ of a $G$-bundle $\scP$ on $D$ and its
trivialization $\varphi$ over $D^*$. The same applies to the variety
of triples, introduced below.

The set $\pi_0(\Gr_G)$ of connected components of $\Gr_G$ is known to
be in bijection to the fundamental group $\pi_1(G)$ of $G$. It is a
classical result that $\pi_1(G)$ is isomorphic to the quotient of the
coweight lattice by the coroot lattice.
\begin{NB}
    So $\pi_0(\Gr_G)\cong\ZZ$ if $G = \GL(r)$.
\end{NB}

Let $\bN$ be a (complex) representation of $G$. We consider a {\it
  variety of triples\/} $\cR\equiv\cR_{G,\bN}$, the moduli space
parametrizing triples $(\scP,\varphi,s)$, where $(\scP,\varphi)$ is in
$\Gr_G$, and $s$ is a section of an associated vector bundle
$\scP_{\bN} = \scP\times_G \bN$ such that it is sent to a regular
section of a trivial bundle under $\varphi$.
We use the notation $\cR$ when $(G,\bN)$ is clear from the context.
It is an ind-scheme of ind-infinite type, as we will explain
below. But we simply call it a variety.

Let us explain the last condition.  Let $\varphi_\bN$ denote the
induced isomorphism $\scP_\bN|_{D^*}\to D^*\times\bN$ of vector
bundles over $D^*$. It sends a section $s\in H^0(\scP_\bN)$ to a
rational section of the trivial bundle $D\times\bN$, i.e., an element
in $\bN_\cK$. It may have a pole at the origin, as $\varphi$ is not
regular there in general. The last condition means that it is regular,
i.e., $\varphi_\bN(s)\in\bN_\cO$. It defines a finite codimensional
subspace in $H^0(\scP_\bN)$.

The variety $\cR$ is a closed subvariety of a variety
$\cT\equiv\cT_{G,\bN}$, the moduli space of $(\scP,\varphi, s)$ as
above, but $s$ is merely a section of $\scP_\bN$, no further condition
on the behavior under $\varphi_\bN$. Under the projection
$\cT\to\Gr_G$, it has a structure of a vector bundle over $\Gr_G$,
whose fiber at $(\scP,\varphi)$ is $H^0(\scP_\bN)$. The rank of the
vector bundle is infinite. This is the reason why $\cT$ is {\it not\/}
an ind-scheme of ind-finite type.

\begin{NB}
    Added on May 5.
\end{NB}%
The closed embedding $\cR\to\cT$ is denoted by $i$. The projection
$\cT\to\Gr_G$ is denoted by $\pi$.

As sets, $\cT$ is the quotient $G_\cK\times_{G_\cO}\bN_\cO$. Let us
write its point as $[g,s]$ with $g\in G_\cK$, $s\in\bN_\cO$. The
equivalence relation is given by $(g,s)\sim (g b^{-1}, bs)$ for $b\in
G_\cO$, and $[g,s]$ denote the equivalence class for $(g,s)$. We have
a map $\cT\ni [g,s]\mapsto gs\in \bN_\cK$. In terms of the description
of $\cT$ as a moduli space, $gs$ is nothing but $\varphi_\bN(s)$.
\begin{NB}
    Added on May 5
\end{NB}%
Let us denote it by $\Pi$.
Together with the projection $\pi\colon \cT\to\Gr_G$, it gives a
closed embedding $(\pi,\Pi)\colon \cT\hookrightarrow
\Gr_G\times\bN_\cK$.
\begin{NB}
Let us keep the notation $\pi$ for something else.

Let us denote it by $\pi$. We have $\cR = \pi^{-1}(\bN_\cO)$.
\end{NB}%
We have $\cR = \cT\cap\left(\Gr_G\times\bN_\cO\right)$.

We have an action of $G_\cK$ on $\cT$ given by the left
multiplication. As a moduli space, it is given by the change of the
trivialization $\varphi$. Its restriction to $G_\cO$ preserves $\cR$.
There is also $\CC^\times$-actions on $\Gr_G$, $\cT$, $\cR$ induced
from the loop rotation of $D$. It is combined to actions of the
semi-direct product $G_\cO\rtimes\CC^\times$.

\begin{Convention}
In order to make formulas look cleaner, we let the loop rotation
$\CC^\times$ act on $\bN$ by weight $1/2$. (So we take a double
covering of $\CC^\times$, but we do not introduce a new notation for
the double cover for brevity.)
\end{Convention}

\begin{Remark}\label{rem:St}
    Let $\St$ be the usual (finite dimensional) Steinberg variety,
    i.e.,
    \begin{equation*}
        \St = \{ (B_1,x,B_2)\in\mathfrak B\times\g\times\mathfrak B \mid
        x\in\mathfrak n_1\cap\mathfrak n_2\},
    \end{equation*}
    where $\mathfrak B$ is the flag variety, considered as the space
    of Borel subgroups, and $\mathfrak n_a$ is the nilradical of the
    Lie algebra of $B_a$ ($a=1,2$). We have an action of
    $G\times\CC^\times$ on $\St$, and the equivariant Borel-Moore homology
    group $H^{G\times\CC^\times}_*(\St)$ gives a geometric realization of
    the degenerate affine Hecke algebra \cite{Lu-cus} (see also
    \cite{CG} for the $K$-theory version).

    If we fix a point $B_2 = B\in\mathfrak B$, we have an isomorphism
    $\mathfrak B = G/B$, and the induced isomorphism
    $\St\cong G\times_B\overline\St$,
    \begin{NB}
        $G\backslash \St \cong B\backslash\overline\St$,
    \end{NB}%
    where
    \begin{equation*}
        \overline\St = \{ (B_1,x)\in\mathfrak B\times\g \mid
        x\in\mathfrak n_1\cap\mathfrak n\},
    \end{equation*}
    where $\mathfrak n$ is the nilradical of the Lie algebra of $B$.
    \begin{NB}
        The isomorphism is given by $G\times_B\St\ni [g, B_1, x] \mapsto
        (gB_1, \operatorname{Ad}_gx, g\bmod B)$.
    \end{NB}%
    Our space $\cR$ is an analogue of $\overline\St$. We have
    \begin{equation*}
        H^{G\times\CC^\times}_*(\St) \cong H^{B\times\CC^\times}_*(\overline\St),
    \end{equation*}
    hence we can understand the geometric realization in terms of
    $\overline\St$.
    \begin{NB}
        Let $EG\to BG$ be the classifying space of $G$. Then $EG\to
        EG/B$ is the classifying space of $B$. We have $EG\times_G \St
        \cong EG\times_B \overline\St$, hence
        $H^{G\times\CC^*\times}_*(\St) \cong
        H^{B\times\CC^\times}_*(\overline\St)$. In fact, we need to
        replace $EG\to BG$ by its finite-dimensional approximation,
        but the argument is the same.
    \end{NB}%
\end{Remark}

Let $Y$ be the coweight lattice of $G$. It is well-known that
$G_\cO$-orbits in $\Gr_G$ are parametrized by dominant coweights $Y^+$
: $\Gr_G = \bigsqcup_{\la\in Y^+} \Gr_{G}^{\la}$. The closure relation
corresponds to the usual order on $Y^+$ : $\overline{\Gr}_{G}^{\la} =
\bigsqcup_{\mu\le\la} \Gr_{G}^{\mu}$.
It is well-known that $\overline{\Gr}_{G}^{\la}$ is a scheme of finite
type. The $G_\cO$-action on $\overline{\Gr}_{G}^{\la}$ factors through a
finite dimensional quotient.
\begin{NB}
    Let $g(z) =
    \begin{pmatrix}
        z^d & 0 \\ 0 & z^{-d}
    \end{pmatrix}\in \SL_2(\cK)$. Take $g'(z) =
    \begin{pmatrix}
        g_{11}(z) & g_{12}(z) \\ g_{21}(z) & g_{22}(z)
    \end{pmatrix}\in \SL_2(\cO)$. Consider
    \begin{equation*}
       g(z)^{-1} g'(z) g(z) =
       \begin{pmatrix}
           g_{11}(z)& z^{2d} g_{12}(z) \\ z^{-2d} g_{21}(z) & g_{22}(z)
       \end{pmatrix}.
    \end{equation*}
    If $g'(z) - 1 \in z^{2d}\cO$, $g(z)^{-1} g'(z) g(z)$ is in
    $\SL_2(\cO)$. It means that $[g(z)]\in\Gr_{\SL_2}$ is fixed by
    the action of $g'(z)$. Therefore $K_i$ (see below) acts trivially
    on the closure of $G_\cO\cdot [g(z)]$.
\end{NB}

Recall $\pi\colon \cT\to \Gr_G$ denote the projection forgetting
sections. Let $\cT_{\le\la} \defeq
\pi^{-1}(\overline{\Gr}_{G}^{\la})$. This is a scheme of infinite type,
and $\cT = \bigcup \cT_{\le\la}$. We define $\cR_{\le\la} \defeq
\cR\cap \pi^{-1}(\overline{\Gr}_{G}^{\la})$. It is also a scheme of
infinite type and $\cR = \bigcup \cR_{\le\la}$.

For a sufficiently large $d\gg 0$ consider the fiberwise translation by
$z^d \bN_\cO$ on $\cT = G_\cK\times_{G_\cO}\bN_\cO$. The quotient is
$\cT = G_\cK\times_{G_\cO}\bN_\cO\to \cT^d\defeq
G_\cK\times_{G_\cO}(\bN_\cO/z^d\bN_\cO)$. We have a surjective vector
bundle homomorphism $p^d_e\colon \cT^d\to \cT^e$ for $d > e$. The
original $\cT$ can be understood as the inverse limit of this system.
\begin{NB}
    The index notation in \cite{MR2135527} is {\it not\/} standard
    matrix notation. In particular $p^d_e p^e_f = p^d_f$ is not
    true. We have $p^e_f p^d_e = p^d_f$ instead. But $(p^d_e)^*
    (p^e_f)^* = (p^d_f)^*$, on the other hand.
\end{NB}%
Let $\cT^d_{\le\la} = (\pi^d)^{-1}(\overline{\Gr}_{G}^{\la})$ where
$\pi^d\colon \cT^d\to\Gr_G$ is the projection. It is a scheme of
finite type. Moreover we have an induced homomorphism
$\cT^d_{\le\la}\to \cT^e_{\le\la}$, and $\cT_{\le\la}$ is the inverse
limit of this system.

The order of pole of $\varphi$ at $0$ is bounded by a constant
depending on $\la$ for $(\scP,\varphi)\in
\overline{\Gr}_{G}^{\la}$. Therefore $\cR_{\le\la}$ is invariant under
the translation by $z^d\bN_\cO$ if we choose $d$ larger than the
order. Let $\cR_{\le\la}^d$ denote the quotient. It is a closed
subscheme of $\cT_{\le\la}^d$. Moreover we have an affine fibration
$\tilde p^d_e \colon \cR_{\le\la}^d\to \cR_{\le\la}^e$, as the
restriction of $p^d_e$ for $d>e$. Then $\cR_{\le\la}$ is the inverse
limit of this system.

\begin{NB}
    I am not sure whether we should use the different notation $\tilde
    p^k_l$ from $p^k_l$, but the pull-back with support depends on
    $p^k_l$, not only on the restriction. Therefore I use a different
    notation at this stage.
\end{NB}

\begin{NB}
    The following is added on Apr. 21.
\end{NB}%

Let $\cR_{\la} = \cR\cap\pi^{-1}(\Gr_{G}^{\la})$, the inverse image of
the $G_\cO$-orbit $\Gr_{G}^{\la}$. It is an open subvariety in a closed
subvariety $\cR_{\le\la}$, hence locally closed in $\cR$. Let
$\cR_{<\la}$ be the complement $\cR_{\le\la}\setminus\cR_{\la}$. It is
closed subvariety. Let us define $\cT_{\la}$, $\cT_{<\la}$ in the same
way.

\begin{Lemma}\label{lem:Rla}
    The restriction of $\pi$ to $\cR_{\la}$ is a vector bundle
    $\cR_{\la}\to\Gr_{G}^{\la}$ of infinite rank. It is a subbundle of
    another infinite rank vector bundle $\cT_{\la}\to\Gr_{G}^{\la}$ such
    that the quotient bundle has a finite rank given by the formula
    \begin{equation*}
        d_\la \defeq \operatorname{rank}(\cT_{\la}/\cR_{\la}) =
        \sum_\chi \max(-\langle\chi, \la\rangle,0) \dim \bN(\chi),
    \end{equation*}
    where $\bN(\chi)$ is the weight $\chi$ subspace of $\bN$.
\end{Lemma}

\begin{NB}
    Misha and Sasha, The formula of $d_\la$ is corrected on Aug. 27.
\end{NB}

\begin{NB}
    Since $\dim \bN(\chi) = \dim \bN(w\chi)$ for a Weyl group element
    $w\in W$, $d_\lambda = d_{w\lambda}$. Since $\Gr_{G}^{\la}$ depends
    only of the Weyl group orbit through $\la$, the formula must be of
    course invariant under $W$.
\end{NB}%

\begin{proof}
    This is obvious since $\cR_{\la}$ is $G_\cO$-invariant and
    $\Gr_{G}^{\la}$ is a $G_\cO$-orbit: Consider a coweight $z^\la$ as
    an element $G_\cK$, and also a point in $\Gr_{G}^{\la}$. Then the
    fiber $\cR\cap\pi^{-1}(\la)$ is $\bN_\cO\cap z^{\la}\bN_\cO = \{
    s\in\bN_\cO \mid z^{-\la} s \in\bN_\cO\}$.
    \begin{NB}
        This is via $\Pi$. Consider $[g(z),s(z)]\in\cT$ with $[g(z)] =
        [z^\lambda]\in \Gr_G$, $s(z)\in\bN_{\cO}$. Then $[g(z),s(z)]$
        is in $\cR$ if and only if $g(z)s(z)\in\bN_{\cO}$.  Therefore
        $g(z)s(z)\in \bN_{\cO}\cap z^\lambda\bN_{\cO}$. Also $\cT$ is
        just $g(z)s(z)\in z^\lambda\bN_{\cO}$.
    \end{NB}%
    This is a subspace of
    $\bN_\cO$ invariant under the stabilizer $\Stab_{G_\cO}(z^\la)$.
    \begin{NB}
        Let $g\in \Stab_{G_\cO}(z^\la)$. Then $g z^\la \in z^\la
        \bmod G_\cO$, i.e., $z^{-\la} g z^\la \in G_\cO$. Then
        $z^{-\la} (g s) = (z^{-\la} g z^{\la})(z^{-\la} s)\in\ \bN_\cO$.
    \end{NB}%
    Then $\cR_\la$ is the vector bundle over $\Gr_{G}^{\la} =
    G_\cO/{\Stab_{G_\cO}(z^\la)}$ associated with $\bN_\cO\cap
    z^{\la}\bN_\cO$.

    The rank of the quotient is the dimension of $z^\la\bN_\cO/\bN_\cO\cap
    z^{\la}\bN_\cO$.
    \begin{NB}
        Misha and Sasha : This was originally $\bN_\cO/\bN_\cO\cap
        z^{\la}\bN_\cO$. It is corrected on Aug. 28.
    \end{NB}%
    This can be computed by decomposing $\bN$ into weight spaces. If
    $s$ is contained in the weight $\chi$ subspace, we replace $\bN$ by
    a $1$-dimensional subspace, and find that the contribution is
    $\dim z^{\langle\chi,\la\rangle}\CC[[z]]/\CC[[z]]\cap
    z^{\langle\chi,\la\rangle}\CC[[z]]$. This is equal to
    $\max(-\langle\chi,\la\rangle, 0)$. The above formula follows.
\end{proof}

\begin{NB}
    The contribution to the monopole formula is
    \begin{equation*}
        |\langle\chi, \la\rangle| =
        2 \max(-\langle\chi,\la\rangle,0) +
        \langle\chi,\la\rangle.
    \end{equation*}
    Therefore we need to correct above by $1/2 \sum_\chi
    \langle\chi,\la\rangle \dim \bN(\chi)$ later. This correction term
    is $0$ for $\SL(2)$, as $\chi$ and $-\chi$ appear in pairs. On the
    other hand, this could be nontrivial if $G$ is a torus. In
    practice, we need to show that $\sum_\chi \langle\la,\chi\rangle
    \dim \bN(\chi)$ depends only on $[\la]\in\pi_1(G)$. (More
    precisely, it depends only the free part of the abelian group
    $\pi_1(G)$.)
    \begin{NB2}
        This should be something elementary. How to prove ?
    \end{NB2}%

    If so, since $\pi_0(\cR) = \pi_1(G)$, the correction term depends
    only on the component, and is harmless. In
    \cite{2015arXiv150303676N}, the correction term appears as the
    index of an associated vector bundle. So it should really depend
    on the Chern class (or the degree), not how it decomposes into
    line bundles.

    \begin{NB2}
        See below. Added on Apr.\ 24.
    \end{NB2}
\end{NB}%

\subsection{Equivariant Borel-Moore homology of the variety of triples}
\label{subsec:equiv-borel-moore}

We shall use the equivariant Borel-Moore homology groups
$H^{G_\cO}_*(\cR)$ and $H^{G_\cO\rtimes\CC^\times}_*(\cR)$ to define
the Coulomb branch and its quantization. Since both $G_\cO$ and $\cR$
are infinite dimensional, these homology groups must be treated
carefully.

We have the ind-scheme $\calR$ together with a map $\pi\colon\calR\to
\Gr_G$.  We would like to define its $G_\calO$-equivariant Borel-Moore
homology.  To this end it is sufficient to define the
$G_\calO$-equivariant Borel-Moore homology of $\cR_{\le\la} =
\cR\cap\pi^{-1}(\overline{\Gr}_{G}^{\la})$ in such a way that an
embedding $\cR_{\le\mu}\hookrightarrow \cR_{\le\la}$ will induce a map
$H_*^{G_\calO}(\cR_{\le\mu})\to H_*^{G_\calO}(\cR_{\le\la})$ for $\mu\le\la$.

Now, given such $\cR_{\le\la}$, we choose an integer $d\ge 0$ so that
we have a finite dimensional scheme $\cR^d_{\le\la}$ as its quotient
as above. We have the induced $G_\cO$-action on $\cR^d_{\le\la}$.
Let $G_i=G(\calO/z^i\calO)$. This is a quotient of $G_\calO$ and for
large $i$ the action of $G_\calO$ factorizes through $G_i$.

\begin{NB}
Now, given such $X$, we have an embedding $\pi^{-1}(X)\hookrightarrow X\times \bfN_\calO$.
Moreover, there exists an integer $d\geq 0$ such that $\pi^{-1}(X)$ is invariant with respect to
$z^d\bfN(\calO)$. Hence we can take the quotient $\pi^{-1}(X)/z^d\bfN(\calO)$.
This is a finite-dimensional scheme on which the group $G(\calO)$ acts. Let $G_i=G(\calO/z^i\calO)$. This is a quotient
of $G(\calO)$ and for large $i$ the action of $G(\calO)$ factorizes through $G_i$.
\end{NB}%

We now set
\begin{equation*}
    H_{*}^{G_\calO}(\cR_{\le\la})\defeq
    H_{G_i}^{-*}(\cR_{\le\la}^d, \DC_{\cR_{\le\la}^d})
    [-2\dim (\bfN_\calO/z^d\bfN_\calO)].
\end{equation*}
We claim that this definition depends neither on $i$ nor on
$d$. Indeed, independence of $i$ follows from the fact that for $i>j$
we have a surjective map $G_i\to G_j$ with unipotent kernel.
Independence of $d$ follows from the fact that for $d > e$ we have a
$G_\calO$-equivariant map $\tilde{p}^{d}_{e}\colon
\cR^{d}_{\le\la}\to \cR^{e}_{\le\la}$ which is a locally trivial
fibration with fibers being affine spaces of dimension $\dim
(z^{e}\bfN_\calO/z^{d}\bfN_\calO)$.
Note that if $p\colon Z\to W$ is a locally trivial fibration of
finite-dimensional schemes over $\CC$ with fibers being affine spaces
of dimension $r$ then we have a canonical isomorphism
$H^*(Z,\DC_Z)\simeq H^*(W,\DC_W)[2r]$ and the same is true for
equivariant Borel-Moore homology with respect to any algebraic group
$K$ acting on $Z$ and $W$ (and such that the morphism $p$ is
$K$-equivariant).

Note also that the degree of this homology group is given relative to
`$2\dim\bN_\cO$'. Namely if a homology classes has degree $k$, it means
that we consider homology classes $\cR^d_{\le\la}$ for all
sufficiently large $d$ whose degree is $k+2\dim(\bN_\cO/z^d\bN_\cO)$. As $d\to\infty$, the degree goes to `$k+2\dim\bN_\cO$'.
Since it is not illuminating to go back to finite dimensional
approximations every time, we use this convention hereafter: we just
write $2\dim\bN_\cO$ (and later $2\dim G_\cO$) without mentioning finite dimensional approximations.

Given an embedding $\cR_{\le\mu}\hookrightarrow \cR_{\le\la}$ as
above, note that for sufficiently large $d$ and $i$ we have a
$G_i$-equivariant closed embedding $\cR_{\le\mu}^d\hookrightarrow
\cR_{\le\la}^d$ and we can use the push-forward with respect to this
closed embedding to define the map
\begin{equation*}
H_{G_i}^*(\cR^d_{\le\mu}, \DC_{\cR^d_{\le\mu}})[-2\dim (\bfN_\calO/z^d\bfN_\calO)]\to 
H_{G_i}^*(\cR^d_{\le\la}, \DC_{\cR^d_{\le\la}})[-2\dim (\bfN_\calO/z^d\bfN_\calO)].
\end{equation*}

\begin{NB}
For the affine Grassmannian $\Gr_G$, it is well-known (see e.g.,
\cite{MV2}): $\Gr_G$ is an ind-scheme $\Gr_G = \varinjlim \Gr_G^n$,
where $\Gr_G^n$ is a finite dimensional scheme with
$G_\cO\rtimes\CC^\times$-action, and $\Gr_G^m\to \Gr_G^n$ ($m \le n$)
is a $G_\cO\rtimes\CC^\times$-equivariant closed embedding. The group
$G_\cO$ acts on $\Gr_G^n$ via a finite dimensional quotient. We define
$H^{G_\cO}_*(\Gr_G)$ as the limit $\varinjlim H^{G_\cO}_*(\Gr_G^n)$.
\end{NB}

The equivariant Borel-Moore homology group $H_*^{G_\cO}(\cR)$ is a
module over $H^*_{G_\cO}(\mathrm{pt})$, the equivariant cohomology
group of a point, defined using $G_i$ above.
Since any $G_i$ acts (trivially) on $\mathrm{pt}$, we have a natural
isomorphism $H^*_{G_\cO}(\mathrm{pt}) \cong H^*_G(\mathrm{pt})$.

The definition of the $G_\cO\rtimes\CC^\times$-equivariant homology group is
the same.

In the definition of the convolution product, we also use equivariant
Borel-Moore homology groups of other spaces (see \eqref{eq:12}). A
prototype of such homology groups is $H^{G_\cO\times
  G_\cO}_*(G_\cK)$. Let us explain how this is defined. Homology
groups of spaces actually needed are simple variants of
$H^{G_\cO\times G_\cO}_*(G_\cK)$, and hence can be defined in the same way.

Let $G_i = G(\cO/z^i\cO)$ as before. We have a surjective homomorphism
$G_\cO\to G_i$, and let $K_i$ be its kernel. Take a dominant coweight
$\la$, and let $G_{\cK}^{\le\la}$ be the inverse image of
$\overline{\Gr}_{G}^{\la}$ under $G_\cK\to\Gr_G$. We take $j\gg i$ so
that $K_j$ acts trivially on $G_{\cK}^{\le\la}/K_i$. (This is possible by
the same well-known argument that $K_i$ acts trivially on
$\overline{\Gr}_{G}^{\la}$ for $i\gg 0$.) Then we have an action of
$G_j\times G_i$ on $G_{\cK}^{\le\la}/K_i$. We define
\begin{equation*}
    H^{G_\cO\times G_\cO}_*(G_{\cK}^{\le\la}) \defeq
    H^{-*}_{G_j\times G_i}(G_{\cK}^{\le\la}/K_i,\DC_{G_{\cK}^{\le\la}/K_i})[-2\dim G_i].
\end{equation*}
This definition is independent of $i$ or $j$ by the same argument as
above. Note also that the degree is given relative to `$2\dim G_\cO$'
in the same sense as above.

We have a homomorphism $H^{G_\cO\times G_\cO}_*(G_{\cK}^{\le\mu})\to
H^{G_\cO\times G_\cO}_*(G_{\cK}^{\le\la})$ for $\mu\le\la$ as
above. Therefore we define $H^{G_\cO\times G_\cO}_*(G_\cK)$ as the
direct limit of this system.

Furthermore, as $G_{\cK}^{\le\la}/K_i\to \overline{\Gr}_{G}^{\la}$ is a
principal $G_i$-bundle, we have
\(
    H^{G_\cO\times G_\cO}_*(G_{\cK}^{\le\la})\cong
    H^{G_\cO}_*(\overline{\Gr}_{G}^{\la}).
\)
As a direct limit, we have an isomorphism
\begin{equation}\label{eq:15}
    H^{G_\cO\times G_\cO}_*(G_\cK) \cong H^{G_\cO}_*(\Gr_G).
\end{equation}

\begin{Remark}
    We consider Borel-Moore homology groups with \emph{complex}
    coefficients. But many of results below remain true for
    \emph{integer} coefficients. In particular, our Coulomb branch
    $\mathcal M_C$ will be an affine scheme over the integers. We
    leave to the interested reader the consideration of possible
    applications, and stick to complex coefficients in this paper.
\end{Remark}

\subsection{Monopole formula}

We compute the equivariant Poincar\'e polynomial of $\cR$ in this
subsection. This computation is essentially the same as \cite[\S8]{2015arXiv150303676N}, but we give the detail as it is simple.

We take the following convention:
\begin{equation*}
    P_t^{G_\cO}(\cR) = \sum_k t^{-k} \dim H^{G_\cO}_k(\cR),
\end{equation*}
and similarly for other spaces.

Let $\Gr_{G}^{\la}$ be the $G_\cO$-orbit corresponding to a dominant
coweight $\la$ as before. Let
\begin{equation*}
    P_G(t;\la) \defeq \prod \frac1{1-t^{2d_i}},
\end{equation*}
where $d_i$ are degrees of generators of invariant polynomials of
$\operatorname{Stab}_G(\la)$. In other words, they are {\it
  exponents\/} plus one.

\begin{Lemma}\label{lem:orbit}
    The equivariant Poincar\'e polynomial of $\Gr_{G}^{\la}$ is given by
    the formula
    \begin{equation*}
        P_t^{G_\cO}(\Gr_{G}^{\la}) 
        = t^{-4\langle\rho,\la\rangle} P_G(t;\la),
    \end{equation*}
    where $\rho$ is the half sum of positive roots, and $\la$ is taken
    so that it is dominant.
\end{Lemma}

If $\lambda$ is dominant, we have
\begin{equation*}
    -4\langle\rho,\la\rangle
    = -2\sum_{\alpha\in\Delta^+} | \langle\alpha,\la\rangle|
    = -\sum_{\alpha\in\Delta} | \langle\alpha,\la\rangle|,
\end{equation*}
where $\Delta^+$ (resp.\ $\Delta$) is the set of positive (resp.\ all)
roots. The right hand side is invariant under the Weyl group action,
while the left hand side is {\it not}. Since
$P_t^{G_\cO}(\Gr_{G}^{\la})$ is $W$-invariant, it is better to replace
$-4\langle\rho,\la\rangle$ by the right hand side.

\begin{proof}
    It is known that $\Gr_{G}^{\la}$ is a vector bundle over a flag
    manifold $G/P_\la$, where $P_\la$ is the parabolic subgroup
    associated with $\la$.  It is also known that dimension of
    $\Gr_{G}^{\la}$ is $2\langle\rho,\la\rangle$. (See e.g.,
    \cite[\S2]{MV2}.
    \begin{NB}
        Claim.
        \begin{equation*}
            2\langle\rho,\la\rangle
            = \sum_{\alpha\in\Delta^+} | \langle\alpha,\la\rangle|
            = \sum_{\alpha\in\Delta} \max(0,\langle\alpha,\la\rangle).
        \end{equation*}
        This is manifestly true, as $2\rho =
        \sum_{\alpha\in\Delta^+}\alpha$ and
        $\langle\alpha,\la\rangle\ge 0$ as we assume $\la$ is
        dominant. The second equality holds for any $\la$, by
        considering $\alpha$, $-\alpha\in\Delta^+$, $-\Delta^+$
        respectively. Now the last expression is manifestly invariant
        under the Weyl group $W$, as $w\Delta = \Delta$. Therefore the
        middle expression is also $W$-invariant. The middle expression
        was used in the monopole formula. See the first term in
        \cite[(4.1)]{2015arXiv150303676N}.
    \end{NB}%
    In fact, the tangent space of $\Gr_{G}^{\la}$ at $z^\mu$ is
    isomorphic to $\bigoplus_{\alpha\in\Delta}
    \bigoplus_{n=0}^{\max(0,\langle\alpha,\mu\rangle)-1} \g_\alpha z^n$.)

    We have
    \begin{equation*}
        \begin{split}
            & H^{G_\cO}_*(\Gr_{G}^{\la}) = H^{G}_*(\Gr_{G}^{\la})
        \cong H^G_{*-4\langle\rho,\la\rangle+2\dim G/P_\la}(G/P_\la)
\\
        \cong \; & H_G^{-*+4\langle\rho,\la\rangle}(G/P_\la)
        \cong H^{-*+4\langle\rho,\la\rangle}_{\operatorname{Stab}_G(\la)}
        (\mathrm{pt}),
        \end{split}
    \end{equation*}
    where $\operatorname{Stab}_G(\la)$ is as above, which is the Levi
    quotient of $P_\la$.
    \begin{NB}
        For the last isomorphism : we have an isomorphism $EG\times_G
        (G/P_\lambda) \cong EG/P_\lambda$ given by $[e, g\bmod
        P_\lambda]\mapsto eg\bmod P_\lambda$, $[e, \id\bmod P_\lambda]
        \leftmapsto e\bmod P_\lambda$. Further we have a homotopy equivalence
        $EG/P_\lambda \simeq EG/\operatorname{Stab}_G(\la)$.
    \end{NB}%
    Now the assertion follows from the well-known
    result $P^{\operatorname{Stab}_G(\la)}_{1/t}(\mathrm{pt}) = \prod
    1/(1-t^{2d_i})$.
\end{proof}

\begin{Lemma}\label{lem:MayerV}
    \textup{(1)} The equivariant Poincar\'e polynomial of
    $H^{G_\cO}_*(\cR_{\la})$ is given by
    \begin{equation*}
        P_t^{G_\cO}(\cR_\la) = t^{2d_\la}P_t^{G_\cO}(\Gr_{G}^{\la}).
    \end{equation*}
    In particular, homology group vanishes in odd degrees.

    \textup{(2)} The homology group 
    $H^{G_\cO}_*(\cR_{\le\la})$ vanishes in odd degrees. Hence the
    Mayer-Vietoris sequence splits into short exact sequences
    \begin{equation*}
        0\to H^{G_\cO}_*(\cR_{<\la})\to H^{G_\cO}_*(\cR_{\le\la})
        \to H^{G_\cO}_*(\cR_{\la})\to 0.
    \end{equation*}
\end{Lemma}

\begin{proof}
    (1)
    Since $\cR_\la\to \Gr_{G}^{\la}$ is a vector bundle (see
    \lemref{lem:Rla}), we have the Gysin isomorphism
    $H^{G_\cO}_*(\cR_\la)\cong H^{G_\cO}_{*+2d_\la}(\Gr_{G}^{\la})$.
    Here note that the rank of $\cR_\la$ is $2\dim\bN_\cO-2d_\la$, as
    the rank of $\cT$ is $2\dim\bN_\cO$. Since the degree of
    $H^{G_\cO}_*(\cR_\la)$ is relative to $2\dim\bN_\cO$, we have the
    above shift of the degree. The formula of the equivariant
    Poincar\'e polynomial follows.
    The vanishing of odd degree homology follows from
    \lemref{lem:orbit} above.

    (2) We prove the vanishing of $H^{G_\cO}_*(\cR_{\le\la})$ by
    induction on $\la$. If $\la$ is a minimal element, $\cR_{\le\la} =
    \cR_{\la}$, and hence the assertion is true by above. For general
    $\la$, we have odd degree vanishing of $H^{G_\cO}_*(\cR_{<\la})$,
    $H^{G_\cO}_*(\cR_\la)$ by the induction hypothesis and the
    above. Looking at the Mayer-Vietoris long exact sequence for the
    triple $(\cR_{<\la}, \cR_{\le\la}, \cR_\la)$, we have the odd
    vanishing of $H^{G_\cO}_*(\cR_{\le\la})$.
\end{proof}

We thus get
\begin{Proposition}\label{prop:monopole_formula}
    Fix a dominant coweight $\overline\la$. Then
    \begin{equation*}
        P_t^{G_\cO}(\cR_{\le\overline\la})
        = \sum_{\la\le\overline\la} t^{2d_\la-4\langle\rho,\la\rangle}
        P_G(t;\la),
    \end{equation*}
    where the sum runs over dominant coweights $\la$ with
    $\la\le\overline\la$.
\end{Proposition}

\begin{Remarks}
\label{discrepancy}
    (1) Taking $\overline\la\to\infty$, we formally get
    \begin{equation}\label{eq:19}
        P_t^{G_\cO}(\cR)
        = \sum_{\la} t^{2d_\la-4\langle\rho,\la\rangle}
        P_G(t;\la).
    \end{equation}
    However this infinite sum may not be well-defined even as a formal
    Laurent series, as we do not have a control on
    $2d_\la-4\langle\rho,\la\rangle$ in general.

    (2) The above formal infinite sum is essentially the same as the
    monopole formula of the Hilbert series of the Coulomb branch of
    the $3$-dimensional $\mathcal N=4$ SUSY gauge theory associated
    with $(G_c,\bN\oplus\bN^*)$, proposed by Cremonesi, Hanany and
    Zaffaroni \cite{Cremonesi:2013lqa}. Here there is a slight
    difference: $d_\la-2\langle\rho,\la\rangle$ is replaced by
    \begin{equation}\label{eq:18}
        \Delta(\la) \defeq
        -\sum_{\alpha\in\Delta^+} |\langle\alpha,\la\rangle|
        + \frac12\sum_\chi |\langle\chi,\la\rangle|\dim\bN(\chi).
    \end{equation}
It is a simple exercise to check that the difference
\begin{equation*}
    d_\la-2\langle\rho,\la\rangle - \Delta(\la)
    = - \frac12  \sum_\chi \langle\chi, \la\rangle \dim\bN(\chi)
\end{equation*}
\begin{NB}
    Though $d_\la$ is changed on Aug. 28, the difference is just
    changed by sign. So it is not essential.
\end{NB}%
depends only on the equivalence class $[\la]$ in $\pi_1(G) =
\pi_0(\cR)$. (In fact, it depends only on the free part of the abelian
group $\pi_1(G)$.)
\begin{NB}
    Recall $\pi_1(G)$ is $Y/X$, where $Y$ is the coweight lattice, $X$
    is the coroot lattice. We claim that $\sum_\chi
    \langle\chi,\la\rangle \dim \bN(\chi)$ depends only on the class of
    $\la$, the free part of $Y/X$.

    It is enough to check that $\sum_\chi \langle\chi,\la\rangle \dim
    \bN(\chi) = 0$ if $\la\in X$. If so, it is also zero if $\la\bmod
    X$ is a torsion element in $Y/X$. Indeed the left hand side is an
    integer, hence it is zero if its multiple is zero. Let $d'_\la =
    \sum_\chi \langle\chi,\la\rangle \dim \bN(\chi)$. Since $\dim
    \bN(\chi) = \dim \bN(w\chi)$ for a Weyl group element $w\in W$,
    $d'_\lambda = d'_{w\lambda}$. On the other hand, $d'_\la$ is
    linear in $\la$. Therefore $d'_{\alpha_i^\vee} =
    -d'_{-\alpha_i^\vee}$ for a simple coroot $\alpha_i^\vee$. The
    Weyl group invariance implies $d'_{\alpha_i^\vee} = 0$. Since
    $\alpha_i^\vee$ gives a base of $X$, we have $d'_\la = 0$ for
    $\la\in X$.
    \end{NB}%
    Therefore this correction term is harmless: we just shift the
    degree on each component of $\cR$. This shift turns out to be
    natural when we identify the Coulomb branch with known
    examples. The degree $\Delta(\la)$ is determined so that the
    corresponding $S^1$-action, the restriction of the
    $\CC^\times$-action in \subsecref{sec:grading}, extends to an
    $\SU(2)$-action on the Coulomb branch which rotates the
    hyperK\"ahler structure.\footnote{The third named author thanks
      Amihay Hanany for his explanation.}
    See \subsecref{subsec:abel_examples} below. See also
    \cite[Remarks~\ref{blowup_pre-compare_degrees_zas} and
      \ref{blowup_pre-compare_degrees_slice}]{blowup}.
    \begin{NB}
        Misha:

        It is a little difficult to identify the $\SU(2)$-action in
        the examples in Remarks~\ref{compare_degrees_zas}
        and~\ref{compare_degrees_slice} for HN. Therefore it is better
        to add a comment that the $\Delta(\la)$-degree should be
        coming from the $\SU(2)$-action.
    \end{NB}%

    It should be remarked also that the monopole formula is proposed
    under the assumption $2\Delta(\la)\ge 1$ for any $\la\neq 0$ (the
    so-called `good' or `ugly' theory.) This is to avoid a possible
    divergence in the infinite sum above pointed out in (1). On the
    other hand, our $H^{G_\cO}_*(\cR)$ is always well-defined even
    without this assumption. We do not have any problem even if its
    degree piece is infinite dimensional.

    Nevertheless the monopole formula is very useful to investigate
    expected properties of the Coulomb branch.

    (3) Suppose that $G$ is possibly disconnected.
    Since $\Gr_G$ depends on the connected component $G^0$ of $G$, our
    variety $\cR$ of triples does not see the component group. However
    the equivariant homology group \emph{does} see the component
    group: $H^{G_\cO}_*(\cR) \cong H_*^{G^0_\cO}(\cR)^\Gamma$, where
    $\Gamma = G/G^0$. (See e.g., \cite[Chap.~3, \S1,
    Example~3]{MR0423384}.)
\begin{NB}
    From the fibration $\Gamma\to EG\times_{G^0} X \to EG\times_G X$.
\end{NB}%
For the Coulomb branch defined in the next section, it means $\mathcal
M_C(G,\bN) = \mathcal M_C(G^0,\bN)/\Gamma$. For the monopole formula
above, we should understand $P_G(t;\la)$ as the Poincar\'e polynomial
of $H^*_{\Stab_G(\la)}(\mathrm{pt})$, where $\Stab_G(\la)$ is possibly
disconnected. See \cite[App.~A]{Cremonesi:2014uva} for examples of
computation.
\end{Remarks}

\section{Definition of Coulomb branches as affine schemes}
\label{sec:definition}

We define the convolution product on $H^{G_\cO}_*(\cR)$, following
\cite[\S7]{MR2135527} in this section. This gives us a definition of
the Coulomb branch as the spectrum of $H^{G_\cO}_*(\cR)$.

We use a sheaf theoretic framework for later applications, hence need
to make some points in the construction \cite[\S7]{MR2135527} to
actual statements, e.g., \lemref{lem:pull}.

\subsection{Convolution diagram}\label{subsec:convolution-diagram}

Recall the convolution diagram for the affine Grassmannian
(\cite[(4.1)]{MV2}):
\begin{equation}\label{eq:1}
    \begin{CD}
        \Gr_G\times\Gr_G @<p<< G_\cK\times\Gr_G 
    @>q>> \Gr_G\tilde\times\Gr_G @>m>> \Gr_G,
    \end{CD}
\end{equation}
Here $\Gr_G\tilde\times\Gr_G$ is the quotient
$G_\cK\times_{G_\cO}\Gr_G$. The maps $p$, $q$ are projections and $m$
is the multiplication. For $G_\cO$-equivariant perverse sheaves
$A_1$,$A_2$, the pullback $p^*(A_1\boxtimes A_2)$ descends to
$\Gr_G\tilde\times\Gr_G$ by the equivariance. Let us denote it by
$A_1\tilde\boxtimes A_2$.
\begin{NB}
$(q^*)^{-1}p^*(A_1\boxtimes A_2)$.
\end{NB}%
Then we define $A_1\star A_2$ by $m_*(A_1\tilde\boxtimes A_2)$.
\begin{NB}
$m_*(q^*)^{-1}p^*(A_1\boxtimes A_2)$.
\end{NB}
It defines a symmetric monoidal structure on the category of
$G_\cO$-equivariant perverse sheaves on $\Gr_G$, and is equivalent to
the monoidal category of finite dimensional representations of the
Langlands dual of $G$ \cite{MV2}.

\begin{NB}
Let us relate this diagram to the usual convolution diagram
\begin{equation*}
    \begin{CD}
    X\times X \times X\times X @<{p_{12}\times p_{23}}<<
    X\times X\times X @>{p_{13}}>> X\times X;
    \qquad p_{13*}(p_{12}^*(\bullet)\otimes p_{23}^*(\bullet)),
    \end{CD}
\end{equation*}
where $X = G/P$. We use the isomorphism of stacks
\begin{equation}\label{eq:3}
    \frac{G/P}{P} \cong \frac{G/P \times G/P}{G}.
\end{equation}
\begin{NB2}
    An object of the left hand side is a principal $P$-bundle $\scP\to
    S$ together with a $P$-equivariant morphism $\xi\colon \scP\to
    G/P$, i.e., $\xi(fp) = p^{-1} \xi(f)$ for $p\in P$. Morphisms are
    morphisms of $P$-bundles compatible with maps to $G/P$. Similarly
    an object in the right hand side is a principal $G$-bundle
    $\scP'\to S$ together with a $G$-equivariant morphism $\xi'\colon
    \scP'\to G/P\times G/P$. Now we define a functor
\begin{equation*}
    (\scP\to S, \xi) \longmapsto (\scP\times_P G\to S, \xi'), \qquad
    \xi'([f,g]) = (g^{-1}\bmod P, g^{-1}\xi(f)).
\end{equation*}
Since $\xi'([fp, p^{-1}g]) = (g^{-1} p\bmod P, g^{-1}p\xi(fp))
= (g^{-1}p\bmod P, g^{-1}\xi(f))$, this is well-defined. The $G$-action on $\scP\times_P G$ is given by $[f,g]\mapsto [f,gg']$ for $g'\in G$. Therefore $\xi'$ is $G$-equivariant.

The inverse functor is given by
\begin{equation*}
    (\scP',\xi') \longmapsto (\scP \defeq \xi^{\prime-1}(e\bmod P,G/P), \xi'|_\scP).
\end{equation*}

A naive map is given by
\begin{equation*}
    \begin{split}
    \frac{G/P}P \ni [g\bmod P] & \mapsto [(e\bmod P, g\bmod P)]
    \in \frac{G/P\times G/P}G
\\
    [g_1^{-1} g_2 \bmod P]
    & \leftmapsto [(g_1\bmod P, g_2\bmod P)].
    \end{split}
\end{equation*}
\end{NB2}%
The above diagram is using the left side picture, where $P$ is
hidden. We have
\begin{equation*}
    \begin{CD}
     \frac{G/P \times G/P}{G}\times
     \frac{G/P \times G/P}{G}
@<{p_{12}\times p_{23}}<< \frac{G/P \times G/P\times G/P}{G}
@= \frac{G/P \times G/P\times G/P}{G}
@>{p_{13}}>> \frac{G/P \times G/P}{G}
\\
  @| @| @| @|
\\
\frac{G/P}P\times \frac{G/P}P
@<<p<
\frac{G\times (G/P)}{P\times P}
@>\cong>q>
\frac{G\times_P (G/P)}P
@>>m>\frac{G/P}P.
    \end{CD}
\end{equation*}
The $G$-actions on the upper row are diagonal ones on the double and
triple products. The $P\times P$-action (resp. $P$-action) on the
second (resp.\ third) term in the lower row is
\begin{equation*}
    (g_1, g_2\bmod P)\cdot
    (p_1,p_2) = (p_1^{-1} g_1 p_2, p_2^{-1} g_2\bmod P),\quad
    \text{(resp.\ }
    [g_1, g_2\bmod P]\cdot p
    = [p^{-1} g_1, g_2\bmod P]\text{)}.
\end{equation*}
Then $p$, $q$, $m$ are equivariant morphisms. Here we take the first
projection $P\times P\to P$ as the group homomorphism for the morphism
$q$. Because $[p_1^{-1}g_1p_2, p_2^{-1}g_2\bmod P] = [p_1^{-1}g_1,
g_2\bmod P]$, it is indeed equivariant.

The map $q$ is an isomorphism of stacks: for a $P\times P$-bundle
$\scP$ with a $P\times P$-equivariant morphism $\xi\colon \scP\to
G\times G/P$, we take the quotient $\nicefrac{\scP}{\{1\}\times P}$ and
the induced morphism $\nicefrac{\xi}{\{1\}\times P}\colon
\nicefrac{\scP}{\{1\}\times P} \to G\times_P G/P$. The inverse is given
by the fiber product: given $\scP'$ with $\xi$, construct $\scP$ with
$\xi$ by
\begin{equation*}
    \begin{CD}
        \scP @>\xi>> G\times G/P
\\
@VVV @VVV
\\
    \scP' @>>\xi'> G\times_P G/P.
    \end{CD}
\end{equation*}
In the third vertical isomorphism, we use an isomorphism
$\frac{G/P\times G/P\times G/P}G\cong \frac{G/P\times G/P}P$, which is
given by just adding one more $G/P$ to \eqref{eq:3}, and an isomorphism
$G\times_P G/P\cong G/P\times G/P$ given by
\begin{alignat*}{3}
    G\times_P G/P \ni
    & [g_1, g_2\bmod P] && \quad\mapsto\quad && [g_1\bmod P, g_1g_2\bmod P]
    \in G/P\times G/P
\\
    & [g_1,g_1^{-1}g_2'\bmod P] && \quad\leftmapsto\quad &&
    [g_1\bmod P, g_2'\bmod P].
\end{alignat*}
\begin{NB2}
    Let us check the commutativity of the rightmost square. Let us
    take a $P$-bundle $\scP$ with a $P$-equivariant map
    $\xi\colon\scP\to G\times_P (G/P)$. We regard $\xi\colon
    \scP \to G/P\times G/P$ via the above isomorphism
    $G\times_P (G/P) \cong G/P \times G/P$.
    Applying $m$, we get $(\scP, \xi_2)$ where $\xi_2\colon\scP\to
    G/P$ is the second component of $\xi$. Then we go up to get
    $(\scP\times_P G, \xi')$ with $\xi'([f,g]) = (g^{-1}\bmod P,
    g^{-1}\xi_2(f))$. On the other hand, if we go up by the third
    vertical isomorphism, we get
    \(
      (\scP\times_P G,\xi'')
    \)
    with $\xi''([f,g]) = (g^{-1}\bmod P, g^{-1}\xi(f))$. Applying
    $p_{13}$, we get $(\scP\times_P G, \xi')$. Therefore the rightmost
    square is commutative.

    By the commutativity of the middle square of the diagram the
    second vertical isomorphism is given as follows: for a $P\times
    P$-bundle $\scP$ with a $P\times P$-equivariant morphism
    $\xi\colon \scP \to G\times (G/P)$, we take a $G$-bundle
    \(
    (\nicefrac{\scP}{\{1\}\times P})\times_P G
    = \scP\times_{P\times P} G
    \)
    with the induced morphism $\xi'$, where $P\times P$ acts on $G$ by
    $(p_1,p_2) g = p_1 g$, and
    \(
       \xi'([f,g]) = \left(g^{-1}\bmod P, g^{-1}\xi_G(f)\bmod P,
       g^{-1}\xi_G(f)\xi_{G/P}(f)\right)
    \)
    with $\xi = (\xi_G, \xi_{G/P})$. We have
    \begin{equation*}
        \begin{split}
            & \xi'([f\cdot(p_1,p_2),g]) = \left(g^{-1}\bmod P,
              g^{-1}\xi_G(f\cdot (p_1,p_2))\bmod P, g^{-1}\xi_G(f\cdot
              (p_1,p_2))\xi_{G/P}(f\cdot (p_1,p_2))\right)
            \\
            =\; &
            \left(g^{-1}\bmod P,
              g^{-1}p_1^{-1} \xi_G(f)\bmod P,
              g^{-1}p_1^{-1}\xi_G(f)p_2\,
              p_2^{-1}\xi_{G/P}(f)\right)
            \\
            =\; &
            \left(g^{-1}\bmod P,
              g^{-1}p_1^{-1} \xi_G(f)\bmod P,
              g^{-1}p_1^{-1}\xi_G(f)\xi_{G/P}(f)\right)
            = \xi'([f, p_1 g]).
        \end{split}
    \end{equation*}
    Therefore $\xi'$ is well-defined.

    Let us check the commutativity of the leftmost square: take a
    $P\times P$-bundle $\scP$ with $\xi\colon \scP\to G\times
    (G/P)$. If we go up and apply $p_{12}\times p_{23}$, we get a pair
    $(\nicefrac{\scP}{\{1\}\times P})\times_P G, \xi'_{12})$,
    $(\nicefrac{\scP}{\{1\}\times P})\times_P G, \xi'_{23})$, where
    $\xi'_{12}$, $\xi'_{23}$ are components of $\xi'$ above, i.e.,
    $\xi'_{12}([f,g]) = (g^{-1}\bmod P, g^{-1}\xi_G(f)\bmod P)$,
    $\xi'_{23}([f,g]) = (g^{-1}\xi_G(f)\bmod P,
    g^{-1}\xi_G(f)\xi_{G/P}(f))$.

    On the other hand, if we apply $p$, we get
    a pair
    \(
    \left(\nicefrac{\scP}{\{1\}\times P},
    \nicefrac{\xi_G}{\{1\}\times P}\bmod P\right),
    \)
    and
    \(
    \left(\nicefrac{\scP}{P\times \{1\}},
    \nicefrac{\xi_{G/P}}{P\times \{1\}}\right).
    \)
    Then we go up to get a pair
    \(
        \left((\nicefrac{\scP}{\{1\}\times P})\times_P G,
          \xi''_{12}
        \right),
        \)
        \(
    \left((\nicefrac{\scP}{P\times \{1\}})\times_P G,
      \xi''_{23}
  \right)
\)
with
\begin{equation*}
    \begin{split}
        & \xi''_{12}([f,g]) = \left(g^{-1}\bmod P,
   g^{-1}(\nicefrac{\xi_G}{\{1\}\times P})(f)\bmod P\right),\\
 &    \xi''_{23}([[f,g]]) = \left(g^{-1}\bmod P,
   g^{-1}(\nicefrac{\xi_{G/P}}{P\times \{1\}})(f)\right),
    \end{split}
\end{equation*}
where $[[f(p_1,p_2),g]] = [[f, p_2 g]]$, $[f(p_1,p_2),g] = [f, p_1
g]$. We have $\xi'_{12} = \xi''_{12}$. Hence the first factor is the
same.
    Note that the morphism $\xi_G\colon\scP\to G$ induces an isomorphism
    of $G$-bundles
    \begin{equation*}
        (\nicefrac{\scP}{P\times \{1\}})\times_P G \ni
        [[f, g]]\longmapsto
        [f,\xi_G(f)g]
        \in(\nicefrac{\scP}{\{1\}\times P})\times_P G.
    \end{equation*}
    Since $\xi_{23}''$ is mapped to $\xi'_{23}$ under this
    isomorphism, the second factor also matches. Hence the leftmost
    square is commutative.
\end{NB2}%
\end{NB}

\begin{NB}
    \begin{Remark}\label{rem:abelian}
    Let us suppose $G$ is abelian. In this case we have
    \begin{equation*}
        p'\colon \Gr_G\tilde\times\Gr_G\ni [g_1,[g_2]]
        \mapsto ([g_1], [g_2]) \in \Gr_G\times\Gr_G
    \end{equation*}
    is well-defined. In fact, the equivalence relation is given by
    $[g_1,[g_2]] = [g_1 b, [b^{-1}g_2]]$ for $b\in G_\cO$, but we have
    $[b^{-1}g_2] = [g_2 b^{-1}] = [g_2]$ as $G$ is abelian. We thus have
    $p = p'\circ q$. Therefore $p^* = q^* p^{\prime*}$. Thus
    \begin{equation*}
        A_1\star A_2 = m_* (q^*)^{-1} p^*(A_1\boxtimes A_2)
        = m_* p^{\prime*} (A_1\boxtimes A_2).
    \end{equation*}

    Moreover, $\Gr_G\tilde\times\Gr_G$ and $\Gr_G\times\Gr_G$ is
    isomorphic. The inverse of $p'$ is given by
    $([g_1],[g_2])\mapsto [g_1, [g_2]]$. This is well-defined again as
    $[g_1b, [g_2]] = [g_1, [bg_2]] = [g_1, [g_2b]] = [g_1, [g_2]]$. Then
    \begin{equation*}
        A_1\star A_2 = m'_*(A_1\boxtimes A_2),
    \end{equation*}
    where $m'\colon\Gr_G\times\Gr_G\to \Gr_G$ is given by
    $([g_1],[g_2])\mapsto [g_1g_2]$.
    \end{Remark}
\end{NB}%

Let us describe the convolution diagram \eqref{eq:1} in terms of
functors, as in \subsecref{triples}. This is given in
\cite[\S5]{MV2}.
\begin{NB}
    Let us just take a test scheme $S$ to be a point for brevity, as a
    generalization to arbitrary affine schemes is automatic. This is
    just saying a scheme is the {\it moduli space\/} parametrizing
    something.
\end{NB}%
The leftmost space $\Gr_G\times \Gr_G$ is the moduli space of
$(\scP_1,\varphi_1,\scP_2,\varphi_2)$, two $G$-bundles on the formal
disk $D$ with trivializations over the punctured disk $D^*$.
The next $G_\cK\times\Gr_G$ is the moduli of
$(\scP_1,\varphi_1,\kappa, \scP_2,\varphi_2)$, the above data together
with a trivialization $\kappa$ of $\scP_1$ over $D$.
The third space $\Gr_G\tilde\times\Gr_G$ is the moduli of
$(\scP_1,\varphi_1,\scP_2,\eta)$, where $\eta$ is an isomorphism
between $\scP_1$ and $\scP_2$ over $D^*$.
\begin{NB}
    In \cite[\S5]{MV2}, $\scP_2$ in $\Gr_G\tilde\times\Gr_G$, more precisely
    in the definition of $q$, is another $G$-bundle $\scP$. But $\scP$
    is the $G$-bundle obtained by gluing $\scP_1|_{X\setminus x_2}$
    and $\scP_2|_{\hat{X}_{x_2}}$ by $\varphi_2^{-1}
    \circ\kappa$. Therefore it is the same.
\end{NB}%

The map $p$ just forgets $\kappa$. The map $q$ is
$(\scP_1,\varphi_1,\kappa,\scP_2,\varphi_2)\mapsto
(\scP_1,\varphi_1,\scP_2,\varphi_2^{-1}\circ\kappa|_{D^*})$. Note
$\varphi_2^{-1}\circ\kappa|_{D^*}$ is an isomorphism from $\scP_1|_{D^*}$
to $\scP_2|_{D^*}$, as required. Finally $m$ is given by
$(\scP_1,\varphi_1,\scP_2,\eta)\mapsto (\scP_2,\varphi_1\circ\eta^{-1})$.

\begin{NB}
    Next we need to introduce the diagram for $\cR$. We first consider
    the finite dimensional situation. The usual convolution diagram is
    \begin{equation*}
        \begin{CD}
            T^*\mathfrak B\times T^*\mathfrak B\times T^*\mathfrak
            B\times T^*\mathfrak B
            @<{\id\times\Delta\times\id}<<
            T^*\mathfrak B\times T^*\mathfrak B\times T^*\mathfrak B
            @>{p_{13}}>> T^*\mathfrak B\times T^*\mathfrak B
\\
    @AAA @AAA @AAA
\\
            \St \times \St @<<< p_{12}^{-1}(\St)\cap p_{23}^{-1}(\St)
            @>>> p_{13}(p_{12}^{-1}(\St)\cap p_{23}^{-1}(\St)) = \St.
        \end{CD}
    \end{equation*}
    Let us rewrite the convolution product in terms of
    $\overline{\St}$. The closed embedding $T^*\mathfrak
    B\times T^*\mathfrak B\supset\St$ is replaced by
    \begin{equation*}
        \{ (B_1,x) \in
        \mathfrak B\times\mathfrak g\mid
        x\in\mathfrak n_1 \} = T^*\mathfrak B
          \supset
          \overline{\St} = \{ (B_1,x)\in T^*\mathfrak B \mid
          x\in\mathfrak n\},
    \end{equation*}
    as we fix $B_2 = B$. The diagram above is replaced as
    \begin{equation*}
        \begin{CD}
            \frac{\St}G
            \times \frac{\St}G
            @=
            \frac{\St}G
            \times \frac{\St}G
            @<<<
            \frac{p_{12}^{-1}(\St)\cap p_{23}^{-1}(\St)}{G}
            @>{p_{13}}>>
            \frac{\St}G
\\
            @| @| @| @|
\\
            \frac{\overline{\St}}B \times \frac{\overline{\St}}B @=
            \frac{\overline{\St}}B\times\frac{\St}G
            @<<< \frac{ \{ (B_1,B_2,x)\in\St \mid x\in\mathfrak n\}}B @>>> \frac{\overline{\St}}B,
        \end{CD}
    \end{equation*}
    where three projections from $\{ (B_1,B_2,x)\in\St\mid
    x\in\mathfrak n\}$ to $\overline{\St}$ (twice) and $\St$ are given
    by taking $(B_1,x)$, $(B_2,x)$ and $(B_1,B_2,x)$.

    From this replacement, it is more or less clear that the above
    convolution is the same as usual one. But we still need smoothness
    to define a pull-back of a homology class.
\end{NB}

The goal of this subsection is to introduce corresponding diagrams for
$\cR$.
Recall that $\cT$ is the quotient $G_\cK\times_{G_\cO}\bN_\cO$, and we
have an embedding $\cT\hookrightarrow \Gr_G\times \bN_\cK$ such that $\cR =
\cT\cap (\Gr_G\times\bN_\cO)$.
\begin{NB}
    Let $\tilde\cR \defeq \{ (g,v) \mid gv\in \bN_\cO\}$ denote the
    inverse image of $\cR$ in $G_\cK\times \bN_\cO$.
\end{NB}%
We consider the induced space
\(
     G_\cK\times_{G_\cO}\cR.
\)
It consists of $\bigl[g_1, [g_2,s]\bigr]$ with $g_1\in G_\cK$,
$[g_2,s]\in \cR\subset \cT = G_\cK\times_{G_\cO}\bN_\cO$. We have
$\bigl[g_1, [g_2,s]\bigr] = \bigl[g_1 b, [b^{-1}g_2,s]\bigr]$ for $b\in
G_\cO$. We consider the diagram
\begin{equation}\label{eq:12}
    \begin{CD}
        \cR \times\cR @<\tilde p<< p^{-1}(\cR\times\cR)
        @>\tilde q>> q(p^{-1}(\cR\times\cR))
        @>\tilde m>> \cR
        \\
        @V{i\times\id_\cR}VV @V{i'}VV @VVV @VV{i}V
        \\
        \cT\times \cR @<p<< G_\cK\times\cR @>q>>
        G_\cK\times_{G_\cO}\cR @>m>> \cT,
    \end{CD}
\end{equation}
\begin{NB}
I slightly change the definition. Apr. 10, 2015.
\begin{equation*}
    \begin{CD}
        \tilde\cR \times\cR @<\tilde p<< p^{-1}(\tilde\cR\times\cR)
        @>\tilde q>> q(p^{-1}(\tilde\cR\times\cR))
        @>\tilde m>> \cR
        \\
        @V{i}VV @V{i'}VV @VVV @VVV
        \\
        G_\cK\times\bN_\cO\times \cR @<p<< G_\cK\times\cR @>q>>
        G_\cK\times_{G_\cO}\cR @>m>> \cT,
    \end{CD}
\end{equation*}
\end{NB}%
where
\begin{NB}
    $p^{-1}(\cR\times\cR) = \{ (g_1, [g_2, s]) \mid g_1 g_2
    s\in\bN_\cO\}$, $q(p^{-1}(\cR\times\cR)) = \{ [g_1, [g_2,
    s]] \mid g_1 g_2 s\in\bN_\cO\}$.
\end{NB}%
the first row consists of closed subvarieties in spaces in the second
row. Maps in the second row are given by
\begin{equation}\label{eq:42}
    \left([g_1,g_2s], [g_2,s]\right) \leftmapsto \left(g_1,[g_2,s]\right)
    \mapsto
    \bigl[g_1, [g_2,s]\bigr] \mapsto [g_1g_2, s].
\end{equation}
Since $p^{-1}(\cR\times\cR) = \{ (g_1, [g_2, s]) \mid g_1 g_2
s\in\bN_\cO\}$, the target of $\tilde m$ is $\cR$ as required.
The map $m$ factors as
$G_\cK\times_{G_\cO}\cR\xrightarrow{\id_{G_\cK}\times_{G_\cO}i}
G_{\cK}\times_{G_\cO}\cT\xrightarrow{m'}\cT$, and $m'$ (and hence also
$m$) is ind-proper.

Let us introduce the following group actions on terms in the second
row:
\begin{equation}\label{eq:41}
    \begin{split}
        G_\cO \times G_\cO\curvearrowright
        \cT\times\cR ;\ &
        (g,h)\cdot \left([g_1,s_1],[g_2,s_2]\right)
        = \left([gg_1, s_1], [hg_2,s_2] \right),
        \\
        G_\cO \times G_\cO \curvearrowright G_\cK\times\cR ;\ &
        (g,h)\cdot \left(g_1, [g_2,s]\right)
        = \left(gg_1h^{-1}, [hg_2, s]\right),
        \\
        G_\cO \curvearrowright G_\cK\times_{G_\cO}\cR ;\ &
        g\cdot \left[g_1, [g_2,s]\right] = \left[gg_1, [g_2, s]\right],
        \\
        G_\cO \curvearrowright \cT ;\ &
        g\cdot [g_1,s] = [gg_1, s].
    \end{split}
\end{equation}
These actions preserve spaces in the first row. Moreover morphisms
$p$, $q$, $m$ are equivariant, where we take a group homomorphism
$p_1\colon G_\cO\times G_\cO\to G_\cO$ for $q$.

Let us understand spaces in \eqref{eq:12} as moduli spaces.
Let us first consider spaces in the lower row.

The space $\cT\times\cR$ is clear. It is the moduli space of
$(\scP_1,\varphi_1,s_1, \scP_2, \varphi_2, s_2)$ such that
$(\scP_1,\varphi_1)$, $(\scP_2,\varphi_2)\in\Gr_G$ and $s_1$, $s_2$
are sections of the associated bundles $\scP_{1,\bN}$,
$\scP_{2,\bN}$. We require $\varphi_{2,\bN}(s_2)\in\bN_\cO$.
The second space $G_\cK\times\cR$ is the moduli of
$(\scP_1,\varphi_1,\kappa,\scP_2, \varphi_2, s_2)$, where $\kappa$ is
a trivialization of $\scP_1$ as in the affine Grassmannian case.
The third space $G_\cK\times_{G_\cO}\cR$ is the moduli of
$(\scP_1,\varphi_1,\scP_2,s_2,\eta)$, where $\eta$ is an isomorphism
between $\scP_1$ and $\scP_2$ over $D^*$.
\begin{NB}
    Added on Apr. 21.
\end{NB}%
We require $\eta_{\bN}^{-1}(s_2)\in H^0(\scP_{1,\bN})$.

Let us describe maps. The map $p$ is given as follows. Let us take
$(\scP_1,\varphi_1,\kappa,\scP_2,\varphi_2,s_2)$ from
$G_\cK\times\cR$.  Since $\varphi_{2,\bN}(s_2)\in\bN_\cO$, it can be
considered as a section of the trivial bundle $D\times\bN$ over
$D$. We transfer it to a section of $\scP_{1,\bN}$ by
$\kappa_\bN\colon \scP_{1,\bN}\to D\times\bN$. We denote it by
$\kappa_\bN^{-1}\circ \varphi_{2,\bN}(s_2)$. Then we have
$(\scP_1,\varphi_1,\kappa_\bN^{-1}\circ
\varphi_{2,\bN}(s_2),\scP_2,\varphi_2,s_2)\in\cT\times\cR$.
The map $q$ is given by
\(
    (\scP_1,\varphi_1,\kappa,\scP_2,\varphi_2,s_2)\mapsto
    (\scP_1,\varphi_1,s_2,\scP_2,\varphi_2^{-1}\circ\kappa|_{D^*}).
\)
\begin{NB}
    Added on Apr. 21.
\end{NB}%
The condition $\varphi_{2,\bN}(s_2)\in\bN_\cO$ is equivalent to
$\eta_\bN^{-1}(s_2) = \kappa_\bN^{-1}\varphi_{2,\bN}(s_2)\in H^0(\scP_{1,\bN})$.
Since we do not need to touch the section $s_2$, it is essentially the
same as the corresponding map in the affine Grassmannian case.
Finally the map $m$ is given by
\(
   (\scP_1,\varphi_1,\scP_2,s_2,\eta)\mapsto
   (\scP_2,\varphi_1\circ\eta^{-1},s_2).
\)
This is again the same as the affine Grassmannian case, but we note that
there is no reason that the trivialization $\varphi_1\circ\eta^{-1}$ sends $s_2$ to $\bN_\cO$. Therefore the target of $m$ is $\cT$, not~$\cR$.

Let us go to the spaces in the upper row. They are sub-ind-schemes of
spaces in the lower row. So we describe the conditions defining the
upper spaces in the lower spaces. The space $\cR\times\cR$ is
clear. We impose $\varphi_{1,\bN}(s_1)\in\bN_\cO$ in
$\cT\times\cR$. The second space $p^{-1}(\cR\times\cR)$ is given by
the condition
$\varphi_{1,\bN}\circ\kappa_\bN^{-1}\circ\varphi_{2,\bN}(s_2)\in\bN_\cO$.
For the third space $q(p^{-1}(\cR\times\cR))$, we need to rewrite this
condition as $\varphi_{1,\bN}\circ\eta_\bN^{-1}(s_2)\in\bN_\cO$. Since
$q$ is given by setting $\eta = \varphi_2^{-1}\circ\kappa|_{D^*}$, two
conditions are equivalent as required.
Finally $m$ sends $q(p^{-1}(\cR\times\cR))$ to $\cR$, as
$\varphi_1\circ\eta^{-1}$ is the new trivialization of $\scP_2$ given
by the map $m$.

\begin{NB}

 We first introduce relevant moduli spaces. Let $X$ denote an
ind-scheme defined as the moduli space for the following data
\begin{enumerate}
      \item $\scP_1,\scP_2$ -- $G$-bundles on the formal disc $D$.
      \item A trivialization $\varphi_1$ of $\scP_1$ on $D^*$ and an isomorphism $\eta$ between $\scP_1$ and $\scP_2$ on $D^*$.
    \item Sections $s_i$ of $\scP_{i,\bN}$ (on $D$) for $i=1,2$.
\end{enumerate}
This is the space corresponding to $\Gr_G\tilde\times\Gr_G$, but we need more
in our case. Let $X_1$, $X_2$ be closed sub-ind-schemes of $X$ where
$X_1$ is given by the condition that $\varphi_{1,\bN}(s_1)\in
\bN_{\cO}$ ({\it a priori\/} it is an element of $\bN_{\cK}$), and
$X_2$ is given by the condition that $s_2=\eta(s_1)$.

Note that we have a natural map $m$ from $X\to \cT$ sending the above
data to $(\scP_2,\varphi_1\circ\eta^{-1},s_2)$.  This map is
ind-proper when restricted to $X_2$ and it sends $X_1\cap X_2$ to
$\cR$. In addition we have a map $p_1\colon X\to \cT$ sending the
above data to $(\scP_1,\varphi_1,s_1)$, which sends $X_1$ to $\cR$.

We consider the space $\tilde X_2$ consisting of the above data (for
$X_2$) together with a trivialization $\kappa$ on $\scP_1$ on $D$.
It is endowed with a map $p_2\colon \tilde X_2\to \cT$ sending the
above data to $(\scP_2,\kappa|_{D^*}\circ\eta^{-1},s_2)$.
But $\eta^{-1}(s_2) = s_1$ is a regular section of $\scP_{1,\bN}$ and
$\kappa$ is defined over $D$. Therefore
$\kappa_\bN\circ\eta_\bN^{-1}(s_2)\in\bN_\cO$, hence $p_2$ sends
$\tilde X_2$ to $\cR$.
Furthermore, the only remaining data to describe a point in $\tilde
X_2$ is just $\varphi_1$, a trivialization of $\scP_1$ on
$D^*$. Therefore $\tilde X_2$ is isomorphic to $G_\cK\times \cR$.
We have a $G_\cO$-action on $\tilde X_2$ changing the trivialization
$\kappa$ on $D$. This action is free and we have $X_2 = \tilde X_2/G_\cO$.
Let $q\colon\tilde X_2\to X_2$ denote the quotient map.
The map $p_1\colon X\to\cT$ induces $\tilde X\to\cT$, which is denoted
also by $p_1$.

Spaces and maps are summarized in the diagram as
\begin{equation*}
    \begin{CD}
        \cR \times\cR @<\tilde p<< p^{-1}(\cR\times\cR)
        @>\tilde q>> q(p^{-1}(\cR\times\cR)) = X_1\cap X_2
        @>\tilde m>> \cR
        \\
        @V{i}VV @V{i'}VV @VVV @VVV
        \\
        \cT\times \cR @<p<< \tilde X_2 @>q>>
        X_2 @>m>> \cT,
    \end{CD}
\end{equation*}
where we denote $(p_1,p_2)$ by $p$.

\begin{NB2}
A diagram including the space $X$:
\begin{equation*}
    \xymatrix{\cR \ar@{->}[d] && X_1 \ar@{->}[ll] \ar@{->}[d] & \\
      \cT && X \ar@{->}[ll] \ar@{->}[r] & \cT \\
      \cR & \tilde X_2 \ar@{->}[l] \ar@{->}[r] & X_2 \ar@{->}[u] &
    }
\end{equation*}
\end{NB2}
\end{NB}

\subsection{Abstract nonsense}\label{subsec:AN}

We need some preparatory material before we give a definition of the
convolution product.

\subsubsection{}\label{subsubsec:a}

Let $A$ be a complex of sheaves on a space $X$. Then we have a canonical
homomorphism $A\otimes \DD A\to \DC_X$. (See \cite[(8.3.17)]{CG}.)

\subsubsection{}\label{subsubsec:b}
We define the pull-back homomorphism with support (cf.\
\cite[8.3.21]{CG}).

Let
\begin{equation*}
    \begin{CD}
        M @>>f> N \\ @A{i}AA @AA{j}A \\ X @>\tilde f>> Y
    \end{CD}
\end{equation*}
be a Cartesian square. Let $A$, $B$ be complexes on $N$, $M$
respectively. For a homomorphism
\(
     \varphi\in \Hom(A,f_* B) = \Hom(f^* A,B),
\)
we define the {\it pull-back with support homomorphism\/} in
\(
     \Hom(j^! A, \tilde f_* i^! B)
\)
as the composite of
\begin{equation*}
    j^! A \to
    j^! f_* f^* A \cong \tilde f_* i^! f^* A
    \xrightarrow{\tilde f_* i^!\varphi} \tilde f_* i^! B,
\end{equation*}
where the first map is the adjunction and the middle isomorphism is
base change. It induces a homomorphism $H^*(j^!A)\to H^*(i^!  B)$ on
hypercohomology. We denote both homomorphisms by $f^*$. Note that it
depends on the map $f$ between $M$ and $N$, though complexes $j^!  A$,
$i^! B$ are on $Y$, $X$, and the map between $X$ and $Y$ is $\tilde
f$.

\subsubsection{}\label{subsubsec:c}

We define an `intersection pairing'. Let $A$ be a complex on a space
$X$. By \ref{subsubsec:a} above, we have
\begin{equation*}
    H^*(A)\otimes H^*(\DD A)\to H^*(\DC_X).
\end{equation*}
It is constructed as follows. Let $\Delta\colon X\to X\otimes X$ be
the diagonal embedding. We have $\Delta^*(A\boxtimes \DD A) = A\otimes
\DD A$. We have the adjunction homomorphism $H^*(A)\otimes H^*(\DD A)
= H^*(A\boxtimes\DD A) \to H^*(\Delta_* \Delta^*(A\boxtimes\DD A)) =
H^*(A\otimes \DD A)$. We now compose \ref{subsubsec:a}.

\subsection{Convolution product}\label{subsec:convolution-product}

We return back to \eqref{eq:12}.
The leftmost square is Cartesian. Thanks to the above definition, a
homomorphism $p^*\DC_{\cT\times\cR}\to \DC_{G_\cK\times\cR}$ induces
pull-back with support homomorphism $\DC_{\cR\times\cR} =
(i\times\id_\cR)^!\DC_{\cT\times\cR}\to \tilde p_* i^{\prime !}\DC_{G_\cK\times\cR} =
\tilde p_* \DC_{p^{-1}(\cR\times\cR)}$.
Therefore we want to understand $p^* \DC_{\cT\times\cR}$.
Let us write $p=(p_\cT,p_\cR)$ according to factors of $\cT\times\cR$. Then
$p^*\DC_{\cT\times\cR} = p_\cT^*\DC_\cT\otimes p_\cR^*\DC_\cR$.

\begin{Lemma}\label{lem:pull}
    We have isomorphisms of $G_\cO\times G_\cO$-equivariant complexes
\begin{equation}\label{eq:13}
    \begin{split}
    p_\cR^*\DC_{\cR} & \cong \CC_{G_\cK}\boxtimes\DC_{\cR},\\
    p_\cT^* \DC_{\cT} 
    & \cong \DC_{G_\cK}\boxtimes\CC_{\cR}[2\dim\bN_\cO-2\dim G_\cO],
\\
    p^* \DC_{\cT\times\cR}
    & \cong \DC_{G_\cK\times\cR}[2\dim\bN_\cO-2\dim G_\cO].
    \end{split}
\end{equation}
\end{Lemma}

Here the degree shift $[2\dim\bN_\cO-2\dim G_\cO]$ is understood by
taking finite dimensional approximation as in
\subsecref{subsec:equiv-borel-moore}. The same applies to degree shifts
appearing in the proof below.

\begin{proof}
    The first isomorphism is obvious as $p_\cR$ is just the projection
    to the second factor.

Note that $p_\cT$ factorizes as
\begin{equation*}
    G_\cK\times\cR \xrightarrow{\id_{G_\cK}\times\Pi} G_\cK\times\bN_\cO
    \xrightarrow{p'_\cT} \cT,
\end{equation*}
where $\Pi\colon \cR\to\bN_\cO$ is the natural projection, and
the second map $p'_{\cT}$ is the quotient by $G_\cO$. Since
$p'_{\cT}$ is a fiber bundle with smooth fibers,
\(
    p_\cT^{\prime*} \DC_{\cT} \cong 
    \DC_{G_\cK}\boxtimes\DC_{\bN_\cO}[-2\dim G_\cO].
\)
Since $\bN_\cO$ is smooth, we have $\DC_{\bN_\cO} =
\CC_{\bN_\cO}[2\dim\bN_\cO]$.
We pull back further by $\id_{G_\cK}\times\Pi$. Since the pull-back of the constant sheaf is again constant sheaf,
\begin{equation*}
    p_\cT^* \DC_{\cT} 
    \cong \DC_{G_\cK}\boxtimes\CC_{\cR}[2\dim\bN_\cO-2\dim G_\cO].
\end{equation*}

Finally a tensor product of any complex $A$ with the constant sheaf
is $A$ itself.
\begin{NB}
    Recall $\Delta^*(A\boxtimes B) = A\otimes B$, while
$\Delta^!(A\boxtimes B) = A\overset{!}\otimes B$.
\end{NB}%
Hence we obtain \eqref{eq:13}. It is an isomorphism of $G_\cO\times
G_\cO$-equivariant sheaves.
\begin{NB2}
    Why ? $\Delta$, $p'$, $p''$ are all $G_\cO\times G_\cO$
    equivariant, so all the sheaves appeared here are $G_\cO\times
    G_\cO$-equivariant sheaves. Then isomorphisms we have used are
    $G_\cO\times G_\cO$-equivariant.
\end{NB2}%
\end{proof}

\begin{NB}
    Here is the original argument:

Note that $p$ factorizes as
\begin{equation*}
    G_\cK\times\cR\xrightarrow{\Delta}
    G_\cK\times\cR\times G_\cK\times\cR
    \xrightarrow{p'}G_\cK\times\bN_{\cO}\times\cR
    \xrightarrow{p''} \cT\times\cR,
\end{equation*}
where $\Delta$ is the diagonal embedding, $p'$ is
$(g_1,[g_2,s],g_3,[g_4,s])\mapsto (g_3, g_2s, [g_4,s])$ and the final
map $p''$ is the quotient by $G_\cO$.
Since $p''$ is a fiber bundle with smooth fibers,
\(
   p^{\prime\prime*}\DC_{\cT\times\cR}
   = \DC_{G_\cK\times\bN_\cO\times\cR}[-2\dim G_\cO].
\)
\begin{NB2}
    Since $\dim G_\cO = \infty$, we need to justify what this mean.
\end{NB2}%
Since $\bN_\cO$ is smooth, we have $\DC_{\bN_\cO} =
\CC_{\bN_\cO}[2\dim\bN_\cO]$.
\begin{NB2}
    Since $\dim\bN_\cO = \infty$, we need to justify what this mean.
\end{NB2}%
Therefore
\begin{equation*}
    p^* \DC_{\cT\times\cR}
    \cong \Delta^* p^{\prime*}(\DC_{G_\cK\times\cR}\boxtimes
    \CC_{\bN_\cO}[2\dim\bN_\cO-2\dim G_\cO]).
\end{equation*}
Since $p'$ is the product of the identity on $G_\cK\times\cR$ and
the map $G_\cK\times\cR\to\bN_\cO$, we have
\(
   p^{\prime*}(\DC_{G_\cK\times\cR}\boxtimes
    \CC_{\bN_\cO})
    \cong \DC_{G_\cK\times\cR}\boxtimes \CC_{G_\cK\times\cR}.
\)
Finally a tensor product of any sheaf $A$ with the constant sheaf
is $A$ itself.
\begin{NB2}
    Recall $\Delta^*(A\boxtimes B) = A\otimes B$, while
$\Delta^!(A\boxtimes B) = A\overset{!}\otimes B$.
\end{NB2}%
Hence we obtain \eqref{eq:13}. It is an isomorphism of $G_\cO\times
G_\cO$-equivariant sheaves.
\begin{NB2}
    Why ? $\Delta$, $p'$, $p''$ are all $G_\cO\times G_\cO$
    equivariant, so all the sheaves appeared here are $G_\cO\times
    G_\cO$-equivariant sheaves. Then isomorphisms we have used are
    $G_\cO\times G_\cO$-equivariant.
\end{NB2}%
\end{NB}

\begin{NB}
    The factorization of $p'$ and $\cT\times\cR$ are not compatible:
    That is $p' = (G_\cK\times\cR \to \bN_\cO)\times (G_\cK\times
    \cR\xrightarrow{\id} G_\cK\times\cR)$, and
    $p''=(G_\cK\times\bN_\cO\to\cT)\times (\cR\to\cR)$, but we cannot
    compose $p'$ and $p''$ so that factorization is preserved, as the
    location of $G_\cK$ are different.
\end{NB}

Using \eqref{eq:13} as $\varphi$ in
\subsecref{subsec:AN}\ref{subsubsec:b}, we get the restriction with
support homomorphism for sheaves and their hypercohomology groups:
\begin{equation}\label{eq:7}
    \begin{gathered}
    p^* \colon \DC_{\cR\times\cR}\to \tilde p_* \DC_{p^{-1}(\cR\times\cR)}
    [2\dim\bN_\cO-2\dim G_\cO],
\\
    p^* \colon H^{G_\cO}_*(\cR)\otimes H^{G_\cO}_*(\cR)
    \to H^{G_\cO\times G_\cO}_{*}
    (p^{-1}(\cR\times\cR)).
    \end{gathered}
\end{equation}
Note that $G_\cK\times\cR$ is not smooth, hence it is different from
the usual restriction with support (e.g., \cite[8.3.21]{CG}). Our
definition uses a special form of $p$.
\begin{NB}
The following no longer applies: Apr. 27

This is a formal argument, but
becomes rigorous when we replace spaces by their finite dimensional
approximations.
\end{NB}

Let us check the degree for the second equation in
\eqref{eq:7}. Recall the degree of $H^{G_\cO}_*(\cR)$ is given
relative to $2\dim\bN_\cO$. Similarly the degree of $H^{G_\cO\times
  G_\cO}_{*}(p^{-1}(\cR\times\cR))$ is given relative to
$2\dim\bN_\cO+2\dim G_\cO$, as $p^{-1}(\cR\times\cR)$ is a closed
subvariety in $G_\cK\times \cR$, whose homology group is shifted by
that. The difference of the degrees is $
\begin{NB}
    2\times 2\dim\bN_\cO - (2\dim\bN_\cO+2\dim G_\cO) =
\end{NB}%
2\dim\bN_\cO-2\dim G_\cO$
appeared above. Therefore $p^*$ preserves the degree.

\begin{Remark}
    \lemref{lem:pull} gives 
    homomorphisms
    \begin{equation*}
        \begin{split}
            &p_\cT^*\colon H^{G_\cO}_*(\cR)\to
            H_{G_\cO\times G_\cO}^{*
              \begin{NB}
                  -2\dim G_\cO+2\dim\bN_\cO
              \end{NB}%
              }
            (i^{\prime!}A),
            \\
            & p_\cR^*\colon H^{G_\cO}_*(\cR)\to
            H_{G_\cO\times G_\cO}^{*}(\DD A),
        \end{split}
    \end{equation*}
    where $A$ denotes $\DC_{G_\cK}\boxtimes\CC_\cR$ for short. Then
    $\CC_{G_\cK}\boxtimes\DC_\cR = \DD A$.

    We compose $i^{\prime*} \colon H_{G_\cO\times G_\cO}^{*}(\DD A)\to
    H_{G_\cO\times G_\cO}^*(i^{\prime*}\DD A)$ with the intersection pairing
    in \subsecref{subsec:AN}\ref{subsubsec:c}, we get
    \begin{equation*}
        H_{G_\cO\times G_\cO}^{*}(i^{\prime!}A)
        \otimes H_{G_\cO\times G_\cO}^{*}(\DD A)
        \to
        H_{G_\cO\times G_\cO}^{*}(\DC_{p^{-1}(\cR\times\cR)}).
    \end{equation*}
    The intersection pairing is defined as $i^{\prime*}\DD A = \DD
    i^{\prime!} A$. Thus $p^*$ is the composition of $p_\cT^*\otimes
    i^{\prime*} p_\cR^*$ and the intersection pairing.
    \begin{NB}
        This is Sasha's original approach.
    \end{NB}%
\end{Remark}

We further have
\begin{equation*}
    H^{G_\cO\times G_\cO}_*(p^{-1}(\cR\times\cR))
    \xleftarrow[\cong]{\tilde q^*}
    H^{G_\cO}_{*}(q(p^{-1}(\cR\times\cR)))
\end{equation*}
as in \eqref{eq:15}.
\begin{NB}
    The degree of the RHS is relative to $2\dim\bN_\cO$. The dimension
    of the fiber is $\dim G_\cO$, so the degree is preserved.
\end{NB}%
Since $\tilde m$ is ind-proper,
\begin{NB}
    why ? OK.
\end{NB}%
we have the push-forward homomorphism for the Borel-Moore homology. We
define a convolution product
\begin{equation*}
    c_1\ast c_2 = \tilde m_*(\tilde q^*)^{-1} p^* (c_1\otimes c_2)
    \quad\text{for $c_1, c_2\in H^{G_\cO}_*(\cR)$}.
\end{equation*}
The degree is preserved, so the product preserves the grading.
\begin{NB}
Original formulation : Keep it as a record.

This is of degree $-2\dim\bN_\cO$. Therefore if we shift the degree as
\begin{equation*}
    H^{G_\cO}_{[*]}(\cR)
    \cong
    H^{G_\cO}_{2\dim\bN_\cO-*}(\cR),
\end{equation*}
it becomes a graded algebra.
\end{NB}

\begin{Remarks}\label{rem:Steinberg}
(1)    We have an isomorphism
\begin{equation*}
    G_\cK\times_{G_\cO}\cR \xrightarrow{\cong} \cT \times_{\bN_\cK}\cT ;
\qquad
    \bigl[g_1,[g_2,s]\bigr] \mapsto ([g_1, g_2 s], [g_1 g_2, s]).
\end{equation*}
Since $g_1(g_2 s) = (g_1 g_2)s$, it is a fiber product over $\bN_\cK$
as required.
\begin{NB}
    The inverse is given by $([g_1,s_1],[g_2',s_2])\mapsto \bigl[g_1,
    [g_1^{-1}g_2', s_2]\bigr]$. Since it is a point in the fiber
    product, we have $g_1s_1 = g_2' s_2$. Therefore $g_1^{-1} g_2' s_2
    = s_1$ is in $\bN_\cO$. Therefore $[g_1^{-1}g_2', s_2]$ is in $\cR$.
\end{NB}%
\begin{NB}
    Added on May 4.
\end{NB}%
The map $m$ is interpreted as the second projection $p_2$ from
$\cT\times_{\bN_\cK}\cT$, while the first projection $p_1$ is the
descent of $p_\cT$. (The remaining $p_\cR$ does not seem to have a
simple interpretation in terms of $\cT\times_{\bN_\cK}\cT$.)

According to \remref{rem:St}, $\cT\times_{\bN_\cK}\cT$ is analog of
$\St$. It is tempting to define a convolution product on
$H^{G_\cK}_*(\cT\times_{\bN_\cK}\cT)$ as usual, namely using three
projections from $\cT\times_{\bN_\cK}\cT\times_{\bN_\cK}\cT$ to
$\cT\times_{\bN_\cK}\cT$. But we are not sure whether even this
homology group is well-defined or has a well-defined convolution
product unless we identify it with $H^{G_\cO}_*(\cR)$. For example, we
do not know how to define a convolution on {\it nonequivariant\/}
homology group $H_*(\cT\times_{\bN_\cK}\cT)$.

(2) There is alternative approach to a definition of the convolution
product in \cite{MR3013034}. The Kashiwara affine flag manifold played
an essential role there. We do not check whether this approach can be
applicable to our setting, nor gives the same convolution product.

(3) Our definition of the convolution product can be modified for the equivariant $K$-theory as follows.

First of all, the equivariant $K$-group of $\cR$ is defined as the
limit of $K^{G_i}(\cR^d_{\le\la})$ as for the equivariant Borel-Moore
homology group. We use the pullback with respect to
$\tilde{p}^{d}_{e}\colon \cR^{d}_{\le\la}\to \cR^{e}_{\le\la}$ and
pushforward with respect to the embedding $\cR^d_{\le\mu}\to \cR^d_{\le\la}$. These are well-defined on equivariant $K$-theory, as $\tilde{p}^{d}_{e}$ is flat and the embedding is proper.
We now omit these $d$, $\le\la$ from the notation and treat $\cR$ (and
other spaces $\cT$, $G_\cO$) as if it would be a finite dimensional
scheme.

Looking back the definition of the product for homology groups, we
only need to replace \lemref{lem:pull} and the pull-back homomorphism
with support by appropriate arguments which make sense for $K$-theory.

Take $E\in K^{G_O}(\cR)$. We consider it as a class of an object in
$D^b_{G_\cO}(\mathrm{Coh}(\cT))$ whose cohomology groups are supported
in $\cR$. We replace $p_{\cT}^{\prime *}(E)$ by its resolution,
consisting of sheaves which are flat over $\bN_\cO$. This is possible
since $\bN_\cO$ is smooth and \emph{finite dimensional} as it is
actually a truncation of the infinite dimensional $\bN_\cO$.
Now we further pull back $p_{\cT}^{\prime*}(E)$ by $\id\times \Pi$. It is
a complex defined over $G_\cK\times\cR$ consisting of sheaves which
are flat over $\cR$.

Taking another $F\in D^b_{G_\cO}(\mathrm{Coh}(\cR))$, we consider
$p_{\cR}^*(F)$, which is flat over $G_{\cK}$. Then thanks to the
flatness, $(\id\times \Pi)^* (p_{\cT}^{\prime*}(E)) \otimes^L
p_{\cR}^*(F)$ has only finitely many higher $\mathrm{Tor}$, hence a
well-defined object in $D^b_{G_O\times G_O\times
  G_O}(\mathrm{Coh}(G_{\cK}\times\cR))$.
Moreover its cohomology groups are supported on
$p^{-1}(\cR\times\cR)$.

This operation sends a distinguished triangle (either for $E$ or $F$) to a distinguished triangle. Therefore descends to the equivariant $K$-theory.

When $\bN = \g$, the adjoint representation of $G$, one can easily
check that this definition coincides with one in
\cite[\S7]{MR2135527}.
\begin{NB}
    The ambient spaces used in here and \cite{MR2135527} are
    different. We use $G_\cK\times\bN_\cO\times\cR$, while
    \cite{MR2135527} uses $\cT\times\cT$. The difference between the
    last factor $\cR$ vs $\cT$ is not essential:

We view $D^b(\mathrm{Coh}(\mathcal R))$ as objects on $\mathcal T$ whose cohomology groups are supported in $\mathcal R$, as analog of relative cohomology groups. Let us denote the latter by $D^b_{\mathcal R}(\mathrm{Coh}(\mathcal T))$.

Because $p_{\mathcal R}$ is just the projection to a factor, is flat. Therefore two pullbacks by $G_{\mathcal K}\times\mathbf N_O\times \mathcal R\to \mathcal R$ for $D^b(\mathrm{Coh}(\mathcal R))$ and $G_{\mathcal K}\times\mathbf N_O\times \mathcal T\to \mathcal T$ for $D^b_{\mathcal R}(\mathrm{Coh}(\mathcal T))$ are the same, under the equivalence $D^b(\mathrm{Coh}(\mathcal R))\cong D^b_{\mathcal R}(\mathrm{Coh}(\mathcal T))$.

Now let us compare $G_\cK\times\bN_\cO\times\cT$ and $\cT\times\cT$. We have
a commutative diagram
\begin{equation*}
    \begin{CD}
        p^{-1}({\mathcal R}\times{\mathcal R}) @>{q\circ i'}>>
        G_\cK\times_{G_\cO}\cR = \widetilde{\Lambda}
        \\
        @V{(\id_{G_\cK}\times\Pi\times (i\circ p_\cR))\circ i'}VV
        @VV{p_\cT\times m}V
        \\
        G_K\times{\mathbf N}_\cO\times{\mathcal T}
        @>>>
        \cT\times\cT,
    \end{CD}
\end{equation*}
where $p_\cT$ is $[g_1,[g_2,s]]\mapsto [g_1, g_2 s]$, descent of the
original $p_\cT$ through $q$. The bottom arrow is
$(g,s',[h,s''])\mapsto ([g,s'], [gh, s''])$.  If we replace
$\cT\times\cT$ by $G_\cK\times\bN_\cO\times\cT$, it is an isomorphism,
as the inverse is $(g,s', [g^{-1}g'',s'']) \leftmapsto (g,s',[g'',s''])$.
Therefore derived tensor products over $G_\cK\times\bN_\cO\times\cT$ vs
$\cT\times\cT$ yield the same result.
\end{NB}%

(4) Note that $\bN_\cK$ is a representation of $G_\cK$, and $\bN_\cO$
is a $G_\cO$-subrepresentation. Then $H^{G_\cO}_*(\cR)$ is an infinite
dimensional example of \emph{Springer theory} formulated in
\cite{2013arXiv1307.0973S}, as a uniform way to look at various
examples of geometric constructions of noncommutative algebras via
convolution. Namely for a reductive group $\mathbf G$, a parabolic
subgroup $\mathbf P$, a representation $V$ of $\mathbf G$, and a
$\mathbf P$-subrepresentation $V'$ of $V$, we introduce a variety
$\mathbf R \defeq \{ [g,s]\in \mathbf G\times_P V'\mid gs\in V' \}$
with $\mathbf P$-action given by $p[g,s] = [pg,s]$. Then
$H^{\mathbf P}_*(\mathbf R)$ is an associative algebra. We take
$(\mathbf G,\mathbf P, V,V') = (G_\cK, G_\cO, \bN_\cK, \bN_\cO)$ for
our $\cR$. Another infinite dimensional examples of Springer theory,
which are cousins of our $\cR$ for quiver gauge theories, are given
recently in \cite[\S4]{2016arXiv161110216B,2016arXiv160905494W}.
\end{Remarks}

\begin{NB}
The following is the original approach.

Recall that $\cT$ is the quotient $G_\cK\times_{G_\cO}\bN_\cO$, and we have a morphism $\pi\colon \cT\to\bN_\cK$ such that $\cR = \pi^{-1}(\bN_\cO)$.
We consider the induced space
\(
     G_\cK\times_{G_\cO}\cR.
\)
It consists of $\bigl[g_1, [g_2,s]\bigr]$ with $g_1\in G_\cK$,
$[g_2,s]\in \cR\subset \cT = G_\cK\times_{G_\cO}\bN_\cO$. We have
$\bigl[g_1, [g_2,s]\bigr] = \bigl[g_1 b, [b^{-1}g_2,s]\bigr]$ for $b\in
G_\cO$. We have an isomorphism
\begin{equation*}
    G_\cK\times_{G_\cO}\cR \xrightarrow{\cong} \cT \times_{\bN_\cK}\cT ;
\qquad
    \bigl[g_1,[g_2,s]\bigr] \mapsto ([g_1, g_2 s], [g_1 g_2, s]).
\end{equation*}
Since $g_1(g_2 s) = (g_1 g_2)s$, it is a fiber product over $\bN_\cK$
as required.
\begin{NB2}
    The inverse is given by $([g_1,s_1],[g_2',s_2])\mapsto \bigl[g_1,
    [g_1^{-1}g_2', s_2]\bigr]$. Since it is a point in the fiber
    product, we have $g_1s_1 = g_2' s_2$. Therefore $g_1^{-1} g_2' s_2
    = s_1$ is in $\bN_\cO$. Therefore $[g_1^{-1}g_2', s_2]$ is in $\cR$.
\end{NB2}%
As a moduli space, $G_\cK\times_{G_\cO}\cR\cong\cT\times_{\bN_\cK}\cT$
parametrizes pairs $(\scP_1,\varphi_1,s_1), (\scP_2,\varphi_2,s_2)$ of
framed $G$-bundles with sections of their associated bundles such that
$(\varphi_1)_\bN(s_1) = (\varphi_2)_\bN(s_2)$.

We have two projections
\begin{equation}\label{eq:4}
    \cT \xleftarrow{p_1}\cT\times_{\bN_\cK}\cT \xrightarrow{p_2} \cT,
\end{equation}
and $\cT\times_{\bN_\cK}\cT$ is a closed subscheme of a nonsingular
scheme $\cT\times\cT$.
\begin{NB2}
    Since the base $G_\cK/G_\cO$ is not smooth, this is not true. It
    is here probably we need the Kashiwara flag manifold.
\end{NB2}

Let $c_1,c_2\in H^{G_\cO}_*(\cR)$. We consider $c_2$ as a class in
$H^{G_\cO}_*(\cT\times_{\bN_\cK}\cT)$ via
\begin{equation*}
    1 \otimes c_2\in H^{G_\cO\times G_\cO}_*(G_\cK\times \cR)
    \cong H^{G_\cO}_*(G_{\cK}\times_{G_{\cO}}\cR).
\end{equation*}
We define the convolution product of $c_1$, $c_2$ by
\begin{equation*}
    c_1\ast c_2 = p_{2*}(p_1^*(c_1)\cap c_2).
\end{equation*}
Here the intersection product with support is defined by using the
smooth scheme $\cT\times\cT$ as an ambient space.
\begin{NB2}
    Here the smoothness seems crucial. I am not sure whether I can
    avoid the intersection product as in the case of $\Gr_G$.
\end{NB2}%
As $p_2(p_1^{-1}(\cR)\cap (\cT\times_{\bN_\cK}\cT))\subset\cR$, $c_1\ast
c_2$ is a cycle in $\cR$.
\begin{NB2}
    We also need $p_2$ to be proper.
\end{NB2}%
\end{NB}

\begin{Theorem}\label{thm:convolution}
    The convolution product $\ast$ defines an associative graded
    algebra structure on $H^{G_\cO}_*(\cR)$. The unit is given by the
    fundamental class of the fiber of $\cR\to\Gr_G$ at the base point
    $[1]\in \Gr_G$. The multiplication is
    $H_{G_\cO}^*(\mathrm{pt})$-linear in the first variable. The same
    assertions are true for $H^{G_\cO\rtimes\CC^\times}_*(\cR)$.
\end{Theorem}

\begin{NB}
    Two diagrams \eqref{eq:1} and \eqref{eq:12} are related as
    \begin{equation*}
        \begin{CD}
            \cR \times\cR
            @<\tilde p<<
            p^{-1}(\cR\times\cR)
            @>\tilde q>>
            q(p^{-1}(\cR\times\cR))
            @>\tilde m>> \cR
\\
    @V{i}VV @V{i'}VV @VVV @VVV
\\
            \cT\times \cR @<p<< G_\cK\times\cR
            @>q>> G_\cK\times_{G_\cO}\cR @>m>> \cT
\\
    @VVV @VVV @VVV @VVV
\\
            \Gr_G\times\Gr_G @<<\bar p< G_\cK\times\Gr_G 
            @>>\bar q> \Gr_G\tilde\times\Gr_G @>>\bar m> \Gr_G.
        \end{CD}
    \end{equation*}
\end{NB}

Since we will prove that $H^{G_\cO}_*(\cR)$ is commutative, the
multiplication turns out to be linear in both first and second
variables. However it is not true for
$H^{G_\cO\rtimes\CC^\times}_*(\cR)$. See the computation in
\subsecref{sec:triv-properties}\ref{item:reduction} below.

\begin{proof}[Proof of \thmref{thm:convolution}]
    We prove the assertions for $G_\cO$ for notational simplicity. All
    the arguments work also for $G_\cO\rtimes\CC^\times$.

    The last assertion is clear from the definition.

    Let $e$ denote the fundamental class of the fiber of $\cR\to\Gr_G$
    at $[1]\in \Gr_G$. We prove
    \begin{equation*}
        e\ast \bullet = \id = \bullet\ast e
        \quad\text{on $H^{G_\cO}_*(\cR)$}.
    \end{equation*}
    Consider $e\ast \bullet$. The class $(\tilde q^*)^{-1}p^*(e\otimes
    \bullet)$ is given by the pushforward homomorphism
    $H^{G_\cO}_*(\cR)\to H^{G_\cO}_*(q(p^{-1}(\cR\times\cR)))$ with
    respect to the embedding $\cR\ni [g,s]\mapsto
    \left[\id,[g,s]\right]\in q(p^{-1}(\cR\times\cR))$. If we compose
    $\tilde m$, the embedding becomes just $\id_\cR$. Therefore $e\ast
    \bullet = \id$. Similarly $(\tilde q^*)^{-1}p^*(\bullet\otimes e)$ is
    given by the pushforward homomorphism of the embedding $\cR\ni
    [g,s]\mapsto \left[g,[\id,s]\right]\in
    q(p^{-1}(\cR\times\cR))$. If we compose $\tilde m$, it becomes
    $\id_\cR$ again. Hence $\bullet\ast e = \id$.

    It remains to prove the associativity.
    \begin{NB}
        For the affine Grassmannian, we consider
        \begin{equation*}
            \Gr_G\times \Gr_G\times\Gr_G
            \xleftarrow{p} G_\cK\times G_\cK\times\Gr_G
            \xrightarrow{q} G_\cK\times_{G_\cO}G_\cK\times_{G_\cO}\Gr_G
            \xrightarrow{m} \Gr_G,
        \end{equation*}
        given by
        \begin{equation*}
            ([g_1],[g_2],[g_3])\leftmapsto
            (g_1,g_2,[g_3])\to \left[ g_1, g_2, [g_3]\right]
            \to [g_1 g_2 g_3].
        \end{equation*}
        The last map $m$ factors in two ways
        \begin{equation*}
            \begin{CD}
                G_\cK\times_{G_\cO}G_\cK\times_{G_\cO}\Gr_G @>>>
                G_\cK\times_{G_\cO}\Gr_G \\
                @VVV @VVV \\
                G_\cK\times_{G_\cO}\Gr_G @>>> \Gr_G,
            \end{CD}
        \end{equation*}
        where the upper horizontal arrow is $\left[ g_1, g_2,
          [g_3]\right] \to \left[g_1 g_2, [g_3]\right]$, the left
        vertical arrow is $\left[ g_1, g_2, [g_3]\right] \to
        \left[g_1, [g_2 g_3]\right]$, and lower horizontal and right
        vertical arrows are multiplication maps in the original
        diagram \eqref{eq:1}. Then it is clear that $(A_1\ast A_2)\ast
        A_3 = m_* (q^*)^{-1} p^*(A_1\boxtimes A_2\boxtimes A_3) =
        A_1\ast(A_2\ast A_3)$.
    \end{NB}%

    \begin{NB}
        Note that $\cT\times_{\bN_\cK}\cT\times_{\bN_\cK}\cT \cong
        G_\cK\times_{G_\cO}(\cR\times_{\bN_\cO}\cR)$ with $
        \cR\times_{\bN_\cO}\cR = \{ ([g_1,s_1],
        [g_2,s_2])\in\cR\times\cR \mid g_1 s_1 = g_2 s_2 \}$.
        The isomorphism is given by
        \(
            G_\cK\times_{G_\cO}(\cR\times_{\bN_\cO}\cR)
            \ni [g_1, [g_2,s_2],[g_3,s_3]]
            \mapsto
            ([g_1,g_2 s_2], [g_1g_2, s_2], [g_1 g_3, s_3])
            \in \cT\times_{\bN_\cK}\cT\times_{\bN_\cK}\cT
        \)
        and its inverse
        \(
           ([g_1,s_1],[g_2,s_2],[g_3,s_3]) \mapsto
           [g_1, [g_1^{-1} g_2, s_2], [g_1^{-1}g_3, s_3]].
        \)
        $p_{13}\colon \cT\times_{\bN_\cK}\cT\times_{\bN_\cK}\cT\to
        \cT\times_{\bN_\cK}\cT$ corresponds to the second projection
        $p_2 \colon\cR\times_{\bN_\cO}\cR\to \cR$, while $p_{12}$
        corresponds to $p_1$.
    \end{NB}%

    We consider the following commutative diagram, which is a
    `product' of two copies of the upper row of \eqref{eq:12}:
    \begin{equation}\label{eq:70}
    \xymatrix@C=20pt{
      \cR\times\cR& 
      p^{-1}(\cR\times\cR) \ar[l]_-{\tilde p} \ar[r]^-{\tilde q}
      & q(p^{-1}(\cR\times\cR)) \ar[r]^-{\tilde m} & \cR
      \\
      q(p^{-1}(\cR\times\cR))\times\cR \ar[u]^-{\tilde m\times\id_\cR}&
      \boxed{3} \ar[l] \ar[u] \ar[r] & \boxed{4} \ar[u] \ar[r] &
      q(p^{-1}(\cR\times\id_\cR)) \ar[u]_{\tilde m}
      \\
      p^{-1}(\cR\times\cR)\times\cR \ar[u]^{\tilde q\times\id_\cR}
      \ar[d]_{\tilde p\times\id_\cR}
      & \boxed{1} \ar[l] \ar[r] \ar[u] \ar[d]
      & \boxed{2} \ar[u] \ar[r] \ar[d] 
      & p^{-1}(\cR\times\cR) \ar[u]_{\tilde q} \ar[d]^{\tilde p}
      \\
      \cR\times\cR\times\cR 
      & \cR\times p^{-1}(\cR\times\cR) 
      \ar[l]^-{\id_\cR\times \tilde p}
      \ar[r]_-{\id_\cR\times \tilde q}
      & \cR\times q(p^{-1}(\cR\times\cR)) 
      \ar[r]_--{\id_\cR\times \tilde m}
      & \cR\times\cR,
  }
    \end{equation}
    where
    \begin{equation*}
        \boxed{1} = \{ (g_1,g_2,[g_3,s])\in G_\cK\times G_\cK\times \cR \mid
        g_2 g_3 s, g_1 g_2 g_3 s\in \bN_\cO \},
    \end{equation*}
    and $\boxed{2}$, $\boxed{3}$, $\boxed{4}$ are quotients of
    $\boxed{1}$ by $1\times G_\cO$, $G_\cO\times 1$, $G_\cO\times
    G_\cO$ respectively. Here $G_\cO\times G_\cO$ acts on $\boxed{1}$ by
    \begin{equation*}
        (h_1,h_2)\cdot (g_1,g_2, [g_3,s]) = (g_1 h_1^{-1}, h_1 g_2 h_2^{-1},
        [h_2 g_3, s])\quad\text{for $(h_1,h_2)\in G_\cO\times G_\cO$}.
    \end{equation*}
    Horizontal and vertical arrows from $\boxed{1}$, $\boxed{4}$ are
    given by
    \begin{equation}\label{eq:14}
        \xymatrix@C=10pt{(g_1, [g_2,g_3 s], [g_3,s]) &
          (g_1,g_2, [g_3,s])\in\boxed{1} \ar@{|->}[l]
        \ar@{|->}[d] \\
        & ([g_1,g_2g_3 s],(g_2,[g_3,s])),
      } \quad
      \xymatrix@C=10pt{
        [g_1g_2, [g_3,s]] &
\\
        \boxed{4}\ni [g_1, [g_2,[g_3, s]]] \ar@{|->}[u] \ar@{|->}[r] &
        [g_1, [g_2g_3, s]].
      }
    \end{equation}
    Arrows from $\boxed{2}$, $\boxed{3}$ are given by trivial
    modification of above ones, as $\boxed{1}\to\boxed{3}$, etc.\ are
    fiber bundles.

    The convolution product $c_1\ast (c_2\ast c_3)$ is given by
    applying induced homomorphisms in the bottom row from left to
    right, and then going up in the rightmost column. Similarly
    $(c_1\ast c_2)\ast c_3$ is given by going the leftmost column and
    the top row. Therefore the associativity follows if we show that
    arrows induce appropriate pull-back or push-forward homomorphisms,
    and they form a commutative diagram for each square.

    Let us first look at the bottom left square. We can extend the
    square to a cube as
\begin{equation*}
    \xymatrix@!C=100pt{
& G_\cK\times\cR\times\cR \ar@{->}'[d]^{p\times\id_\cR}[dd]
& & G_\cK\times p^{-1}(\cR\times\cR) \ar@{->}[dd]^P\ar@{->}[ll]_{\id_{G_\cK}\times\tilde p}
\\
p^{-1}(\cR\times\cR)\times\cR \ar@{->}[ur]\ar@{->}[dd]
& & \boxed{1} \ar@{->}[ur]\ar@{->}[ll]\ar@{->}[dd]
\\
& \cT\times\cR\times\cR
& & \cT\times p^{-1}(\cR\times\cR) \ar@{->}'[l]_{\id_\cT\times\tilde p}[ll]
\\
\cR\times\cR\times\cR\ar@{->}[ur]
& & \cR\times p^{-1}(\cR\times\cR) \ar@{->}[ll]\ar@{->}[ur]
}
\end{equation*}
Arrows from spaces in the front square to those in the rear square are
closed embeddings. Arrows in the rear square are as indicated, and one
remaining $P\colon G_\cK\times p^{-1}(\cR\times\cR)\to \cT\times
p^{-1}(\cR\times\cR)$ is given by the formula for the corresponding map in the front square in \eqref{eq:14}.

Homomorphisms between dualizing complexes $\DC$ and their pull-backs
have been already constructed for $p$ and $\tilde p$ in \eqref{eq:13}
and \eqref{eq:7} respectively. For $P$, we decompose $P =
(P_\cT,P_{p^{-1}(\cR\times\cR)})$ and see that
\begin{equation*}
    \begin{split}
    P_{p^{-1}(\cR\times\cR)}^*\DC_{p^{-1}(\cR\times\cR)} & \cong
    \CC_{G_\cK}\boxtimes \DC_{p^{-1}(\cR\times\cR)},
\\
    P_\cT^*\DC_\cT & \cong \DC_{G_\cK}\boxtimes \CC_{p^{-1}(\cR\times\cR)}
    [2\dim\bN_\cO-2\dim G_\cO],
    \end{split}
\end{equation*}
as in \lemref{lem:pull}. We then construct pull-back homomorphisms
with support (\subsecref{subsec:AN}\ref{subsubsec:b}) for the front
square, as the top and bottom squares are cartesian.

In order to show commutativity of pull-back homomorphisms, it is
enough to consider the rear square by the construction of pull-back homomorphisms with support.
Let us factor $\DC_{\cT\times\cR\times\cR} =
\DC_\cT\boxtimes\DC_{\cR\times\cR}$ as before, and consider the
pull-backs of $\DC_\cT$ and $\DC_{\cR\times\cR}$ separately.

Let us first consider $\DC_{\cR\times\cR}$. We have two homomorphisms
\begin{equation*}
    \begin{CD}
     P^* (\id_\cT\times\tilde p)^* (\CC_\cT\boxtimes\DC_{\cR\times\cR})
     @>>> \CC_{G_\cK}\boxtimes\DC_{p^{-1}(\cR\times\cR)}[2\dim\bN_\cO-2\dim G_\cO]
\\
   @| @|
\\
     (p\times\id_\cR)^*(\id_{G_\cK}\times\tilde p)^*
     (\CC_\cT\boxtimes\DC_{\cR\times\cR})
     @>>> \CC_{G_\cK}\boxtimes\DC_{p^{-1}(\cR\times\cR)}[2\dim\bN_\cO-2\dim G_\cO]
    \end{CD}
\end{equation*}
by following left, top arrows and bottom, right arrows. They are the
same, as both are essentially given by $p^*\colon \tilde p^*
\DC_{\cR\times\cR}\to \DC_{p^{-1}(\cR\times\cR)}[2\dim\bN_\cO-2\dim
G_\cO]$ constructed in \eqref{eq:7}.

Next consider $\DC_{\cT}$. The $\cT$-component of
\(
(\id_\cT\times\tilde p)\circ P =
(\id_{G_\cK}\times\tilde p)\circ(p\times\id_\cR)
\)
\begin{NB}
    which is $(g_1,g_2, [g_3,s])\mapsto [g_1, g_2g_3 s]$,
\end{NB}%
factors as
\[
   G_\cK\times p^{-1}(\cR\times\cR)
   \xrightarrow{\id_{G_\cK}\times\Pi'}
   G_\cK\times\bN_\cO\xrightarrow{p'_\cT}\cT,
\]
where $\Pi'\colon p^{-1}(\cR\times\cR)\to\bN_\cO$ is $(g_2, [g_3,
s])\mapsto g_2g_3 s$. As in the proof of \lemref{lem:pull}, we have
\begin{equation*}
    ((\id_\cT\times\tilde p)\circ P)^*(\DC_\cT\boxtimes\CC_{\cR\times\cR})
    \cong \DC_{G_\cK}\boxtimes\CC_{p^{-1}(\cR\times\cR)}
    [2\dim\bN_\cO-2\dim G_\cO].
\end{equation*}
Two homomorphisms which are constructed by going along left, top arrows and
bottom, right arrows are the same, as they are constructed in the same
way. This completes the proof of the commutativity at the bottom left
square.

Since $\tilde q\colon p^{-1}(\cR\times\cR)\to q(p^{-1}(\cR\times\cR))$
is a fiber bundle with fibers $G_\cO$, commutativity for squares
involving $\tilde q$ is obvious. Let us consider the right bottom
square. We extend it to a cube:
\begin{equation*}
    \xymatrix@!C=100pt{
& G_\cK\times q(p^{-1}(\cR\times\cR)) \ar@{->}'[d]^{P'}[dd]
\ar@{->}[rr]_{\id_{G_\cK}\times\tilde m}
&& G_\cK\times\cR \ar@{->}[dd]^p
\\
\boxed{2} \ar@{->}[ur]\ar@{->}[rr]\ar@{->}[dd] && p^{-1}(\cR\times\cR)
\ar@{->}[dd] \ar@{->}[ur]
&
\\
& \cT\times q(p^{-1}(\cR\times\cR)) \ar@{->}'[r][rr]_{\id_\cT\times\tilde m}
&& \cT\times\cR
\\
\cR\times q(p^{-1}(\cR\times\cR)) \ar@{->}[rr]\ar@{->}[ur]
&& \cR\times\cR\ar@{->}[ur]
}
\end{equation*}
Arrows from the front to rear are closed embeddings. The map $P'\colon  G_\cK\times q(p^{-1}(\cR\times\cR)) \to \cT\times q(p^{-1}(\cR\times\cR))$ is given by the formula in \eqref{eq:14}.
\begin{NB}
    $(g_1, [g_2, [g_3, s])\mapsto ([g_1, g_2g_3 s], [g_2, [g_3, s]])$.
\end{NB}%
Recall that the pull-back with support homomorphism $p^*$ from
$\cR\times \cR$ to $p^{-1}(\cR\times\cR)$ is defined via
$p$. Similarly the pull-back from $\cR\times q(p^{-1}(\cR\times\cR)))$
to $\fbox{2}$ is defined via $P'$. Therefore it is enough to check the commutativity in the rear square, i.e.,
\[
   p^* (\id_{\cT}\times\tilde m)_!
   = (\id_{G_\cK}\times \tilde m)_! P^{\prime*}
   \colon H^{G_\cO\times G_\cO}_*(\cT\times q(p^{-1}(\cR\times\cR)))
   \to H^{G_\cO\times G_\cO}_*(G_\cK\times\cR)
\]
(with appropriate degree shift). But this follows from the base
change, as the rear square is cartesian.

The commutativity of the right top square is clear, as it involves
only pushforward homomorphisms.
\end{proof}

\subsection{Definition of the Coulomb branch}\label{sec:commutativity}

Let us denote $(H_*^{G_\cO}(\cR),\ast)$ by $\cA$ or $\cA[G,\bN]$ when
we want to emphasize $(G,\bN)$.
We will prove that $\cA$ is commutative later. Therefore we can
consider its spectrum. Here is our main proposal:

\begin{Definition}\label{def:main}
    We define the {\it Coulomb branch\/} $\mathcal M_C$ (as an affine
    scheme) by
    \begin{equation*}
        \mathcal M_C \equiv \mathcal M_C(G,\bN) \defeq
        \Spec \cA.
    \end{equation*}
\end{Definition}

If we add the loop rotation, the equivariant homology group
$(H^{G_\cO\rtimes\CC^\times}_*(\cR),\ast)$ is a noncommutative
deformation of the Coulomb branch $\mathcal M_C$.
Let us denote it by $\cAh$ or $\cAh[G,\bN]$. We call it the {\it
  quantized Coulomb branch}.
In particular,
$\mathcal M_C$ has a natural Poisson structure.

We will show that $\cA$ is finitely generated and integral later (see
\propref{prop:fg}, \corref{cor:integral}). We also prove that $\cA$ is
normal (\propref{prop:normality}). It could be compatible with the
following example in \cite[Table~10]{Cremonesi:2014uva}: Take a
nilpotent orbit $\mathcal O$ whose Lusztig-Spaltenstein dual (in the
Langlands dual Lie algebra) is not normal. Take a gauge theory whose
Higgs branch is the intersection of the nilpotent cone and the Slodowy
slice to $\mathcal O$. Such a gauge theory exists for classical groups
(see \cite[Appendix A]{2015arXiv150303676N}). A naive guess gives us
Lusztig-Spaltenstein dual of $\mathcal O$ as the Coulomb branch, hence
is not normal. But \cite[Table~10]{Cremonesi:2014uva} suggests us the
normalization of the Lusztig-Spaltenstein dual instead.
Note however that they are \emph{not} of cotangent type, hence our
construction does not apply. Therefore it is a little early to make
any conclusion, but it seems natural to conjecture that $\cA$ is
normal even if not necessarily of cotangent type.

In many examples, $\cA$ is Cohen-Macaulay. We will use these
properties crucially in \ref{QGT} and we do not have counter-examples
at this moment. Moreover we show that the Poisson structure is
symplectic on the regular locus of $\mathcal M_C$
(\propref{prop:symplectic}) and believe that there is a hyper-K\"ahler
structure there. This is what physicists have expected. We
optimistically conjecture that $\mathcal M_C$ has only
\emph{symplectic singularities} \cite{MR1738060}, i.e., the symplectic
form on the smooth locus extends to a holomorphic $2$-form on a
resolution $\tilde{\mathcal M}_C\to \mathcal M_C$.

We will give two proofs of the commutativity of $\cA$. The first proof
is a reduction to the abelian case, and is given in \thmref{abel} and \propref{prop:commutative}.

The second proof is a well-known argument: Using the
Beilinson-Drinfeld Grassmannian to deform a situation where the
product $c_1\ast c_2$ is manifestly symmetric under
$c_1\leftrightarrow c_2$. Then we use nearby cycle functors and (dual)
{\it specialization\/} homomorphism. In fact, we will prove a
commutativity at the level of an object in the $G_\cO$-equivariant
derived category of constructible sheaves on $\Gr_G$. See
\cite[\cref{affine_pre-sec:sheav-affine-grassm,affine_pre-sec:commute}]{affine}
\begin{NB}
    Check an external reference !
\end{NB}%
for more detail.

\begin{Remark}
    In the same way we define the \emph{$K$-theoretic Coulomb branch}
    as $\mathcal M_C^K \defeq \Spec K^{G_\cO}(\cR)$ thanks to
    \remref{rem:Steinberg}(3). The first
    proof of the commutativity is a reduction to abelian cases, and
    hence works for $K^{G_\cO}(\cR)$. (The second proof seems
    only applicable for homology groups.) The proofs that $\scA$ is finitely
    generated and integral work for $K^{G_\cO}(\cR)$. 
    The proof of the normality
    does not work, as we will use the grading in \lemref{lem:SL2}(1).
\end{Remark}

\begin{NB}
    Let us describe the analog of $G_\cK\times_{G_\cO}\cR$. Let $X$ be
    a smooth curve. Then analog is the moduli space of
    \begin{itemize}
          \item $(x_1,x_2)\in X^2$,
          \item $\scP_1$, $\scP$ : $G$-bundles on $X$,
          \item $s$ : a section of $\scP_{\bN}$,
          \item $\varphi_1$ : a trivialization of $\scP_1$ on
        $X\setminus \{x_1\}$, and
          \item $\eta\colon \scP_1|_{X\setminus\{x_2\}}\xrightarrow{\cong}
    \scP|_{X\setminus\{x_2\}}$ such that
    $\eta^{-1}_\bN(s)\in H^0(\scP_{1,\bN})$.
    \end{itemize}
    For $G_\cK\times \cR$ :
    \begin{itemize}
          \item $(x_1,x_2)\in X^2$,
          \item $\scP_1$, $\scP_2$ : $G$-bundles on $X$,
          \item $s_2$ : a section of $\scP_{2,\bN}$,
          \item $\varphi_i$ : a trivialization of $\scP_i$ on
        $X\setminus \{x_i\}$ ($i=1,2$) such that
        $\varphi_{2,\bN}(s_2)\in \bN(X)$, and
          \item $\kappa$ : a trivialization of $\scP_1$ on
        $\widehat{X}_{x_2}$, the formal neighborhood of $x_2$ in $X$.
    \end{itemize}
    We define $q$ by setting $\scP$ as the $G$-bundle gotten by gluing
    $\scP_1$ on $X\setminus\{x_2\}$ and $\scP_2$ on
    $\widehat{X}_{x_2}$ using $\varphi_2^{-1}\circ
    \kappa\colon\scP_1\to\scP_2$ over $(X\setminus\{x_2\})\cap
    \widehat{X}_{x_2}$.
    \begin{NB2}
        In the formal disk case, we take $\scP = \scP_2$ as $\kappa$
        is a `global' trivialization.
    \end{NB2}
\end{NB}%

\subsection{Grading and group action}\label{sec:grading}
(See \cite[\S4(iv), Properties (a,c)]{2015arXiv150303676N} for
original sources in physics.)

\begin{NB}
    The following is an original consideration of the grading, but it
    is now already built in the construction.

We have homological grading on $H^{G_\cO}_*(\cR)$, but it
is not compatible with the convolution product.
It is natural to set the grading on
$H^{G_\cK}_{*}(\cT\times_{\bN_\cK}\cT)$ as $2\dim_\CC\cT-*$
so that the convolution algebra becomes a graded ring. For example,
the diagonal is unit, and has degree $0$. Hence we set
\begin{equation*}
    H^{G_\cK}_{[*]}(\cT\times_{\bN_\cK}\cT)
    = H^{G_\cK}_{2\dim_\CC\dim\cT-*}(\cT\times_{\bN_\cK}\cT).
\end{equation*}
For the space $\cR$, we set
\begin{equation*}
    H^{G_\cO}_{[*]}(\cR)
    \cong
    H^{G_\cO}_{2\dim_\CC\dim\cT-2\dim_\CC G_\cK/G_\cO-*}(\cR).
\end{equation*}
We formally calculate
\begin{equation*}
    \dim_\CC\dim\cT - \dim_\CC G_\cK/G_\cO
    = \dim \bN_\cO.
\end{equation*}
This is infinite, but we consider cycles whose degrees differ from
$\dim\bN_\cO$ finitely.
\end{NB}%

Recall that connected components of $\Gr_G$ are parametrized by
$\pi_1(G)$. Since $\cR$ is homotopy equivalent to $\Gr_G$, we also
have $\pi_0(\cR)\cong \pi_1(G)$. Thus we have a decomposition
\begin{equation*}
    H^{G_\cO}_*(\cR) \cong \bigoplus_{\gamma}
    H^{G_\cO}_*(\cR^\gamma),
\end{equation*}
where $\cR^\gamma$ is the connected component corresponding to
$\gamma\in\pi_1(G)$. This decomposition is compatible with the convolution product:
\(
    H^{G_\cO}_*(\cR^{\gamma_1}) \ast
    H^{G_\cO}_*(\cR^{\gamma_2})
    \subset H^{G_\cO}_*(\cR^{\gamma_1+\gamma_2}),
\)
where $\gamma_1+\gamma_2$ is the sum of $\gamma_1$ and $\gamma_2$ in
the abelian group $\pi_1(G)$.
Therefore $\cA$ is a $\ZZ\times\pi_1(G)$-graded algebra, where the
first $\ZZ$ is the half of the cohomological grading. (The odd degree
part vanishes by \propref{prop:monopole_formula}.) This gives an action
of $\CC^\times\times \pi_1(G)^\wedge$ on the spectrum $\mathcal M_C$,
where $\pi_1(G)^\wedge$ is the Pontryagin dual of $\pi_1(G)$.

\begin{NB}
    This naive grading is not correct, and does not match with the
    grading expected from the monopole formula, and also natural
    grading on various examples.
\end{NB}%

\subsection{Cartan subalgebra -- a commutative subalgebra in the
  quantized Coulomb branch}\label{subsec:Cartan}

Recall that $\cAh = H_*^{G_\cO\rtimes\CC^\times}(\cR)$ is a module
over $H^*_{G\times\CC^\times}(\mathrm{pt})$.  Let $T$ be a maximal
torus of $G$ with the Lie algebra $\ft$. Let $W$ be the Weyl group. We
have $H^*_G(\mathrm{pt})\cong \CC[\ft]^W$. We add a variable $\hbar$
for the $\CC^\times$-part, so
$H^*_{G\times\CC^\times}(\mathrm{pt})\cong \CC[\hbar,\ft]^W$.
It is also known that $\ft/W \cong\CC^\ell$, where $\ell$ is the rank
of $G$.

\begin{Proposition}\label{prop:comm}
    Let $1$ be the unit of the algebra $\cAh$. Then
    $H^*_{G\times\CC^\times}(\mathrm{pt}) 1 = \CC[\hbar,\ft]^W1$
    forms a commutative subalgebra of $\cAh$.
\end{Proposition}

\begin{Definition}
    We call the commutative subalgebra
    $\CC[\hbar,\ft]^W 1$ the {\it Cartan
      subalgebra\/}\footnote{This commutative subalgebra is called
      {\it Gelfand-Tsetlin algebra\/} in related contexts. But our
      subalgebra is hardly worth this name, as the proof of the
      commutativity is just a tautology. This alternative name is
      proposed to us by Boris Feigin.} of the quantized Coulomb
    branch.
\end{Definition}

\begin{proof}[Proof of \propref{prop:comm}]
    \begin{NB}
    Let $c_1$, $c_2\in H^*_{G\times\CC^\times}(\mathrm{pt})$. In
    the definition of the product $\ast$, $c_2$ is considered as a
    class in $H^{G_\cO\rtimes\CC^\times}(\cT\times_{\bN_\cK}\cT)$. It
    is nothing but $c_2\cap[\Delta\cT]$, where $\Delta\cT$ is the
    diagonal in $\cT\times\cT$. Therefore the assertion is clear.
    \end{NB}%
    It can be checked directly, but this is a formal consequence of
    properties that have been already established in
    \thmref{thm:convolution}. Let $c_1$, $c_2\in
    H^*_{G\times\CC^\times}(\mathrm{pt})$. Then
    \begin{equation*}
        (c_1 1) \ast (c_2 1) = c_1 (1 \ast (c_2 1)) = c_1 (c_2 1)
        = (c_1 c_2) 1.
    \end{equation*}
    The first equality is the linearity of the multiplication $\ast$
    with respect to the first variable. The second equality holds as
    $1$ is unit. The third is true, as
    $H_*^{G_\cO\rtimes\CC^\times}(\cR)$ is a module over
    $H^*_{G\times\CC^\times}(\mathrm{pt})$.

    Since $H^*_{G\times\CC^\times}(\mathrm{pt})$ is commutative,
    the above implies the assertion.
\end{proof}

By specializing \propref{prop:comm} at $\hbar=0$, we have a Poisson
commuting subalgebra $H^*_{G}(\mathrm{pt})\cong \CC[\ft]^W$ of
$\cA$. Therefore we have a morphism
\begin{equation}\label{eq:17}
    \intsys\colon \mathcal M_C = \Spec \cA \to
    \Spec(H^*_G(\mathrm{pt})) \cong \ft/W \cong \CC^\ell,
\end{equation}
and all the functions factoring through $\intsys$ are Poisson commuting. 
It will be proved in
\propref{prop:integrable} that the generic fiber is $T^\vee$, the dual
torus of $T$. Hence $\intsys$ is a complete integrable system.

The Poisson bracket $\{\ ,\ \}$ is the commutator $[\ ,\ ]$ divided by
$\hbar$.
\begin{NB}
    $\{ H_k^{G_\cO}(\cR), H_l^{G_\cO}(\cR)\} \subset
    H_{k+l+2}^{G_\cO}(\cR)$.
\end{NB}%
Therefore it is of degree $-\deg\hbar = -1$.

\begin{NB}
    The following proof depends on the localization theorem, which is
    not yet established (so that it preserves the
    multiplication). Anyway the proof should be postponed until the
    localization theorem is given, i.e., in \secref{sec:abel}.

    \begin{NB2}
    Added on Apr. 21 : The following proof cannot be made rigorous, as
    the localization theorem is given in a different form. Therefore
    the proof will be really postponed.
    \end{NB2}

\begin{Proposition}
    A generic fiber of $\intsys$ is $T^\vee$, the dual torus of
    $T$.
\end{Proposition}

\begin{proof}
    It is enough to study a generic fiber corresponding to
    $H^*_T(\mathrm{pt})\to H^T_*(\cR)$, as $H^G_*$ is the Weyl group
    invariant part of $H^T_*$. For $\xi\in\mathfrak h$, the
    specialization of $H^T_*(\cR)$ at $\xi$ is the ordinary cohomology
    of the vanishing locus $\cR^\xi$ of the vector field generated by
    $\xi$. If $\xi$ is generic, $\cR^\xi$ is nothing but the $T$-fixed
    point locus $\cR^T$. Therefore a generic fiber is $\Spec
    H_*(\cR^T)$.

    \begin{NB2}
        I need to clarify how multiplication is specialized.
    \end{NB2}

    The fixed point locus $\Gr_G^T$ is $\Gr_T$ (see e.g., ***),
    consisting of framed $T$-bundles. It is because the $T$-action on
    the trivialization $\varphi$ can be absorbed into an automorphism
    of a $T$-bundle $\scP$. If $(\scP,\varphi,s)$ is fixed by $T$, $s$
    must be fixed by the automorphism of $\scP$. It happens if and
    only if $s$ takes values in $\scP\times_T\bN_0$, where $\bN_0$ is
    the $0$-weight subspace of $\bN$. Then $\varphi_\bN(s)$ is
    automatically regular, hence $\cR^T = \Gr_T\times (\bN_0)_\cO$.
    Hence $H_*(\cR^T)\cong H_*(\Gr_T) = \CC[Y]$, where $Y$ is the
    coweight lattice of $T$. This isomorphism is compatible with the
    multiplication.
    \begin{NB2}
        why ?
    \end{NB2}%
    Therefore a generic fiber is $T^\vee$.
\end{proof}
\end{NB}

\subsection{Trivial properties of Coulomb branches}\label{sec:triv-properties}

Let us list a few trivial properties of $\mathcal M_C$ and $\cAh$,
which follow directly from the definition.

\subsubsection{}\label{item:product}
Suppose $(G,\bN) = (G_1\times G_2, \bN_1\oplus \bN_2)$, where $\bN_i$
is a representation of $G_i$ ($i=1,2$). Then
\begin{equation*}
    \begin{split}
    \mathcal M_C(G,\bN) &= \mathcal M_C(G_1,\bN_1)\times
    \mathcal M_C(G_2,\bN_2),
\\
    \cAh[G,\bN] &= \cAh[G_1,\bN_1]\otimes_{\CC[\hbar]}
    \cAh[G_2,\bN_2].
    \end{split}
\end{equation*}

These follow from $\cR_{G,\bN} = \cR_{G_1,\bN_1}\times
\cR_{G_2,\bN_2}$, and the K\"unneth formula $H^{G_\cO}_*(\cR_{G,\bN})
= H^{(G_1)_{\cO}}_*(\cR_{G_1,\bN_1})\otimes_\CC
H^{(G_2)_{\cO}}_*(\cR_{G_2,\bN_2})$ and its $\CC^\times$-equivariant
version.

\subsubsection{}\label{item:trivial}

Suppose $(G,\bN) = (G,\bN_1\oplus \bN')$, where $\bN'$ is a trivial
representation of $G$. Then
\begin{equation*}
    \mathcal M_C(G,\bN) = \mathcal M_C(G,\bN_1), \qquad
    \cAh[G,\bN] = \cAh[G,\bN_1].
\end{equation*}


Take $G_1 = G$, $G_2 = \{e\}$, $\bN_2 = \bN'$ in
\ref{item:product}. We have $\Gr_{\{1\}}$ is just a single point and
$\cR_{\{e\}, \bN'} \cong \bN'_\cO$. Its homology $H_*(\bN'_\cO)$ is
spanned by the fundamental class of $\bN'_\cO$. We have
$H_*(\bN'_\cO)\cong \CC$, as an algebra. Therefore $\mathcal
M_C(\{1\},\bN')$ is a single point. For the quantized version, we have
$H^{\CC^\times}_*(\bN'_\cO)\cong\CC[\hbar]$.

\subsubsection{}\label{item:finite_quotient}
Let $G'\to G$ be a finite covering, and let $\pi_1(G')\subset
\pi_1(G)$ be the corresponding cofinite subgroup of $\pi_1(G)$. Let
$\Gamma$ be the Pontryagin dual of $\pi_1(G)/\pi_1(G')$, considered as
a subgroup of $\pi_1(G)^\wedge$. It acts on $\mathcal M_C(G,\bN)$ by
the construction in \subsecref{sec:grading}.
Let us consider $\bN$ as a representation of $G'$ through the
projection $G'\to G$. Then
\begin{equation*}
    \mathcal M_C(G',\bN) = \mathcal M_C(G,\bN)/\Gamma,
\qquad
    \cAh[G',\bN] = \cAh[G,\bN]^\Gamma.
\end{equation*}

It is known that $\Gr_{G'}$ is the union of components of $\Gr_G$
corresponding to $\pi_1(G')\subset \pi_1(G)\cong\pi_0(\Gr_G)$. (See
e.g., \cite[4.5.6]{Beilinson-Drinfeld}.) The same is true for
$\cR_{G',\bN}$. Note also that there is no difference between
equivariant homology groups for ${G_\cO}$ and $G'_\cO$ as we consider
over complex coefficients. 
\begin{NB}
    Consider $EG'\to BG'$. We consider $EG'/\varXi\to BG'$, where
    $\varXi$ is the kernel of $G'\to G$. Then it can be used as the
    classifying space for $G$: the equivariant homology is defined as
    $H_*(EG'/\varXi \times_G \cR_{G',\bN})$ (a finite dimensional
    approximation, more precisely). This is basically because
    $H^*(EG'/\varXi)$ vanishes for complex coefficients. (Remember the
    double fibration argument to show the well-definedness of the
    equivariant homology.) Now observe $EG'/\varXi \times_G
    \cR_{G',\bN} = EG'\times_{G'} \cR_{G',\bN}$.
\end{NB}%
Therefore $H^{G'_\cO}_*(\cR_{G',\bN})$ is just the $\Gamma$-invariant
part of $H^{G_\cO}_*(\cR_{G,\bN})$. It means the assertion.

\subsubsection{}\label{item:reduction}
Next one is similar to the above, but the case when the Pontryagin dual of
$\pi_1(\tilde G)/\pi_1(G)$ is a torus.
Let $1\to G\to \tilde G\to T_F\to 1$ be a short exact sequence of
connected reductive groups where $T_F$ is a torus. The subscript `F'
stands for \emph{flavor} symmetry that will be discussed more
generally in \S\S\ref{subsec:flavor}, \ref{subsec:flav-symm-group2}.
\begin{NB}
    This remark is probably meaningless: We can replace $G_F = \tilde
    G/G$ by its maximal torus $T_F$, and replace $\tilde G$ by the
    inverse image $\tilde G'$ of $T_F$ in $\tilde G$. Then we have an
    action of $T_F^\vee$ on $\mathcal M_C(\tilde G',\bN)$.

    It would be nice if we can generalize the following when the
    quotient (flavor symmetry group) $\tilde G/G$ is not necessarily
    abelian. As $\dim \mathcal M_C(G,\mathbf N) = \dim \mathcal
    M_C(\tilde G,\mathbf N) - 2\operatorname{rank} (\tilde G/G)$, it
    may be still a reduction by a maximal torus of $\tilde G/G$, but it is
    not even clear how to construct an action of $\tilde G/G$.
\end{NB}%
For any representation $\bN$ of $\tilde G$ we can consider the corresponding
Coulomb branch $\mathcal M_C({\tilde G,\bN})$. It acquires an action of the
dual torus $T_F^{\vee} = \pi_1(T_F)^\wedge$ by
\subsecref{sec:grading}. Then

\begin{Proposition}\label{prop:reduction}
\begin{equation*}
    \mathcal M_C({G,\bN}) \cong
    \text{Hamiltonian reduction of $\mathcal M_C(\tilde G,\bN)$ by 
      $T_F^\vee$}.
\end{equation*}
 \end{Proposition}

 Since $\mathcal M_C(\tilde G,\bN)$ has singularities in general, this
 statement means an algebraic counterpart, i.e., $\cA[{G,\bN}]$ is the
 $T_F^\vee$-invariant part of $\cA[\tilde
 G,\bN]/\{\mu_{T_F^\vee}=0\}$, where $\mu_{T_F^\vee}$ is the moment
 map for the $T_F^\vee$-action, which is described as follows.

Let $\ft_F = \operatorname{Lie}T_F \cong (\operatorname{Lie}T_F^\vee)^*$.
Recall that a map $\mu_{T_F^\vee}\colon \mathcal M_C(\tilde
G,\bN)\to \ft_F$ is a moment map if $\xi\circ\mu_{T_F^\vee}$ is a
hamiltonian for a vector field $\xi^*$ generated by $\xi\in
\operatorname{Lie}T_F^\vee= \ft_F^*$. It is equivalent to say that
the Poisson bracket satisfies $\{ f, \xi\circ\mu_{T_F^\vee}\} =
\xi^*(f)$ for any function $f\in\cA[\tilde G,\bN]$. This notion makes
sense for Poisson algebras. A moment map is not unique in general, but
we have the canonical one given by the composite of $\mathcal
M_C(\tilde G,\bN)\xrightarrow{\intsys} \operatorname{Spec} H^*_{\tilde
  G}(\mathrm{pt})\to \operatorname{Spec} H^*_{T_F}(\mathrm{pt})$,
or in other words:
\begin{Lemma}\label{lem:comoment}
    The composite of
\begin{equation*}
    \CC[\ft_F] = H^*_{T_F}(\mathrm{pt}) \to H^*_{\tilde G}(\mathrm{pt})
    \xrightarrow{\intsys^*}
    H^{\tilde G_\cO}_*(\cR_{\tilde G,\bN}) = \cA[\tilde G,\bN]
\end{equation*}
is the comoment map, i.e., the pull-back by the moment map
$\mu_{T_F^\vee}\colon \mathcal M_C(\tilde G,\bN)\to\ft_F$. Here the first
homomorphism is induced by $\tilde G\to T_F$, and the second one is
multiplication to the unit $1\in H^{\tilde G_\cO}_*(\cR_{\tilde G,\bN})$ as in
\propref{prop:comm}, which is the pull-back by $\intsys$ in \eqref{eq:17}.
\end{Lemma}

\begin{proof}
    We take a character $\chi\colon T_F\to\CC^\times$ and consider
    its first Chern class $c_1(\chi)$ as an element in $H^2_{
     T_F}(\mathrm{pt})$. We also consider it as $\tilde
    G\to\CC^\times$, and hence as an element in $H^2_{\tilde
      G}(\mathrm{pt})$. We have the induced $\CC^\times$-action
    through $\CC^\times = (\CC^\times)^\vee \to T_F^\vee$. We need
    to show that the comoment map for $\CC^\times$ is
    $c_1(\chi)\mapsto c_1(\chi)1\in H^{\tilde G_\cO}_*(\cR_{\tilde G,\bN})$.

    We consider the quantized version of the homomorphism
    \begin{equation*}
        H^*_{\tilde G\times\CC^\times}(\mathrm{pt}) \to
        H^{\tilde G_\cO\rtimes\CC^\times}_*(\cR_{\tilde G,\bN})
    \end{equation*}
    and the commutator $[\bullet, c_1(\chi)1]$. The assertion follows
    from the following lemma below.

    Indeed, it implies that the Poisson bracket $\{\bullet,
    c_1(\chi)1\}$ is $\pi_1(\chi)\id$, which is the action of the Lie
    algebra. This is nothing but the definition of the comoment map.
\end{proof}

\begin{Lemma}\label{lem:former_claim}
    Let $\chi\colon \tilde G\to \CC^\times$ be a character and
    consider its first Chern class $c_1(\chi)$ as an element in
    $H^2_{\tilde G}(\mathrm{pt})$. Then
    $[\bullet, c_1(\chi)1] = \hbar \pi_1(\chi)\id$, where
    $\pi_1(\chi)$ is the $\ZZ$-valued function given by
    $\pi_0(\cR_{\tilde G,\bN})=\pi_1(\tilde
    G)\xrightarrow{\pi_1(\chi)}\pi_1(\CC^\times) = \ZZ$.
\end{Lemma}

\begin{proof}
    By the definition of the convolution product, $[\bullet,
    c_1(\chi)1]$ is given by the cup product with respect to the first
    Chern class of the line bundle over $\cR_{\tilde G,\bN}$ induced
    from the composite of $\tilde G_\cO\to \tilde G
    \xrightarrow{\chi}\CC^\times$, where the first homomorphism is
    given by taking the constant term.
    \begin{NB}
        The $\tilde G$-action on the fiber becomes trivial, as we are taking
        the commutator.
    \end{NB}%
    The line bundle is the pull-back from $\Gr_{\tilde G}$
    \begin{NB}
        that is $\tilde G_\cK\times_{\tilde G_\cO}\CC\to \tilde
        G_\cK/\tilde G_\cO = \Gr_{\tilde G}$,
    \end{NB}%
    by the projection $\cR_{\tilde G,\bN}\to\Gr_{\tilde G}$. It is
    further the pull-back from $\Gr_{\CC^\times}$
    \begin{NB}
        that is $\CC^\times_\cK\times_{\CC^\times_\cO}\CC\to
        \CC^\times_\cK/\CC^\times_\cO = \Gr_{\CC^\times}$,
    \end{NB}%
    by the morphism $\Gr_{\tilde G}\to \Gr_{\CC^\times}$ given by
    $\chi$. Let us note that the identification
    $\Gr_{\CC^\times}\simeq \ZZ$ is given by
    $\CC^\times_\cK/\CC^\times_\cO\ni[z^n]\mapsto n$. Then the line
    bundle is trivial, equipped with the
    $\CC^\times_\cO\rtimes\CC^\times$-equivariant structure by the
    $n^{\mathrm{th}}$ power map $\CC^\times\to\CC^\times$ on the
    component for $n\in\ZZ$.
    \begin{NB}
        Consider the (trivial) line bundle $(z^n
        \CC^\times_\cO)\times_{\CC^\times_\cO}\CC \to
        (z^n\CC^\times_\cO)/{\CC^\times_\cO}$. The rotation action
        $z\mapsto \lambda z$ gives the multiplication by $\lambda^n$
        on the fiber.
    \end{NB}

    Now $\pi_1(\chi)$ is given by $\cR_{\tilde
      G,\bN}\to\Gr_{\CC^\times}$,
    \begin{NB}
        and $\hbar$ is the first Chern class of the line bundle
        associated with the identity $\CC^\times\to\CC^\times$,
    \end{NB}%
    and $[\bullet, c_1(\chi)1] = \hbar \pi_1(\chi)\id$ follows.
\end{proof}

\begin{proof}[Proof of \propref{prop:reduction}]
Since $T_F$ is torus, its affine Grassmannian is a discrete
lattice. Hence $\cR_{G,\bN}$ is a union of components of $\cR_{\tilde
  G,\bN}$ as in \ref{item:finite_quotient}. Hence $H^{\tilde G_\cO}_*(\cR_{G,\bN})$ is the
   $T_F^\vee$-invariant part of $H^{\tilde G_\cO}_*(\cR_{\tilde G,\bN})$.
Next the equivariant homology for $G_\cO$ is given by 
\(
   H^{\tilde G_\cO}_*(\cR_{G,\bN})\otimes_{H^*_{\tilde G}(\mathrm{pt})}
   H^*_{G}(\mathrm{pt})
   =
   H^{\tilde G_\cO}_*(\cR_{G,\bN})\otimes_{H^*_{T_F}(\mathrm{pt})}
  \CC.
\)
This is given by cutting out the ideal generated by $c_1(\chi)1$ for
various $\chi\in H^2_{T_F}(\mathrm{pt})$.

We need to check that the induced product and the original product on
$H^{G_\cO}_*(\cR_{G,\bN})$ are equal. Namely both maps
\begin{equation*}
\begin{CD}
     H^{\tilde G_\cO}_*(\cR_{\tilde G,\bN})@<<< H^{\tilde G_\cO}_*(\cR_{G,\bN}) @>>>
     H^{G_\cO}_*(\cR_{G,\bN})
\end{CD}
\end{equation*}
are algebra homomorphisms. The left arrow is the embedding of the degree $0$ part with respect to the grading given by $\pi_1(T_F)$ in \subsecref{sec:grading}, and hence it is an algebra embedding.
The right arrow will be studied in more general setting in \propref{prop:deformation}. It will be shown
that it is an algebra homomorphism.
Moreover the multiplication on $H^{\tilde G_\cO}_*(\cR_{G,\bN})$ in \propref{prop:deformation} is the one given as a subalgebra: The diagram \eqref{eq:43} is the degree $0$ part of the diagram \eqref{eq:12} for $\cR_{\tilde G,\bN}$. 
\begin{NB}
\begin{equation*}
    \begin{CD}
    \cT\times \cR @<{p}<< G_\cK\times\cR
    @>{q}>> G_\cK\times_{G_\cO}\cR
    @>{m}>> \cT
\\
    @| @VVV @VV{\cong}V @|
\\        
    \cT\times \cR @<{p}<< \Iw\times\cR
    @>{q}>> \Iw\times_{\tilde G_\cO}\cR
    @>{m}>> \cT
\\
    @VVV @VVV @VVV @VVV
\\        
    \cT_{\tilde G,\bN}\times \cR_{\tilde G,\bN} @<{p}<< 
    \tilde G_\cK\times\cR_{\tilde G,\bN}
    @>{q}>> \tilde G_\cK\times_{\tilde G_\cO}\cR_{\tilde G,\bN}
    @>{m}>> \cT_{\tilde G,\bN}
    \end{CD}
\end{equation*}
\end{NB}%
Therefore $\mathcal M_C(G,\bN)$ is the Hamiltonian reduction of
$\mathcal M_C(\tilde G,\bN)$ by $T_F^\vee$.
 \end{proof}

Now the corresponding statement for quantized Coulomb
branches is clear.
\begin{equation*}
    \text{$\cAh[{G,\bN}]$ is the quantum
      Hamiltonian reduction of $\cAh[\tilde G,\bN]$ by $T_F^\vee$}.
\end{equation*}
Recall a homomorphism $\mu^*\colon
U(\operatorname{Lie}T_F^\vee)[\hbar] \to \cAh[\tilde G,\bN]$ is a
\emph{quantum comoment map} if $[f, \mu^*(\xi)] = \hbar \xi^*(f)$ for
$\xi\in\operatorname{Lie}T_F^\vee$, $f\in\cAh[\tilde G,\bN]$. Since
$T_F^\vee$ is torus, we have $U(\operatorname{Lie}T_F^\vee) \cong
S(\operatorname{Lie}T_F^\vee) \cong \CC[\ft_F]$. The above proof 
of~\lemref{lem:comoment}, in fact, shows that the composite of $\CC[\ft,\hbar] =
H^*_{T_F\times\CC^\times}(\mathrm{pt})\to H^*_{\tilde
  G\times\CC^\times}(\mathrm{pt}) \to\cAh[\tilde G,\bN]$ is a quantum
comoment map.
\begin{NB}
    The subalgebra $\cAh[\tilde G,\bN]^{T_F^\vee}$ consists of
    $f\in\cAh[\tilde G,\bN]$ with $[f,\mu^*(\xi)] = 0$. This is true for a
    general quantum comoment map.

    It is also clear from \lemref{lem:former_claim} above, as
    $\pi_1(\chi) = 0$ on the (union of the) component
    $\cR=\cR_{G,\bN}$ as it factors as $\pi_1(\tilde
    G)\to\pi_1(T_F)\to\pi_1(\CC^\times) = \ZZ$.
\end{NB}%
Furthermore $H^{\tilde G_\cO\rtimes\CC^\times}_*(\cR_{G,\bN}) =
\cAh[\tilde G,\bN]^{T_F^\vee} $ and
$H^{G_\cO\rtimes\CC^\times}_*(\cR_{G,\bN}) = H^{\tilde
  G_\cO\rtimes\CC^\times}_*(\cR_{G,\bN}) \otimes_{H^*_{\tilde
    G\times\CC^\times}(\mathrm{pt})}H^*_{G\times\CC^\times}(\mathrm{pt})$.
Hence $\cAh[G,\bN]$ is the quotient of $\cAh[\tilde G,\bN]^{T_F^\vee}$ by
the intersection of $\cAh[\tilde G,\bN]^{T_F^\vee}$ and the right ideal
generated by the image of the quantum comoment map. (The intersection
is a two-sided ideal, and the quotient is actually a ring.)
\begin{NB}
    Let $c\in \cAh[\tilde G,\bN]^{T_F^\vee}$, $b=\sum \mu^*(\xi_i) a_i
    \in \mu^*(\operatorname{Lie}T_F^\vee)\cdot\cAh[\tilde
    G,\bN]$. Then $c \sum \mu^*(\xi_i) a_i = \sum \mu^*(\xi_i) c a_i$,
    as $\xi_i^*(f) = 0$.
\end{NB}%
This is nothing but the definition of the quantum Hamiltonian reduction.
See \cite[\S4]{Etingof-CM}.

\begin{Example}\label{ex:typeD}
Let us give an example of this construction. Let $(\tilde G,\bN) = (\GL(2),
(\CC^2)^{\oplus N_f})$, and $G = \SL(2)$. Here $\CC^2$ is the vector
representation of $\GL(2)$ and $(\CC^2)^{\oplus N_f}$ is the direct
sum of its $N_f$ copies.
Then $(\tilde G,\bN)$ is a quiver gauge theory of type $A_1$ with
$\dim V = 2$, $\dim W = N_f$. Assume $N_f\ge 4$ so that the
corresponding vector is dominant. As we will prove in \ref{quivar},
the Coulomb branch for $(\tilde G,\bN)$ is a quiver variety of type
$A_{N_f-1}$ with dimension vectors $\dim V = (1,2,\dots,2,1)$, $\dim W
= (0,1,0,\dots,0,1,0)\in \ZZ^{N_f-1}$. Moreover, the torus
$\pi_1(\tilde G)^\vee$ action is identified with the action of
$\GL(W_2) \cong \CC^\times$. (See \ref{Cartan_grading}.)
Therefore the Coulomb branch of $(G,(\CC^2)^{N_f})$ is the Hamiltonian
reduction of $\prod \GL(V_i)\times \GL(W_2)$, which is a quiver
variety of type $D_{N_f}$ of dimension vectors $\dim V =
(\genfrac{}{}{0pt}{0}{1}{1} 2\dots 21)$, $\dim W =
(\genfrac{}{}{0pt}{0}{0}{0} 0\dots 010)\in\ZZ^{N_f}$. This quiver
variety is Kronheimer's original construction of a simple singularity
of type $D_{N_f}$. This coincides with the expectation in
\cite{MR1490862}.

By \cite[(2.18)]{MR1490862}, $\mathcal M_C(G,\bN)$ is in general
expected to be $y^2 = x^2 v - v^{N_f-1}$ even for $N_f=1,2,3$. In this
case, $\mathcal M_C(\tilde G,\bN)$ is \emph{not} a quiver variety,
hence the above argument does not work. It should be possible to check
this by using the Coulomb branch of $(\tilde G,\bN)$ in
\ref{nondom}. But we give an alternative argument in \lemref{lem:SL2}.

\begin{NB}
    Misha : On Aug.22:

    I do not know how to write down equations for the nondominant
    slices, but Joel does (for $PGL(2)$, see page 5 of his letter of
    June 22). If you perform the hamiltonian reduction with respect to
    Cartan of his 4-dimensional slices with $N_f=1,2,3$, you do get
    the answer of \cite[(2.18)]{MR1490862}. Thus everything checks, the only
    problem being my inability to identify our slices with Joel's.
\end{NB}
\end{Example}

Further examples are given as toric hyper-K\"ahler manifolds. See
\subsecref{sec:toric-hyper-kahler}.

\begin{NB}
    From the claim, $c_1(\chi)$ is central if $\pi_1(\chi) = 0$. In
    particular, if $\g = \g_{\mathrm{ss}} \oplus Z(\g)$ be the
    decomposition of $\g$ into a direct sum of the semisimple part
    $\g_{\mathrm{ss}} = [\g,\g]$ and the center $Z(\g)$,
    $\CC[\g_{\mathrm{ss}}]$ is central. Therefore $\cAh$ is an algebra
    over $\CC[\g_{\mathrm{ss}}]$.
\end{NB}%

Let us also remark that this proposition is naturally predicted from the monopole formula \eqref{eq:19}, as was observed in \cite[\S5.1]{Cremonesi:2013lqa}.

\subsection{Flavor symmetry group -- deformation}
\label{subsec:flavor}

Suppose that we have a larger group $\tilde G$ containing $G$ as a
normal subgroup. Let $G_F = \tilde G/G$. This is called the {\it
  flavor symmetry group\/} in physics literature. We suppose our
$G$-module $\bN$ extends to a $\tilde G$-module. We denote it by the
same notation $\bN$.
The Hamiltonian reduction in \ref{item:reduction} above is an example
when $G_F$ is a torus. Note that $\mathcal M_C(G,\bN)$ in
\ref{item:reduction} has a natural deformation and a family of
quasi-projective varieties which are projective over $\mathcal
M_C(G,\bN)$. The former is given by changing the level of the moment
map, and the latter is given by considering GIT quotients for
characters of $T_F^\vee$.
We will give both constructions for arbitrary $G_F$. This property
was expected in \cite[\S5]{2015arXiv150303676N}. See original physics literature given there.

The deformation is easy, and is given here. Quasi-projective varieties
will be given later in \subsecref{subsec:flav-symm-group2}.

Since $G$ is a normal subgroup of $\tilde G$, the $G_\cO$-action on
$\Gr_G$ extends to $\tilde G_\cO$.
\begin{NB}
    By the formula $\tilde g[g] = [\tilde g g \tilde g^{-1}]$.
\end{NB}%
Moreover, as $\bN$ is a representation of $\tilde G$, we have $\tilde
G_\cO$-actions on $\cT$, $\cR$, etc.
\begin{NB}
    By the formula $\tilde g[g,s] = [\tilde g g \tilde g^{-1},\tilde gs]$.
\end{NB}%
Therefore we can consider the $\tilde G$-equivariant homology group
$H^{\tilde G_\cO}_*(\cR)$. It is a module over $H^*_{\tilde
  G}(\mathrm{pt})$ and has extra directions parametrized by $\Spec
(H^*_{G_F}(\mathrm{pt}))$. We have the restriction homomorphism
$H_*^{\tilde G_\cO}(\cR)\to H_*^{G_\cO}(\cR) = H^{\tilde
  G_\cO}_*(\cR)\otimes_{H^*_{G_F}(\mathrm{pt})}\CC$.

\begin{Proposition}\label{prop:deformation}
    A convolution product $\ast$ defines an associative graded algebra
    structure on $H^{\tilde G_\cO}_*(\cR)$. The restriction
    homomorphism $H^{\tilde G_\cO}_*(\cR)\to H^{G_\cO}_*(\cR)$
    is an algebra homomorphism. The same is true for $\tilde
    G_\cO\rtimes\CC^\times$, $G_\cO\rtimes\CC^\times$ equivariant
    homology groups.
\end{Proposition}

Applying \lemref{lem:former_claim}
to the setting in the proof below, we see that
$H^*_{G_F}(\mathrm{pt})$ is central in $H^{\tilde
  G_\cO\rtimes\CC^\times}_*(\cR)$. (Note that $\pi_1(\chi)$ is zero on
$\cR$.)
\begin{NB}
    As it factors through $\pi_1(G_F)$.
\end{NB}%
Therefore $H^{\tilde
  G_\cO\rtimes\CC^\times}_*(\cR)\otimes_{H^*_{G_F}(\mathrm{pt})} \CC$
has an induced multiplication. Then the second assertion means that
$H^{\tilde
  G_\cO\rtimes\CC^\times}_*(\cR)\otimes_{H^*_{G_F}(\mathrm{pt})} \CC
\cong H^{G_\cO\rtimes\CC^\times}_*(\cR)$ is an algebra isomorphism.

\begin{NB}
    Let $a$, $b\in H^{\tilde G_\cO\rtimes\CC^\times}_*(\cR)$, $f$,
    $g\in H^*_{G_F}(\mathrm{pt})$. We have $fa\otimes 1 = a\otimes
    f(0)$, etc. Then $(a\otimes 1)\ast^{\tilde G} (b\otimes 1) =
    (a\ast^{\tilde G} b)\otimes 1$ is well-defined. Now the second
    assertion means $(a\otimes 1)\ast^G(b\otimes 1) = (a\ast^{\tilde G}
    b)\otimes 1$. Therefore the two products are the same.

    The second assertion implies
    \begin{equation*}
        (fa\ast^{\tilde G} gb)\otimes 1
        = (a\otimes f(0))\ast (b\otimes g(0))
        = (a\otimes 1\ast b\otimes 1) f(0)g(0)
        = (a\ast^{\tilde G}b)\otimes f(0)g(0).
    \end{equation*}
    Therefore it turns out that a multiplication is well-defined on
    $H^{\tilde
      G_\cO\rtimes\CC^\times}(\cR)\otimes_{H^*_{G_F}(\mathrm{pt})}
    \CC$. But unless we put it a priori, the algebra isomorphism
    assertion does not make sense.
\end{NB}%

A proof of commutativity of the product on $H^{\tilde G_\cO}_*(\cR)$
is the same as one for $H^{G_\cO}_*(\cR)$, hence is postponed until
\secref{sec:abel}.

\begin{proof}
Let us regard $\Gr_G$ as the moduli space of pairs $(\scP,\varphi)$ as
before. We have the induced $\tilde G$-bundle $\tilde\scP\defeq
\scP\times_G\tilde G$ and its trivialization $\tilde\varphi
\defeq\varphi\times_G\tilde G\colon \scP\times_G\tilde G \to \tilde
G\times D^*$. Moreover for the further induced $G_F$-bundle
$\tilde\scP\times_{\tilde G}G_F$, the trivialization
$\tilde\varphi\times_{\tilde G} G_F$ extends across the origin $0\in D$.
\begin{NB}
    Let $[[p,\tilde g],g_F] \in \scP\times_G\tilde G\times_{\tilde G}
    G_F$. We have a section $x\mapsto [[p,1],1]$, which is
    well-defined as $[[pg,1],1] = [[p,g],1] = [[p,1], \xi(g)] =
    [[p,1],1]$ where $\xi\colon \tilde G\to G_F$. Then 
    \(
    \tilde\varphi\times_{\tilde G} G_F([[p,1],1]) =
    \varphi(p)\times_G\times\tilde G\times_{\tilde G} G_F.  
    \)
    In effect, $\varphi(p)\times_G\tilde G\times_{\tilde G} G_F$ is
    $\varphi(p)\times_G\tilde G$ means the image of $\varphi(p)$ under
    the inclusion $G\to\tilde G$. And $\varphi(p)\times_G\tilde
    G\times_{\tilde G} G_F$ is $\xi(\varphi(p))$. Therefore it is
    $1$. Thus the trivialization extends across the origin.
\end{NB}%
Conversely a pair $(\tilde\scP,\tilde\varphi)$ such that the
trivialization $\tilde\varphi\times_{\tilde G}G_F$ extends is coming
from a pair $(\scP,\varphi)$. Thus $\Gr_G$ can be regarded as the
moduli space of such pairs $(\tilde\scP,\tilde\varphi)$.

Let $\Iw$ be the inverse image of $(G_F)_\cO$ under $\tilde G_\cK\to
(G_F)_\cK$.
\begin{NB}
    It contains both $G_\cK$ and $\tilde G_\cO$. Since $1\to G_\cK\to
    \tilde G_\cK\to (G_F)_\cK\to 1$ is exact, $\Iw/G_\cK \cong
    (G_F)_\cO$. Also 
\[
     \Gr_G = G_\cK/G_\cO \xrightarrow{\cong}
     \Iw/\tilde G_\cO
\]

The simplest case $\tilde G = G\times G_F$, we have $\tilde G_\cO =
G_\cK\times (G_F)_\cO$, $\Iw = G_\cK\times (G_F)_\cO$. The above identification is trivial in this case.
\end{NB}%
The homomorphism $G_\cK\to \Iw$ induces a bijection $\Gr_G =
G_\cK/G_\cO \cong \Iw/\tilde G_\cO$, where the latter quotient
$\Iw/\tilde G_\cO$ is compatible with the above scheme structure of
$\Gr_G$. As an analog of \eqref{eq:15}, we have
\begin{equation*}
    H^{\tilde G_\cO\times\tilde G_\cO}_*(\Iw) \cong H^{\tilde G_\cO}_*(\Gr_G).
\end{equation*}

Let us modify (the lower row of) the diagram \eqref{eq:12} as
\begin{equation}\label{eq:43}
    \begin{CD}
    \cT\times \cR @<{p}<< \Iw\times\cR
    @>{q}>> \Iw\times_{\tilde G_\cO}\cR
    @>{m}>> \cT,
    \end{CD}
\end{equation}
where maps are given by
\begin{equation*}
    \left([g_1,g_2s], [g_2,s]\right) \leftmapsto \left(g_1,[g_2,s]\right)
    \mapsto
    \bigl[g_1, [g_2,s]\bigr] \mapsto [g_1g_2, s].
\end{equation*}
It is exactly the same formula as \eqref{eq:42} above, but we have
used the description $\cR = \{ [g,s]\in \Iw\times_{\tilde
  G_\cO}\bN_\cO \mid gs\in\bN_\cO\}$, etc. The upper row of
\eqref{eq:12} is defined in the same way.

The same formula as in \eqref{eq:41} gives actions
\begin{equation*}
        \tilde G_\cO \times\tilde G_\cO\curvearrowright
        \cT\times\cR,
        \qquad
        \tilde G_\cO \times \tilde G_\cO \curvearrowright
        \Iw\times\cR,
        \qquad
        \tilde G_\cO \curvearrowright \Iw\times_{\tilde G_\cO}\cR,
        \qquad
        \tilde G_\cO \curvearrowright \cT.
\end{equation*}
The above diagram \eqref{eq:43} is equivariant.
\begin{NB}
\begin{equation*}
    \begin{split}
        \tilde G_\cO \times\tilde G_\cO\curvearrowright
        \cT\times\cR ;\ &
        (g,h)\cdot \left([g_1,s_1],[g_2,s_2]\right)
        = \left([gg_1, s_1], [hg_2,s_2] \right),
        \\
        \tilde G_\cO \times \tilde G_\cO \curvearrowright
        \Iw\times\cR ;\ &
        (g,h)\cdot \left(g_1, [g_2,s]\right)
        = \left(gg_1h^{-1}, [hg_2, s]\right),
        \\
        \tilde G_\cO \curvearrowright \Iw\times_{\tilde G_\cO}\cR ;\ &
        g\cdot \left[g_1, [g_2,s]\right] = \left[gg_1,
          [g_2, s]\right],
        \\
        \tilde G_\cO \curvearrowright \cT ;\ &
        g\cdot [g_1,s] = [gg_1, s].
    \end{split}
\end{equation*}
\end{NB}%
Also the diagram is given by morphisms of schemes as before.

We then define the convolution product $\ast$ on $H^{\tilde
  G_\cO}_*(\cR)$ using \eqref{eq:43} instead of \eqref{eq:12}.

The compatibility of two products on $H^{\tilde G_\cO}_*(\cR)$ and
$H^{G_\cO}_*(\cR)$ follows from a commutative diagram connecting
\eqref{eq:12} to \eqref{eq:43}.
\begin{NB}
    \begin{equation*}
    \begin{CD}
    \cT\times \cR @<{p}<< G_\cK\times\cR
    @>{q}>> G_\cK\times_{G_\cO}\cR
    @>{m}>> \cT
\\
    @| @VVV @VV{\cong}V @|
\\        
    \cT\times \cR @<{p}<< \Iw\times\cR
    @>{q}>> \Iw\times_{\tilde G_\cO}\cR
    @>{m}>> \cT.
    \end{CD}
\end{equation*}
\end{NB}%
The detail is omitted.
%
\end{proof}

\begin{Remarks}\label{rem:W-cover}
    (1) When $G_F$ is a torus, we are in the setting
    \subsecref{sec:triv-properties}\ref{item:reduction}. As we have
    remarked in the beginning of this subsection, the quotient
    $\mu_{T_F^\vee}\colon \mathcal M_C(\tilde G,\bN)\dslash T_F^\vee
    \to \ft_F$ gives a deformation of $\mathcal M_C(G,\bN)$
    parametrized by $\ft_F$. Taking the quotient by $T_F^\vee$, but
    not imposing the moment map equation $\mu_{T_F}=0$ means that we
    change the space from $\cR_{\tilde G,\bN}$ to $\cR_{G,\bN}$, but
    keep the group as $\tilde G_\cO$. Therefore we have $\mathcal
    M_C(\tilde G,\bN)\dslash T_F^\vee = \Spec H^{\tilde G_\cO}_*(\cR)$
    in this case.

    In general, we take a maximal torus $T_F$ of $G_F$, and set
    $\tilde G'$ as its inverse image in $\tilde G$. Then $\mathcal
    M_C(\tilde G',\bN)\dslash T_F^\vee = \Spec H^{\tilde
      G'_\cO}_*(\cR)$ is a $W_F$-covering of $\Spec H^{\tilde
      G_\cO}_*(\cR)$, where $W_F$ is the Weyl group of $G_F$.
    (Checking that $H^{\tilde G'_\cO}_*(\cR) \cong H^{\tilde
      G_\cO}_*(\cR)\otimes_{H^*_{\tilde G}(\mathrm{pt})} H^*_{\tilde
      G'}(\mathrm{pt})$ respects the multiplication is left as an
    exercise for the reader.)

(2) If $\tilde G$ acts on $\bN$ through $G$, i.e., we have a group
    homomorphism $\rho\colon \tilde G\to G$ such that the $\tilde
    G$-action on $\bN$ factors through $\rho$, we have $H^{\tilde
      G_\cO}_*(\cR) \cong
    H^{G_\cO}_*(\cR)\otimes_{H^*_{G}(\mathrm{pt})} H^*_{\tilde
      G}(\mathrm{pt})$, where $H^*_{G}(\mathrm{pt})\to H^*_{\tilde
      G}(\mathrm{pt})$ is given by $\rho$. Our deformation is
    trivial in this case. An example is the dilatation action of
    $\CC^\times$ on $\bN$. (See \subsecref{sec:another}.) Although
    $\tilde G = G\times\CC^\times$ acts on $\bN$, it factors through
    $G$ in many occasions, say quiver gauge theories whose underlying
    graphs have no cycles.
\end{Remarks}

\subsection{Flavor symmetry group -- resolution} \label{subsec:flav-symm-group2}

Let us continue the setting in the previous subsection. 
We have just constructed a deformation of $\mathcal M_C$
parametrized by $H^*_{G_F}(\mathrm{pt})$.
Since $\mathcal M_C$ is supposed to be a hyper-K\"ahler manifold, a
deformation and a (partial) resolution should come together. We shall
construct the latter in this subsection. In fact, in the
hyper-K\"ahler setting, one can construct simultaneous resolution of
the deformation, after the base change to a $W_F$-cover.
Therefore in view of \remref{rem:W-cover} we take a maximal torus
$T_F$ of $G_F$, and set $\tilde G'$ as its inverse image in $\tilde
G$.
Therefore we are in the setting of
\subsecref{sec:triv-properties}\ref{item:reduction}. We consider the
GIT quotient of $\mathcal M_C(\tilde G',\bN)$ by $T_F^\vee$ with respect to a
character $\laF$ of $T_F^\vee$. Let us denote it by $\mathcal M_C(\tilde G',\bN)\dslash_\laF T_F^\vee$.
We have a natural projective morphism 
\(
   \mathcal M_C(\tilde G',\bN)\dslash_\laF T_F^\vee
   \to \mathcal M_C(\tilde G',\bN)\dslash T_F^\vee,
\)
which fits in the diagram
\begin{equation*}
    \begin{CD}
        @. \Spec H^{\tilde G_\cO}_*(\cR) @>>> \ft_F/W_F
        \\
        @. @AA{/W_F}A @AA{/W_F}A
        \\
        \mathcal M_C(\tilde G',\bN)\dslash_\laF T_F^\vee @>>> \mathcal
        M_C(\tilde G',\bN)\dslash T_F^\vee @>>> \ft_F
        \\
        @AAA @AAA @AAA
        \\
        \mu_{T_F^\vee}^{-1}(0)\dslash_\laF T_F^\vee
        @>>>
        \mu_{T_F^\vee}^{-1}(0)\dslash T_F^\vee
        \overset{\rm\subsecref{sec:triv-properties}\ref{item:reduction}}=
        \mathcal M_C(G,\bN)
        @>>> \{0\}.
    \end{CD}
\end{equation*}
This is reminiscent of the Grothendieck-Springer resolution of $\g^*$
and the corresponding diagram for quiver varieties.

Let us give a description of $\mathcal M_C(\tilde G',\bN)\dslash_\laF
T_F^\vee$ in terms of a homology group of a modification of $\cR$.

In order to simplify the notation, let us replace $\tilde G$ (resp.\
$G_F$) by $\tilde G'$ (resp.\ $T_F$), and assume $\tilde G=\tilde G'$,
$G_F=T_F$. We also denote $\cT_{\tilde G,\bN}$, $\cR_{\tilde G,\bN}$
by $\tilde\cT$, $\tilde\cR$ respectively for short.

A $\tilde G$-bundle $\scP$ induces a $T_F$-bundle by
$\scP\times_{\tilde G}T_F$. This gives a morphism $\Gr_{\tilde G}\to
\Gr_{T_F}$.  Composing it with $\tilde\cT\to\Gr_{\tilde G}$ or
$\tilde\cR\to\Gr_{\tilde G}$, we have
\begin{equation*}
    \tilde\pi\colon \text{$\tilde\cT$ or $\tilde\cR$}\to \Gr_{T_F}.
\end{equation*}

Let $\tilde\cT^{\laF}$ or $\tilde\cR^{\laF} = \tilde\pi^{-1}(\laF)$
be a fiber of this projection at a coweight $\laF$ of $G_F$.
It is preserved under the action of $\tilde G_\cO$.
We have
\(
   H^{\tilde G_\cO}_*(\tilde\cR) = \bigoplus_{\laF} H^{\tilde G_\cO}_*(\tilde\cR^\laF)
\)
by \subsecref{sec:grading}. It corresponds to $T_F^\vee$-action on $\mathcal M_C(\tilde G,\bN)$.
\begin{Proposition}\label{prop:GITquotient}
    Consider the $\ZZ_{\ge 0}$-graded algebra $\bigoplus_{n\ge 0}
    	H^{\tilde G_\cO}_*(\tilde \cR^{n\laF})$.
	We have
\begin{equation*}
    \mathcal M_C(\tilde G,\bN)\dslash_\laF T_F^\vee \cong \Proj(\bigoplus_{n\ge 0}
    	H^{\tilde G_\cO}_*(\tilde\cR^{n\laF})).
\end{equation*}
Similarly we have
\begin{equation*}
    \mu_{T_F^\vee}^{-1}(0)\dslash_\laF T_F^\vee 
    \cong \Proj(\bigoplus_{n\ge 0}
    	H^{G_\cO}_*(\tilde\cR^{n\laF})).
\end{equation*}
\end{Proposition}

This is clear from the definition of the left hand side. It is Proj of the ring of $T_F^\vee$-semi-invariants with respect to the character $\laF\colon T_F^\vee\to\CC^\times$. The semi-invariants ring is the graded algebra in question.

For the second isomorphism, note that $H^{G_\cO}_*(\tilde\cR) =
H^{\tilde G_\cO}_*(\tilde\cR)\otimes_{H^*_{T_F}(\mathrm{pt})}\CC$ has
the induced multiplication. This is obvious if we use the
commutativity, which will be proved later. But it is true even without
commutativity as we have already know $H^*_{G_F}(\mathrm{pt})$ is central, as remarked after \propref{prop:deformation}.

Let us denote
$\Proj(\bigoplus_{n\ge 0} H^{G_\cO}_*(\tilde\cR^{n\laF}))$ by
$\mathcal M_C^{\laF}(G,\bN)$. We have a projective morphism
$\boldsymbol\pi\colon \mathcal M^{\laF}_C(G,\bN)\to \mathcal M_C(G,\bN)$ as
$H^{G_\cO}_*(\tilde\cR^{0}) = H^{G_\cO}_*(\cR) = \CC[\mathcal M_C(G,\bN)]$.
Since $\bigoplus_{n\ge 0} H^{G_\cO}_*(\tilde\cR^{n\laF})$ is finitely
generated, we can take sufficiently large $n_0$ so that
$\bigoplus_{n\ge 0} H^{G_\cO}_*(\tilde\cR^{n n_0\laF})$ is generated by
$H^{G_\cO}_*(\tilde\cR^{n_0\laF})$ as a $\CC[\mathcal M_C(G,\bN)]$-algebra.
Therefore we have the corresponding line bundle over
$\mathcal M_C^{\la_F}(G,\bN)$. Let $\cL_{\laF}$ denote its $1/n_0$-th
power, considered as a $\mathbb Q$-line bundle.

\begin{Remark}\label{rem:HLmonopole}
    The dimension of $H^{G_\cO}_*(\tilde\cR^{n\laF}))$ is given by a
    formula similar to \propref{prop:monopole_formula} and
    \eqref{eq:19}. In fact, we take the stratification $\tilde\cR =
    \bigsqcup\tilde\cR_{\tilde\la}$ parametrized by dominant coweights
    $\tilde\la$ of $\tilde G$. Then $\tilde\cR^\laF$ is the union of
    strata $\tilde\cR_{\tilde\la}$ such that $\tilde\la$ is sent to
    $\laF$ under $\tilde G\to T_F$. Such an extension of the monopole
    formula was given in \cite{Cremonesi:2014kwa}. See
    \cite[\S5(i)]{2015arXiv150303676N} for a review.
\end{Remark}

\subsection{Previously known examples}\label{subsec:prev}

Let us identify previous constructions in the literature as special
cases of our Coulomb branches.

\subsubsection{Pure gauge theories}\label{subsub:pure}

Consider $\bN = 0$. Then $\cR = \Gr_G$. The convolution algebra
$H^{G_\cO}_*(\Gr_G)$ was calculated in \cite[Th.~2.12]{MR2135527}, and was
attributed to an earlier work by Peterson \cite{Peterson,MR1403352}.
It is the algebraic variety ${\mathfrak Z}^{G^\vee}_{{\mathfrak
    g}^\vee}$ formed by the pairs $(g,x)$ such that $x$ lies in a
(fixed) Kostant slice in ${\mathfrak g}^\vee$, and $g\in G^\vee$
satisfies $\operatorname{Ad}_g(x)=x$.
\begin{NB}
It is a symplectic reduction of the cotangent bundle $T^* G^\vee$ of the Langlands dual group $G^\vee$ with respect to a nondegenerate character of 
$N^\vee\times N^\vee$, where $N^\vee\subset G^\vee$ is a unipotent subgroup.
\end{NB}%
Combining with \cite{MR1438643}, we will prove that the Coulomb branch
$\mathcal M_C$ is the moduli space of solutions of Nahm's equations
for the Langlands dual group $G_c^\vee$ in \thmref{thm:pure}. For $G =
\SL(k)$, it recovers results proved by physical arguments
(\cite{MR1490862} for $k=2$, \cite{MR1443803} for arbitrary $k$). See \secref{sec:monop-moduli-spac} for detail.


\subsubsection{Adjoint matters}\label{subsub:adjoint}
Consider $\bN = \g$, the adjoint representation of $G$. We consider the
dilatation action on $\bN$ as a flavor symmetry group ($\tilde G =
G\times\CC^\times$, $G_F=\CC^\times$). The space $\cR$ in this case is
\begin{equation*}
    \cR = \{ (\xi,[g])\in \g_\cO\times \Gr_G \mid
    \operatorname{Ad}_{g^{-1}}(\xi) \in \g_\cO\},
\end{equation*}
which is a variant of the affine Grassmannian Steinberg variety,
denoted by $\Lambda$ in \cite[\S7]{MR2135527}. (In fact, $\Lambda$ is
a closed subvariety of $\cR$, and the inclusion induces
$H^{G_O}_*(\Lambda) \cong H^{G_O}_*(\cR)$.)
\begin{NB}
    $\mathfrak u$ in \cite[\S7]{MR2135527} is $z\g[z] = z\g_\cO$.
\end{NB}%
In \cite[\S7]{MR2135527}, it was shown that $K^{G_\cO}(\Lambda)$ is
isomorphic to $\CC[T^\vee\times T]^W$. 
\begin{NB}
The argument for $K$-theory in \cite{MR2135527} uses its specific
features (certain coherent sheaves obtained as associated graded of
certain $D$-modules).
\end{NB}%
The argument for $K$-theory in \cite{MR2135527} uses its specific
features (certain coherent sheaves obtained as associated graded of
certain $D$-modules). The homology case can be deduced from the
$K$-theory case via Chern character homomorphism as we will see in
\ref{prop:ad}. But we also give another proof in \propref{prop:ad2}.
\begin{NB}
In \propref{prop:ad}, we prove $H^{G_\cO}_*(\cR)\cong \CC[\mathfrak
t\times T^\vee]^W$ by deducing required computation from one done in
\cite{MR2135527} via the Chern character homomorphism.
\end{NB}%
Recall that it was shown that the equivariant $K$-group of the affine
flag variety analog of $\Lambda$ is isomorphic to the Cherednik double
affine Hecke algebra in \cite{MR3013034}. (See also an earlier work \cite{MR2115259}.)
Therefore, when we include the loop rotation $\CC^\times$ and the
flavor symmetry group $G_F=\CC^\times$, the equivariant homology group
$H^{\tilde G_\cO\rtimes\CC^\times}_*(\cR)$ is expected to be the
spherical subalgebra of the graded Cherednik algebra (alias the
trigonometric degeneration of the double affine Hecke algebra).
\footnote{After this paper was written, it is proved in \cite{2016arXiv160800875K} in type $A$.}
In fact, representations of the whole graded Cherednik algebra were
constructed on equivariant homology groups of affine Springer fibers
\cite{oy}. Therefore we believe that the only remaining task is a
matter of checking.

\begin{NB}
    Misha and Sasha, is this fine ? Any suggestion on the line bundle
    over a (partial) resolution corresponding to $G_F$ ?
\end{NB}%



\newcommand\grt{\mathfrak t}
\newcommand\Sym{\operatorname{Sym}}
\newcommand\tilV{\widetilde V}
\newcommand\calA{\mathcal A}
\newcommand\bfV{\mathbf V}
\newcommand\tcalA{\widetilde{\calA}}
\newcommand\tilA{\widetilde{A}}
\newcommand\tilT{\widetilde{T}}
\newcommand\lam{\lambda}
\newcommand\slt{\textsf{t}}

\section{The abelian case}\label{sec:abelian}

\begin{NB}
    This is the original Sasha's notation. It is recorded in case the
    old notation would remain.

Notation: I will denote by $\calA_{G,\bfN}$ the $G(\calO)$-equivariant Borel-Moore homology of $\calR_{G,\bfN}$.
Similarly, I will denote by $\tcalA_{G,\bfN}$ the $T$-equivariant BM homology of $\calR_{G,\bfN}$, where $T$ is the maximal
torus of $G$. We have $\calA_{\bfN}=\tcalA_{G,\bfN}^W$ where $W$ is the Weyl group of $G$.

In addition we let $\calA_{G,\bfN,\hbar}$ denote the $G\times \CC^*$-equivariant BM homology of $\calR_{G,\bfN}$. Also we denote
by $\calA_{G,\bfN,\hbar,\slt}$ the $G\times \CC^*\times \CC^*$-equivariant BM homology of $\calR_{G,\bfN}$ (here $\hbar$ and $\slt$
stand for the corresponding equivariant parameters).
\end{NB}

In this section we determine the Coulomb branch and its quantization
when the group is a torus.
We obtain an explicit presentation of the ring
$H^{T_\cO}_*(\cR)$. This presentation is the same as one proposed in
\cite[\S3]{2015arXiv150304817B} by a physical intuition.

\subsection{The main result - the non-quantized case}
Let $T$ be a torus and let $Y$ denote its coweight lattice. We denote by $\grt$ the Lie algebra of $T$ and by
$\grt^*$ its dual space.
Let $\bN$ be a representation of $T$ given by a bunch of characters $\xi_1,\cdots,\xi_n$. Note that each $\xi_i$ can be viewed as an element of $\grt^*$.

For two integers $k,l$ let us set
$$
d(k,l)=
\begin{cases}
0 & \text{if $k$ and $l$ have the same sign,}\\
\min(|k|,|l|)&  \text{if $k$ and $l$ have different signs.}
\end{cases}
$$

\begin{Theorem}
\label{abel}
The $T$-equivariant Borel-Moore homology of $\cR\equiv\cR_{T,\bN}$ is
generated by the algebra $\Sym(\grt^*)$ together with symbols
$r^{\lambda}$, $\lambda\in Y$ subject to the following relation:
\begin{equation}\label{relation}
  r^{\lambda}r^{\mu}=\prod\limits_{i=1}^n \xi_i^{d(\xi_i(\lambda),\xi_i(\mu))}
  \,r^{\lambda+\mu}.
\end{equation}
\end{Theorem}

\begin{NB}
In the future we shall denote that algebra by $\calA_{T,\bfN}$.
\end{NB}

\begin{Remark}
It is easy to see that the above algebra is graded if the degree of any $y\in \grt^*$ is 2 and we also set
$$
\deg r^{\lambda}=\sum\limits_{i=1}^n |\xi_i(\lambda)|.
$$
This is the grading given by $\Delta$ in \eqref{eq:18}.
\end{Remark}

\begin{proof}
First, the algebra $H^{T}_*(\cR)$ clearly contains $H^*_T(\mathrm{pt})=\Sym(\grt^*)$. 
The multiplication is $H^*_T(\mathrm{pt})$-linear in the first variable by definition.
It is also in the second variable by \lemref{lem:former_claim}.
\begin{NB}
    (Note that \lemref{lem:former_claim} is true for any $\chi\in
    H^2_{\tilde G}(\mathrm{pt})$.)
\end{NB}%
On the other hand,
the affine Grassmannian of $T$ has connected components parametrized by $Y$. Each such components consists of one point (we ignore nilpotents here). The space $\cT$ is identified with the disjoint union $\bigsqcup_{\lambda\in Y}\{\lambda\}\times z^\lambda \bN_\cO$ via the embedding $\cT\subset\Gr_T\times\bN_\cK$. We also have
$\cR = \bigsqcup_{\lambda\in Y}\{\lambda\}\times (z^\lambda \bN_\cO\cap\bN_\cO)$.
\begin{NB}
    Misha and Sasha, please check. Added on Aug. 27.
\end{NB}%
We denote by $r^{\lambda}$ the fundamental class of the component of $\cR$, corresponding to $\lambda$.
Clearly, as a vector space we have
$$
H^{T_\cO}_{*}(\cR)=\bigoplus\limits_{\lambda\in Y} \Sym(\grt^*)r^{\lambda}
$$
It remains to show that \eqref{relation} holds. For this let us describe the part of the convolution diagram, corresponding to the convolution of the components, corresponding to $\lambda$ and $\mu$.
\begin{NB}
    Following \remref{rem:abelian}, we try to define
    \begin{equation*}
        p'\colon G_\cK\times_{G_\cO}\cR \ni \left[g_1,[g_2,s]\right]
        \mapsto ([g_1,g_2s], [g_2,s])\in\cT\times\cR.
    \end{equation*}
    But the equivalence relation is $[g_1, [g_2,s]] = [g_1b,
    [b^{-1}g_2,s]]$, hence this is \emph{not} well-defined.

    Instead
\end{NB}%

We use the identification
\begin{equation*}
    T_\cK\times_{T_\cO}\cR \cong
    \bigsqcup_{\la,\nu\in Y} \{\la\} \times \{\nu\}
    \times (z^{\nu} \bN_\cO\cap z^\la\bN_\cO)
\end{equation*}
given by $[g_1, [g_2,s]]\mapsto ([g_1], [g_1g_2], g_1g_2s)$. We define
$p'\colon T_\cK\times_{T_\cO}\cR\to \cT\times\cR$ by
\begin{multline}\label{eq:51}
    \{\la\} \times \{\nu\} \times (z^{\nu}
    \bN_\cO\cap z^\la\bN_\cO) \ni (\la, \nu, s)
    \\
    \longmapsto (\la, s, {\nu-\la}, z^{-\la} s) \in \{
    \la\} \times z^\la\bN_\cO \times \{ {\nu-\la}\} \times
    (z^{\nu-\la}\bN_\cO\cap\bN_\cO).
\end{multline}
Let us decompose $p' = (p'_\cT,p'_\cR)$ as before. We have $p'_\cT
\circ q = p_\cT$. For $\cR$, we define $a\colon T_\cK\times\cR\to
T_\cK\times\cR$ by $(g_1, [g_2,s]) \mapsto (g_1, [z^\la g_1^{-1}
g_2,s])$ when the class of $g_1$ is $\la$. (Note $z^\la g_1^{-1}\in
T_\cO$.) Then $p'_\cR \circ q \circ a = p_\cR$. Therefore
\begin{equation}
    \label{eq:50'}
        (q^*)^{-1} p_\cR^* \DC_\cR = p^{\prime*}_\cR\DC_\cR,
        \qquad
        (q^*)^{-1} p_\cR^* \DC_\cR = (q^*)^{-1} a^* q^* p^{\prime*}_\cT \DC_\cT.
\end{equation}

\begin{NB}
    This was wrong, as pointed out by Justin Hilburn.

We have $p' \circ q = p$. Therefore $(q^*)^{-1}p^* = p^{\prime*}$. We have
an isomorphism
\begin{equation}\label{eq:50}
    p^{\prime*}\DC_{\cT\times\cR} \cong (q^*)^{-1}p^*\DC_{\cT\times\cR}
    \cong \DC_{G_\cK\times_{G_\cO}\cR}[2\dim\bN_\cO]
\end{equation}
by \eqref{eq:13}. 
\end{NB}%
In the abelian case, $\cT\times\cR$, $T_\cK\times_{T_\cO}\cR$ (resp.\
$T_\cK\times \cR$) are unions of vector spaces (resp.\ products of
groups and vector spaces), in particular they are smooth and hence
their dualizing sheaves are isomorphic to constant sheaves up to
shifts by the Poincar\'e duality.
\begin{NB}
    $\DC_{\cT\times\cR}\cong \CC_{\cT\times\cR}[2(\dim\cT+\dim\cR)]$,
    $\DC_{T_\cK\times_{T_\cO}\cR}\cong
    \CC_{T_\cK\times_{T_\cO}\cR}[2(\dim\Gr_T+\dim\cR)]$.
\end{NB}%
Therefore $p^{\prime*}_\cR\DC_\cR \cong \CC_{T_\cK\times_{T_\cO}\cR}$,
$p^{\prime*}_\cT\DC_\cT \cong \CC_{T_\cK\times_{T_\cO}\cR}$ up to shifts.
Looking at the construction of isomorphisms in \eqref{eq:13}, we find
that this Poincar\'e duality and \eqref{eq:13} are the same under
\eqref{eq:50'}.
We further replace $p'$ by multiplying $z^\la$, composing the
inclusion $z^\nu\bN_\cO\cap z^\la\bN_\cO\subset z^\nu\bN_\cO$ in the
second factor, and dropping unnecessary entries $\la$, $\nu$:
\begin{equation*}
    p'\colon z^\la\bN_\cO\cap z^\nu\bN_\cO \ni
    s\mapsto (s, s)
    \in 
    z^\la\bN_\cO\times 
    z^\nu\bN_\cO.
\end{equation*}
Then the morphism $m$ is given by the projection to the second factor
$z^\nu\bN_\cO$. As $(q^*)^{-1}a^*q^*$ in \eqref{eq:50'} does not
affect the computation, we have $r^\la r^\mu = m_* p^{\prime*}
[
(z^\la\bN_\cO\cap\bN_\cO)\times
z^\nu\bN_\cO]$.
\begin{NB}
    This is because the class $[(z^\la\bN_\cO\cap\bN_\cO)\times
    z^\nu\bN_\cO]$ is pull-backed to $[(z^\la\bN_\cO\cap\bN_\cO)\times
    (z^\nu\bN_\cO\cap z^\la\bN_\cO)]$ under the inclusion
    $z^\la\bN_\cO\times (z^\nu\bN_\cO\cap z^\la\bN_\cO)\to
    z^\la\bN_\cO\times z^\nu\bN_\cO$.
\end{NB}%

\begin{NB}
    The proof is \emph{not} yet rigorous, as we do not know the
    multiplication is linear in the second variable. Therefore we need
    to compute $r^\la \ast f(\xi) r^\mu$.

    \begin{NB2}
        It is not clear when this \verb+NB+ is added, but the
        linearity is established at the third sentence in the
        beginning of this proof.
    \end{NB2}%
\end{NB}%

To simplify the notation, let us give a proof in the case $n=1$ (the general case is essentially a word-by-word repetition). We set $\xi=\xi_1$. Also, set $a=\xi(\lambda)$, $b=\xi(\mu)$. In the above notation we use $\nu = \la + \mu$, so $a+b=\xi(\nu)$.

From the above argument, the convolution product $r^{\lambda}r^{\mu}$
is given as follows.
First, consider the vector space $W=z^a\calO\oplus z^{a+b}\calO$. 
It contains subspaces $V_{12}=\calO\cap z^a\calO\oplus z^{a+b}\calO$,
$V_{23}=z^a\calO\cap z^{a+b}\calO$ (where $V_{23}$ is embedded into $W$ by means of the diagonal embedding). Let also $m$ denote the projection from $W$ to $z^{a+b}\calO$. We denote by $V$ the subspace of $z^{a+b}\calO$ equal to $\calO\cap z^{a+b}\calO$. Clearly, $m$
sends $V_{12}\cap V_{23}=\calO\cap z^a\calO\cap z^{a+b}\calO$ to $V$ (and when restricted to $V_{12}\cap V_{23}$ the map $p$ is injective).

Let $x_{12}$ and $x_{23}$ be the fundamental
classes of $V_{12}$ and $V_{23}$ viewed as elements of the $T$-equivariant Borel-Moore homology of $W$.
There exists a class $x$ in $H_{*}^T(V)$ such that $m_*(x_{12}\cap x_{23})$ is equal to the direct
image of $x$ from $V$ to $z^{a+b}\calO$. Then $x$ is the product of $r^{\lambda}r^{\mu}$.

On the other hand, let $L$ be any vector space with a $T$-action and let $L_1,L_2$ be two $T$-invariant subspaces.
Let $x_{L_i}$ (resp.\ $x_{L_1\cap L_2}$) be the fundamental class of $L_i$ (resp.\ $L_1\cap L_2$) viewed as an element of $H_{*}^T(L)$. Then
$x_{L_1}\cap x_{L_2}=x_{L_1\cap L_2}\cdot e(\Coker(L/L_1\cap L_2\to L/L_1\oplus L/L_2))$, where $e(\ )$ denotes the equivariant Euler class.
\begin{NB}
$x_{L_1}\cap x_{L_2}=x_{L_1\cap L_2}\cdot \det (\Coker(L/L_1\cap L_2\to L/L_1\oplus L/L_2))^*$.
\end{NB}%
Hence in our case we see that $r^{\lambda}r^{\mu}$ equals $r^{\lambda+\mu}$ multiplied by
\begin{equation}\label{product-torus}
e(\Coker(W/V_{12}\cap V_{23}\to W/V_{12}\oplus W/V_{23}))\cdot
e(V/V_{12}\cap V_{23}).
\end{equation}
\begin{NB}
    \begin{equation*}
        \det(\Coker(W/V_{12}\cap V_{23}\to W/V_{12}\oplus W/V_{23}))^*\otimes \det(V/V_{12}\cap V_{23})^*
    \end{equation*}
\end{NB}
Let us now look at the different options for signs of $a$ and $b$.

1) $a\geq 0, b\geq 0$.

In this case $V_{12}=z^a\calO\oplus z^{a+b}\calO, V_{23}=z^{a+b}\calO=V_{12}\cap V_{23}=V$.
Hence the 2nd equivariant Euler class
\begin{NB}
determinant     
\end{NB}%
in \eqref{product-torus} is equal to 1. Also $W/V_{12}\cap V_{23}=z^a\calO,
W/V_{12}=0, W/V_{23}=z^a\calO$, hence the 1st equivariant Euler class
\begin{NB}
determinant     
\end{NB}%
is equal to 1 as well.

2) $a\leq 0, b\leq 0$.

In this case $V_{12}=\calO\oplus z^{a+b}\calO, V_{23}=z^a\calO, V=V_{12}\cap V_{23}=\calO$.
Hence the 2nd equivariant Euler class
\begin{NB}
determinant     
\end{NB}%
is equal to 1 again. Also $W/V_{12}=z^a\calO/\calO, W/V_{23}=z^{a+b}\calO$ and
$W/V_{12}\cap V_{23}=z^a\calO/\calO\oplus z^{a+b}\calO$. Hence the 1st equivariant Euler class 
\begin{NB}
determinant 
\end{NB}%
is again 1.

3) $a\geq 0, b\leq 0$.

In this case we have $V_{12}=z^a\calO\oplus z^{a+b}\calO$, $V_{23}=z^a\calO$, $V_{12}\cap V_{23}=z^a\calO$, $V=\calO\cap z^{a+b}\calO$.
So $W/V_{12}=0, W/V_{23}=z^{a+b}\calO, W/V_{12}\cap V_{23}=z^{a+b}\calO$ and hence the first equivariant Euler class \begin{NB} determinant \end{NB}%
 in \eqref{product-torus} is equal to 1. On the other hand the 2nd equivariant Euler class \begin{NB} determinant \end{NB}%
 is equal to $e(\calO\cap z^{a+b}\calO/z^a\calO)$
\begin{NB}
$\det(\calO\cap z^{a+b}\calO/z^a\calO)^*$      
\end{NB}%
which is equal to $\xi^{-b}$
\begin{NB}
$(-\xi)^{-b}$    
\end{NB}%
if $a+b\geq 0$ and $\xi^a$
\begin{NB}
$(-\xi)^a$    
\end{NB}%
if $a+b\leq 0$. This is exactly $\xi^{d(a,b)}$.
\begin{NB}
$(-\xi)^{d(a,b)}$.
\end{NB}%

4) $a\leq 0, b\geq 0$.

In this case we have $V_{12}=\calO\oplus z^{a+b}\calO$, $V_{23}=z^{a+b}\calO$, $V_{12}\cap V_{23}=\calO\cap z^{a+b}\calO=V$.
Hence in this case the 2nd equivariant Euler class \begin{NB} determinant \end{NB}%
\eqref{product-torus} is equal to 1. On the other hand, assume that $a+b\geq 0$.
Then $V_{12}\cap V_{23}=z^{a+b}\calO$ and $W/V_{12}\cap V_{23}=z^a\calO$. On the other hand,
$W/V_{12}=z^a\calO/\calO$, $W/V_{23}=z^a\calO$ and hence the first equivariant Euler class \begin{NB} determinant \end{NB}%
in \eqref{product-torus} is equal to
$e(z^a\calO/\calO) = \xi^{|a|}
\begin{NB}
\det(z^a\calO/\calO)= (-\xi)^{|a|}    
\end{NB}%
$. But since $a+b\geq 0$ we have $|a|=d(a,b)$.
Similarly, if $a+b\leq 0$, then $V_{12}\cap V_{23}=\calO$. Hence $W/V_{12}=z^a\calO/\calO, W/V_{23}=z^a\calO$ and
$W/V_{12}\cap V_{23}=z^a\calO\oplus z^{a+b}\calO/\calO$. Hence the equivariant Euler class \begin{NB} determinant \end{NB}%
in question is equal
to $e(z^a\calO/z^{a+b}\calO)=\xi^b=\xi^{d(a,b)}
\begin{NB}
\det (z^a\calO/z^{a+b}\calO)^*=(-\xi)^b=(-\xi)^{d(a,b)}
\end{NB}%
$.
\end{proof}

\subsection{The quantized version}

We now want to do everything $\CC^\times$-equivariantly (where $\CC^\times$ acts 
by loop rotation). Let $\hbar$ denote the generator of 
$H^*_{\CC^\times}(\mathrm{pt})$.
Then the description of $H^{T_\cO\rtimes \CC^\times}_{*}(\calR)$ is similar, except 
the relation \eqref{relation} now reads
\begin{equation}\label{q-relation}
    r^{\lambda}r^{\mu}=\prod\limits_{i=1}^n A_i(\lambda,\mu)\,r^{\lambda+\mu},
\end{equation}
where
\begin{NB}
$$
A_i(\lambda,\mu)=
\begin{cases}
\prod\limits_{j=1}^{d(\xi_i(\lambda),\xi_i(\mu))}(-\xi_i+(j-\xi_i(\lambda))\hbar)\
\text{if $\xi_i(\lambda)\geq 0\geq \xi_i(\mu)$}\\
\prod\limits_{j=1}^{d(\xi_i(\lambda),\xi_i(\mu))}(-\xi_i+(1-j-\xi_i(\lambda))\hbar)\
\text{if $\xi_i(\lambda)\leq 0\leq \xi_i(\mu)$}
\end{cases}
$$
and $A_i(\lambda,\mu)=1$ if $\xi_i(\lambda)$ and $\xi_i(\mu)$ have the same sign.    
\end{NB}%
$$
A_i(\lambda,\mu)=
\begin{cases}
\prod\limits_{j=1}^{d(\xi_i(\lambda),\xi_i(\mu))}(\xi_i+(\xi_i(\lambda)-j + \frac12)\hbar) & 
\text{if $\xi_i(\lambda)\geq 0\geq \xi_i(\mu)$,}\\
\prod\limits_{j=1}^{d(\xi_i(\lambda),\xi_i(\mu))}(\xi_i+(\xi_i(\lambda)+j-\frac12)\hbar)& 
\text{if $\xi_i(\lambda)\leq 0\leq \xi_i(\mu)$,}\\
1 & \text{otherwise.}
\end{cases}
$$
The proof easily follows from the same calculations as above.
Note that $1/2$ is coming from the action of $\CC^\times$ on $\bN$
with weight $1/2$.

\begin{NB}
    Corrected on Aug. 31:

    The cases 1),2) are trivial. In case 3), we compute the weights of
    $\shfO\cap z^{a+b}\shfO/z^a\shfO$. We have $z^{a-1}$, $z^{a-2}$,
    \dots ($d(a,b)$ times). Therefore $\xi_i + (a-1)\hbar$, $\xi_i +
    (a-2)\hbar$, \dots. In case 4), we compute weights of
    $z^a\shfO/\shfO$ or $z^a\shfO/z^{a+b}\shfO$ according to $a+b\ge
    0$ or $a+b\le 0$. We have $z^a$, $z^{a+1}$, \dots ($d(a,b)$
    times). Therefore $\xi_i + a\hbar$, $\xi_i + (a+1)\hbar$, \dots.
\end{NB}%

In addition, $\Sym_{\CC[\hbar]}(\grt^*)$ now does not commute with $r^{\lambda}$'s. But it is easy to see that for any $\alpha\in \grt^*$ we have
\begin{equation}
    \label{eq:27}
[r^{\lambda},\alpha]=\hbar \alpha(\lambda) r^{\lambda}.
\begin{NB}
[r^{\lambda},\alpha]=-\hbar \alpha(\lambda) r^{\lambda}.
\end{NB}%
\end{equation}
See \lemref{lem:former_claim}.

\begin{NB}
    Misha and Sasha : This is incompatible with \lemref{lem:comoment}.

    Aug. 31, it is corrected.
\end{NB}%

\subsection{Flavor symmetries}Let $\widetilde{T}=T\times (\CC^\times)^n$. 
Then $\widetilde{T}$ acts naturally on $\bfN$ and thus also on $\cR$ and we can consider the algebra $H^{\tilT_\cO\rtimes \CC^\times}_*(\cR)$.
This algebra has a presentation similar to the above, except that the relation \eqref{q-relation} now reads
\begin{equation}\label{qflavor-relation}
r^{\lambda}r^{\mu}=\prod\limits_{i=1}^n \tilA_i(\lambda,\mu)\,r^{\lambda+\mu},
\end{equation}
where 
\begin{NB}
$$
\tilA_i(\lambda,\mu)=
\begin{cases}
\prod\limits_{j=1}^{d(\xi_i(\lambda),\xi_i(\mu))}(-b_i-\xi_i+(j-\xi_i(\lambda))\hbar)\
\text{if $\xi_i(\lambda)\geq 0\geq \xi_i(\mu)$}\\
\prod\limits_{j=1}^{d(\xi_i(\lambda),\xi_i(\mu))}(-b_i-\xi_i+(1-j-\xi_i(\lambda))\hbar)\
\text{if $\xi_i(\lambda)\leq 0\leq \xi_i(\mu)$}
\end{cases}
$$
\end{NB}%
$$
\tilA_i(\lambda,\mu)=
\begin{cases}
\prod\limits_{j=1}^{d(\xi_i(\lambda),\xi_i(\mu))}(b_i+\xi_i+(\xi_i(\lambda)-j+\frac12)\hbar) &
\text{if $\xi_i(\lambda)\geq 0\geq \xi_i(\mu)$,}\\
\prod\limits_{j=1}^{d(\xi_i(\lambda),\xi_i(\mu))}(b_i+\xi_i+(\xi_i(\lambda)+j-\frac12)\hbar)&
\text{if $\xi_i(\lambda)\leq 0\leq \xi_i(\mu)$,}\\
1 & \text{otherwise},
\end{cases}
$$
where $b_1,\dots,b_n$ denote the equivariant parameters for the torus $(\CC^\times)^n$.

\subsection{Examples}\label{subsec:abel_examples}
\begin{NB}
Let $T=\CC^\times, n=1$ and take $\xi$ to be the $r$-th power the standard character.
In this case $\Lambda=\ZZ$. Set $x=z^{-1}, y=z^1$ (note that $z^{-1}$ is not the inverse of $z$!). Let also $a$ denote the generator of $\grt^*$. Then $x$, $y$ and $a$ generate $H_{*}^{T\rtimes \CC^\times}(\calR)$ and
we have the following relations:
$$
xy=\prod\limits_{j=1}^n (a+j\hbar);\quad yx=\prod\limits_{j=1}^n (a-(j-1)\hbar), [x,a]=\hbar x, [y,a]=-\hbar y
$$
When $\hbar=0$ we get a commutative ring, generated by $x,y,a$ with relation $xy=a^r$.
\end{NB}%

Let $T=\CC^\times, n=1$ and take $\xi$ to be the $N$-th power the standard character.
In this case $\Lambda=\ZZ$. Set $x=r^1, y=r^{-1}$ (note that $r^{-1}$ is not the inverse of $r^1$!). Let also $a$ denote the generator of $\grt^*$. Then $x$, $y$ and $a$ generate $H_{*}^{T_\cO\rtimes \CC^\times}(\cR)$ and
we have the following relations:
$$
xy=\prod\limits_{j=1}^N (a+(j-\frac12)\hbar),\quad
yx=\prod\limits_{j=1}^N (a-(j-\frac12)\hbar), \quad [x,a]=\hbar x,\quad [y,a]=-\hbar y
$$
When $\hbar=0$ we get a commutative ring, generated by $x,y,a$ with relation $xy=a^N$.

On the other hand, let $\bfN$ be the direct sum of $N$ standard characters. We define $x,y$ and $a$ in the same way as above.
Then we have
$$
xy=(a+\frac12\hbar)^N,\quad yx=(a-\frac12\hbar)^N, \quad [x,a]=\hbar x, \quad [y,a]=-\hbar y.
$$
Note that for $\hbar=0$ we get the same algebra as before.

Let $\deg_h$ denote the half of the homological degree. We have
$\deg_h x = 0$, $\deg_h y = N$, $\deg_h a = 1$. On the other hand, the
degree $\deg$ given by $\Delta(\la)$ by \eqref{eq:18} is $\deg x =
N/2$, $\deg y = N/2$, $\deg a = 1$. If we multiply it by $2$, it is
the weight of the $\CC^\times$-action, which is the restriction of the
$\SU(2)$-action on $\{xy=a^N\} = \CC^2/(\ZZ/N\ZZ)$, rotating the
hyper-K\"ahler structure. Namely we identify $\CC^2$ with the
quaternions $\HH = \{ 
  z_1 + z_2 j\mid z_1,z_2\in\CC\}.
\)
The hyper-K\"ahler structure is given by multiplication of $i$, $j$,
$k$ from the left, with the standard inner product. The action of
$\ZZ/N\ZZ$ is given by multiplication of $N$th roots of unity from the
right. Then $x = z_1^N$, $y=z_2^N$, $a=z_1 z_2$.
The action of $\SU(2)=\grpSp(1)$ is the multiplication of unit
quaternions from the left.
If we restrict to $\U(1)$, unit complex numbers, $z_1$, $z_2$ both
have weight $1$.

\subsection{Changing \texorpdfstring{$\xi$}{\textxi} to 
\texorpdfstring{$-\xi$}{-\textxi}}\label{subsec:change_dual}
Let us assume that we change one of the $\xi$ to $-\xi$; let us denote
by $\bfN_i$ the corresponding new representation of $T$. We claim that
the resulting algebra $\cA[{T,\bfN_i}]$ is isomorphic to
$\cA[{T,\bfN}]$.
\begin{NB}
(it is not clear to me whether this is true if we also
add the loop rotation).

This problem is fixed by introducing $1/2$.
\end{NB}%
\begin{NB}
    Misha and Sasha, I do not see the reason why $k$ is necessary.

Indeed, let $k$ be the minimal positive value of $\xi_i(\lambda)$.
\end{NB}%
 Then define a map
$\sigma\colon \cA[{T,\bfN_i}]\to \cA[{T,\bfN}]$ which is equal to identity on $\grt^*$ and such that
\begin{NB}
$$
\sigma(r^{\lambda})=
\begin{cases}
r^{\lambda}\quad \text{if $\xi_i(\lambda)\leq 0$}\\
(-1)^{k\xi(\lambda)} r^{\lambda} \quad \text{if $\xi_i(\lambda)> 0$}
\end{cases}
$$
\end{NB}%
$$
\sigma(r^{\lambda})=
\begin{cases}
r^{\lambda}\quad \text{if $\xi_i(\lambda)\leq 0$}\\
(-1)^{\xi_i(\lambda)} r^{\lambda} \quad \text{if $\xi_i(\lambda)> 0$}
\end{cases}
$$
In particular, we see that the algebra $\calA({T,\bfN})$ depends only on $\bM=\bfN\oplus \bfN^*$ and not on $\bfN$.

\begin{NB}
    Observe that
    \begin{equation*}
        (-\xi_i)^{d(-\xi_i(\la),-\xi_i(\mu))}
        = (-1)^{d(\xi_i(\la),\xi_i(\mu))}\xi_i^{d(\xi_i(\la),\xi_i(\mu))}.
    \end{equation*}
    Since
    \begin{equation*}
        \frac{\sigma(r^\la)\sigma(r^\mu)}{\sigma(r^{\la+\mu})}
        = (-1)^{\max(\xi_i(\la),0)+\max(\xi_i(\mu),0)-\max(\xi_i(\la+\mu),0)} 
        \frac{r^\la r^\mu}{r^{\la+\mu}},
    \end{equation*}
    it is enough to note that
    \begin{equation}\label{eq:26}
        \begin{split}
            & \max(\xi_i(\la),0)+\max(\xi_i(\mu),0)-\max(\xi_i(\la+\mu),0)
\\
        = \; &
        \begin{cases}
            \min (|\xi_i(\la)|,|\xi_i(\mu)|) & 
            \text{if $\xi_i(\la)$ and $\xi_i(\mu)$ have different signs}
          \\
          0 & \text{otherwise}
        \end{cases}
        = d(|\xi_i(\la)|,|\xi_i(\mu)|).
        \end{split}
    \end{equation}
\end{NB}%

The same is true for the quantized case. (It is not true if we do not
introduce the weight $1/2$ action of $\CC^\times$ on $\bN$.)

\begin{NB}
    In the quantized case, we observe
    \begin{equation*}
        \prod_{j=1}^{d(-\xi_i(\la),-\xi_i(\mu))}
        (-\xi_i + (-\xi_i(\la) \pm (j - \frac12) \hbar)
        = (-1)^{d(\xi_i(\la),\xi_i(\mu))}
        \prod_{j=1}^{d(\xi_i(\la),\xi_i(\mu))}
        (\xi_i + (\xi_i(\la) \mp (j - \frac12) \hbar).
    \end{equation*}
    Since two cases of $A_i(\la,\mu)$ are exchanged when we replace
    $\xi_i$ by $-\xi_i$, the argument remains the same as the
    non-quantized case.
\end{NB}%

\subsection{Changing the representation}\label{sec:chang-repr}
Let $\bfV$ be another representation. Then we claim that there is a natural embedding $\cA[T,\bfN\oplus \bfV]\hookrightarrow \cA[{T,\bfN}]$.
To construct it, it is enough to assume that $\bfV=\xi$ for some $\xi$.
Then the $\eta$ which is equal to identity of $\Sym(\grt^*)$ and
such that
\begin{NB}
\begin{equation*}
\eta(r^{\lambda})=
\begin{cases}
(-\xi)^{\xi(\lambda)}r^{\lambda}\quad \text{if $\xi(\lambda)\geq 0$}\\
r^{\lambda}\quad\text{otherwise}
\end{cases}
\end{equation*}
\end{NB}%
\begin{equation}\label{eqV}
\eta(r^{\lambda})=
\begin{cases}
\xi^{-\xi(\lambda)}r^{\lambda}& \text{if $\xi(\lambda) < 0$}\\
r^{\lambda} & \text{otherwise}
\end{cases}
\end{equation}
defines the embedding.

\begin{NB}
    Since $\eta(r^\la) = \xi^{\max(-\xi(\la),0)}r^\la$, we have
    \begin{equation*}
        \begin{split}
        \eta(r^\la)\eta(r^\mu)
        & = \xi^{\max(-\xi(\la),0)}\xi^{\max(-\xi(\mu),0)} 
        r^{\la+\mu} \prod_{\bN} \xi_i^{d(\xi_i(\la),\xi_i(\mu))}
\\
        & = \xi^{\max(-\xi(\la),0) + \max(-\xi(\mu),0) - \max(-\xi(\la+\mu),0)}
         \prod_{\bN} \xi_i^{d(\xi_i(\la),\xi_i(\mu))}
        \eta(r^{\la+\mu}),
        \end{split}
    \end{equation*}
    where $\xi_i$ runs over weights of $\bN$. Now by \eqref{eq:26},
    \begin{equation*}
        \eta(r^\la) \eta(r^\mu) = \eta(r^{\la+\mu})
        \prod_{\bN\oplus\mathbf V}\xi_i^{d(\xi_i(\la),\xi_i(\mu))},
    \end{equation*}
    where $\xi_i$ runs over the weights of $\bN\oplus\mathbf V$ this time.
\end{NB}%

This embedding is given by the pull-back homomorphism
$H^{T_\cO}_*(\cR_{T,\bN\oplus\mathbf V})\to H^{T_\cO}_*(\cR_{T,\bN})$
with respect to the embedding $\cT_{T,\bN}\subset
\cT_{T,\bN\oplus\mathbf V}$ of a subbundle. In \subsecref{sec:z}, we
consider a similar pull-back with respect to
$\bz\colon\Gr_G\hookrightarrow\cT_{G,\bN}$ for general $(G,\bN)$.
\begin{NB}
    Misha and Sasha, I correct Sasha's formula on Sep.\ 8.

    The embedding $\cR\subset\cT$ at the coweight $\la$ is
    $z^\la\bN_\cO\cap \bN_\cO\subset z^\la\bN_\cO$. We also have
    $z^\la(\bN\oplus\mathbf V)_\cO\cap (\bN\oplus\mathbf V)_\cO\subset
    z^\la(\bN\oplus\mathbf V)_\cO$. Therefore for $\mathbf
    V$-component, we have $z^\la\mathbf V_\cO\cap \mathbf V_\cO =
    z^{\max(\xi(\lambda),0)}\cO \subset z^\la\mathbf V_\cO =
    z^{\xi(\la)}\cO$. Therefore we should have
    \begin{equation*}
        \eta(r^{\lambda})= \xi^{\max(-\xi(\la),0)} r^\la.
    \end{equation*}
    Since
    \begin{equation*}
        \max(-\xi(\la),0) + {\max(-\xi(\mu),0)}{-\max(-\xi(\la+\mu),0)}
      = d(\xi(\la),\xi(\mu)),
    \end{equation*}
    the proof that $\eta$ is an algebra homomorphism remains to work.
\end{NB}%
The formula \eqref{eqV} is understood as the multiplication of
$e(z^\la \mathbf V_\cO/z^\la \mathbf V_\cO\cap \mathbf V_\cO) = e(z^{\xi(\la)}\cO/z^{\max(\xi(\la),0)}\cO)$.

In the quantized case, we define
\begin{equation*}
    \eta(r^\la) = 
    \begin{cases}
        \displaystyle \prod_{j=0}^{-\xi(\la)-1} (\xi + (\xi(\la) +
        j + \frac12)\hbar) r^\la & \text{if $\xi(\la)< 0$},
        \\
        r^\lambda & \text{otherwise}.
    \end{cases}
\end{equation*}
Using \eqref{eq:27}, we can check that this is an algebra
homomorphism.

\begin{NB}
    Let us give a proof.

    (a) Suppose $\xi(\la)$, $\xi(\eta) \ge 0$. Then $\eta(r^\la)\eta(r^\mu)
    = r^\la r^\mu = r^{\la+\mu} = \eta(r^{\la+\mu})$.

    (b) Suppose
    $\xi(\la)$, $\xi(\eta) < 0$. We have
    \begin{equation*}
        \begin{split}
            & \eta(r^\la) \eta(r^\mu)
        =  \prod_{j=0}^{-\xi(\la)-1} (\xi + (\xi(\la) + j + \frac12)\hbar) r^\la
        \prod_{k=0}^{-\xi(\mu)-1}(\xi + (\xi(\mu) + k + \frac12)\hbar) r^\mu
\\
=\; & \prod_{j=0}^{-\xi(\la)-1} (\xi + (\xi(\la) + j + \frac12)\hbar) 
        \prod_{k=0}^{-\xi(\mu)-1}(\xi + (\xi(\la) + \xi(\mu) + k + \frac12)\hbar) r^\la r^\mu
\\
=\; & \prod_{j=-\xi(\mu)}^{-\xi(\la)-\xi(\mu)-1}
(\xi + (\xi(\la) + \xi(\mu) + j + \frac12)\hbar) 
        \prod_{k=0}^{-\xi(\mu)-1}(\xi + (\xi(\la) + \xi(\mu) + k + \frac12)\hbar) r^\la r^\mu
\\
=\; & \prod_{j=0}^{-\xi(\la)-\xi(\mu)-1}
(\xi + (\xi(\la) + \xi(\mu) + j + \frac12)\hbar) r^\la r^\mu
= \eta(r^{\la+\mu}),
        \end{split}
    \end{equation*}
    where we have used \eqref{eq:27} in the second equality.

    (c) Suppose $\xi(\la)\ge 0 \ge \xi(\mu)$. We have
\allowdisplaybreaks{
\begin{equation*}
        \begin{split}
            & \eta(r^\la) \eta(r^\mu)
        =  r^\la
        \prod_{k=0}^{-\xi(\mu)-1}(\xi + (\xi(\mu) + k + \frac12)\hbar) r^\mu
\\
=\; & \prod_{k=0}^{-\xi(\mu)-1}(\xi + (\xi(\la) + \xi(\mu) + k + \frac12)\hbar)
r^\la r^\mu
\\
=\; & \prod_{k=0}^{-\xi(\mu)-1}(\xi + (\xi(\la) + \xi(\mu) + k + \frac12)\hbar) 
r^{\la+\mu}
\\
=\;
& 
\begin{cases}
    \displaystyle\prod_{k=0}^{-\xi(\mu)-1}(\xi + (\xi(\la) + \xi(\mu) + k + \frac12)\hbar) 
\eta(r^{\la+\mu})
    & \text{if $\xi(\la)+\xi(\mu)\ge 0
      \Leftrightarrow d(\xi(\la),\xi(\mu)) = -\xi(\mu)$}
\\
    \frac{\displaystyle\prod_{k=0}^{-\xi(\mu)-1}(\xi + (\xi(\la) + \xi(\mu) + k + \frac12)\hbar)}
    {\displaystyle\prod_{j=0}^{-\xi(\la)-\xi(\mu)-1}
    (\xi + (\xi(\la)+\xi(\mu)+j + \frac12)\hbar)}
    \eta(r^{\la+\mu})
    & \text{if $\xi(\la)+\xi(\mu)\le 0
      \Leftrightarrow d(\xi(\la),\xi(\mu)) = \xi(\la)$}
\end{cases}
\\
=\; &
\prod_{j=1}^{d(\xi(\la),\xi(\mu))} (\xi + (\xi(\la)-j + \frac12)\hbar) \eta(r^{\la+\mu}).
        \end{split}
    \end{equation*}

    (d) Suppose $\xi(\la)\le 0 \le \xi(\mu)$. We have
}
    \begin{equation*}
        \begin{split}
            & \eta(r^\la) \eta(r^\mu)
            =  \prod_{j=0}^{-\xi(\la)-1} (\xi + (\xi(\la) + j + \frac12)\hbar) r^\la
        r^\mu
\\
=\; & \prod_{j=0}^{-\xi(\la)-1}(\xi + (\xi(\la) + j + \frac12)\hbar) r^{\la+\mu}
\\
=\;
& 
\begin{cases}
    \displaystyle\prod_{j=0}^{-\xi(\la)-1}(\xi + (\xi(\la) + j + \frac12)\hbar)
    \eta(r^{\la+\mu})
    & \text{if $\xi(\la)+\xi(\mu)\ge 0
      \Leftrightarrow d(\xi(\la),\xi(\mu)) = -\xi(\la)$}
\\
    \frac{\displaystyle\prod_{j=0}^{-\xi(\la)-1}(\xi + (\xi(\la) + j + \frac12)\hbar)}
    {\displaystyle\prod_{k=0}^{-\xi(\la)-\xi(\mu)-1}
    (\xi + (\xi(\la)+\xi(\mu)+k + \frac12)\hbar)}
    \eta(r^{\la+\mu})
    & \text{if $\xi(\la)+\xi(\mu)\le 0
      \Leftrightarrow d(\xi(\la),\xi(\mu)) = \xi(\mu)$}
\end{cases}
\\
=\; &
\prod_{j=1}^{d(\xi(\la),\xi(\mu))} (\xi + (\xi(\la)+j- \frac12)\hbar) \eta(r^{\la+\mu}).
        \end{split}
    \end{equation*}
    In the second case, replace $k$ in the denominator by $k-\xi(\mu)$
    and absorb it into the numerator. Then the product is from $j=0$ to $\xi(\mu)-1$. Then replace $j$ by $j+1$.
\end{NB}

The following observation might be important: although for general
$\bfV$ the above map does not respect the gradings on
$\cA[{T,\bfN\oplus \bfV}]$ and $\cA[{T,\bfN}]$, the gradings
\emph{are} preserved if $\bfV$ is self-dual.

\begin{Remark}
  We can recover \thmref{abel} and its quantized
  version~\eqref{q-relation} from the above formulas together with
  $r^\lambda r^\mu = r^{\lambda+\mu}$ when $\bN=0$.
\end{Remark}

\begin{NB}
    Note that the degree depends on the representation. So let us
    denote the degree as $\deg_\bN$. Suppose $\mathbf V = \xi$ as
    above. We have
    \begin{equation*}
        \deg_\bN \eta(r^\la) = 
        \deg_\bN r^\la + 2\max(\xi(\lambda),0)
    \end{equation*}
    by the definition of $\eta$. On the other hand, by the definition
    of the degree, we have
    \begin{equation*}
        \deg_{\bN\oplus\mathbf V} r^\la = \deg_\bN r^\la
        + |\xi(\la)|.
    \end{equation*}
    Hence if we add $\xi\oplus(-\xi)$, the degree is preserved as
    \begin{equation*}
        2\max(\xi(\lambda),0) + 2\max(-\xi(\lambda),0)
        = 2|\xi(\lambda)|
        = |\xi(\lambda)| + |(-\xi)(\lambda)|.
    \end{equation*}
\end{NB}%

\begin{NB}
Also note that the above map does not respect the Poisson brackets on
$\calA({T,\bfN\oplus \bfV})$ and $\calA({T,\bfN})$, hence it doesn't
extend to quantization.

\begin{NB2}
    Misha and Sasha, this is not compatible with
    \lemref{lem:convolution}(3). Where the contradiction comes from ?

    It is confirmed that this is not true.
\end{NB2}
\end{NB}%

\begin{NB}
\begin{NB2}
    Sasha says the following two subsections are not correct, so is
    moved to NB.
\end{NB2}

\subsection{The algebra \texorpdfstring{$\tcalA_{G,\bfN}$}{\textasciitilde A[G,N]} via cell decomposition}

Let now $\bfN$ be a representation of a reductive group $G$. Recall that we have set
$$
\calA_{G,\bfN}=H^G_{*,BM}(\calR_{G,\bfN}), \quad\tcalA_{G,\bfN}=H^T_{*,BM}(\calR_{G,\bfN}).
$$
Note that $\calA_{G,\bfN}=\tcalA_{G,\bfN}^W$ where $W$ is the Weyl group.

Also, as we have previously discussed there is a natural embedding $\iota_{\bfN}:\calA_{T,\bfN}\hookrightarrow\tcalA_{G,\bfN}$, which becomes an isomorphism after localization. In fact, it is easy to describe $\tcalA_{G,\bfN}$ as a subspace in the field of fractions of $\calA_{T,\bfN}$. Namely, the space $\calR_{G,\bfN}$ has a $T$-invariant cell decomposition - the cells are preimages of
the Iwahori orbits on $\Gr_G$. These orbits are indexed by coweights $\lambda\in \Lambda$ and we denote the corresponding orbit by $I_{\lambda}$. Also, we set $\calR_{\lambda}$ to be its preimage in $R_{G,\bfN}$. We let $x^{\lambda}$ denote the
corresponding element (fundamental class) in $\tcalA_{G,\bfN}$. Note that if $G=T$ then $x^{\lambda}=z^{\lambda}$.

In the general case we can only claim the following two (obvious) things:

1) The algebra $\tcalA_{T,\bfN}$ is freely generated over $\Sym(\grt^*)$ by the $x^{\lambda}$'s.

2) For any $\lambda$ let $\lambda_+$ be the unique element of $\Lambda_+$ (dominant coweights) on the $W$-orbit of $\lambda$.
Then we have
$$
\iota_{\bfN}(z^{\lambda})=\sum\limits_{\mu\ \text{such that }\mu_+<\lambda_+\ \text{or}\ \lambda_+=\mu_+,\mu\leq \lambda}
a_{\lambda,\mu}x^{\mu},
$$
where $a_{\lambda,\mu}\in \Sym(\grt^*)$. Moreover, we have
$$
a_{\lambda,\lambda}=\pm \prod\limits_{\alpha\in \Delta_+}\alpha^{|\alpha(\lambda)|-i(\lambda,\alpha)}
$$
where $i(\lambda,\alpha)=0$ if $\alpha(\lambda)\geq 0$ and 1 otherwise.
\begin{NB2}
    Sasha, I do not see where this formula comes from....
\end{NB2}

Let us set $\Delta(\lambda)=\prod\limits_{\alpha\in \Delta_+}\alpha^{|\alpha(\lambda)|}$. 
\begin{NB2}
    Sasha, where $-i(\la,\alpha)$ gone ?
\end{NB2}%
Note that $\Delta(\lambda)$ is $W$-invariant up to sign. Then we have
\begin{Corollary}
$\calA_{G,\bfN}$ is contained in $W$-invariants in the $\Sym(\grt^*)$-
span of all the $\frac{z^{\lam}}{\Delta(\lambda)}$.
\end{Corollary}

\subsection{Example}Let now $G=\GL(N)$, $\bfN=(\CC^N)^r\oplus Ad$ where $Ad$ denotes the adjoint representation.
Note that $Ad$ is self-dual.
Then from \eqref{eqV} it follows that
$\calA_{T,\bfN}$ embeds into $\calA_{T,(\CC^N)^r)}$ which is equal to functions on $(\CC^2/\ZZ_r)^N$.
So, $\calA_{T,\bfN}^{S_N}$ sits inside both $\calO(\Sym^N(\CC^2/\ZZ_r))$ and $\calA_{G,\bfN}$ (and all the maps are isomorphisms
after localization).
But \eqref{eqV} together with the above corollary actually implies that $\calA_{G,\bfN}$ is contained in $\calO(\Sym^N(\CC^2/\ZZ_r))$ and since $Ad$ is self-adjoint
this embedding preserves grading. Hence it is an isomorphism by Hiraku's calculation of characters.

\begin{NB2}
    Sasha, I do not see the argument. Do you mean
    $\eta(z^\la/\Delta(\la)) = z^\la$, for $\eta$ removing
    $\operatorname{Ad}$ ? In this case, $\eta$ multiply
    \begin{equation*}
        \prod_{\alpha\in\Delta}(-\alpha)^{\max(\alpha(\la),0)}   
        = \pm \prod_{\alpha\in\Delta_+} \alpha^{|\alpha(\la)|}
    \end{equation*}
    right ?
\end{NB2}

The above calculation can be generalized to the algebra $\calA_{G,\bfN,\hbar,\slt}$. We claim that in the above situation this
algebra is naturally isomorphic to the spherical subalgebra in the cyclotomic Cherednik algebra of type A (I will try to write it down later).

\begin{NB2}
    Sasha, Are you still planning to do this ?
\end{NB2}

\end{NB}

\subsection{Toric hyper-K\"ahler manifolds}\label{sec:toric-hyper-kahler}

Consider a short exact sequence
\begin{equation*}
  0 \to \ZZ^{d-n}\xrightarrow{\alpha} \ZZ^d \xrightarrow{\beta} \ZZ^n \to 0,
\end{equation*}
and the associated sequence
\begin{equation}\label{eq:toric}
    1 \to G = T^{d-n} \xrightarrow{\alpha} T^d
    \xrightarrow{\beta} G_F = T^n \to 1.
\end{equation}
Let $\bN = \CC^d$, considered as a representation of $G$ through
$\alpha$.  By \subsecref{sec:triv-properties}\ref{item:reduction}, the
Coulomb branch $\mathcal M_C(G,\bN)$ is the Hamiltonian reduction of
$\mathcal M_C(T^d,\CC^d)$ by $G_F^\vee$. We have
$\mathcal M_C(T^d,\CC^d)\cong \CC^{2d}$ by \thmref{abel}. Hence
$\mathcal M_C(G,\bN)$ is, by definition, a toric hyper-K\"ahler
manifold associated with the dual exact sequence of \eqref{eq:toric},
introduced by Bielawski and Dancer in \cite{MR1792372}.\footnote{The
  name `toric hyper-K\"ahler manifold' and a special class of examples
  were introduced earlier in \cite{MR1187554}. Many people use a new
  name `hypertoric manifold', but we use the original name to respect
  the contribution of Goto, Bielawski-Dancer.}

The coordinate ring of the Hamiltonian reduction of $\CC^{2d}$ by
$G_F^\vee$ has the presentation given by \thmref{abel}. In fact, this
can be checked directly. See \cite[\S5(ii)]{2015arXiv150303676N}.

\begin{NB}
\renewenvironment{NB}{
\color{blue}{\bf NB2}. \footnotesize
}{}
\renewenvironment{NB2}{
\color{purple}{\bf NB3}. \footnotesize
}{}
\input{hikita_temp}
\end{NB}

\section{Codimension \texorpdfstring{$1$}{1} reduction}\label{sec:abel}

Recall that $T$ is a maximal torus of $G$ with Lie algebra $\ft$. Let $W$
be the Weyl group of $G$.

In this section, we develop a tool to analyze $\mathcal M_C$ based on
the idea in \cite{MR2135527}. It says that it is enough to find an
affine scheme $\mathcal M$, flat over $\ft/W$ such that it is equal to
$\mathcal M_C$ over an open subset $\ft^\bullet/W$ of $\ft/W =
\Spec(H^*_G(\mathrm{pt}))$ whose complement has codimension at least
$2$. Then automatically $\mathcal M_C = \mathcal M$. See \thmref{prop:flat} for more detail.

It is easy to determine $\mathcal M_C$ on a smaller open subset
$\ft^\circ/W$ of $\ft/W$, the complement of certain hyperplanes. We
prove $H^{G_\cO}_*(\cR)|_{\loc/W} \cong \CC[\ft\times
T^\vee]|_\loc^W$, where $T^\vee$ is the dual torus of $T$. This will
be done in \subsecref{sec:classical} after preparation.
At a point $t$ in $(\ft^\bullet\setminus \ft^\circ)/W$, we will show
that $\mathcal M_C$ is the Coulomb branch of another pair $(G',\bN')$
such that $G'$ has semisimple rank at most $1$. (See
\lemref{lem:fixedpts}.)
We can determine $\mathcal M_C(G',\bN')$: \secref{sec:abelian} for
abelian, \lemref{lem:SL2} for $G'=\PGL(2)$ or $\SL(2)$, and the proof
of \propref{prop:normality} in general.
Therefore we just need to look for $\mathcal M$ such that it is
isomorphic to $\mathcal M_C(G',\bN')$ for each $t\in
(\ft^\bullet\setminus \ft^\circ)/W$.
Using these, we can determine, for example, $(G,\bN)$ associated with
a quiver gauge theory of type ADE. See \ref{QGT}.

\subsection{Fixed points and generalized roots}\label{sec:fixed-points}

\begin{Lemma}\label{lem:fixedpts}
    Let $t\in\Lie T$. Let $\cR^t_{G,\bN}$ be the fixed point set of
    $\exp(\RR t)$ in $\cR_{G,\bN}$, i.e., the zero locus of the vector
    field generated by $t$. Then
    \begin{equation*}
        \cR^t_{G,\bN} \cong \cR_{Z_G(t), \bN^t},
    \end{equation*}
    where $Z_G(t)$ is the centralizer of $t$ in $G$, and $\bN^t$ is
    the subspace of $t$ invariants, considered as a representation of
    $Z_G(t)$.
    Similarly $\cT^t_{G,\bN}\cong \cT_{Z_G(t),\bN^t}$.
\end{Lemma}

\begin{proof}
    For the Grassmann part, it is known that $\Gr_G^t \cong
    \Gr_{Z_G(t)}$. (We are unable to find an exact reference, but it
    is implicit in \cite[\S6]{bfgm}: $Z_G(t)$ is a Levi subgroup $L$
    in a parabolic $P$ with radical $U$. The Grassmannian $\Gr_G$ is a disjoint union
    of ``semiinfinite orbits": the connected components of
    $\Gr_P=L(\cK)\cdot U(\cK)/L(\cO)\cdot U(\cO)$. Since $t$ acts
    trivially on $L$ and contracts $U$ to the origin, its fixed points
    on $\Gr_P$ are $L(\cK)/L(\cO)=\Gr_L$.)
    \begin{NB}
        Misha, this is a copy of your message on Aug.23. Please check.

        Yes, checked.
    \end{NB}%

    For the representation part, it
    is clear from the embedding $\cR_{G,\bN}\subset\Gr_G\times\bN_\cO$.
\end{proof}

From this description, it is natural to introduce the following
definition, which is given directly in terms of $(G,\bN)$ without
reference to $\cR_{G,\bN}$.

\begin{Definition}
    Fix a maximal torus $T$ of $G$ as before.
    A {\it generalized root\/} $\alpha$ for a pair $(G,\bN)$ is either
    (I) a nonzero weight of $\bN$ or (II) a root of $\Lie G$.

    Generalized roots define hyperplanes in $\ft \defeq \Lie T$. Let
    $\ft^\circ$ denote the complement of the union of all generalized
    root hyperplanes.
\end{Definition}

From the above lemma, the fixed point subset $\cR^t_{G,\bN}$ of $t$ is
strictly larger than the fixed point $\cR^T_{G,\bN}$ of $T$ if and
only if $\langle\alpha, t\rangle = 0$ for some generalized root
$\alpha$.
Hence $\ft^\circ$ consists of $t$ such that $\cR^t_{G,\bN} =
\cR^T_{G,\bN}$.

We call a nonzero weight $\alpha$ of $\bN$ as a generalized root of type
(I). More precisely, we say $\alpha$ is of type (I) further if there is
no roots of $G$ in $\QQ\alpha$.
We may assume $\alpha$ is primitive, i.e., it is not a positive integer
multiple of another integral weight.
Suppose $\langle t, \alpha\rangle = 0$ and $t$ is not contained in any
other generalized root hyperplanes.
Then $Z_G(t) = T$, $\bN^t = \bN^T\oplus\bigoplus_{m\in\ZZ} \bN(m
\alpha) $, where $\bN(m \alpha)$ (resp.\ $\bN^T$) is the weight $m
\alpha$ (resp.\ $0$) subspace. We understand that $\bN(m\alpha) = 0$
if $m\alpha$ is not a weight of $\bN$.
The first factor $\bN^T$ plays no role by
\subsecref{sec:triv-properties}\ref{item:trivial}. Since $Z_G(t)$ is
abelian, we already know $H^{Z_G(t)}_*(\cR_{Z_G(t),\bN^t})$
thanks to \secref{sec:abelian}.

Other generalized roots are of type (II). They are just roots of $G$.
%
We further suppose $t$ is not contained in any other generalized root
hyperplanes. In particular, $\langle\mu,t\rangle\neq 0$ for any
nonzero weight $\mu\notin\QQ\alpha$. It could happen that a multiple
of $\alpha$ is also a weight of $\bN$. Let $\bN(m\alpha)$ be the weight
$m\alpha$ subspace, understanding it is $0$ if $m\alpha$ is not a
weight. Then $Z_G(t)$ is of semisimple rank $1$ and $\bN^t =
\bN^T\oplus \bigoplus_{m\in\QQ} \bN(m \alpha)$.
For example, if $(G,\bN) = (\GL_r,\gl_r)$ and
$t=\operatorname{diag}(t_1,t_1,t_3,\cdots,t_r)$ with $t_1$,
$t_3$,\dots, $t_r$ distinct, we have $Z_G(t) = \GL_2\times T^{r-2}$,
$\bN^t = \gl_2\oplus\CC^{r-2}$. Another example is $(G,\bN) =
(\GL_r,\CC^r)$ with the same $t$. We have $\bN^t = \{ 0\}$.

\begin{NB}
    A reductive group $G'$ of semisimple rank $1$ is $G' \cong
    Z(G')^\circ\cdot[G',G']$ with $Z(G')^\circ\cap [G',G']$ being
    finite.  See e.g., \cite[Cor.~25.3]{MR0396773}. Since the Coulomb
    branch for $Z(G')^\circ\cap [G',G']$ is larger than that for
    $Z(G')^\circ\times [G',G']$ (see
    \subsecref{sec:triv-properties}\ref{item:finite_quotient}), we
    cannot make a reduction to $Z(G')^\circ\times [G',G']$. 
\end{NB}%

\subsection{Torus equivariant homology}

The $T_\cO$-equivariant Borel-Moore homology group
$H^{T_\cO}_*(\cR)$ is defined in the same way as in
\subsecref{subsec:equiv-borel-moore}. The same applies also for
$T_\cO\rtimes\CC^\times$-equivariant homology.
Recall $H^{G_\cO}_*(\cR)$ (resp.\ $H^{T_\cO}_*(\cR)$) is a module over
$H^*_{G}(\mathrm{pt}) = \CC[\ft/W]$ (resp.\ $H^*_{T}(\mathrm{pt}) =
\CC[\ft]$).
\begin{NB}
    Moved from \subsecref{sec:bimodule}. Apr. 30.
\end{NB}%
Also $W$ acts on $H^{T_\cO}_*(\cR)$ induced by the $N(T)$-action on
$\cR$, where $N(T)$ is the normalizer of $T$ in $G$ as usual
(and $W=N(T)/T$).
\begin{NB}
previous version, changed on Apr. 27 : $W$-action on $T$.
\end{NB}%
\begin{NB}
    More precise definition: We have an $N(T)$-action $\cR$ as the
    restriction from the $G$-action. (Hence it extends the $T$-action
    on $\cR$.) Let us write it as a right action. We have the induced
    $N(T)$-action on $\cR\times_T EG$ by $[x,e]\mapsto [x n^{-1}, ne]$
    for $n\in N(T)$. This is well-defined as $[x,e] = [xt^{-1}, t e]
    \mapsto [xt^{-1} n^{-1}, n t e] = [xn^{-1} nt^{-1} n^{-1},
    ntn^{-1} n e] = [x n^{-1}, ne]$. It factorizes $W =
    N(T)/T$. Moreover $\cR\times_T EG\to \cR\times_{N(T)} EG$ is the
    quotient by $W$. Therefore we have a $W$-action on
    $H^{T_\cO}_*(\cR)$ and $H^{T_\cO}_*(\cR)^W = H^{N(T)}_*(\cR) =
    H^{G_\cO}_*(\cR)$, as $G/N(T)$ is $\QQ$-acyclic.
\end{NB}%

\begin{Lemma}\label{lem:flat}
    The $H^*_{T}(\mathrm{pt})$-module $H_*^{T_\cO}(\cR)$ is flat, and
    the $H^*_{G}(\mathrm{pt})$-module $H_*^{G_\cO}(\cR)$ is
    flat. Moreover, the natural map
    $H^*_{T}(\mathrm{pt})\otimes_{H^*_{G}(\mathrm{pt})}
    H_*^{G_\cO}(\cR)\to H_*^{T_\cO}(\cR)$ is an isomorphism, and
    $H_*^{G_\cO}(\cR)=(H_*^{T_\cO}(\cR))^{W}$. The same applies for
    $T_\cO\rtimes\CC^\times$ and $G_\cO\rtimes\CC^\times$ equivariant
    homology groups.
\end{Lemma}

\begin{proof}
Same as the one of~\cite[Lemma~6.2]{MR2135527}
\end{proof}

\subsection{Bimodule}\label{sec:bimodule}

Let us consider the diagram \eqref{eq:12}, and we restrict the
$G_\cO\times G_\cO$ (resp.\ $G_\cO$) action on the first and second
(resp.\ third and fourth) columns to $T_\cO\times G_\cO$ (resp.\
$T_\cO$). Then we have a right $H_*^{G_\cO}(\cR)$-module structure on
$H^{T_\cO}_*(\cR)$,
\begin{NB}
    The following is added. Apr. 17, HN.
\end{NB}%
i.e., we have $c_1\ast (c_2\ast c_3) = (c_1\ast
c_2)\ast c_3$ for $c_1\in H^{T_\cO}_*(\cR)$, $c_2$, $c_3\in
H^{G_\cO}_*(\cR)$. This is obvious from the proof of the associativity
in \thmref{thm:convolution}, where we restrict $G_\cO$-action to
$T_\cO$ at appropriate places.

Let $\bNT$ be the restriction of the $G$-module $\bN$ to $T\subset
G$. Then we can introduce the space $\cR_{T,\bNT}$ of triples for
$T$. It is nothing but the preimage in $\cR \equiv \cR_{G,\bN}$ of
$\Gr_T\subset\Gr_G$ under the natural projection $\cR\to \Gr_G$. We
modify the diagram \eqref{eq:12} to
\begin{equation}\label{eq:24}
    \begin{CD}
    \cT_{T,\bNT}\times \cR_{G,\bN} @<{p}<<
    T_\cK\times \cR_{G,\bN} @>{q}>> T_\cK\times_{T_\cO}\cR_{G,\bN}
    @>{m}>> \cT_{G,\bN}.
    \end{CD}
\end{equation}
(We only write the bottom row, as the top row is given by closed
embeddings.) We have $T_\cO\times T_\cO$ (resp.\ $T_\cO$) action on
the first and second (resp. third and fourth) spaces. Then the diagram
gives a left $H^{T_\cO}_*(\cR_{T,\bNT})$-module structure on
$H^{T_\cO}_*(\cR)$,
\begin{NB}
    The following is added. Apr. 17, HN.
\end{NB}%
i.e., we have $c_1\ast (c_2\ast c_3) = (c_1\ast c_2)\ast c_3$ for
$c_1$, $c_2\in H^{T_\cO}_*(\cR_{T,\bNT})$, $c_3\in
H^{T_\cO}_*(\cR)$. This again follows from the proof of the
associativity in \thmref{thm:convolution} with appropriate changes of
spaces.

\begin{NB}
    The following is added. Apr. 20, HN.
\end{NB}%
Two actions are commuting, i.e., $c_1\ast (c_2\ast c_3) = (c_1\ast
c_2)\ast c_3$ for $c_1\in H^{T_\cO}_*(\cR_{T,\bNT})$, $c_2\in
H^{T_\cO}_*(\cR)$, $c_3\in H^{G_\cO}_*(\cR)$.
This also follows from the proof of the associativity in
\thmref{thm:convolution} with appropriate changes of spaces.
Therefore

\begin{Lemma}\label{lem:bimodule}
    The convolution product
    \begin{equation*}
        H^{T_\cO}_*(\cR_{T,\bNT})\otimes H^{T_\cO}_*(\cR)\otimes
        H^{G_\cO}_*(\cR) \to H^{T_\cO}_*(\cR)
    \end{equation*}
    gives an $(H^{T_\cO}_*(\cR_{T,\bNT}),
H^{G_\cO}_*(\cR))$-bimodule structure on $H^{T_\cO}_*(\cR)$. The same applies for
    $T_\cO\rtimes\CC^\times$ and $G_\cO\rtimes\CC^\times$ equivariant
    homology groups.
\end{Lemma}

Let us remark that the convolution product $c\ast c'$ may not be
linear in the second variable $c'$, and is {\it not\/} so if we
include the rotation $\CC^\times$-action. See the computation in
\subsecref{sec:triv-properties}\ref{item:reduction}.

Also the $W$-action on $H^{T_\cO}_*(\cR)$ commutes with the right
action of $H^{G_\cO}_*(\cR)$ and normalizes the left action of
$H^{T_\cO}_*(\cR_{T,\bNT})$. Indeed, we have

\begin{Lemma}\label{lem:Winv}
    We have a $W$-action on $H^{T_\cO}_*(\cR_{T,\bNT})$ so that its
    algebra structure and its left module structure on
    $H^{T_\cO}_*(\cR)$ in \lemref{lem:bimodule} are
    $W$-equivariant. The same applies to $T_\cO\rtimes\CC^\times$
    equivariant homology groups.
\end{Lemma}

\begin{proof}
    Let $N(T)$ be the normalizer of $T$. The $N(T)$-action on $\cR$
    preserves $\cR_{T,\bNT}$.
\begin{NB}
    Added on Apr. 27:

    Note that the $N(T)$-action on $\cR$ preserves $\cR_{T,\bNT}$ as
    \begin{equation*}
        \cR_{T,\bNT}\ni [t,s]\mapsto [nt, s] = [ntn^{-1}, ns],
        \qquad \text{for $t\in T_\cK$, $s\in\bN_\cO$,
        $n\in N(T)$},
    \end{equation*}
    and $ntn^{-1}\in T_\cK$.
\end{NB}%
Hence we have the induced $W$-action on $H^{T_\cO}_*(\cR_{T,\bNT})$.
Diagrams \eqref{eq:12} for $(T,\bNT)$ and \eqref{eq:24}, used to
define convolution products, are $N(T)$-equivariant. 
\begin{NB}
    From the commutativity of the diagram
    \begin{equation*}
        \begin{CD}
            ([t_1, t_2 s], [t_2, s]) @<<< (t_1, [t_2, s]) @>>>
            \left[t_1, [t_2, s]\right] @>>> [t_1t_2, s]
            \\
            @VVV @VVV @VVV @VVV
            \\
            ([nt_1 n^{-1}, n t_2 s], [nt_2 n^{-1}, ns])
            @<<<
            (nt_1 n^{-1}, [nt_2 n^{-1}, ns])
            @>>>
            \left[nt_1 n^{-1}, [nt_2 n^{-1}, ns]\right]
            @>>>
            [nt_1 t_2 n^{-1}, ns],
        \end{CD}
    \end{equation*}
    we see that diagrams are $N(T)$-equivariant.
\end{NB}%
Therefore the convolution products are $W$-equivariant.
\end{proof}

Let us consider $e$, the fundamental class of of the fiber of
$\cR\to\Gr_G$ at the base point $[1]\in \Gr_G$ as in
\thmref{thm:convolution}. It was the unit in $H^{G_\cO}_*(\cR)$, but
we consider it as an element in the bimodule $H^{T_\cO}_*(\cR)$ instead.

\begin{Lemma}\label{lem:convolution_by_e}
    \textup{(1)} The left multiplication $H^{T_\cO}_*(\cR_{T,\bNT})\ni
    c\mapsto c\ast e\in H^{T_\cO}_*(\cR)$ is the pushforward
    homomorphism $\iota_*$ for the embedding $\iota\colon
    \cR_{T,\bNT}\to \cR$.

    \textup{(2)} The right multiplication $H^{G_\cO}_*(\cR)\ni
    c'\mapsto e\ast c'\in H^{T_\cO}_*(\cR)$ is the homomorphism given by
    the restriction from $G_\cO$ to $T_\cO$. In particular, it is
    $H_G^*(\mathrm{pt})$-linear under the forgetting homomorphism
    $H_G^*(\mathrm{pt})\to H_T^*(\mathrm{pt})$.
    \begin{NB}
        It seems that the original is useless:

        $H^{G_\cO}_*(\cR)$-linear, where $H^{G_\cO}_*(\cR)$ is a right
        $H^{G_\cO}_*(\cR)$-module by the multiplication from the
        right.
    \end{NB}%

    \textup{(3)} Both \textup{(1)}, \textup{(2)} are true for
    $T_\cO\rtimes\CC^\times$, $G_\cO\rtimes\CC^\times$ equivariant
    homology groups.
\end{Lemma}

The proof is exactly the same as the proof that $e$ is unit in \thmref{thm:convolution}.

By (2), we have a well-defined homomorphism
\begin{equation}\label{eq:20}
    e\ast\bullet\colon
    H^{G_\cO}_*(\cR)\otimes_{H^*_G(\mathrm{pt})} H^*_T(\mathrm{pt})
    \xrightarrow{\cong} H^{T_\cO}_*(\cR).
\end{equation}
It is an isomorphism thanks to \lemref{lem:flat}.

\begin{Lemma}\label{lem:iota_injective}
    The homomorphism $\iota_*$ becomes an isomorphism over $\loc$, the
    complement of the union of all generalized root hyperplanes. In
    particular, it is injective.

    For $T_\cO\rtimes\CC^\times$-equivariant homology groups,
    $\iota_*$ is an isomorphism over
    $\loc\times\operatorname{Lie}(\CC^\times)$.
\end{Lemma}

\begin{proof}
    The last assertion is a consequence of the first and
    \lemref{lem:flat}.

By
    \lemref{lem:fixedpts} (or a direct consideration), we have
    $\cR_{T,\bNT}^T = \cR^T$. The localization theorem for equivariant
    homology groups implies the assertion.
\end{proof}

Hereafter we often use the notation `$|_\loc$' meaning the
localization at the ideal given by generalized root hyperplanes, i.e.,
tensor product with $\CC[\loc]$ over ${H^*_T(\mathrm{pt})}$.

\begin{NB}
    We have
    \begin{equation*}
        T_\cK\times_{T_\cO} \cR \cong \cT_{T,\bNT}\times_{\bN_\cK} \cT,
    \end{equation*}
    as an analog of \remref{rem:Steinberg}(1). Therefore
    $H^{T_\cO}_*(\cR)$ can be regarded as
    $H^{T_\cK}_*(\cT_{T,\bNT}\times_{\bN_\cK}\cT)$. Although we cannot
    define the convolution on this, it is at least philosophically
    apparent that this is a  $(H^{T_\cK}_*(\cT_{T,\bNT}\times_{\bN_\cK}\cT_{T,\bNT}),
    H^{G_\cK}_*(\cT\times_{\bN_\cK}\cT))
    = (H^{T_\cO}_*(\cR_{T,\bNT}), H^{G_\cO}_*(\cR))$
    bimodule.
\end{NB}

\begin{NB}
    The following is added on Aug. 31.
\end{NB}%

\begin{Lemma}\label{lem:bimodule-1}
    \textup{(1)}
    $\iota_*\colon H^{T_\cO}_*(\cR_{T,\bNT})^W\to H^{G_\cO}_*(\cR)$
    is an algebra homomorphism. 

    \textup{(2)} The same is true for $T_\cO\rtimes\CC^\times$ and
    $G_\cO\rtimes\CC^\times$ equivariant homology groups. In
    particular, $\iota_*$ in \textup{(1)} respects the Poisson
    structures.\footnote{The statement will make sense after we prove
the commutativity in~\propref{prop:commutative} below. The same applies 
to~\lemref{lem:convolution}(3).}
\end{Lemma}

\begin{proof}
    Our idea of the proof is based on \cite[\S5.3]{MR3118615}.

    Let $c_1$, $c_2\in H^{T_\cO}_*(\cR_{T,\bNT})$. By
    \lemref{lem:convolution_by_e}(1) $\iota_*(c_a) = c_a\ast e$
    ($a=1,2$). If $c_a$ is $W$-invariant, $c_a\ast e$ is also by
    \lemref{lem:Winv}, hence we have $c_a \ast e = e\ast c'_a$ for some
    $c'_a\in H^{G_\cO}_*(\cR)$ by
    \lemref{lem:convolution_by_e}(2). The map $\iota_*$ in the
    statement is nothing but $c_a\mapsto c'_a$.

    Now the associativity implies
    \begin{equation*}
        (c_1\ast c_2)\ast e = c_1 \ast (c_2\ast e) = c_1\ast (e\ast c_2')
        = (c_1\ast e)\ast c_2' = (e\ast c_1')\ast c_2' = e\ast (c_1'\ast c_2').
    \end{equation*}
    It means that $\iota_*$ is an algebra homomorphism. The argument
    works also when the loop rotation $\CC^\times$ is included.
\end{proof}

\subsection{From the variety of triples to the affine Grassmannian}
\label{sec:z}

In the same (and simpler) way as in \subsecref{sec:bimodule}, we have
a natural $(H^{T_\cO}_*(\Gr_T), H^{G_\cO}_*(\Gr_G))$-module structure
on $H^{T_\cO}_*(\Gr_G)$. Alternatively it is a special case of the
above construction with $\bN=0$.

Let $\bz\colon \Gr_G\to\cT$ be the closed embedding given by considering
$\Gr_G$ as the $0$ section of the vector bundle $\cT$. We have
$\bz^{-1}(\cR) = \Gr_G$. Since $\cT\to\Gr_G$ is a vector bundle, we have
the pull-back homomorphism $\bz^*\DC_{\cT}\to \DC_{\Gr_G}[2\dim\bN_\cO]$,
where $\dim\bN_\cO$ is the rank of the vector bundle.
\begin{NB}
    This is the Gysin homomorphism, defined in \cite{MR1644323} for
    finite dimensional schemes. It is compatible with the
    inverse/direct systems in \subsecref{triples}. Therefore it is
    well-defined.
\end{NB}%
Let $\tilde\bz$ denote the inclusion $\Gr_G\to\cR$ (so that $\bz =
i\circ\tilde\bz$). We have the pull-back with support homomorphisms
$\bz^*\colon \DC_{\cR}[-2\dim\bN_\cO]\to \tilde \bz_*\DC_{\Gr_G}$ and
$\bz^*\colon H^{G_\cO}_*(\cR)\to H^{G_\cO}_*(\Gr_G)$, where the degree
is given relative to $2\dim\bN_\cO$ for the former so that the degree
is preserved.
In fact, it is also easy to check that $\bz^*$ is the same as the
composite of $H^{G_\cO}_*(\cR)\xrightarrow{i_*}
H^{G_\cO}_*(\cT)\xrightarrow[\cong]{\bz^*} H^{G_\cO}_*(\Gr_G)$.
The second $\bz^*$ is an isomorphism and its inverse is $\pi^*$ where
$\pi\colon\cT\to\Gr_G$ is the projection.
\begin{NB}
    Claim : The pull-back with support $\bz^*\colon
H^{G_\cO}_*(\cR)\to H^{G_\cO}_*(\Gr_G)$ is the same as the composite of
$H^{G_\cO}_*(\cR)\xrightarrow{i_*} H^{G_\cO}_*(\cT)\xrightarrow{\bz^*}
H^{G_\cO}_*(\Gr_G)$.

\begin{proof}
    The pull-back with support is the composite
    \begin{equation*}
        i_*\DC_\cR = i_! i^! \DC_\cT \to i_! i^! \bz_* \bz^* \DC_\cT
        \xrightarrow{\cong} i_! \tilde \bz_* \bz^*\DC_\cT
        = i_*\tilde \bz_* \bz^*\DC_\cT
        \to \bz_* \DC_{\Gr_G}[2\dim\bN_\cO],
    \end{equation*}
    where the first arrow is the adjunction, the second is the base
    change, and the third is constructed as above.
    On the other hand the second homomorphism is the composite of
    \begin{equation*}
        i_*\DC_\cR = i_! i^!\DC_\cT \to \DC_\cT \to
        \bz_* \DC_{\Gr_G}[2\dim\bN_\cO],
    \end{equation*}
    where the first is adjunction and the second is constructed as
    above. The second map can be re-written as the composite of
    $\DC_\cT\to \bz_*\bz^*\DC_\cT\to \bz_*\DC_{\Gr_G}[2\dim\bN_\cO]$. So we
    need to check that two homomorphisms $i_! i^! \bz_* A \to \bz_* A$,
    one given by adjunction $i_!i^!\to\id$, another by the base change
    $i^! \bz_*\xrightarrow{\cong} \tilde \bz_*\id^!$ composed with
    $i_!\tilde \bz_* = i_* \tilde \bz_* = \bz_*$ are the same. (Here $A =
    \bz^*\DC_\cT$.)
    It means that the base change homomorphism in $\Hom(\tilde
    \bz_*\id^!  A, i^! \bz_* A)$ is the same as the identification
    $i_!\tilde \bz_*=\bz_*$ in $\Hom(i_!\tilde \bz_* \id^! A, \bz_* A)$ under
    the adjunction.
    The base change homomorphism is constructed (see
    \cite[Prop.~2.5.11]{KaSha}) so that its adjunction is
    \begin{equation*}
        i_! \tilde \bz_* \id^! A\xrightarrow{\cong} \bz_* \id_! \id^!A
        \to \bz_* A,
    \end{equation*}
    where the first is given by considering a section of $i_!\tilde
    \bz_* \id^! A$ over $V\subset\cT$ as a section $s$ of $\id^!A$ over
    $(i\circ\tilde \bz)^{-1}(V) = V\cap\Gr_G$ with appropriate
    conditions on the support, and the second is the
    adjunction. Looking at the proof of \cite[Prop.~2.5.11]{KaSha}, we
    see that the support condition is $\operatorname{supp}(s)\subset
    \tilde \bz^{-1}(Z)$ for a subset $Z$ of $i^{-1}(V) = V\cap \cR$
    proper over $V$. This is unnecessary in our case. Therefore it is
    clear that two homomorphisms are the same.
\end{proof}
\end{NB}%

We also have $\bz^*\colon H^{T_\cO}_*(\cR)\to H^{T_\cO}_*(\Gr_G)$, and
$\bz^*\colon H^{T_\cO}_*(\cR_{T,\bNT})\to H^{T_\cO}_*(\Gr_T)$. The
second $\bz$ should be understood as $\Gr_T\to\cT_{T,\bNT}$, but is
denoted by the same letter for brevity.

\begin{Lemma}\label{lem:convolution}
    \textup{(1)}
    $\bz^*\colon H^{G_\cO}_*(\cR)\to H^{G_\cO}_*(\Gr_G)$ is an algebra
    homomorphism. The same is true for
    $\bz^*\colon H^{T_\cO}_*(\cR_{T,\bNT})\to H^{T_\cO}_*(\Gr_T)$.

    \textup{(2)} $\bz^*\colon H^{T_\cO}_*(\cR)\to H^{T_\cO}_*(\Gr_G)$ is
    a homomorphism of bimodules.

    \textup{(3)} Both \textup{(1)}, \textup{(2)} are true for
    $T_\cO\rtimes\CC^\times$, $G_\cO\rtimes\CC^\times$ equivariant
    homology groups. In particular, $\bz^*$ in \textup{(1)} respects the
    Poisson structures.
\end{Lemma}

\begin{proof}
    We give the proof of the first statement of (1). The second
    statement is the special case $G=T$. The proof of (2) is
    straightforward modification of the proof of (1), and is
    omitted. The proof of (3) is the same as (1), just check that
    everything is $\CC^\times$-equivariant.

    The idea of the proof is similar to one of \lemref{lem:bimodule-1}.

    We consider the diagram
    \begin{equation}\label{eq:23}
    \begin{CD}
                \cT\times\Gr_G @<<<
    G_\cK\times \Gr_{G} @>>> G_\cK\times_{G_\cO}\Gr_{G}
    @>>> \Gr_{G},
    \end{CD}
    \end{equation}
    which is nothing but the diagram \eqref{eq:1} with the leftmost
    term replaced by $\cT\times\Gr_G$ via the inclusion
    $\Gr_G\times\Gr_G\xrightarrow{\bz\times\id_{\Gr_G}}\cT\times\Gr_G$.
    As in \subsecref{sec:bimodule} we can view $H^{G_\cO}_*(\Gr_G)$
    as a left $H^{G_\cO}_*(\cR)$-module. Then we have $c\bar\ast c' =
    \bz^*(c)\ast c'$, where the left $\bar\ast$ is the module action, and
    the right $\ast$ is the multiplication in
    $H^{G_\cO}_*(\Gr_G)$. Now we take $c' = 1$, the unit of
    $H^{G_\cO}_*(\Gr_G)$. Then the associativity implies
    \begin{equation*}
        \bz^*(c_1\ast c_2) =
        (c_1 \ast c_2)\bar\ast 1 = c_1 \bar\ast (c_2 \bar\ast 1)
        = c_1\bar\ast \bz^*(c_2)  = \bz^*(c_1)\ast \bz^*(c_2).
    \end{equation*}
    This means that $\bz^*$ is an algebra homomorphism. The proof of the
    associativity is the same as one in \thmref{thm:convolution},
    hence is omitted.
    \begin{NB}
        The associativity means the commutativity of the diagram:
        \begin{equation*}
            \begin{CD}
            H^{G_\cO}_*(\cR)\otimes H^{G_\cO}_*(\cR) \otimes H^{G_\cO}_*(\Gr_G)
            @>{\ast\otimes\id}>> H^{G_\cO}_*(\cR) \otimes H^{G_\cO}_*(\Gr_G)
            \\
            @V{\id\otimes\bar\ast}VV @VV{\bar\ast}V
            \\
            H^{G_\cO}_*(\cR)\otimes H^{G_\cO}_*(\Gr_G) @>>\bar\ast>
            H^{T_\cO}_*(\Gr_G).
            \end{CD}
        \end{equation*}
        The proof is basically written in the {\bf NB} two below, as
        the `original proof'.
    \end{NB}%
    \begin{NB}
        Here is the proof of (2):

        The diagram \eqref{eq:23} together with $T_\cO\times
        G_\cO$-action gives $H^{T_\cO}_*(\cR) \otimes
        H^{G_\cO}_*(\Gr_G) \to H^{T_\cO}_*(\Gr_G)$ such that
        \begin{equation*}
            c \ast c' = \bz^*(c)\ast c'\quad\text{for
              $c\in H^{T_\cO}_*(\Gr_T)$, $c'\in H^{T_\cO}_*(\cR)$}.
        \end{equation*}
        In particular, taking $c' = 1_{H^{G_\cO}_*(\Gr_G)}$, we get
        $c\ast 1_{H^{G_\cO}_*(\Gr_G)} = \bz^*(c)$.

        For $c_1\in H^{T_\cO}_*(\cR)$, $c_2\in H^{G_\cO}_*(\cR)$, we have
        \begin{multline*}
            \bz^*(c_1\ast c_2) = (c_1\ast c_2)\ast 1_{H^{G_\cO}_*(\Gr_G)}
            = c_1 \ast (c_2\ast 1_{H^{G_\cO}_*(\Gr_G)}) = c_1 \ast \bz^*(c_2)
\\
            = c_1 \ast (1_{H^{G_\cO}_*(\Gr_G)} \ast \bz^*(c_2))
            = (c_1\ast 1_{H^{G_\cO}_*(\Gr_G)}) \ast \bz^*(c_2))
            = \bz^*(c_1)\ast \bz^*(c_2).
        \end{multline*}
        We have used the associativity twice: the first is at the
        second equality, and is the commutativity of
        \begin{equation*}
            \begin{CD}
            H^{T_\cO}_*(\cR)\otimes H^{G_\cO}_*(\cR) \otimes H^{G_\cO}_*(\Gr_G)
            @>>> H^{T_\cO}_*(\cR) \otimes H^{G_\cO}_*(\Gr_G)
            \\
            @VVV @VVV
            \\
            H^{T_\cO}_*(\cR)\otimes H^{G_\cO}_*(\Gr_G) @>>> H^{T_\cO}_*(\Gr_G).
            \end{CD}
        \end{equation*}
        The second is at the fifth equality, and is the commutativity of
        \begin{equation*}
            \begin{CD}
            H^{T_\cO}_*(\cR)\otimes H^{G_\cO}_*(\Gr_G) \otimes H^{G_\cO}_*(\Gr_G)
            @>>> H^{T_\cO}_*(\cR) \otimes H^{G_\cO}_*(\Gr_G)
            \\
            @VVV @VVV
            \\
            H^{T_\cO}_*(\cR)\otimes H^{G_\cO}_*(\Gr_G) @>>> H^{T_\cO}_*(\Gr_G).
            \end{CD}
        \end{equation*}

        Next take $c_1\in H^{T_\cO}_*(\cR_{N,\bNT})$, $c_2\in
        H^{T_\cO}_*(\cR)$. We have
        \begin{equation*}
            \bz^*(c_1\ast c_2) = (c_1\ast c_2) \ast 1_{H^{G_\cO}_*(\Gr_G)}
            = c_1\ast (c_2\ast 1_{H^{G_\cO}_*(\Gr_G)}) = c_1 \ast \bz^*(c_2).
        \end{equation*}
        The associativity at the second equality is the commutativity of
        \begin{equation*}
            \begin{CD}
            H^{T_\cO}_*(\cR_{T,\bNT})\otimes H^{T_\cO}_*(\cR) \otimes
            H^{G_\cO}_*(\Gr_G)
            @>>> H^{T_\cO}_*(\cR) \otimes H^{G_\cO}_*(\Gr_G)
            \\
            @VVV @VVV
            \\
            H^{T_\cO}_*(\cR_{T,\bNT})\otimes H^{T_\cO}_*(\Gr_G) @>>>
            H^{T_\cO}_*(\Gr_G).
            \end{CD}
        \end{equation*}
        The bottom arrow is given by the diagram
        \begin{equation*}
            \begin{CD}
            \cT_{T,\bNT} \times \Gr_G@<{p}<<
            T_\cK\times \Gr_{G} @>{q}>> T_\cK\times_{T_\cO}\Gr_{G}
            @>{m}>> \Gr_{G},
            \end{CD}
        \end{equation*}
        which is just
        \begin{equation*}
            \begin{CD}
            \Gr_T \times \Gr_G@<{p}<<
            T_\cK\times \Gr_{G} @>{q}>> T_\cK\times_{T_\cO}\Gr_{G}
            @>{m}>> \Gr_{G},
            \end{CD}
        \end{equation*}
        composed with $\Gr_T\times \Gr_G\to \cT_{T,\bNT}\times \Gr_G$.
        Now it is clear $c_1\ast \bz^*(c_2) = \bz^*(c_1)\ast \bz^*(c_2)$. Hence
        $\bz^*(c_1\ast c_2) = \bz^*(c_1)\ast \bz^*(c_2)$.
    \end{NB}%
\end{proof}

\begin{NB}
Here is the original proof:

Let us combine the diagrams \eqref{eq:1} and \eqref{eq:12}:
\begin{equation}\label{eq:11}
    \begin{CD}
        \Gr_G\times\Gr_G @<\overline{p}<< G_\cK\times \Gr_G
        @>\overline{q}>> G_K\times_{G_\cO}\Gr_G @>\overline{m}>> \Gr_G
        \\
        @V{\tilde z\times\tilde z}VV @VVV @VVV @VV{\tilde z}V
        \\
        \cR\times\cR @<\tilde p<< p^{-1}(\cR\times\cR) @>\tilde q>>
        q(p^{-1}(\cR\times\cR)) @>\tilde m>> \cR
        \\
        @V{i\times\id_\cR}VV @V{i'}VV @VVV @VV{i}V
        \\
        \cT\times \cR @<<p< G_\cK\times\cR @>>q>
        G_\cK\times_{G_\cO}\cR @>>m> \cT,
    \end{CD}
\end{equation}
where $\tilde z$ is the inclusion $\Gr_G\to\cR$ (so that $z =
i\circ\tilde z$) so that the first row consists of closed subschemes
of the second row, and we put `bar' on arrow in the top row to
distinguish from arrows in the bottom. We need to show the
commutativity of pushforward/pull-back with support homomorphisms for
squares in the first and second rows. And the third row is necessary to define pull-back with support.

Let us just ignore the degree shifts in the following discussion, as
it is clear that $\bz^*$ preserves the degree.

We first consider the left most square. According to factors $\cT$,
$\cR$, and $\Gr_G$, we decomposed into two parts. First consider
\begin{equation*}
    \begin{CD}
     \Gr_G
     \begin{NB2}
         \text{second $\Gr_G$}
     \end{NB2}%
@<{p_2}<< G_\cK\times \Gr_G
\\
     @V{\tilde z}VV @VV{\id_{G_\cK}\times\tilde z}V
\\
        \cR @<<{p_\cR}< G_\cK\times \cR,
    \end{CD}
\end{equation*}
where the middle row is omitted. Recall $p^*_\cR\DC_\cR \cong
\CC_{G_\cK}\boxtimes \DC_\cR$ as $p_\cR$ is the projection to the
second factor. (See \eqref{eq:13}.) The upper row is also the
projection to the second factor, hence $p_2^*\DC_{\Gr_G}\cong
\CC_{G_\cK}\boxtimes\DC_{\Gr_G}$. These two homomorphisms are
obviously commuting with $\DC_\cR\to \tilde z_*\DC_{\Gr_G}$,
constructed above.
\begin{NB2}
    It means that
    \begin{equation*}
        \begin{CD}
            p_2^*\DC_{\Gr_G} @>>> \CC_{G_\cK}\boxtimes\DC_{\Gr_G}
\\
@AAA @AAA
\\
            (\id_{G_\cK}\times\tilde z)^* p_\cR^* \DC_\cR
            = p_2^* \tilde z^* \DC_\cR
            @>>> (\id_{G_\cK}\times\tilde z)^* (\CC_{G_\cK}\boxtimes\DC_{\cR})
        \end{CD}
    \end{equation*}
    is commutative. But this is obvious.
\end{NB2}

Next consider the diagram for the first factor. Both horizontal arrows
factor as
\begin{equation*}
    \begin{CD}
     \Gr_G
     \begin{NB2}
         ^1
     \end{NB2}%
@<p_1'<< G_\cK @<{p_1}<< G_\cK\times \Gr_G
     \begin{NB2}
         ^2
     \end{NB2}%
\\
     @V{z}VV @VV{z'}V @VV{\id_{G_\cK}\times\tilde z}V
\\
     \cT @<<{p'_\cT}< G_\cK\times\bN_\cO @<<{\id_{G_\cK}\times\Pi}<
     G_\cK\times \cR.
    \end{CD}
\end{equation*}
The middle vertical arrow $z'$ is the embedding of $G_\cK$ regarded as
$G_\cK\times\{0\}$. For the left square, we need to check that two
homomorphisms $z^{\prime*}p_T^{\prime*}\DC_\cT=
p_1^{\prime*}z^*\DC_\cT \to \DC_{G_\cK}
\begin{NB2}
[2\dim G_\cO]
\end{NB2}%
$ are the same. This is clear, as both $p'_1$ and $p'_\cT$ are
$G_\cO$-principal bundles.
Note also that we have an isomorphism $\DC_{\bN_\cO}\cong
\CC_{\bN_\cO}
\begin{NB2}
[2\dim\bN_\cO]
\end{NB2}%
$ as $\bN_\cO$ is smooth. Then
$z^{\prime*}\colon\DC_{G_\cK\times\bN_\cO}\to
\DC_{G_\cK}
\begin{NB2}
[2\dim\bN_\cO]
\end{NB2}%
$ can be understood also as a composite of
$\DC_{G_\cK\times\bN_\cO}\xrightarrow{\cong} \DC_{G_\cK}\boxtimes \CC_{\bN_\cO}
\begin{NB2}
[2\dim\bN_\cO]
\end{NB2}%
\to \DC_{G_\cK}
\begin{NB2}
[2\dim\bN_\cO]
\end{NB2}%
$.

Let us move the right square. We need to check that two homomorphisms
$\DC_{G_\cK}\boxtimes\CC_{\bN_\cO}\to \DC_{G_\cK}\boxtimes\CC_{\Gr_G}$
are the same. But this is obvious. Both maps decompose as the product
of identity $G_\cK\to G_\cK$ and `$\Gr_G\to \{0\} \to \bN_\cO$ or
$\Gr_G\to\cR\to\bN_\cO$'. The sheaf on the second factor $\bN_\cO$ is
the constant sheaf, and we just consider the natural homomorphisms
among pull-backs of constant sheaves.

We now take the tensor product $\DC_\cT\boxtimes\DC_\cR$. The two
homomorphisms $\overline{p}^*(z\times\tilde
z)^*(\DC_\cT\boxtimes\DC_\cR) = (\id_{G_\cK}\times \tilde z)^* p^*
(\DC_\cT\boxtimes\DC_\cR) \to \DC_{G_\cK}\boxtimes \DC_{\Gr_G}
\begin{NB2}
[2\dim G_\cO+2\dim\bN_\cO]
\end{NB2}%
$ are the same. Moreover it is also clear from the
above computation that $(\id_{G_\cK}\times\tilde
z)^*(\DC_{G_\cK}\boxtimes \DC_{\cR}) \to \DC_{G_\cK}\boxtimes
\DC_{\Gr_G}
\begin{NB2}
[2\dim\bN_\cO]
\end{NB2}%
$ is the tensor product of the identity on
$\DC_{G_\cK}$ and $z^*\colon \tilde z^*\DC_{\cR}\to
\DC_{\Gr_G}
\begin{NB2}
[2\dim\bN_\cO]
\end{NB2}
$.

The commutativity for the middle square in \eqref{eq:11} is clear, as $\overline{q}$, $\tilde q$ are $G_\cO$-bundles.

Let us consider the right most square in \eqref{eq:11}. Recall that
$m$ factors as $m'\circ (\id_{G_\cK}\times_{G_\cO}i)$. Since the
pull-back is defined via $z$, we modify the right most square as
\begin{equation*}
    \begin{CD}
        G_\cK\times_{G_\cO}\Gr_G @>>{\overline{m}}> \Gr_G
\\
@V{\id_{G_\cK}\times_{G_\cO} z}VV @VV{z}V
\\
        G_\cK\times_{G_\cO}\cT @>>{m'}>\cT.
    \end{CD}
\end{equation*}
It is enough to show that we have an isomorphism
between $z^* m'_!\DC_{G_\cK\times_{G_\cO}\cT}$ and
$\overline{m}_! (\id_{G_\cK\times_{G_\cO}z})^* \DC_{G_\cK\times_{G_\cO}\cT}$
so that two natural homomorphisms to $\DC_{\Gr_G}$ are the same.
The isomorphism is provided by the base change, as the above diagram
is cartesian.
\end{NB}

\begin{Lemma}\label{lem:injective}
    The homomorphisms $\bz^*$ in three cases in \lemref{lem:convolution}
    become isomorphisms over $\ft^\circ$, the complement of the union
    of all generalized root hyperplanes. In particular, they are
    injective.
\end{Lemma}

\begin{proof}
    Since $H^{G_\cO}_*(\bullet) = (H^{T_\cO}_*(\bullet))^W$ for both
    $\bullet = \cR$ and $\Gr_G$, it is enough to check the assertion
    for $\bz^*$ in (2). (The second $\bz^*$ in (1) is the special case
    $G=T$.)

    Recall that $\bz^*$ is the composite of
    $H^{T_\cO}_*(\cR)\xrightarrow{i_*}
    H^{T_\cO}_*(\cT)\xrightarrow[\cong]{\bz^*} H^{T_\cO}_*(\Gr_G)$, as
    we remarked above. Therefore it is enough to show that $i_*$
    becomes an isomorphism over $\loc$.
    \begin{NB}
        Let `$|_\loc$' denote the localization at the ideal given by
        generalized root hyperplanes, i.e., tensor product with
        $\CC[\loc]$ over ${H^*_T(\mathrm{pt})}$.
    \end{NB}%
    The pushforward homomorphisms of inclusions $\cR^T\subset \cR$,
    $\cT^T\subset \cT$ induce isomorphisms
    $H^{T_\cO}_*(\cR^T)|_\loc\xrightarrow{\cong}H^{T_\cO}_*(\cR)|_\loc$,
    $H^{T_\cO}_*(\cT^T)|_\loc\xrightarrow{\cong}H^{T_\cO}_*(\cT)|_\loc$
    respectively. Therefore it is enough to show that $i_*$ for
    $H^{T_\cO}_*(\cR^T)|_\loc\to H^{T_\cO}_*(\cT^T)|_\loc$ is an
    isomorphism. But this is clear from \lemref{lem:fixedpts} : $\cR^T
    \cong \cR_{T,\bN^T} \cong \Gr_G\times\bN^T_\cO \cong \cT^T$.
    \begin{NB}
        Alternative argument:

        It is clear that we may assume $\bN^T = 0$.
        Let $\Gr_T\to \cT_{T,\bNT}$ be the embedding as the
        $0$-section. It factors through $\Gr_T\to\cR_{T,\bNT}$. Let us
        denote the latter by $\iota$. Then $\bz^*\iota_*$ is the
        multiplication by the equivariant Euler class of
        $\cT_{T,\bNT}$. The assertion $\cT_{T,\bNT}^T = \Gr_T$ means
        no zero weight subbundle. Therefore its equivariant Euler
        class is invertible over $\ft^\circ$.
    \end{NB}%
\end{proof}

\begin{Remark}\label{rem:further}
    Let $\bfV$ be another representation. Then there is a natural
    embedding $H^{G_\cO}_*(\cR_{G,\bfN\oplus \bfV})\hookrightarrow
    H^{G_\cO}_*(\cR_{G,\bfN})$ given by the pull-back homomorphism
    with respect to the embedding $\cT_{T,\bN}\subset
    \cT_{T,\bN\oplus\mathbf V}$ of a subbundle. (cf.\
    \subsecref{sec:chang-repr}) This can be proved by the same
    argument, or by observing $H^{G_\cO}_*(\cR_{G,\bfN\oplus \bfV})
    \to H^{G_\cO}_*(\Gr_G)$ factors as $H^{G_\cO}_*(\cR_{G,\bfN\oplus
      \bfV})\to H^{G_\cO}_*(\cR_{G,\bfN})\to H^{G_\cO}_*(\Gr_G)$.
\end{Remark}

\subsection{Classical description of Coulomb branches}\label{sec:classical}

We study $\mathcal M_C$ over $\ft^\circ$ in this subsection. This is
called a classical description of the Coulomb branch $\mathcal M_C$ in
the physics literature. This description first leads to a proof of the
commutativity of $H^{G_\cO}_*(\cR)$, as we promised in
\subsecref{sec:commutativity}.

\begin{NB}
    The proof of the commutativity of $H^{G_\cO}_*(\Gr)$.

We consider $H^{T_\cO}_*(\Gr)$ with the commuting left action by
$H^{T_\cO}_*(\Gr_T)$ and right action by $H^{G_\cO}_*(\Gr)$. By Localization
theorem, $H^{T_\cO}_*(\Gr)$ embeds into the free
$H^{T_\cO}_*(\Gr_T)_{loc}$-module of rank 1 $H^{T_\cO}_*(\Gr)_{loc}$. Now
since $H^{G_\cO}_*(\Gr)=(H^{T_\cO}_*(\Gr))^W$, the algebra $H^{G_\cO}_*(\Gr)$
embeds into the ring of endomorphisms of $H^{T_\cO}_*(\Gr)$ (already the
action on the skyscraper sheaf at the base point is effective). Hence
it further embeds into the localized algebra of endomorphisms of the
free $H^{T_\cO}_*(\Gr_T)_{loc}$-module of rank 1
$H^{T_\cO}_*(\Gr)_{loc}$. But the latter one is nothing but
$H^{T_\cO}_*(\Gr_T)_{loc}$ which is finally commutative.
\end{NB}

\begin{Proposition}\label{prop:commutative}
    $H^{G_\cO}_*(\cR)$ is commutative.
\end{Proposition}

\begin{NB}
    Added on Apr. 27
\end{NB}%

One way to prove this is to use \lemref{lem:injective} and the
commutativity of $H^{G_\cO}_*(\Gr_G)$, proved in \cite{MR2135527}.
The proof of the commutativity of $H^{G_\cO}_*(\Gr_G)$, as was
explained in \subsecref{sec:commutativity}, uses the Beilinson-Drinfeld
Grassmannian. This proof will be given at a categorical level in
\ref{sec:commute}.
We present another argument, completely avoiding Beilinson-Drinfeld
Grassmannian now.

\begin{proof}[Proof of \propref{prop:commutative}]
    Since we have proved that $H^{T_\cO}_*(\cR_{T,\bNT})$ is
    commutative (see \secref{sec:abelian}), the multiplication is
    $H_{T_\cO}^*(\mathrm{pt})$-linear in both first and second
    variables. In particular, its localization
    $H^{T_\cO}_*(\cR_{T,\bNT})|_\loc$ inherits the algebra structure.

    Recall the embedding $\iota\colon \cR_{T,\bNT}\to\cR$.  By
    \lemref{lem:iota_injective} it induces an isomorphism between the
    localized equivariant homology groups: ${\iota_*}\colon
    H^{T_\cO}_*(\cR_{T,\bNT})|_\loc\xrightarrow{\cong}
    H^{T_O}_*(\cR)|_\loc$. We have a chain of injective maps
    \begin{equation*}
        H^{G_\cO}_*(\cR) \to H^{T_\cO}_*(\cR)
            \to H^{T_\cO}_*(\cR)|_{\ft^\circ}
            \xrightarrow[\sim]{\iota_*^{-1}=(\bullet\ast e)^{-1}}
            H^{T_\cO}_*(\cR_{T,\bNT})|_{\ft^\circ},
    \end{equation*}
    where the injectivity of the first two maps follows from
    \lemref{lem:flat}. By the proof of \lemref{lem:bimodule-1}, the
    composite respects the multiplication. Therefore
    $H^{G_\cO}_*(\cR)$ is commutative.
    \begin{NB}
        Here is the record of the original manuscript. Since I add
        \lemref{lem:bimodule-1}, the most of argument below are moved
        there.

    We consider $H^{T_\cO}_*(\cR)$ with the commuting left action by
    $H^{T_\cO}_*(\cR_{T,\bNT})$ and right action by
    $H^{G_\cO}_*(\cR)$.

    Let $e\in H^{T_\cO}_*(\cR)$ be as in \lemref{lem:convolution_by_e}.
    \begin{NB2}
        the fundamental class of the fiber over the base point in
        $\Gr_G$
    \end{NB2}%
    Let us consider $H^{T_\cO}_*(\cR_{T,\bNT})\ni c\mapsto c\ast e\in
    H^{T_\cO}_*(\cR)$. By \lemref{lem:convolution_by_e}(1), it is
    given by the pushforward homomorphism with respect to the
    embedding $\iota\colon \cR_{T,\bNT}\to \cR$. By
    \lemref{lem:iota_injective}, it becomes an isomorphism
    ${\iota_*}\colon
    H^{T_\cO}_*(\cR_{T,\bNT})|_\loc\xrightarrow{\cong}
    H^{T_O}_*(\cR)|_\loc$.
    \begin{NB2}
        where `$|_\loc$' means the localization at the ideal given by
        generalized root hyperplanes, i.e., tensor product with
        $\CC[\loc]$ over ${H^*_T(\mathrm{pt})}$.
    \end{NB2}
    \begin{NB2}
        It seems the following is unnecessary any more:

    This is an isomorphism of
    $H^{T_\cO}_*(\cR_{T,\bNT})|_{\loc}$-modules, where
    $H^{T_\cO}_*(\cR_{T,\bNT})|_{\loc}$ is a
    $H^{T_\cO}_*(\cR_{T,\bNT})|_\loc$-module by a left multiplication.
    Thus $H^{T_\cO}_*(\cR)|_\loc$ is a free
    $H^{T_\cO}_*(\cR_{T,\bNT})|_\loc$-module of rank $1$.
    \end{NB2}%

    \begin{Claim}
        The composite of
        \begin{equation*}
            H^{G_\cO}_*(\cR) \to H^{T_\cO}_*(\cR)
            \to H^{T_\cO}_*(\cR)|_{\ft^\circ}
            \xrightarrow{\iota_*^{-1}=(\bullet\ast e)^{-1}}
            H^{T_\cO}_*(\cR_{T,\bNT})|_{\ft^\circ}
        \end{equation*}
        is an algebra homomorphism. Here the first map is given by the
        restriction from $G_\cO$ to $T_\cO$, and the second map is
        restriction to $\ft^\circ$.
    \end{Claim}

    By \lemref{lem:flat}, the composite is injective. Since we already
    know that $H^{T_\cO}_*(\cR_{T,\bNT})$ is commutative, it implies
    the commutativity of $H^{G_\cO}_*(\cR)$.

    Let us give the proof of the claim. The idea is similar to the
    proof of \lemref{lem:convolution}. The first map is $e\ast\bullet$
      by \lemref{lem:convolution_by_e}.
      \begin{NB2}
          We also know that $e\ast\bullet$ is
          $H_{G_\cO}^*(\mathrm{pt})$-linear.
      \end{NB2}%
    Let us write the composite as $c\mapsto c'\defeq(\bullet\ast
    e)^{-1}(e\ast c)$.
    \begin{NB2}
        This is $H_{G_\cO}^*(\mathrm{pt})$-linear.
    \end{NB2}%
    We have $c'\ast e = e\ast c$. We take $c_1$,
    $c_2\in H^{G_\cO}_*(\cR)$ and corresponding elements $c_1'$,
    $c_2'\in H^{T_\cO}_*(\cR_{T,\bNT})|_{\ft^\circ}$. Then $c'_1\ast e
    = e\ast c_1$, $c'_2\ast e = e\ast c_2$. The associativity of
    products implies
    \begin{equation*}
        (c_1'\ast c_2') \ast e = e \ast (c_1\ast c_2).
    \end{equation*}
    It means that the map respects the multiplication.
    \end{NB}%
    \begin{NB}
        An older proof. In particular, it shows that the composite
        algebra homomorphism in Claim is given by the composite of
        \eqref{eq:16}.

    Consider a chain of algebra homomorphisms
    \begin{multline}
        \label{eq:16}
        H^{G_\cO}_*(\cR)\to \End_{H^{T_\cO}_*(\cR_{T,\bNT})}(H^{T_\cO}_*(\cR))^W
\\
        \to \End_{H^{T_\cO}_*(\cR_{T,\bNT})|_\loc}(H^{T_\cO}_*(\cR)|_\loc)^W
        \cong H^{T_\cO}_*(\cR_{T,\bNT})|_\loc^W.
    \end{multline}
    \begin{NB2}
        The first is given by the bimodule structure. The second is
        the restriction. The last isomorphism is given by
        $c\operatorname{id}\to c$ for $c\in
        H^{T_\cO}_*(\cR_{T,\bNT})|_\loc$.
    \end{NB2}%
    If we compose $\iota_*$ at the end,
    \begin{NB2}
        i.e., the evaluation $c\ast 1$,
    \end{NB2}%
    the composite $H^{G_\cO}_*(\cR)\to H^{T_\cO}_*(\cR)|_\loc$ is
    nothing but the induced map from $T_\cO\subset G_\cO$. By
    \lemref{lem:flat} it is injective. In particular, \eqref{eq:16}
    gives an injective algebra homomorphism from $H^{G_\cO}_*(\cR)$ to
    $H^{T_\cO}_*(\cR_{T,\bNT})|_\loc$. Therefore $H^{G_\cO}_*(\cR)$ is
    commutative.
    \end{NB}
\end{proof}

We endow $H^{T_\cO}_*(\cR)\cong
H^{G_\cO}_*(\cR)\otimes_{H^*_G(\mathrm{pt})} H^*_T(\mathrm{pt})$ with an algebra
structure induced from $H^{G_\cO}_*(\cR)$.
Note that {\it a priori\/} it does not have an algebra structure. The
multiplication $\ast$ on $H^{G_\cO}_*(\cR)$ is
$H_G^*(\mathrm{pt})$-linear only in the first variable. In our case, we
have just proved that $H^{G_\cO}_*(\cR)$ is commutative, and hence
also linear in the second variable. Therefore the multiplication is
well-defined. This construction does not work in general, say
$H^{T_\cO\rtimes\CC^\times}_*(\cR)$, as it is noncommutative.

\begin{Remark}
    Let us emphasize further how the commutativity of
    $H^{G_\cO}_*(\cR)$ is important in this construction.

    Consider the usual finite dimensional Steinberg variety $\St$, and
    its analog $\overline\St$.
    \begin{NB}
        Misha and Sasha: I do not yet understand what we can/cannot
        say for $H^{G_\cO\rtimes\CC^\times}_*(\cR)$. Since the
        noncommutative multiplication is not
        $H_{G_\cO\rtimes\CC^\times}^*(\mathrm{pt})$-linear in the
        second variable, I cannot see how to formulate the result. But
        the computation of the convolution still makes sense (at least
        to me).... Maybe a finite dimensional case is useful. But.....

        I am a little confused what is the analog of $\cR_{T,\bNT}$
        here. We should have a counter example.....

        This remark should be expanded. How those claims fail in
        noncommutative case.....
    \end{NB}%
    We also consider $\St_P$, $\overline\St_P$ corresponding to a
    parabolic subgroup $P$. Our $H^{G_\cO}_*(\cR)$ is an analog of
    $H^P_*(\overline\St_P)$, which is isomorphic to $H^G_*(\St_P)$. In
    this situation, $H^P_*(\overline\St_P)$ is an algebra, but
    $H^T_*(\overline\St_P)$ is not in general. It is clear from the
    following: We have $H^P_*(\overline\St_P) \cong e_P
    H^B_*(\overline\St) e_P$ for an idempotent $e_P$. Then
    $H^T_*(\overline\St_P) \cong H^B_*(\overline\St) e_P$. 
    \begin{NB}
        Reference ?
    \end{NB}%
    Thus $H^T_*(\overline\St_P)$ is a bimodule, but not an algebra.
\end{Remark}

We have
\begin{Lemma}\label{lem:iotahom}
    The pushforward homomorphism $\iota_*$ of the embedding
    $\iota\colon \cR_{T,\bNT}\to \cR$, composed with the inverse of
    \eqref{eq:20}, gives an algebra homomorphism
    $H^{T_\cO}_*(\cR_{T,\bNT})\to H^{T_\cO}_*(\cR)\cong
    H^{G_\cO}_*(\cR)\otimes_{H^*_G(\mathrm{pt})} H^*_T(\mathrm{pt})$.
    It becomes an isomorphism over $\loc$. In particular, it is
    injective.
\end{Lemma}

Thus $H^{G_\cO}_*(\cR)$ and $H^{T_\cO}_*(\Gr_T)$ are related as
\begin{equation}\label{eq:21}
    H^{T_\cO}_*(\Gr_T)^W
    \xleftarrow{\bz^*} H^{T_\cO}_*(\cR_{T,\bNT})^W
    \xrightarrow{\iota_*}
    H^{T_\cO}_*(\cR)^W     \cong H^{G_\cO}_*(\cR).
\end{equation}
Both $\iota_*$ and $\bz^*$ are algebra homomorphisms, and become
isomorphisms over $\loc$. (See \lemref{lem:injective} for the last
assertion.)

\begin{Proposition}\label{prop:integrable}
    \textup{(1)} We have a $\CC[\ft]^W$-algebra isomorphism
    \begin{equation*}
        H^{T_\cO}_*(\Gr_T)^W \cong \CC[\ft\times T^\vee]^W,
    \end{equation*}
    where $T^\vee$ is the dual torus of $T$.
    \begin{NB}
        Added on Apr. 27.
    \end{NB}
    The $W$-action on $\ft\times T^\vee$ is the usual one.

    \textup{(2)} The quantized algebra
    $H^{T_\cO\rtimes\CC^\times}_*(\Gr_T)^W$ is isomorphic to the
    $W$-invariant part of the ring of $\hbar$-differential operators
    $D_\hbar(T^\vee)$ on $T^\vee$. The homomorphism
    $H_{G_\cO\rtimes\CC^\times}^*(\mathrm{pt})\to
    H^{T_\cO\rtimes\CC^\times}_*(\Gr_T)^W$ is given by invariant
    vector fields via the isomorphism
    $H_{G_\cO\rtimes\CC^\times}^*(\mathrm{pt})\cong
    H_{T\times\CC^\times}^*(\mathrm{pt})^W \cong
    \operatorname{Sym}(\ft^\vee)^W[\hbar]$.
\end{Proposition}

\begin{NB}
    \cite{MR2135527} uses an isomorphism $\ft\cong \ft^\vee$ via a
    choice of an invariant form. But it seems it is more natural to
    use $\ft$.
\end{NB}%

\begin{proof}
    Since $\Gr_T$ consists of points parametrized by the coweight
    lattice $Y$ of $T$, we have $H^{T_\cO}_*(\Gr_T)\cong \CC[\ft\times
    T^\vee]$. (This is a special case of \thmref{abel}.) 
    The $W$-action on $H^{T_\cO}_*(\Gr_T)$ is given by the
    $N(T)$-action on $\Gr_T$, and the induced $W=N(T)/T$-action on
    $\Gr_T\times_T EG\cong Y\times (EG/T) = Y\times BT$, where $EG\to
    BG$, $ET\to BT$ are the classifying spaces for $G$ and $T$. This
    induces the usual $W$-action on $\ft\times T^\vee$.

    The argument for the quantized version is the same.
\end{proof}

Taking spectrum of \eqref{eq:21}, \lemref{lem:convolution}, we get
\begin{equation}\label{eq:75}
    \ft\times T^\vee/W \to \mathcal M_C(T,\bNT)/W
    \leftarrow \mathcal M_C(G,\bN), \qquad
    	\mathcal M_C(G,\bN) \leftarrow \mathcal M_C(G,0).
\end{equation}
Since homomorphisms are injective, those morphisms are \emph{dominant}.

When $\bN=0$, the left and right morphisms are the identity. The
middle morphism is an affine blowup described in
\cite[\S2.5]{MR2135527}. This is not explicitly stated in
\cite{MR2135527}, but is clear from the proof of
\cite[Th.~2.12]{MR2135527}.
\begin{NB}
Misha, could you check that this explanation is fine ? Yes, fine
\end{NB}%

Next consider examples in \subsecref{subsec:abel_examples} (hence
$G=T$). The middle morphism is the identity. The left (and the right)
morphism is given by an open embedding $\CC\times\CC^\times\ni (a,x)\mapsto
(a,x,y=a^N x^{-1})\in \{ xy = a^N \}$.
\begin{NB}
    This embedding is compatible with $N$: $(a,x) \mapsto (a,x,y=a^N
    x^{-1}) \mapsto (a,x,y'=a^{N+1}x^{-1}=ay)$. See
    \remref{rem:further}.
\end{NB}%

\begin{Corollary}\label{cor:coordinate}
    \textup{(1)} We have a birational isomorphism
    \begin{equation*}
        \mathcal M_C(G,\bN)
         \approx 
        \ft\times T^\vee/W
    \end{equation*}
    given by $\bz^*\iota_*^{-1}$, where $\intsys$ in \eqref{eq:17}
    corresponds to the first projection in $\ft\times T^\vee/W$. It is
    an isomorphism over $\ft^\circ\times T^\vee/W$.
    In particular, the generic fiber of $\intsys$ is
    $T^\vee$.\footnote{The third named author thanks Kentaro Hori for
      correcting his mistake.}

    \textup{(2)} Moreover the Poisson structure on $\mathcal
    M_C(G,\bN)|_{\intsys^{-1}(\ft^\circ/W)}$ corresponds to the standard
    one on $\ft^\circ\times T^\vee/W$, the restriction of the
    symplectic structure on $T^*T^\vee/W$ to the open subset
    $\ft^\circ\times T^\vee/W$.
\end{Corollary}

\begin{Corollary}\label{cor:integral}
    $\mathcal M_C(G,\bN)$ is an integral scheme.
\end{Corollary}

\begin{proof}
    We have an injective homomorphism $H^{G_\cO}_*(\cR)\to
    \CC[\ft^\circ\times T^\vee]$. The latter is an integral
    domain. Therefore $H^{G_\cO}_*(\cR)$ is also.
\end{proof}

\begin{Remark}\label{rem:BDG1}
    Let us consider the open subvariety
    $\varpi^{-1}(\ft\setminus\{\text{root hyperplanes}\}/W)$ in
    $\mathcal M_C$. The morphism $\mathcal M_C(G,\bN)\to \mathcal
    M_C(T,\bNT)/W$ in \eqref{eq:75} becomes an isomorphism over this
    open subvariety, thanks to the localization theorem in equivariant
    homology groups. Therefore
    $\CC[\varpi^{-1}(\ft\setminus\{\text{root hyperplanes}\}/W)]$ has
    an explicit presentation from that of $\cA[T,\bNT] = \CC[\mathcal
    M_C(T,\bNT)]$.
    Furthermore, we have an embedding $\cA[T,\bNT]\hookrightarrow
    \cA[T,0]$, where the latter is just a Laurent polynomial ring
    $\CC[\ft\times T^\vee]$.

    These embeddings make sense for quantized Coulomb branches: we
    have $\cAh[G,\bN] \hookrightarrow \cAh[T,\bNT]
    [\hbar^{-1}, (\text{root}+m\hbar)^{-1}]_{m\in{\mathbb Z}}
    \hookrightarrow \cAh[T,0][\hbar^{-1}, (\text{root}+m\hbar)^{-1}]_{m\in{\mathbb Z}}$. 
    Here one can easily check that the multiplicative subset generated by 
    $\hbar, \text{roots}+{\mathbb Z}\hbar$ satisfies
    the Ore condition in $\cAh[T,\bNT]$ and $\cAh[T,0]$ thanks to
    \lemref{lem:former_claim}. Therefore the localization as
    $H^*_T(\mathrm{pt})$-modules have algebra structure. Moreover
    $\cAh[T,\bNT][\hbar^{-1}, (\text{root}+m\hbar)^{-1}]_{m\in{\mathbb Z}}$ is the 
    (localized) ring of $\hbar$-difference operators on $\ft$.

The ring homomorphism $\iota_*\colon \cAh[T,\bNT]^{W}\to\cAh$
becomes an isomorphism if we invert the expressions $\hbar, \alpha+m\hbar$ 
where $\alpha$ is a root of $G$ and $m$ is an integer, considered
as elements in $H^*_{T\times{\mathbb C}^\times}(\mathrm{pt})$ (\ref{sec:bimodule}).
%
Indeed, it suffices to check that for any dominant coweight $\lambda$,
the homomorphism of $(H^*_{T\times\CC^\times}(\operatorname{pt}))^W$-modules
$\iota_*\colon (H_*^{T_\shfO\rtimes\CC^\times}(\cR_{T,\bN_T})_{W\lambda})^W\to
H_*^{G_\shfO\rtimes\CC^\times}(\cR_\lambda)$ becomes an isomorphism after inverting
the expressions $\hbar, \alpha+m\hbar$. To this end note that at a $T$-fixed
point $w\lambda\in\cR_\lambda,\ w\in W$, the quotient of the tangent space to 
$\cR_\lambda$ modulo the tangent space to $\cR_{T,\bN_T}$ equals the tangent space
to $\Gr^\lambda_G$, and the weights of $T\times\CC^\times$ in the latter tangent
space are all of the form $\hbar, \alpha+m\hbar$. 

We further have a ring homomorphism $\bz^*\colon \cAh[T,\bNT]\to
\cAh[T,0]$ (\ref{sec:z}).
By \ref{prop:integrable} $\cAh[T,0]$ is a $\CC[\hbar]$-algebra
generated by $w_{r}$, ${\mathsf u}^{\pm 1}_{r}$ ($1\le r\le \dim T$) with
relations $[{\mathsf u}_{r}^{\pm1}, w_{s}] = \pm\delta_{r,s} \hbar
{\mathsf u}_{r}^{\pm1}$. (Here we take coordinates of $T^\vee$ and the induced
coordinates on $\ft$.)

    In \cite[\S4]{2015arXiv150304817B} it is argued that $\CC[\mathcal
    M_C]$ is embedded into an explicit combinatorial ring, denoted by
    $\CC[\mathcal M_C^{\mathrm{abel}}]$ from a physical intuition.  It
    turns out that $\CC[\mathcal M_C^{\mathrm{abel}}]$ is nothing but
    the coordinate ring of $\varpi^{-1}(\ft\setminus\{\text{root
      hyperplanes}\}/W)$.
    This is obvious from \cite[(4.11)]{2015arXiv150304817B}.
    The explicit presentation in \cite[(4.9)]{2015arXiv150304817B}
    coincides with one induced from $\cA[T,\bNT]$ explained above.
    The quantized case is mentioned in \cite[\S4.5]{2015arXiv150304817B}.
\end{Remark}

\subsection{Flatness guarantees that codimension \texorpdfstring{$1$}{1} is enough}

Let $\ft^\bullet$ be the complement to all pairwise intersections of
generalized root hyperplanes. We have
$\codim_\ft(\ft\setminus\ft^\bullet) = 2$.  We set
$\ft^\bullet=\ft$ if $\dim\ft = 1$, and $\ft^\bullet =\ft\setminus\{0\}$ if
$\dim\ft=2$.

Let $t\in\ft^\bullet\setminus\ft^\circ$. Let $G' = Z_G(t)$ and $\bN' =
\bN^t$, where $Z_G(t)$, $\bN^t$ are as in \lemref{lem:fixedpts}. We
consider $T$ as a maximal torus of $G'$.

We consider the diagrams \eqref{eq:21} for both $(G,\bN)$, $(G',\bN')$
and will study their compatibilities. Let us denote the maps for
$(G',\bN')$ by $\bz^{\prime*}$, $\iota'_*$.

Note that the embedding $\bz\colon \Gr_T\to\cT_{T,\bNT}$ factors as
$\Gr_T\xrightarrow{\bz'} \cT_{T,\bNT'} \xrightarrow{\bz''} \cT_{T,\bNT}$
where $\bz''$ is an embedding of a subbundle. We have also another
embedding $\bz'''\colon \cT_{G',\bN'}\to \cT_{G',\bN}$ of a
subbundle. We have pull-back with support homomorphisms
\(
  \bz^{\prime\prime*}\colon H^{T_\cO}_*(\cR_{T,\bNT})
  \to H^{T_\cO}_*(\cR_{T,\bNT'})
\)
together with $\bz^*$, $\bz^{\prime*}$, $\bz^{\prime\prime\prime*}$ as in
\subsecref{sec:z}.

%
Note also that $\iota\colon\cR_{T,\bNT}\to\cR$ factors as
$\cR_{T,\bNT}\to \cR_{G',\bN}\to\cR$. Let $\iota''$ (resp.\
$\iota'''$) denote the left (resp.\ right) map.

Let us consider the following diagram:
\begin{equation}\label{eq:22}
    \begin{CD}
        \CC[\ft\times T^\vee] = H^{T_\cO}_*(\Gr_T)
        @<{\bz^*}<< H^{T_\cO}_*(\cR_{T,\bNT})
        @>{\iota_*}>> H^{T_\cO}_*(\cR)
\\
        @| @| @AA{\iota'''_*}A
\\
        \CC[\ft\times T^\vee] = H^{T_\cO}_*(\Gr_T)
        @<{\bz^*}<< H^{T_\cO}_*(\cR_{T,\bNT})
        @>{\iota''_*}>> H^{T_\cO}_*(\cR_{G',\bN})
\\
        @| @V{\bz^{\prime\prime*}}VV @VV{\bz^{\prime\prime\prime*}}V
\\
        \CC[\ft\times T^\vee] = H^{T_\cO}_*(\Gr_T)
        @<{\bz^{\prime*}}<< H^{T_\cO}_*(\cR_{T,\bNT'})
        @>{\iota'_*}>> H^{T_\cO}_*(\cR_{G',\bN'}).
    \end{CD}
\end{equation}

\begin{Lemma}\label{lem:sqcomm}
    All squares are commutative.
\end{Lemma}

\begin{proof}
    The commutativity is obvious except for the right bottom
    square. The commutativity of the right bottom square follows from
    the base change, noticing that $\bz^{\prime\prime*}$,
    $\bz^{\prime\prime\prime*}$ are $\bz^{\prime\prime!}$,
    $\bz^{\prime\prime\prime!}$ up to appropriate shifts respectively, and
    \begin{equation*}
        \begin{CD}
            \cT_{T,\bNT} @>>> \cT_{G',\bN}
            \\
            @AAA @AAA
\\
            \cT_{T,\bNT'} @>>> \cT_{G',\bN'}
        \end{CD}
    \end{equation*}
    is Cartesian.
\end{proof}


\newcommand{\cM}{\mathcal M}

\newcommand{\nc}{\newcommand}
\newcommand{\iso}{\stackrel{\sim}{\longrightarrow}}
\nc{\BA}{{\mathbb{A}}}
\nc{\BC}{{\mathbb{C}}}
\nc{\BF}{{\mathbb{F}}}
\nc{\BG}{{\mathbb{G}}}
\nc{\BN}{{\mathbb{N}}}
\nc{\BQ}{{\mathbb{Q}}}
\nc{\BP}{{\mathbb{P}}}
\nc{\BR}{{\mathbb{R}}}
\nc{\BZ}{{\mathbb{Z}}}

\nc{\bT}{{\mathbf{T}}}
\nc{\bp}{{\mathbf{p}}}
\nc{\bq}{{\mathbf{q}}}
\nc{\bfr}{{\mathbf{r}}}
\nc{\bfs}{{\mathbf{s}}}
\nc{\bfo}{{\mathbf{o}}}
\nc{\bt}{{\mathbf{t}}}

\nc{\CB}{{\mathcal{B}}}
\nc{\CF}{{\mathcal{F}}}
\nc{\CG}{{\mathscr{G}}}
\nc{\CH}{{\mathcal{H}}}
\nc{\CI}{{\mathcal{I}}}
\nc{\CJ}{{\mathcal{J}}}
\nc{\CK}{{\mathcal{K}}}
\nc{\CL}{{\mathcal{L}}}
\nc{\CM}{{\mathcal{M}}}
\nc{\CO}{{\mathcal{O}}}
\nc{\CR}{{\mathcal{R}}}
\nc{\cS}{{\mathcal{S}}}
\nc{\CT}{{\mathcal{T}}}
\nc{\CU}{{\mathcal{U}}}
\nc{\CV}{{\mathcal{V}}}
\nc{\CW}{{\mathcal{W}}}
\nc{\oW}{\overline{\mathcal{W}}{}}
\nc{\bW}{\overset{\bullet}{\mathcal W}{}}
\nc{\uoW}{{\overline{\underline{\mathcal{W}}}}{}}
\nc{\bGZ}{\overset{\bullet}{\CG\CZ}{}}
\nc{\oGZ}{\overset{\circ}{\CG\CZ}{}}
\nc{\CZ}{{\mathscr{Z}}}

\nc{\fg}{{\mathfrak{g}}}
\nc{\fri}{{\mathfrak{i}}}
\nc{\fgl}{{\mathfrak{gl}}}
\nc{\fsl}{{\mathfrak{sl}}}
\nc{\fz}{{\mathfrak{z}}}
\nc{\fB}{{\mathfrak{B}}}
\nc{\fC}{{\mathfrak{C}}}
\nc{\fF}{{\mathfrak{F}}}
\nc{\fM}{{\mathfrak{M}}}
\nc{\fU}{{\mathfrak{U}}}
\nc{\oU}{\overset{\circ}{\mathfrak U}{}^\alpha_{G;B}}

\nc{\sy}{{\mathsf{y}}}
\nc{\sx}{{\mathsf{x}}}
\nc{\sfu}{{\mathsf{u}}}
\nc{\sw}{{\mathsf{w}}}

\nc{\unl}{\underline}
\nc{\ol}{\overline}
\nc{\on}{\operatorname}

\nc{\oZ}{\overset{\circ}{Z}{}}
\nc{\bZ}{\overset{\bullet}{Z}{}}
\nc{\oE}{\overset{\circ}{E}{}}
\nc{\bE}{\overset{\bullet}{E}{}}
\nc{\uZ}{\underline{Z}}
\nc{\uoZ}{\overset{\circ}{\underline Z}{}}
\nc{\uB}{\underline{\mathfrak B}}
\nc{\oB}{\overset{\circ}{\mathfrak B}{}}
\nc{\uoB}{\overset{\circ}{\underline{\mathfrak B}}{}}
\nc{\oA}{\overset{\circ}{\mathbb A}{}}
\nc{\oG}{\overset{\circ}{\mathbb G}{}_m}
\nc{\bA}{\overset{\bullet}{\mathbb A}{}}

\begin{Theorem}\label{prop:flat}
    Let $\cM\stackrel{\varPi}{\longrightarrow}\ft/W$ 
be an affine scheme over $\ft/W$, and let $\cM^\bullet:=\varPi^{-1}(\ft^\bullet/W)$.
We assume that the natural morphism $\varPi_*\CO_\cM\to j_*\varPi_*\CO_{\cM^\bullet}$
is an isomorphism where $j\colon \ft^\bullet/W\hookrightarrow\ft/W$ is an
open embedding. We further assume $\cM$ is equipped with the following data.
    \begin{enumerate}
          \item Assume we are given an isomorphism between the localizations
        \begin{equation*}
            \CC[\cM]\otimes_{\CC[\ft/W]}\CC[\ft^\circ/W]\xrightarrow{\Xi}
            \CC[\ft\times T^\vee]^W\otimes_{\CC[\ft/W]}\CC[\ft^\circ/W]
            \iso H_*^{G_\CO}(\CR_{G,\bN})\otimes_{\CC[\ft/W]}\CC[\ft^\circ/W],
        \end{equation*}
        (the second isomorphism is $\iota_*(\bz^*)^{-1}$ of~\eqref{eq:21}
        plus~\propref{prop:integrable}). The composition is denoted 
        $\Xi^\circ\colon \CC[\cM]\otimes_{\CC[\ft/W]}\CC[\ft^\circ/W]\iso
        H_*^{G_\CO}(\CR_{G,\bN})\otimes_{\CC[\ft/W]}\CC[\ft^\circ/W]$.
          \item Let $t\in\ft^\bullet\setminus\ft^\circ$. Let $G' =
        Z_G(t)$ and $\bN' = \bN^t$, where $Z_G(t)$, $\bN^t$ are as in
        \lemref{lem:fixedpts}. We consider $T$ as a maximal torus of
        $G'$. 
        Assume we are given an isomorphism
        \begin{equation*}
            \Xi^t\colon (\CC[\cM]\otimes_{\CC[\ft/W]}\CC[\ft])_t
            \xrightarrow{\cong}
            (H^{G'_\cO}_*(\cR_{G',\bN'})
            \otimes_{H^*_{G'}(\mathrm{pt})}\CC[\ft])_t
        \end{equation*}
        such that
        \begin{equation*}
            \bz^{\prime*}(\iota'_*)^{-1}
            \left(\Xi^t\otimes_{\CC[\ft]_t}\CC[\ft^\circ]\right) 
            = \Xi\otimes_{\CC[\ft^\circ/W]}\CC[\ft^\circ].
        \end{equation*}
    \end{enumerate}
    Here $\Xi$ is a $\CC[\ft^\circ/W]$-algebra isomorphism, while
    $\Xi^t$ is a $\CC[\ft]_t$-algebra
    \begin{NB}
        module 
    \end{NB}%
    isomorphism. 
    \begin{NB2}
        Misha, I do not think that `$\Xi^t$ is a $\CC[\ft]_t$-module
        isomorphism' is a reasonable assumption. By the compatibility
        above, it is an algebra homomorphism over
        $\CC[\ft^\circ]$. Therefore it should be so also over
        $\CC[\ft]_t$, right ? In practice, I hardly imagine that we
        can construct $\Xi^t$, which is not clear to be an algebra
        isomorphism. I agree.
    \end{NB2}%
    \begin{NB}
    The restriction of the latter one to the generic point is supposed
    to be $W$-equivariant, and its restriction to the $W$-invariants
    must coincide with the generic fiber of the former one (it is
    understood that the generic fiber
    $(H_*^{T_\CO}(\cR_{G',\bN'}))_{\on{loc}}$ of
    $(H^{G'_\cO}_*(\cR_{G',\bN'})\otimes_{H^*_{G'}(\mathrm{pt})}\CC[\ft])_t=
    (H_*^{T_\CO}(\cR_{G',\bN'}))_t$ is identified with
    $(H_*^{T_\CO}(\cR))_{\on{loc}}$ via
    $i'''_*(z'''{}^*_{\on{loc}})^{-1}$).
    \end{NB}
    \begin{NB2}
        Do we really need the $W$-invariance ? No, with your explicit
formulation with $\Xi$, the $W$-invariance is automatic.
    \end{NB2}

    Then $\Xi^\circ$ extends to an isomorphism between $\cM$ and 
    $\Spec(H^{G_\cO}_*(\cR))$ as schemes over $\ft/W$.
\end{Theorem}

\begin{NB}
    Misha, please correct above, if there is a mistake. You can change
    the notation as you like.
\end{NB}

This proposition enables us to determine $\cM$ in two steps. We first
determine the Coulomb branch $H^{G'_\cO}_*(\cR_{G',\bN'})$ for another
pair $(G',\bN')$. Since $G'$ has semisimple rank $1$, it should be
easier than $H^{G_\cO}_*(\cR)$. Then we look for $\cM$ so that various
$H^{G'_\cO}_*(\cR_{G',\bN'})$ are `glued' to form a flat family.

\begin{proof}
\begin{NB}Let $j\colon \ft^\bullet/W\hookrightarrow\ft/W$ denote the open embedding.
Let also $j\colon \cM^\bullet\hookrightarrow\CM$ denote the open embedding of the
preimage of $\ft^\bullet/W$ in $\CM$. We have 
$j_*\varPi_*\CO_{\CM^\bullet}=\varPi_*j_*\CO_{\CM^\bullet}=\varPi_*\CO_\CM$ because $\CM$
is Cohen-Macaulay, and the codimension of $\CM\setminus\CM^\bullet$ in $\CM$
is at least 2 because all the fibers of $\varPi$ have the same dimension 
by flatness. (according to [Theorem~23.1, H.~Matsumura, 
Commutative ring theory, 
Cambridge Studies in Advanced Mathematics 8 (1986), xiv+320pp].
\end{NB}%
Recall that for a dominant coweight $\lambda$ we denote by $\CR_\lambda$
(resp.\ $\CR_{\leq\lambda}$) the preimage of the $G_\CO$-orbit $\Gr_G^\lambda$
(resp.\ of its closure $\ol\Gr{}_G^\lambda$); we also have 
$\CR_{<\lambda}:=\CR_{\leq\lambda}\setminus\CR_\lambda$. Then the closed embedding
$\CR_{<\lambda}\hookrightarrow\CR_{\leq\lambda}$ gives rise to the exact sequence
$0\to H^{G_\CO}_*(\CR_{<\lambda})\to H^{G_\CO}_*(\CR_{\leq\lambda})\to
H^{G_\CO}_*(\CR_\lambda)\to0$~(\lemref{lem:MayerV}). 
Since $H^{G_\CO}_*(\CR_\lambda)=H^{G_\CO}_*(\Gr_G^\lambda)$
is a finitely generated flat $H_{G_\CO}^*(\mathrm{pt})=\BC[\ft/W]$-module, we conclude
inductively that $H^{G_\CO}_*(\CR_{\leq\lambda})$ is a finitely generated flat 
$\BC[\ft/W]$-module. Hence it is a finitely generated projective 
$\BC[\ft/W]$-module. Let $\CH_{\leq\lambda}$ denote the corresponding locally
free coherent sheaf on $\ft/W$. Then the natural morphism
$\CH_{\leq\lambda}\to j_*j^*\CH_{\leq\lambda}$ is an isomorphism (since $\ft/W$
is smooth). Taking the union over all $\lambda$ we obtain that the natural
morphism $\CH\to j_*j^*\CH$ is an isomorphism, where $\CH$ is the quasicoherent
sheaf on $\ft/W$ localizing the $H_{G_\CO}^*(\mathrm{pt})=\BC[\ft/W]$-module
$H^{G_\CO}_*(\CR)$.

We see that in order to identify the $\ft/W$-schemes $\CM$ and
$\Spec(H^{G_\cO}_*(\cR))$ it suffices to identify the quasicoherent
$\ft^\bullet/W$-modules $\varPi_*\CO_{\CM^\bullet}$ and $j^*\CH$: indeed,
due to~(1), both $\BC[\CM]$ and $H^{G_\cO}_*(\cR)$ are subalgebras in
$\BC[\ft^\circ\times T^\vee]^W$, so it suffices to identify them as 
{\em subsets}. 
Now the desired identification is given in~(1) upon restriction to 
$\ft^\circ/W\subset\ft^\bullet/W$, 
and~(2) guarantees that the latter identification extends to $\ft^\bullet/W$.
More precisely,
let $\Xi^\circ = \iota_* (\bz^*)^{-1} \circ \Xi$. Note that maps in the diagram \eqref{eq:22} are $W$-equivariant. Hence $\Xi^\circ$ gives an isomorphism
from $\CC[\mathcal M]\otimes_{\CC[\ft/W]}\CC[\ft^\circ/W]$
to $H^{G_\cO}_*(\cR)\otimes_{H^*_G(\mathrm{pt})}\CC[\ft^\circ/W]$.
We have
\begin{equation*}
    \bz^{\prime*} (\iota'_*)^{-1}
            \left(\Xi^t\otimes_{\CC[\ft]_t}\CC[\ft^\circ]\right) 
            = \bz^*(\iota_*)^{-1}
            \left(\Xi^\circ\otimes_{\CC[\ft^\circ/W]}\CC[\ft^\circ]\right)
\end{equation*}
by the assumption in (2). By \lemref{lem:sqcomm}
\begin{equation*}
    \Xi^t\otimes_{\CC[\ft]_t}\CC[\ft^\circ]
    = \bz^{\prime\prime\prime*}(\iota'''_*)^{-1}
    \left(\Xi^\circ\otimes_{\CC[\ft^\circ/W]}\CC[\ft^\circ]\right).
\end{equation*}
Note that $\bz^{\prime\prime\prime*}(\iota'''_*)^{-1}$ is an isomorphism
over $\CC[\ft]_t$ by the localization theorem in equivariant homology
groups, as $\cR_{G',\bN'}$ is the fixed point set of $t$. Since
$\Xi^t$ extends to $\CC[\ft]_t$ by the assumption,
$\Xi^\circ\otimes_{\CC[\ft^\circ/W]}\CC[\ft^\circ]$ also extends to
$\CC[\ft]_t$. We apply this argument to all generalized roots, we get
an extension to $\CC[\ft^\bullet]$. Taking the $W$-invariant part, we see that
$\Xi^\circ$ extends to $\CC[\ft^\bullet/W]$.
\begin{NB2}
    Do we need any assumption for the uniqueness ? flatness is enough ?
No assumption at all is needed for uniqueness since this extension is not
an extra datum but a property of the isomorphism at the generic point.
\end{NB2}%

    \begin{NB}
        I have looked at the proof of \cite{MR2135527}. I do not yet
        understand why an isomorphism in (2) can be compatible with
        the given isomorphism in (1) at the generic point.
    \end{NB}
\end{proof}

\begin{Remark}
\label{rem:conditions}
There are various ways to guarantee the condition
$\varPi_*\CO_\cM\xrightarrow{\cong} j_*\varPi_*\CO_{\cM^\bullet}$ of~\thmref{prop:flat}.
For instance, it is enough to assume that $\varPi$ is flat, and $\CM$ is
Cohen-Macaulay. In effect, let $j\colon \cM^\bullet\hookrightarrow\CM$ denote also 
the open embedding of the preimage of $\ft^\bullet/W$ in $\CM$. We have 
$j_*\varPi_*\CO_{\CM^\bullet}=\varPi_*j_*\CO_{\CM^\bullet}=\varPi_*\CO_\CM$ because $\CM$
is Cohen-Macaulay, and the codimension of $\CM\setminus\CM^\bullet$ in $\CM$
is at least 2 because all the fibers of $\varPi$ have the same dimension 
by flatness. 
\begin{NB}
(according to \cite[Theorem~23.1]{matsumura}.)
\end{NB}%

Alternatively, it is enough to assume that all the fibers of $\varPi$ have
the same dimension, and $\CM$ satisfies the Serre condition $S_2$, e.g.\ 
$\CM$ is normal.
\end{Remark}


\begin{NB}
The following is an old text, which is more or less, meaningless.

We have an action of $G_\cK$ on $\Gr_G$ changing the trivialization
$\varphi$ on $D^*$. If we restrict it to $T$ (constant maps with
values in $T$), the fixed point set is $\Gr_T$, the affine
Grassmannian for $T$. Similarly we have $(\cR_{G,\bN})^T =
\cR_{T,\bN|_{T}}$, where $\bN|_T$ is the restriction of $\bN$ to
$T$. By the localization theorem in equivariant homology groups, we have
\begin{equation*}
    H^{G_\cO\rtimes\CC^\times}_*(\cR_{G,\bN})_{\mathrm{loc}}
    \cong
    H^{T_\cO\rtimes\CC^\times}_*(\cR_{G,\bN})_{\mathrm{loc}}^W
    \cong
    H^{T_\cO\rtimes\CC^\times}_*(\cR_{T,\bN|_T})_{\mathrm{loc}}^W,
\end{equation*}
where the subscript `loc' means the localization by the fractional
fields of the equivariant cohomology of a point, and $W$ is the Weyl
group. This isomorphism is compatible with the convolution product.
\end{NB}

\begin{NB}
    See the note on 2014-12-16, for a trial of a study for $(G_c,\bN)
    = (\SU(2),(\CC^2)^{\oplus N_f})$.
\end{NB}

\begin{NB}
    The following subsection is not considered in detail. If this is
    really useful, we should change all the above consideration in
    more general settings..... In particular, I need to understand
    what is actually point of \lemref{lem:convolution}.

\subsection{Further analysis of convolution algebras}

We cannot use the same trick for the quantized Coulomb branch as
$H^{G_\cO}_*(\cR)$ is not commutative, and it is not clear to make
$H^{T_\cO}_*(\cR)$ into an algebra, as we remarked above. However we
still have
\begin{equation*}
    H^{G_\cO\rtimes\CC^\times}_*(\cR) \to
    H^{T_\cO\rtimes\CC^\times}_*(\cR) \to
    H^{T_\cO\rtimes\CC^\times}_*(\cR)|_{\ft^\circ}
    \xrightarrow{\iota_*^{-1} = (\bullet\ast e)^{-1}}
    H^{T_\cO\rtimes\CC^\times}_*(\cR_{T,\bNT})|_{\ft^\circ}
\end{equation*}
as in the proof of \propref{prop:commutative} above. This is an algebra homomorphism.

This construction itself may not be so interesting.... But the same
argument should work also for the fixed point set of $t$, where the
final target is the homology of $\cR^t$. This should be useful to
study simple modules of $H^{G_\cO\rtimes\CC^\times}_*(\cR)$, as in
Varagnolo-Vasserot.....
\end{NB}

\begin{NB}
    The following section is probably unnecessary. It is of expository
    in nature, anyway.

\section{Integrable systems}\label{sec:integrable-systems}

By \propref{prop:integrable} our Coulomb branch is equipped with an
integrable system $\intsys\colon \mathcal M_C\to \Spec(H^*_G(\mathrm{pt})) \cong
\CC^\ell$, where $\ell$ is the rank of $G$.

We will propose conjectural explicit description of these integrable
systems, when known hyper-K\"ahler manifolds are expected to be
realized as Coulomb branches.

\subsection{Easy examples}

Let $\CC^\times$ act on $\CC^2$ by $t(x,y) = (tx,t^{-1}y)$. It
preserves the symplectic form $dx\wedge dy$. The moment map
$\intsys\colon\CC^2\to\CC$ is given by $\intsys(x,y) = xy$. A generic fiber
$\{ xy = t\} $ ($t\in\CC^\times$) is $\CC^\times$. We consider
$\intsys\colon\CC^2\to\CC$ as an integrable system, and we will use it as
a building block for more general integrable systems.

It should be noted that $\intsys$ is the complex part of the
hyper-K\"ahler moment map $\CC^2\to \RR^3$. Since we do not describe
the hyper-K\"ahler structure in our proposal~\ref{def:main}, we do not
know how to interpret this observation in more general context.

The above can be generalized to the case of a simple singularity of
type $A$, i.e., $\CC^2/\ZZ_k$. Let $\CC^\times$ act on $\CC^2/\ZZ_k$ by
$t(x,y)\bmod{\ZZ_k} = (t^{1/k}x, t^{-1/k}y)\bmod{\ZZ_k}$.
\begin{NB2}
If we represent $\CC^2/\ZZ_k$ as $\{ XY = Z^k\}$ in $\CC^3$ by
$X = x^k$, $Y = y^k$, $Z = xy$, the action is given by
$t(X,Y,Z) = (tX,t^{-1}Y,Z)$.
\end{NB2}%
The moment map $\intsys$ is given by $\intsys((x,y)\bmod{\ZZ_k}) = xy
\begin{NB2}
    = Z
\end{NB2}%
$.
\begin{NB2}
    I still need to check this assertion.
\end{NB2}

Starting from $\intsys\colon\CC^2\to\CC$, we get
\begin{equation*}
    S^n\intsys \colon S^n\CC^2 \to S^n\CC \cong \CC^n.
\end{equation*}
Components of $S^n\intsys$ are Poisson commuting and a generic fiber is
$(\CC^\times)^n$. Thus we have an integrable system.

\subsection{Toric hyper-K\"ahler manifolds}

\subsection{Affine Grassmannian}

\subsection{Uhlenbeck spaces}

Let $\Uh{d}$ be the Uhlenbeck space, which is a partial
compactification of $\Bun{d}$ the moduli space of algebraic
$G$-bundles over the projective plane $\proj^2$ with the instanton
number $d$ and with trivialization at the line infinity.
It is an affine variety of dimension $2dh^\vee$, where $h^\vee$ is the
dual Coxeter number.

We construct a complete integrable system
$\pi\colon\Uh{d}\to\CC^{dh^\vee}$ using the space $\mathscr GZ$,
introduced in \cite{MR3134906}.

Let us consider the following coordinates on $\proj^2$ and its blowup
at the origin:
\begin{equation*}
    \begin{split}
    & \proj^2 = \{ [x:y:z] \},
\\
    & \widehat\proj^2 = \{ ([X:Y:Z],[s:t]) \mid Xt = Ys \}.
    \end{split}
\end{equation*}
The blowup morphism is $([X:Y:Z],[s:t])\mapsto [X:Y:Z]$. There is an
open embedding of $\CC^2$ in $\widehat\proj^2$ given by
\begin{equation*}
    \CC^2\ni (x,y) \mapsto ([xy:y:1],[x:1]) \in\widehat\proj^2,
\end{equation*}
and the complement is the union of $Z=0$ and $t=Y=0$. Therefore a
(generalized) $G$-bundle on $\proj^2$ trivialized on the line at
infinity can be viewed as a (generalized) $G$-bundle over
$\widehat\proj^2$, trivialized on the union of $Z=0$ and $t=Y=0$.
\end{NB}

\section{Degeneration and its applications}\label{sec:degeneration}

We introduce a natural degeneration of $\mathcal M_C$ and study its
applications in this section.

\subsection{Filtration}\label{sec:filtra}

Let $\overline{\Gr}_{G}^{\la}$ and $\cR_{\le\la} = \cR\cap
\pi^{-1}(\overline{\Gr}_{G}^{\la})$ as in \subsecref{triples}. In the
diagram \eqref{eq:1}, it is known that $m q
p^{-1}(\overline{\Gr}_{G}^{\la}\times\overline{\Gr}_{G}^{\mu})\subset
\overline{\Gr}_{G}^{\la+\mu}$ (see e.g., \cite[Lemma~4.4]{MV2}).
Therefore we have
\begin{equation*}
    \tilde m \tilde q \tilde p^{-1}(\cR_{\le\la}\times\cR_{\le\mu})
    \subset\cR_{\le\la+\mu}
\end{equation*}
as the diagram \eqref{eq:12} is compatible with \eqref{eq:1} under
$\pi\colon\cT\to\Gr_G$. From the definition of the convolution
product, we have
\begin{equation*}
    H^{G_\cO}_*(\cR_{\le\la})\ast H^{G_\cO}_*(\cR_{\le\mu})
    \subset H^{G_\cO}_*(\cR_{\le\la+\mu}).
\end{equation*}
Thus
\begin{Proposition}
    $\cA$ is a filtered algebra with respect to the filtration $\cA =
    H^{G_\cO}_*(\cR) = \bigcup H^{G_\cO}_*(\cR_{\le\la})$. The same is
    true for $\cAh$.
\end{Proposition}

Let $\gr\cA$ denote the associated graded algebra. It is graded by
$Y^+ = Y/W$, where $Y^+$ is the semi-group of dominant coweights.

Thanks to \lemref{lem:MayerV}(2) the associated graded algebra
$\gr\cA$ is identified with $\bigoplus H^{G_\cO}_*(\cR_\la)$, as an
$H_{G_\cO}^*(\mathrm{pt})$-module. Moreover Lemmas~\ref{lem:orbit} and
\ref{lem:MayerV}(1) imply that
\begin{equation*}
    H^{G_\cO}_*(\cR_\la) \cong H^*_{\operatorname{Stab}_G(\la)}(\mathrm{pt})
    \cap [\cR_\la] \cong \CC[\ft]^{W_\la} [\cR_\la],
\end{equation*}
where $W_\la$ is the stabilizer of $\la$ in the Weyl group $W$.
We regard $[\cR_\la]$ as a class of $\gr H^{G_\cO}_*(\cR)$.

Let us replace $G$ by its maximal torus $T$.
The affine Grassmannian $\Gr_T$ for the torus $T$ consists of discrete
points parametrized by the coweight lattice $Y$ of $T$. Therefore we
have a direct sum decomposition $H^{T_\cO}_*(\cR_{T,\bNT}) =
\bigoplus_{\mu\in Y} H^{T_\cO}_*(\cR_{\mu;T,\bNT})$, where
$\cR_{\mu;T,\bNT}$ is the component corresponding to a coweight
$\mu$. The above filtration for $T$ is just the union of
$H^{T_\cO}_*(\cR_{\mu;T,\bNT})$ such that $\mu$ is in a Weyl group
orbit of a dominant coweight $\mu'$ with $\mu'\le\la$.
\begin{NB}
The following is wrong. May 15, HN.

In particular, the associated graded algebra $\gr\cA[T,\bNT]$
is isomorphic to $\cA[T,\bNT]$
itself, as an algebra. The grading is induced from the $Y$-grading
on $\cA[T,\bNT]$ via the projection $Y\to Y/W$.
\end{NB}%
Thanks to the calculation in \thmref{abel}, this is a ring having an
explicit presentation by generators and relations.

In order to calculate $\gr\cA$, let us relate $\gr\cA$ to
$\gr\cA[T,\bNT]$.
Recall the embedding $\iota\colon\cR_{T,\bNT}\to\cR$ appearing in
\lemref{lem:convolution_by_e}. It is compatible with the filtration
$\cR = \bigcup \cR_{\le\la}$ : $\cR_{T,\bNT}\cap \cR_{\le\la}$
consisting of the inverse image of $W\la$ in $\cR_{T,\bNT}$, where
$W\la$ is the Weyl group orbit of $\la$ considered as a subset of
$\Gr_T$. Therefore we have a homomorphism
\begin{equation*}
    \gr\iota_*\colon \gr \cA[T,\bNT] 
    \to \gr H^{T_\cO}_*(\cR)
    = \gr \cA[G,\bN] 
    \otimes_{H^*_G(\mathrm{pt})} H^*_T(\mathrm{pt}).
\end{equation*}
Thanks to \lemref{lem:iotahom}, it is a graded algebra homomorphism,
which becomes an isomorphism over $\loc$. In particular, it is
injective.

Let us compute structure constants of the multiplication in
$[\cR_\la]$, via $\gr\cA(T,\bNT)$. Recall $r^{\la}$ is the
fundamental class of the fiber of $\cR_{T,\bNT}\to\Gr_T$ at a point
corresponding to a (not necessarily dominant) coweight $\la$. 
(See the proof of \thmref{abel}.)
Let $a_{\la,\mu}\in \CC[\ft]$ denote the coefficient given by $r^\la
r^\mu = a_{\la,\mu} r^{\la+\mu}$. See \eqref{relation} for the
explicit form.

\begin{Proposition}\label{prop:degen}
    Let $\la$, $\mu$ be dominant coweights. Let $f$, $g\in
    \CC[\ft]^{W_\la}$, $\CC[\ft]^{W_\mu}$ respectively. We have the
    following identity in the associated graded ring $\gr \cA$:
    \begin{equation*}
        f [\cR_\la] \ast g [\cR_\mu] =
        a_{\la,\mu} fg [\cR_{\la+\mu}].
    \end{equation*}
    The same is true for $\cAh$.
\end{Proposition}

\begin{proof}
    We have
\begin{equation}\label{eq:76}
   (\gr\iota_*)^{-1}f[\cR_\lambda] = \sum_{\la'\in W\la}
   \frac{wf \times r^{\la'}}{e(T_{\la'} \Gr_{G}^{\la})},
\end{equation}
where $T_{\la'}\Gr_{G}^{\la}$ is the tangent space of $\Gr_{G}^{\la}$ at
the point $z^{\la'}$ with $\la'\in W\la$, and $e(T_{\la'}
\Gr_{G}^{\la})$ is its equivariant Euler class. Furthermore $w$ for $wf$
is given so that $\la' = w\la$. Then $wf$ is independent of the choice
as $f$ is invariant under $W_\la$.
Since $\gr\iota_*$ is an algebra homomorphism, we can calculate
$f[\cR_\la]\ast g[\cR_\mu]$ from the right hand side. Since it is
enough to compute the coefficient of $r^{\la+\mu}$, we do not need to
worry other terms than $r^\la$, $r^\mu$. Therefore we obtain
\begin{equation*}
        f [\cR_\la] \ast g [\cR_\mu] =
        \frac{e(T_{\la+\mu} \Gr_{G}^{\la+\mu})}
        {e(T_\la\Gr_{G}^{\la}) e(T_\mu \Gr_{G}^{\mu})}
        a_{\la,\mu} fg [\cR_{\la+\mu}].
\end{equation*}
Now we use $e(T_\la \Gr_{G}^{\la}) = \prod_{\alpha\in\Delta^+}
\alpha^{\langle\la,\alpha\rangle}$, where $\alpha$ is considered as an
element of $\CC[\ft]$ (cf.\ the proof of \lemref{lem:orbit}). Therefore
$e(T_\la \Gr_{G}^{\la}) e(T_\mu\Gr_{G}^{\mu})$ cancels with $e(T_{\la+\mu} \Gr_{G}^{\la+\mu})$ and we get the assertion.
\end{proof}

\begin{Remark}
    \begin{NB}
        Why the coefficient is in $\CC[\ft]^{W_{\la+\mu}}$ ? Sasha
        wrote that it is easy to see....
    \end{NB}%
    Let us check that $a_{\la,\mu} fg\in \CC[\ft]^{W_{\la+\mu}}$. From
    the description of $a_{\la,\mu}$ (see \eqref{relation}), it is clear that
    $a_{\la,\mu}\in \CC[\ft]^W$. Note also that $W_\la$ is the subgroup
    of $W$ generated by simple reflections $s_i$ such that the
    corresponding simple roots $\alpha_i$ are perpendicular to $\la$.
    \begin{NB}
        Suppose $w\in W_\la$. If $\alpha\in\Delta^+\cap w(\Delta^-)$,
        we have $0\le \langle \la, \alpha\rangle = \langle w\la,
        \alpha\rangle = \langle \la, w\alpha\rangle \le 0$. Therefore
        $\langle \la,\alpha\rangle = 0$. Now consider a reduced
        expression of $w$.
    \end{NB}%
    Therefore we have $W_{\la+\mu} = W_{\la} \cap W_\mu$, and hence
    $fg\in \CC[\ft]^{W_{\la+\mu}}$.
\end{Remark}

\begin{Remark}\label{rem:BDG2}
    In \cite[\S4.3]{2015arXiv150304817B}, a monopole operator
    $M_{A,p}\in\CC[\mathcal M_C]$ corresponding to a cocharacter $A$
    and an $W_A$-invariant polynomial $p$ is considered. Here $W_A$ is
    the Weyl group of $\Stab_G(A)$. Moreover a formula for the product
    $M_{A,p} M_{A',p'}$ is proposed. (See
    \cite[(4.16,17)]{2015arXiv150304817B}.) They satisfy the triangular
    property, which is related to our filtration.
    Therefore $M_{A,p}$ could be related to our $f[\cR_\la]$ ($f\in
    \CC[\ft]^{W_\la}$) under $A\leftrightarrow \la$, $p\leftrightarrow
    f$. However our $f[\cR_\la]$ lives in $\gr\cA$. It is not clear
    for us how to lift $f[\cR_\la]$ to $\CC[\mathcal M_C]$
    \emph{canonically}. Also it is not clear for us how to define the
    equivariant integration in \cite[(4.16,17)]{2015arXiv150304817B}
    rigorously.
\end{Remark}

\subsection{Closed \texorpdfstring{$G_\cO$}{G\textunderscore O}-orbits}

The discussion in the previous subsection was about the associated
graded, but we can also say something if $\Gr_G^\la$ is a closed
$G_\cO$-orbit.

First note that the fundamental class $[\cR_\la]$ is well-defined in
$H^{G_\cO}_*(\cR)$ as $\cR_{<\la} = \emptyset$. Let $\iota_*$ as in
\lemref{lem:convolution_by_e}, and $r^{\la'}$ denote the fundamental
class of the fiber of $\cR_{T,\bNT}\to\Gr_T$ at a coweight $\la'$ as
above. Then \eqref{eq:76} remains true:
\begin{Proposition}\label{prop:useful}
    Let $\la$ be a dominant weight such that $\Gr_G^\la$ is a closed
    $G_\cO$-orbit. Let $f\in\CC[\ft]^{W_\la}$. Then
    \begin{equation*}
        (\iota_*)^{-1}f[\cR_\lambda] = \sum_{\la'\in W\la}
        \frac{wf \times r^{\la'}}{e(T_{\la'} \Gr_{G}^{\la})}.
    \end{equation*}
\end{Proposition}

Recall that we have explicit structure constants for $r^{\la'}$
(\thmref{abel}) and $\iota_*$ is an algebra homomorphism
(\lemref{lem:bimodule-1}). Therefore this proposition can be used to
compute multiplication of elements of forms $f[\cR_\la]$ with
$\Gr^\la_G$ closed.

\begin{Remark}\label{rem:BDG3}
   It has been noted that monopole operators $M_{A,p}$ have explicit presentation when $A$ is minuscule in \cite[\S4.3]{2015arXiv150304817B}. This observation is compatible with the above proposition, i.e., $f[\cR_\la]$ has well-defined lift when $\la$ is minuscule, in view of \remref{rem:BDG2}. As is noted \cite{2015arXiv150304817B}, there are many examples such that minuscule monopole operators generate $\cA$, for example when $G$ is a product of $\GL$ and $\PGL$ like a quiver gauge theory.  See the proof of \lemref{lem:SL2}(2) below, for example.
\end{Remark}

\subsection{Finite generation}

\begin{Proposition}\label{prop:fg}
    $\cA$ and $\cAh$ are finitely generated. They are noetherian.
\end{Proposition}

\begin{proof}
    As for finite generation, it is enough to check that for $\cA$
    since $\cAh$ is a flat deformation of $\cA$. Since $\cA$ is
    commutative, it is noetherian if it is finitely generated. Then
    $\cAh$ is also noetherian. Thus it is enough to check that $\cA$
    is finitely generated. Furthermore, it is sufficient to show that
    $\gr\cA$ is finitely generated.

    We consider generalized (closed) Weyl chambers given by
    complements of generalized root hyperplanes. Our $\lambda$, $\mu$
    are in the dominant Weyl chamber in the usual sense, but the
    dominant Weyl chamber is further decomposed into union of
    generalized Weyl chambers as there might be weights which are not
    roots.

    Now each generalized chamber is a rational polyhedral cone, hence
    its intersection with the coweight lattice is a finitely generated
    semigroup.
    \begin{NB}
        This is a standard result, e.g., a textbook in toric geometry.
    \end{NB}%
    We take finitely many $[\cR_\lambda]$ so that $\lambda$
    is a generator of the semigroup. If we multiply $[\cR_\lambda]$,
    $[\cR_\mu]$ when $\la$, $\mu$ are in the same generalized Weyl
    chamber, $a_{\lambda\mu}$ is 1 from its definition. Therefore
    $[\cR_\lambda][\cR_\mu] = [\cR_{\lambda+\mu}]$. Therefore
    $[\cR_\lambda]$ as above for each generalized chamber, together
    with generators of $\CC[\ft]$ generate $\gr\cA$.
\end{proof}

\subsection{\texorpdfstring{$\SL(2)$}{SL(2)} and \texorpdfstring{$\PGL(2)$}{PGL(2)} cases}

We determine $\mathcal M_C$ when $G=\SL(2)$ or $\PGL(2)$ in this
subsection.\footnote{This result was taught by Amihay Hanany as a
  result of physical intuition. The third named author thanks him for
  his explanation.}

\begin{Lemma}\label{lem:SL2}
    \textup{(1)} Suppose $G=\SL(2)$. Then $\mathcal M_C(G,\bN)$ is a
    hypersurface in $\CC^3$ of the form $\xi^2 = \delta\eta^2
    -\delta^{N-1}$ \textup($N\ge 1$\textup) or $\xi^2=\delta\eta^2 +
    \eta$ \textup($N=0$\textup), where
    $N=\sum_\mu|\langle\chi,\la_0\rangle|\dim\bN(\chi)/2$. Here
    $\la_0$ is the generator of the coweight lattice, which is
    dominant.

    \textup{(2)} Suppose $G=\PGL(2)$. Then $\mathcal M_C(G,\bN)$ is
    also $\xi^2 = \delta\eta^2 - \delta^{N-1}$ where $N = \sum_\mu
    |\langle\chi,\la_0\rangle|\dim \bN(\chi)/2+1$. Furthermore we have
    $\cR_{\le\la_0} = \cR_{\la_0}$ in this case, and generators are given by
    $\eta = [\cR_{\la_0}]$, $\xi = t[\cR_{\la_0}]$, $\delta = t^2 \in
    H^4_{\CC^\times}(\mathrm{pt})^{\{\pm 1\}}$, up to multiplicative
    constants, where $H^*_{\CC^\times}(\mathrm{pt}) = \CC[t]$. Here
    $\la_0$ is again the generator of the coweight lattice, which is
    dominant. \textup(It is the half of $\la_0$ in \textup{(1)}.\textup).
\end{Lemma}

\begin{proof}
    Let us first consider both $G=\SL(2)$ and $\PGL(2)$ together.

    We use $\gr\cA$.

    Let us identify the torus $T$ with $\CC^\times$ and the coweight
    lattice with $\ZZ$ so that dominant coweights are $\ZZ_{\ge
      0}$. (Hence $\la_0$ in the statement is $1$.) 
    Since $\la$,$\mu\in\ZZ_{> 0}$ are equal up to multiple of
    $\QQ_{>0}$, $\langle\chi,\la\rangle$ and $\langle\chi,\mu\rangle$
    always have the same sign for any weight $\chi$. Therefore
    $a_{\la,\mu}$ in \propref{prop:degen} is $1$. Moreover $W_\la = \{
    \pm 1\}$ if $\la=0$, $=\{1\}$ if $\la > 0$. Therefore $\gr\cA$ is
    generated by $\delta = t^2\in H^4_{\CC^\times}(\mathrm{pt})^{\{\pm
      1\}}$, $\eta = [\cR_1]$, $\xi = t[\cR_1]$, where $t$ is a
    generator of $\CC[\ft]\cong
    H^*_{\CC^\times}(\mathrm{pt})$. Moreover the relation is
    \begin{equation}\label{eq:53}
        \xi^2 = (t[\cR_1])^2 = t^2 [\cR_2] = \delta \eta^2.
    \end{equation}
    The singularities of $\Spec(\gr\cA)$ is at $\xi = \eta = 0$ and
    arbitrary $\delta$. (In particular $\Spec(\gr\cA)$ is \emph{not}
    normal.)

    Since $\gr\cA$ is a degeneration of $\cA$, $\delta$, $\eta$, $\xi$
    are generators of $\cA$, hence $\mathcal M_C$ is also a
    hypersurface in $\CC^3$.

    Note that \eqref{eq:53} is true modulo $H^{G_\cO}_*(\cR_{\le 1})$
    and the defining equation must be homogeneous with $\deg\xi =
    \deg\eta + 1$, $\deg\delta = 2$. The only possibility is
    \begin{equation*}
        \xi^2 - \delta \eta^2 = a \delta^m 
        \text{$+ b \delta^n\xi$ or $+b\delta^n\eta$}
    \end{equation*}
    for some $m$, $n\in\ZZ_{\ge 0}$, $a$, $b\in\CC$.
    ($m = \deg \eta + 1$, $n = (\deg \eta + 1)/2$ or $(\deg\eta+2)/2$.)
    
    \begin{NB}
    The differential of LHS - RHS is
    \begin{equation*}
        \begin{split}
            & 2\xi d\xi - (\eta^2 + a m \delta^{m-1}) d\delta 
            - 2 \delta \eta d\eta
        \text{$+ (nb \delta^{n-1}\xi d\delta + b\delta^n d\xi)$
          or
          $+ (nb \delta^{n-1}\eta d\delta + b\delta^n d\eta)$}
\\
    =\; &
    \begin{cases}
    (2\xi + b\delta^n) d\xi 
    - (\eta^2 + am\delta^{m-1} - nb\delta^{n-1}\xi) d\delta
    - 2\delta\eta d\eta,
    \\
    2\xi d\xi 
    - (\eta^2 + am\delta^{m-1} - nb\delta^{n-1}\eta) d\delta
    - (2\delta\eta - b\delta^n) d\eta.
    \end{cases}
        \end{split}
    \end{equation*}

    Let us look at points where the differential vanishes.

    We already know that $\mathcal M_C$ is nonsingular for $\delta\neq
    0$ by \corref{cor:coordinate} as $0\neq
    \delta\in\ft^\circ$. Therefore we may assume $\delta = 0$. 
    \begin{NB2}
        Therefore $\xi^2 = a\delta_{m0} \text{$+ b\delta_{n0}\xi$ or
          $+b\delta_{n0}\eta$}$.
    \end{NB2}%
    From the coefficient of $d\xi$, we find $\xi = 0$ or $-b/2$ if the
    first case with $n=0$.
    \begin{NB2}
        ($\deg\eta = -1$ or $-2$.)
    \end{NB2}%
    Next look at the coefficient of $d\delta$, $\eta$ must be a
    solution of a quadratic equation, therefore we have at most two
    possibilities for $\xi$. Therefore $\mathcal M_C(G,\bN)$ has at
    most two isolated singularities.
    \end{NB}

    Consider the first case $\xi^2 - \delta\eta^2 = a\delta^m +
    b\delta^n\xi$. We have $m=2n$. We set $\xi' \defeq \xi -
    b\delta^n/2$. Then $(\xi')^2 - \delta\eta^2 = (a -
    b^2/4)\delta^m$. Note that $a=b^2/4$ is not possible, as we know
    that $\mathcal M_C$ is nonsingular for $\delta\neq 0$ by
    \corref{cor:coordinate} as $0\neq \delta\in\ft^\circ$. Therefore
    by rescaling $\delta$ and $\eta$, we get the equation $(\xi')^2 -
    \delta\eta^2 = \delta^m$.

    The other case $\xi^2 - \delta\eta^2 = a\delta^m + b\delta^n\eta$
    is the same if $n\neq 0$. If $n=0$, $m=-1$. Therefore $a\delta^m$
    cannot appear. Hence we have $\xi^2 - \delta\eta^2 = \eta$ by
    rescaling.

    \begin{NB}
        $N = m + 1 = \deg \eta + 2$. This extra $2$ is equal to the
        dimension of the base of $\cR_{1}$, i.e., $\Gr_{1}$.
    \end{NB}%

    We now suppose $G=\PGL(2)$. In this case $\Gr_G^{\la_0}$ is a closed
    $G_\cO$-orbit for $\la_0=1$. Hence \propref{prop:useful} is
    applicable for $\eta$, $\xi$. We have $\Gr_G^{\la_0}\cong \CP^1$ and
    $W\la_0 = \{ \pm \la_0\}$ is the north and south poles. We have
  \begin{equation*}
      (\iota_*)^{-1}(\eta) = \frac{r^{\la_0} - r^{-\la_0}}{t}, \qquad
      (\iota_*)^{-1}(\xi) = {r^{\la_0} + r^{-\la_0}}.
  \end{equation*}
  \begin{NB}
      $\xi = t[\cR_1]$ is identified with $c_1(\shfO(2))\cap[\CP^1]$,
      as the standard homomorphism $B \ni \{ \left[
        \begin{smallmatrix}
            x & a \\ 0 & y
        \end{smallmatrix}\right] \mid x,y\in\CC^\times, a\in\CC \} /\CC^\times 
      \mapsto xy^{-1}\in \CC^\times$ gives $\shfO(2)$. Therefore
      $\int_{\Gr^\la_G} \xi = 2$. This is compatible with the above
      formula, as $\int_{\Gr^\la_G} \iota_*(r^{\pm \la}) = 1$.
  \end{NB}%
  Therefore
  \begin{equation*}
      (\iota_*)^{-1}(\xi^2 - \delta \eta^2)
      = 4 r^{\la_0} r^{-\la_0} = 4 \prod_\mu 
      (\langle \mu,\la_0\rangle t)^{|\langle \mu,\la_0\rangle|},
  \end{equation*}
  where the last equality is by \eqref{relation}. Note that the power
  of $t$ is even as weights appear in pairs $\mu$, $-\mu$. Thus the
  right hand side is a power of $\delta$.
  After rescaling $\xi$, $\delta$, $\eta$, we get an equation $\xi^2 -
  \delta\eta^2 = \delta^m$.

  Finally the formula of $N$ is given by computing the degree of $\eta
  = [\cR_1]$. The rank of $\cT/\cR$ over $\Gr_G^1$ is given by
  \lemref{lem:Rla}, and $\dim \Gr_G^1 = 1$ for $\PGL(2)$, $2$ for
  $\SL(2)$. (For $\PGL(2)$, the above computation also gives the
  formula of $m=N-1$.)
\end{proof}

\begin{NB}
    Let us consider the trivial case $(G,\bN) = (\PGL(2),0)$. In this
    case we have $\xi^2 - \delta\eta^2 = 4$. This is different from
    \cite[\S3.10]{MR2135527}. It is likely that choices of generators
    are different.
\end{NB}%

\begin{NB}
    Let us consider the case $(G,\bN)=(\PGL(2),\mathfrak{pgl}(2))$. We
    have a decomposition $\mathfrak{pgl}(2)= \CC e\oplus\CC h\oplus\CC
    f$, which can be regarded as the root space decomposition. If
    $\la$ is a generator as above, we have $\langle e,\la\rangle = 1$,
    $\langle f,\la\rangle = -1$, $\langle h,\la\rangle = 0$. Therefore
    $\xi^2 - \delta\eta^2 = -4\delta$.

    On the other hand, we have $\mathcal M_C = (\ft\times T^\vee)/W$
    by \subsecref{subsec:prev}\ref{subsub:adjoint}. For $G=\PGL(2)$,
    let $t$, $a^\pm$ be coordinates of $\ft$ and $T^\vee$
    respectively. The Weyl group action is $t\leftrightarrow -t$,
    $a\leftrightarrow a^{-1}$. Therefore $\delta = t^2$, $\eta = (a +
    a^{-1})$, $\xi = t(a-a^{-1})$ are generators with relations
    $\xi^2 - \delta\eta^2 = -4\delta$.
\end{NB}%

\begin{NB}
    Let us check the compatibility between the formula of $N$ for
    $\SL(2)$ and $\PGL(2)$. Since $\la_{\SL(2)} = 2\la_{\PGL(2)}$, we
    have $N_{\SL(2)} = 2N_{\PGL(2)} - 2$. Suppose $\mathcal
    M_C(\PGL(2),\bN) = \{ x^2 = \delta z^2 -
    \delta^{N_{\PGL(2)}-1}\}$. We have $\mathcal M_C(\SL(2),\bN) =
    \mathcal M_C(\PGL(2),\bN)/\ZZ/2$ by
    \subsecref{sec:triv-properties}\ref{item:finite_quotient}. The
    $\ZZ/2$ action is $-1$ on $x$, $z$ and $1$ on $\delta$ by its
    definition in \subsecref{sec:grading}. Therefore the invariant
    ring is generated by $\eta = z^2$, $\delta$, $\xi = zx$ and $x^2 =
    \delta z^2 - \delta^{N_{\PGL(2)}-1} = \delta\eta -
    \delta^{N_{\PGL(2)}-1}$. The defining equation is $\xi^2 = (zx)^2 
    = \eta (\delta \eta - \delta^{N_{\PGL(2)}-1})$. We set
    $\eta' = \eta - \delta^{N_{\PGL(2)}-2}/2$. Then
    $\xi^2 = \delta (\eta')^2 - \delta^{2N_{\PGL(2)}-3}/4$. Since
    $2N_{\PGL(2)} - 3 = N_{\SL(2)} - 1$, the formula is correct.
\end{NB}%

\begin{NB}
  Let us continue the calculation. Since $\xi = zx$, $\eta = z^2$, we
  have
  \begin{equation*}
    \begin{split}
      & (\iota_*)^{-1}(\xi) = \frac{r^{2\lambda_0} - r^{-2\lambda_0}}t,\\
      & (\iota_*)^{-1}(\eta) = \frac{r^{2\lambda_0} + r^{-2\lambda_0}}{t^2} - \frac{2 r^{\lambda_0} r^{-\lambda_0}}{t^2}.
    \end{split}
  \end{equation*}
  Hence $(\iota_*)^{-1}(\eta')$ is
  $\frac{r^{2\lambda_0} + r^{-2\lambda_0}}{t^2} +
  a\delta^{N_{\PGL(2)}-2}$ for a constant $a$. Since we have
  $\xi^2 \equiv \delta(\eta')^2$ modulo $\CC[\delta]$, we must have
  $a=0$. Then
  \begin{equation*}
    (\iota_*)^{-1}(\xi^2 - \delta(\eta')^2)
    = -\frac{4 r^{2\lambda_0} r^{-2\lambda_0}}{t^2} \in \CC \delta^{2(N_{\PGL_2}-1) - 1}
    = \CC \delta^{N_{\SL(2)}-1}.
  \end{equation*}
  Thus the power is also correct.
\end{NB}%

\begin{Remark}\label{rem:Dnsing}
    The hypersurface $\xi^2 = \delta\eta^2 - \delta^{N-1}$ is a simple
    singularity of type $D_N$ if $N\ge 4$. The cases $\xi^2 =
    \delta\eta^2 + \eta$ (`type $D_0$') and $N=1$ are nonsingular. $N=2$ has
    two $A_1$ singularities at $\xi=\delta=0$, $\eta=\pm 1$. $N=3$ is
    isomorphic to the $A_3$-singularity.  (The natural hyper-K\"ahler
    metrics on $D_3$ and $A_3$ are different.)  Compare with
    Example~\ref{ex:typeD}.
\end{Remark}

\subsection{Normality}

\begin{Proposition}\label{prop:normality}
    $\mathcal M_C(G,\bN)$ is a normal variety.
\end{Proposition}

The proof occupies this subsection.
We start with the following lemma.

\begin{Lemma}
    Let $X$ be an affine scheme of finite type over a field $\Bbbk$
    and let $U$ be an open subset of $X$ such that the complement has
    dimension $\geq 2$. Assume that a) $U$ is normal, b) any regular
    function on $U$ extends to $X$. Then $X$ is normal.
\end{Lemma}

\begin{proof}
    We use Serre's criterion $(R_1)$ and $(S_2)$. By a) $U$ satisfies both $(R_1)$ and $(S_2)$.

    By $(R_1)$ for $U$ and $\codim(X\setminus U) \geq 2$, $X$ satisfies the
    condition $(R_1)$. The condition $(S_2)$ is guaranteed by $(S_2)$ for $U$ and b),
    as $(S_2)$ is equivalent to that any regular function on $V$
    extends to $X$ for arbitrary open subscheme $V\subset X$ with
    $\codim (X\setminus V)\ge 2$.
    \begin{NB}
        Consider $\mathcal H^i_{X\setminus V}(X,\shfO_X)$ for $i=0,1$.
    \end{NB}%
\end{proof}

\begin{NB}
Sasha's message on Sep.~12:

\begin{verbatim}
let X be an affine scheme
 of finite type over a field k and let U be an open subset of 
X such that the complement has dimension \geq 2. Assume that
a) U is normal
b) Any regular function on U extends to X
Then X is normal.

If we assume the lemma, then we can take X={\mathcal M}_C and take U
to be the preimage of the corresponding
open subset in \mathfrak t/W (take all elements which lie in at most
one generalized root hyperplane).
Then normality of U probably reduces to SL(2).
\end{verbatim}
\end{NB}

We take $X = \mathcal M_C(G,\bN)$, $U = \intsys^{-1}(\ft^\bullet/W)$. The
condition b) is satisfied thanks to the first part of the proof of
\thmref{prop:flat}. We also know $\intsys$ is flat (\lemref{lem:flat})
and $X$ is irreducible (\corref{cor:integral}). Therefore fibers of
$\intsys$ all have the same dimension, hence the complement of $U$ has
codimension $2$.
\begin{NB}
    See \cite[III.\@ Cor.\@ 9.6]{MR0463157}.
\end{NB}%
Therefore it is enough to show that $U$ is normal.

\begin{NB}
    Added on Sep.\@ 27:
    
    Probably it is enough to consider the $T_\cO$-equivariant
    cohomology groups, but the following is enough.
\end{NB}%
By the localization theorem in equivariant cohomology groups, we have
an isomorphism
\begin{equation*}
    H^{Z_G(t)_\cO}_*(\cR_{Z_G(t),\bN^t})\otimes_{H^*_{Z_G(t)}(\mathrm{pt})}
    \CC[\ft]_t^{W_t}
        \cong 
    H^{G_\cO}_*(\cR_{G,\bN})\otimes_{H^*_G(\mathrm{pt})} \CC[\ft]_t^{W_t}
\end{equation*}
for $t\in\ft$. (Recall \secref{sec:abel}, especially \eqref{eq:22}.)
Here $W_t$ is the Weyl group of $Z_G(t)$, which is the subgroup of $W$
fixing $t$. Therefore it is enough to show the normality of $\mathcal
M_C(Z_G(t),\bN^t)$ for each $t\in\ft^\bullet$.
If $t\in\ft^\circ$, we have $\mathcal M_C(Z_G(t),\bN^t)\cong T^*
T^\vee$. So it is smooth.
Therefore we may assume $t\in\ft^\bullet\setminus\ft^\circ$.
By \subsecref{sec:fixed-points}, there is a unique generalized root
$\alpha$ such that $\langle t,\alpha\rangle = 0$.

As remarked at the beginning of \secref{sec:abel}, $Z_G(t)$ has
semisimple rank at most $1$, and we know the answer for $\SL(2)$,
$\PGL(2)$ and a torus $T$. The remaining task is a reduction to these
cases.

\begin{proof}[Proof of \propref{prop:normality}]
    The same (or simpler) argument as in \lemref{lem:SL2} shows that
    $\mathcal M_C(G,\bN)$ for $G=\CC^\times$ is a hypersurface
    $xy=w^N$ for some $N=0,1,2,\dots$.
    \begin{NB}
        Here $w$ is a generator of $\CC[\ft] =
        H^*_{\CC^\times}(\mathrm{pt})$, $x$, $y$ are $[\cR_1]$,
        $[\cR_{-1}]$ or opposite.
    \end{NB}%

    Slightly more generally, suppose that $\alpha$ is a generalized
    root of type (I). Then $Z_G(t) = T$, $\bN^t =
    \bN^T\oplus\bigoplus_{m\in\ZZ}\bN(m\alpha)$. Let us consider
    $\cA[Z_G(t),\bN^t] = \cA[T,\bigoplus_{m\in\ZZ}\bN(m\alpha)]$. Let
    $Y(T)$ be the coweight lattice of $T$. Since $\alpha$ is a weight of
    $T$, we have a homomorphism $\Phi\colon Y(T)\to \ZZ$ given by the
    pairing with $\alpha$.
    \begin{NB}
        $\Phi$ could be non surjective, e.g., $\mu(t) = t^2$,
        $\mu\colon\CC^\times\to\CC^\times$.
    \end{NB}%
    Let us consider the kernel and image
    \begin{equation*}
        0 \to \Ker\Phi \to Y(T)\to \Ima \Phi \to 0.
    \end{equation*}
    Since $\Ker\Phi$ and $\Ima\Phi$ are both free,
    \begin{NB}
        as both are submodules of free modules,
    \end{NB}%
    this exact sequence splits, hence $Y(T) \cong
    \Ker\Phi\oplus\Ima\Phi$. Let us take $\la_g\in \Ima\Phi\cong\ZZ$, a
    generator. We suppose $\langle\la_g,\alpha\rangle > 0$. Then by
    \thmref{abel}, $\cA[T,\bigoplus_{m\in\ZZ}\bN(m\alpha)]$ is
    generated by $\CC[\ft]$, $r^\la$ ($\la\in\Ker\Phi$), $r^{\pm
      \la_g}$ with the relation
    \begin{equation*}
        r^\la r^{\mu} = r^{\la+\mu}
        \quad(\la,\mu\in\Ker\Phi), \quad
        r^{\la_g} r^{-\la_g} = \prod_{m\in \ZZ} (m\alpha)^{|m|
          \langle\la_g,\alpha\rangle\dim \bN(m\alpha)},
    \end{equation*}
    where $\alpha$ is regarded as an element in $\ft^*$. Therefore
    $\mathcal M_C(Z_G(t),\bN^t)$ is a product of $T^* T'$ for one
    dimensional lower torus $T'$ and a simple singularity $xy = w^N$
    of type $A$ for some $N=1,2,\dots$.

    If $\alpha$ is a type (II) generalized root, $Z_G(t)$ has
    semisimple rank $1$. Let us denote by $Z$ the neutral connected component
of the center of $Z_G(t)$. It acts trivially on $\bN^t$ since its Lie algebra
$\on{Lie}(Z)\subset\ft$ is the kernel of $\alpha$ in the Cartan subalgebra, 
but the weights of $\bN^t$
are multiples of $\alpha$. Hence the action of $Z_G(t)$ on $\bN^t$ factors 
through $H:=Z_G(t)/Z$. Note that $H$ is isomorphic to $\SL(2)$ or $\PGL(2)$ 
(recall that $Z_G(t)$ is connected).
Let $D\subset Z_G(t)$ be the derived subgroup, and let $T':=Z_G(t)/D$ be the
quotient torus. The kernel of the diagonal morphism $Z_G(t)\to T'\times H$
is a finite abelian subgroup $\Gamma$ (in fact, $\Gamma$ is either $\{\pm1\}$
or trivial), so that we have an exact sequence
$1\to\Gamma\to Z_G(t)\to T'\times H\to1$. As we have just seen, the 
representation of $Z_G(t)$ in $\bN^t$ factors through its quotient 
$Z_G(t)\to T'\times H\to H$. Hence $H^{Z_G(t)_\shfO}(\CR_{Z_G(t),\bN^t})=
H^{(T'\times H)_\shfO}(\CR_{T'\times H,\bN^t})^{\Gamma^\wedge}=\CC[T^*T^{\prime\vee}\times{\mathcal M}_C(H,\bN^t)]^{\Gamma^\wedge}$
by \secref{sec:triv-properties}\ref{item:product},\ref{item:finite_quotient}.
Since ${\mathcal M}_C(H,\bN^t)$
is normal according to~\lemref{lem:SL2}, $(T^*T^{\prime\vee}\times{\mathcal M}_C(H,\bN^t))/\Gamma^\wedge$
is normal as well.
\end{proof}

\subsection{Adjoint matters}

The purpose of this subsection is to prove the following result,
mentioned in \subsecref{subsec:prev}\ref{subsub:adjoint}.

\begin{Proposition}
\label{prop:ad2}
For a reductive group $G$ and its adjoint representation $\fg$, the
birational isomorphism $\bz^*\iota_*^{-1}\colon
\CM_C(G,\fg)|_{\Phi^{-1}(\ft^\circ/W)}\simeq (\ft^\circ\times
T^\vee)/W$ of \corref{cor:coordinate} extends to a biregular
isomorphism $\CM_C(G,\fg)\iso (\ft\times T^\vee)/W$.
\end{Proposition}

\begin{proof}
    We use the criterion in \thmref{prop:flat}.\footnote{The third
      named author thanks Ryosuke Kodera for his suggestion to give
      this proof.} In this case generalized roots are nothing but
    usual roots. For $t\in\ft^\bullet\setminus\ft^\circ$, there is a
    single root $\alpha$ with $\langle\alpha,t\rangle = 0$. By the
    commutativity of \eqref{eq:22}, $\Xi^t$ is also given by
    \corref{cor:coordinate}, but for $Z_G(t)$. Therefore it is enough
    to check the assertion for $Z_G(t)$. By the argument in the last
    part of the proof in \propref{prop:normality}, we can replace
    $Z_G(t)$ by $\PGL(2)$. ($\SL(2)$ can be replaced by $\PGL(2)$, as
    we are considering the adjoint representation.)
    \begin{NB}
        Let $T_{Z_G(t)}$ (resp.\ $T_H$) denote a maximal torus of
        $Z_G(t)$ (resp. $H$) in the proof. We can make them so that
        $1\to\Gamma\to T_{Z_G(t)}\to T'\times T_H\to 1$. Then we have
        the induced exact sequence $1\to \pi_1(T_{Z_G(t)})\to
        \pi_1(T'\times T_H) \to \pi_0(\Gamma)=\Gamma \to 1$.
        Taking the Pontryagin duals, we have $1\to \Gamma^\wedge \to
        T^{\prime\vee}\times T_H^\vee \to T_{Z_G(t)}^\vee \to 1$.
    \end{NB}%

    For $G=\PGL(2)$, let us use the computation in the proof of
    \lemref{lem:SL2}(2). Thanks to \subsecref{sec:chang-repr}, we have
    $\bz^*(\iota_*)^{-1}(\eta) = r^{\la_0} + r^{-\la_0}$,
    $\bz^*(\iota_*)^{-1}(\xi) = t(r^{\la_0} - r^{-\la_0})$, where $r^{\pm
      \la_0}$ is the fundamental class of the point $\pm\la_0$ ($=\pm 1$)
    in $\Gr_T$. Now we see that $\bz^*\iota_*^{-1}$ is an isomorphism,
    as $r^{\la_0}$ and $t$ are coordinates of $T^\vee$ and $\ft$
    respectively.
\end{proof}

\subsection{Symplectic form}

\begin{Proposition}\label{prop:symplectic}
    The Poisson structure is symplectic on the smooth locus of
    $\mathcal M_C$.
\end{Proposition}

\begin{proof}
    By \corref{cor:coordinate}, we already know the assertion
    over $\intsys^{-1}(\ft^\circ/W)$. By Hartogs theorem, it is enough
    to check that the symplectic form extends and is nondegenerate up
    to codimension $2$. As in the proof of \propref{prop:normality},
    we check it for $\intsys^{-1}(\ft^\bullet/W)$, then it is enough to
    assume $G=\CC^\times$, $\SL(2)$ or $\PGL(2)$.

    Let us consider $G=\CC^\times$. Then $\mathcal M_C$ is a
    hypersurface $xy=w^N$ in $\CC^3$ ($N=0,1,\dots$). The birational
    isomorphism $\bz^*(\iota_*)^{-1}\colon \mathcal M_C
    \overset{\approx}{\dashrightarrow} T^*\CC^\times =
    \CC\times\CC^\times$ in \corref{cor:coordinate} is given by
    $(x,y,w)\mapsto (w,x)$ defined over $w\neq 0$. Moreover the
    symplectic structure on $T^*\CC^\times$ is $x^{-1}dx\wedge dw$. We
    can rewrite it as $- y^{-1}dy\wedge dw$.
\begin{NB}
    $xdy + y dx = Nw^{N-1} dw$.
\end{NB}%
Hence it is a well-defined symplectic form over $\{x\neq 0\}\cup
\{y\neq 0\} = \mathcal M_C\setminus \{ x=y=w=0\}$. If $N=0$, $\mathcal
M_C = \CC\times\CC^*$, i.e., it is well-defined everywhere.
If $N=1$, we have $w=xy$, hence $x^{-1} dx \wedge dw= dx\wedge
dy$. Hence it is also well-defined and non-degenerate over the whole
$\mathcal M_C$. (It is also a consequence of Hartogs theorem.)
For general $N$, it is a symplectic form on $\mathcal M_C \cong
\CC^2/(\ZZ/N)$, descending from the standard one on $\CC^2$,
divided by $N$.
\begin{NB}
    Let $(z_1,z_2)\in\CC^2$, hence $x=z_1^N$, $y=z_2^N$,
    $w=z_1z_2$. Then $x^{-1} dx\wedge dw = N dz_1/z_1 \wedge (z_1 dz_2
    + z_2 dz_1) = N dz_1\wedge dz_2$.
\end{NB}%

Let us next consider the case $G=\PGL(2)$. Let $\xi = t[\cR_\la]$,
$\eta = [\cR_\la]$ as in the proof of \lemref{lem:SL2}.
We use $\bz^*(\iota_*)^{-1}\colon \mathcal
M_C\overset{\approx}{\dashrightarrow} (\CC\times\CC^\times)/(\ZZ/2)$.
Let $t$, $a^\pm$ be coordinates of $\CC$ and $\CC^\times$ respectively. The Weyl group action is $t\leftrightarrow -t$, $a\leftrightarrow a^{-1}$.
First note that $\delta$ is sent to $t^2$.
Recall $(\iota_*)^{-1}(\eta) = t^{-1}(r^\la - r^{-\la})$,
$(\iota_*)^{-1}(\xi) = r^\la + r^{-\la}$ in the proof of
\lemref{lem:SL2}(2). Thanks to \subsecref{sec:chang-repr}, we have
$\bz^* (\iota_*)^{-1}(\eta) = t^{N-2}(a \pm a^{-1})$,
$\bz^*(\iota_*)^{-1}(\xi) = t^{N-1}(a \mp a^{-1})$ up to the same
multiplicative constant.
\begin{NB}
    For example, $N=1$ if $\bN=0$. Then $\bz^*=\id$.
\end{NB}%
Here $\pm$ is determined according to the parity of $N$
so that those are invariant under the Weyl group action.
\begin{NB}
    $\bz^*(\iota_*)^{-1}(\xi^2) = t^{2N-2} (a^2 + a^{-2} \mp 2)
    = t^2 \bz^*(\iota_*)^{-1}(\eta^2) \mp 4 t^{2N-2}
    = \delta \bz^*(\iota_*)^{-1}(\eta^2) \mp 4 \delta^{N-1}$.
\end{NB}%
Let us rescale $\xi$, $\eta$ so that these formulas are true without
ambiguity. (Hence the defining equation is $\xi^2 = \delta\eta^2 \mp
4\delta^{N-1}$.)

The standard symplectic form on $\CC\times\CC^\times$ is
$a^{-1}da\wedge dt$, and descends to
$\CC\times\CC^\times/(\ZZ/2)$. Its pull-back is 
\begin{multline*}
    \frac12 \xi^{-1} d\eta\wedge d\delta \quad \text{over $\xi\neq 0$},\qquad
    \frac12 \delta^{-1}\eta^{-1} d\xi\wedge d\delta \quad
    \text{over $\delta\eta\neq 0$}, \qquad
\\
    \left( \eta^2 \mp 4(N-1)\delta^{N-2}\right)^{-1}d\eta\wedge d\xi \quad
    \text{over $\eta^2\mp 4(N-1)\delta^{N-2} \neq 0$}.
\end{multline*}
\begin{NB}
Since
    \begin{equation*}
        \begin{split}
            & d\delta = 2t dt, \\
            & d\eta = t^{N-2}\left( (N-2)(a\pm a^{-1})t^{-1} dt +
              (a\mp a^{-1})a^{-1}da\right),\\
            & d\xi = t^{N-1}\left( (N-1)(a\mp a^{-1}) t^{-1}dt +
              (a\pm a^{-1})a^{-1}da\right),
        \end{split}
    \end{equation*}
we have
\begin{equation*}
    \begin{split}
        d\eta \wedge d\delta &= 2t^{N-1} (a\mp a^{-1}) a^{-1} da
        \wedge dt = 2\xi a^{-1} da \wedge dt,
        \\
        d\xi\wedge d\delta &= 2 t^N (a\pm a^{-1}) a^{-1}da \wedge dt
        = 2 \delta \eta a^{-1} da \wedge dt,
        \\
        d\eta\wedge d\xi & = t^{2N-3}\left\{(N-2)(a\pm a^{-1})^2 -
          (N-1)(a\mp a^{-1})^2\right\}t^{-1}dt\wedge a^{-1}da
        \\
        &= - \left( (N-2) \eta^2 -
          (N-1)\delta^{-1}\xi^2\right) a^{-1}da\wedge dt
        \\
        & = \left( \eta^2 \mp 4(N-1)\delta^{N-2}\right)
        a^{-1}da\wedge dt.
    \end{split}
\end{equation*}
\end{NB}%
Therefore the pull-back is well-defined and nondegenerate over
$\mathcal M_C\setminus \{ \xi = \delta\eta = \eta^2 \mp
4(N-1)\delta^{N-2} = 0\}$. The complement is empty if $N=1$; $\xi =
\delta = 0$, $\eta = 2$, $-2$ if $N=2$; and $\xi=\delta=\eta=0$ if
$N\ge 3$. It is exactly the singular locus of $\mathcal M_C$. In fact,
it is the descent of the standard symplectic form on $\CC^2$ up to
constant if $N\ge 3$.

For $G=\SL(2)$, lifts of $[\cR_\la]$, $t[\cR_\la]$ are well-defined
only up to $\CC\delta^m$ for appropriate $m$. (In fact, either of two
is well-defined by the degree reason.) By adjusting the ambiguity, we
take $\eta$, $\xi$ so that $\bz^* (\iota_*)^{-1}(\eta) = t^{N-2}(a \pm
a^{-1})$, $\bz^*(\iota_*)^{-1}(\xi) = t^{N-1}(a \mp a^{-1})$ are
true.
\begin{NB}
    Other possibility is $\bz^* (\iota_*)^{-1}(\eta) = t^{N-2}(a \pm
    a^{-1}) + c_1 t^{N-2}$, $\bz^*(\iota_*)^{-1}(\xi) = t^{N-1}(a \mp
    a^{-1}) + c_2 t^{N-1}$ for constants $c_1$, $c_2$. Since they must
    be invariant under the Weyl group, $c_1 = 0$ if $N$ is odd, $c_2 =
    0$ otherwise. As $c_1 = 0$ if $N=1$, it means that we do not have
    negative power of $t$, and hence it can be actually absorbed.
\end{NB}%
Then the remaining argument is the same. The case $N=0$ is
exceptional. In this case, there is no ambiguity of $\CC\delta^m$ by
degree reason. We determine $\eta$ as $\bz^* (\iota_*)^{-1}(\eta) =
t^{-2}(a + a^{-1} - 2)$ so that the defining equation $\xi^2
\begin{NB}
    = t^{-2}(a^2 - 2 + a^{-2}) = \delta (\eta + 2\delta^{-1})^2 - 4
    \delta^{-1}
\end{NB}%
= \delta \eta^2 + 4\eta$ has no negative powers of $\delta$. 
Then the symplectic form is $\xi^{-1}
d\eta\wedge d\delta/2$ over $\xi\neq 0$, $(\delta\eta+2)^{-1}d\xi\wedge
d\delta/2$ over $\delta\eta+2\neq 0$ and $\eta^{-2} d\eta\wedge d\xi$ over $\eta\neq 0$.
\begin{NB}
    We have
    \begin{equation*}
        \begin{split}
            & d\delta = 2t dt, \\
            & d\eta = t^{-2} \left( -2(a+a^{-1}-2) t^{-1}dt + (a -
              a^{-1})
              a^{-1}da\right),\\
            & d\xi = t^{-1}\left( -(a - a^{-1}) t^{-1}dt + (a +
              a^{-1})a^{-1}da\right),
        \end{split}
    \end{equation*}
    Hence
    \begin{equation*}
        \begin{split}
        d\eta\wedge d\delta & = 2t^{-1} (a - a^{-1}) a^{-1} da\wedge dt
        = 2\xi a^{-1}da\wedge dt,\\
        d\xi\wedge d\delta &= 2 (a + a^{-1}) a^{-1}da \wedge dt
        = 2 (\delta \eta + 2 )a^{-1} da \wedge dt,
        \\
        d\eta\wedge d\xi & = t^{-4} \left\{ - (a - a^{-1})^2 
          + 2(a + a^{-1} - 2)(a+a^{-1})
          \right\} a^{-1} da\wedge dt \\
        &= t^{-4}  \left\{
          (a+a^{-1}-2)^2
                    \right\} a^{-1} da\wedge dt 
                    = \eta^2 a^{-1} da\wedge dt.
        \end{split}
    \end{equation*}
\end{NB}%
Since $\xi=\eta=0$ and $\delta\eta+2=0$ cannot happen simultaneously, we
conclude that the symplectic form is well-defined and non-degenerate
on the whole $\mathcal M_C$. (In this case $\mathcal M_C$ is nonsingular.)
\end{proof}

\subsection{Changing a summand to its dual}\label{subsec:change_dual2}

Let us combine \subsecref{subsec:change_dual} with \thmref{prop:flat}
to show that $\cM_C(G,\bN)$ depends only on $\bM=\bN\oplus\bN^*$.
Let $\underline{\bN}$ be a summand of $\bN$, and let $\bN'$ be the
representation of $G$ obtained by replacing $\underline{\bN}$ by its dual
$\underline{\bN}^*$.

Let us define an automorphism of
$H^{T_\cO}_*(\Gr_T)\cong\CC[\ft\times T^\vee]$ by
\begin{equation*}
    r^\lambda \mapsto (-1)^{\sum_\chi\max(\chi(\lambda),0)
      \dim \underline{\bN}(\chi)} r^\lambda,
\end{equation*}
where $\underline{\bN}(\chi)$ is the weight $\chi$ subspace of
$\underline{\bN}$. This definition is based on the definition in
\subsecref{subsec:change_dual}. This is $W$-equivariant, hence induces
an automorphism of $\ft\times T^\vee /W$.
\begin{NB}
For $w\in W$ we have
\begin{equation*}
    r^{w\lambda} \mapsto 
    (-1)^{\sum_\chi\max(\chi(w\lambda),0)\dim\underline{\bN}(\chi)} r^{w\lambda}
    = (-1)^{\sum_\chi\max(w\chi(\lambda),0)\dim\underline{\bN}(w\chi)} r^{w\lambda}
    = (-1)^{\sum_\chi\max(\chi(\lambda),0)\dim\underline{\bN}(\chi)} r^{w\lambda}.
\end{equation*}
\end{NB}%
Composing with birational isomorphisms
$\cM_C(G,\bN) \dashrightarrow \ft\times T^\vee/W$,
$\cM_C(G,\bN') \dashrightarrow \ft\times T^\vee/W$ in
\corref{cor:coordinate}, we consider the above as a birational
isomorphism $\cM_C(G,\bN)\dashrightarrow \cM_C(G,\bN')$.
We claim that it extends to a biregular isomorphism
$\cM_C(G,\bN)\xrightarrow{\cong} \cM_C(G,\bN')$.
By \thmref{prop:flat} and a similar argument in the proof of
\propref{prop:normality}, we may assume $G$ is either torus, $\PGL(2)$
or $\SL(2)$.
\begin{NB}
  Let us check how to deal with a finite group $\Gamma$. It appears as
  $1\to\Gamma\to Z_G(t)\to T'\times H\to 1$ so that the representation
  $\bN^t$ factors through $H$. ($T'$ acts trivially.)  Let us denote
  the surjection $Z_G(t)\to T'\times H$ by $p$. In order to define
  birational isomorphisms for $\cM_C(Z_G(t),\bN^t)$ and
  $\cM_C(T'\times H,\bN^t)$, let us take a maximal torus $T$ of
  $Z_G(t)$ and its image $p(T)$ in $T'\times H$, and
  $\CC^\times\to p(\CC^\times)$ is a $\Gamma$-covering. Then
  $r^\lambda$'s for $H^{T_\cO}_*(\Gr_T)$ are those for
  $H^{p(T)_\cO}_*(\Gr_{p(T)})$ corresponding to
  $T^\vee\subset p(T)^\vee$. Since $\bN^t$ factors through $H$, the
  definition
  $r^\lambda \mapsto (-1)^{\sum_\chi\max(\chi(\lambda),0) \dim
    \underline{\bN}(\chi)} r^\lambda$ is unchanged, hence it is
  compatible with the quotient by $\Gamma^\vee$.
\end{NB}%

By \subsecref{subsec:change_dual} the assertion is true when $G$ is
torus.

Suppose $G = \PGL(2)$. We may assume $\underline{\bN} = S^k(\CC^2)$
with $k$ even. Then $\chi(\lambda_0)$ is $k/2$, $k/2-1$,\dots,
\begin{NB}
$0$, \dots,
\end{NB}%
$-k/2$.
\begin{NB}
  Note
  $\lambda_0(t) = \left[\begin{smallmatrix} t^{1/2} & 0 \\ 0 &
      t^{-1/2}\end{smallmatrix}\right]$.
\end{NB}%
Therefore $\sum_\chi \max(\chi(\pm\lambda^0),0)\dim\underline{\bN}(\chi)
= k(k+2)/8$.
\begin{NB}
  $l(l+1)/2$ if $k=2l$. $1$ if $l=1$, $3$ if $l=2$, $6$ if $l=3$, $10$ if $l=4$, \dots. $(4m+1)(2m+1)$ if $l=4m+1$, $(2m+1)(4m+3)$ if $l=4m+2$,
  $(4m+3)(2m+2)$ if $l=4m+3$, $2m(4m+1)$ if $l=4m$.
\end{NB}%
Since $(\iota_*)^{-1}\xi = r^{\lambda_0} + r^{-\lambda_0}$,
$(\iota_*)^{-1}\eta = (r^{\lambda_0} - r^{-\lambda_0})/t$,
$\delta = t^2$ (see the proof of \lemref{lem:SL2}),
$(\xi,\eta,\delta)$ is mapped to $(\pm\xi,\pm\eta,\delta)$, hence the
equation $\xi^2 = \delta\eta^2 - c\delta^{N-1}$ ($c\in\CC$) is preserved.

Suppose $G=\SL(2)$ and $\underline{\bN} = S^k(\CC^2)$. Then
$\chi(\lambda_0)$ is $k$, $k-2$,\dots, $-k$. Therefore
$\sum_\chi \max(\chi(\pm\lambda^0),0)\dim\underline{\bN}(\chi)
= k(k+2)/4$ if $k$ is even, $(k+1)^2/4$ if $k$ is odd.
\begin{NB}
  $1$ if $k=1$, $2$ if $k=2$, $4$ if $k=3$, $6$ if $k=4$,
  $5+3+1 = 9$ if $k=5$, $7+5+3+1 = 16$ if $k=7$, \dots
\end{NB}%
As in the proof of \propref{prop:symplectic}, we adjust $\xi$, $\eta$
by $\CC\delta^m$ so that $(\iota_*)^{-1}\xi = (r^{\lambda_0} - r^{-\lambda_0})/t$,
$(\iota_*)^{-1}\eta = (r^{\lambda_0} + r^{-\lambda_0})/t^2$.
\begin{NB}
  See {\bf NB} before \remref{rem:Dnsing}.
\end{NB}%
Thus the birational isomorphism is given by
$(\xi,\eta,\delta)\mapsto (\pm\xi,\pm\eta,\delta)$, hence the equation
$\xi^2 = \delta\eta^2 - c\delta^{N-1}$ ($c\in\CC$) is preserved.

\subsection{Another degeneration}
\label{sec:another}

We have a $\CC^\times$-action on $\cR$ induced from the dilatation on
$\bN$. Let us denote the variable for the equivariant cohomology of
this $\CC^\times$ by $\mathbf t$, i.e., $H^*_{\CC^\times}(\mathrm{pt})
= \CC[\mathbf t]$.
Let us consider the embedding $\bz^*\colon
H^{\CC^\times\times G_\cO\rtimes\CC^\times}_*(\cR)\to H^{\CC^\times\times G_\cO\rtimes\CC^\times}_*(\Gr_G)
\cong H^{G_\cO\rtimes\CC^\times}_*(\Gr_G)[\mathbf t]$ (where the second factor 
$\CC^\times$ stands for the loop rotation). We extend it to 
$H^{\CC^\times\times G_\cO\rtimes
  \CC^\times}_*(\cR)\otimes_{\CC[\mathbf t]}\CC[\mathbf t,\mathbf
t^{-1}] \to H^{G_\cO\rtimes\CC^\times}_*(\Gr_G)[\mathbf t,\mathbf t^{-1}]$.
Let $\cL\equiv \cL_{G,\bN}$ be the pull-back of the $\CC[\mathbf
t^{-1}]$-lattice $H^{G_\cO\rtimes\CC^\times}_*(\Gr_G)[\mathbf t^{-1}]$ by $\bz^*$. We
have the induced injective ring homomorphism
\begin{equation}\label{eq:55}
    \overline \bz^*\colon \cL/\mathbf t^{-1}\cL \to H^{G_\cO\rtimes\CC^\times}_*(\Gr_G).
\end{equation}

\begin{Proposition}
\label{prop:another}
$\overline \bz^*\colon \cL/\mathbf t^{-1}\cL \iso H^{G_\cO\rtimes\CC^\times}_*(\Gr_G)$ 
is an isomorphism.
\end{Proposition}

\begin{proof}
Recall the multifiltrations on $H^{\CC^\times\times G_\cO\rtimes\CC^\times}_*(\cR),\ 
H^{\CC^\times\times G_\cO\rtimes\CC^\times}_*(\Gr_G)$ introduced
in~\subsecref{sec:filtra}. They induce the multifiltrations on 
$H^{\CC^\times\times G_\cO\rtimes\CC^\times}_*(\cR)\otimes_{\CC[\mathbf t]}\CC[\mathbf t,
\mathbf t^{-1}],\ H^{G_\cO\rtimes\CC^\times}_*(\Gr_G)[\mathbf t,\mathbf t^{-1}]$, 
and $\bz^*$ is compatible
with the filtrations. Hence it induces a ring homomorphism
\begin{equation}\label{eq:55.5}
\on{gr}\overline \bz^*\colon \on{gr}(\cL/\mathbf t^{-1}\cL)=
(\on{gr}\cL)/\bt^{-1}(\on{gr}\cL)\to\on{gr}H^{G_\cO\rtimes\CC^\times}_*(\Gr_G). 
\end{equation}
The first equality is a version of the Zassenhaus lemma. More generally,
let $F$ be a filtered quasicoherent sheaf on an algebraic variety $S$ with a 
point $s\in S$. Then we have an isomorphism $\on{gr}(F_s)\iso(\on{gr}F)_s$.
In effect, the Rees construction associates to $F$ a $\BG_m$-equivariant
quasicoherent sheaf $RF$ on $S\times\BA^1$, and $\on{gr}F$ is the restriction
of $RF$ to $S\times\{0\}$. Hence both $\on{gr}(F_s)$ and $(\on{gr}F)_s$ are
naturally isomorphic to the fiber of $RF$ at $(s,0)\in S\times\BA^1$.

So it suffices to prove that $\on{gr}\overline \bz^*$
is surjective. But $\on{gr}\overline \bz^*[\CR_\lambda]=
(\bt^{d_\lambda}+\on{lower})\cap[\Gr_\lambda]$ where `lower' stands for the 
terms in $\CC[\ft]^{W_\lambda}[\bt]$ of degree lower than $d_\lambda$ in $\bt$.
Therefore $\on{gr}\cL$ is spanned by $\mathbf
t^{-d_\la}[\CR_\la]$ and $\mathbf t^{-d_\la}[\CR_\la]$ is sent to $[\Gr_\la]$
at $\mathbf t = \infty$.
\end{proof}

\begin{NB}
    Why taking $\gr$ and specialization at $\mathbf t=\infty$ commute ?

    Sasha's message on Oct.\ 22:

    Let $F$ be a filtered quasi-coherent sheaf on a scheme $S$ over a
    field $k$ and let $s$ be a $k$-point of $S$. I want to construct an
    isomorphism between $\gr(F)_s$ and $\gr(F_s)$. But a filtered $F$ on $S$ is
    the same as a $\mathbb G_m$-equivariant quasi-coherent sheaf
    $\widetilde{F}$ on $S\times \mathbb A^1$ and $\gr(F)$ is just the
    restriction of $\widetilde{F}$ to $S\times \{ 0\}$. Thus both $\gr(F)_s$
    and $\gr(F_s)$ are naturally isomorphic to the fiber of $\widetilde{F}$
    at the point $(s,0)$.
\end{NB}

\begin{Remark}
Thus we obtain a sheaf of algebras over $\BP^1$ with coordinate $\bt$.
From the proof of~\propref{prop:another}, it is a sheaf of filtered algebras, 
and the associated graded sheaf is clearly flat, so the initial sheaf is flat
as well. If we drop the loop rotation equivariance, we obtain a flat sheaf of
commutative algebras; let $\cM_C^{\bt}\to\BP^1$ denote its relative spectrum.
The grading by degree of equivariant homology gives rise to an action of
$\BG_m$ on $\cM_C^{\bt}$ such that the projection $\cM_C^{\bt}\to\BP^1$ is
equivariant with respect to the natural action of $\BG_m$ on $\BP^1$.
Hence all the fibers over $\bt\ne0,\infty$ are isomorphic.
\end{Remark}

\begin{Remark}
It seems likely that when $\bN=\fg$ is the adjoint representation,
$H^{\CC^\times\times G_\cO\rtimes\CC^\times}_*(\cR)$ is the spherical subalgebra
in the graded Cherednik algebra (alias trigonometric DAHA)~\cite[4.1]{oy},
while $H^{G_\cO\rtimes\CC^\times}_*(\Gr_G)$ is the spherical subalgebra in the
trigonometric Nil-DAHA (cf.~\cite[1.2]{cf}). The above degeneration is
related to the Inozemtsev degeneration~\cite[Section~2]{ino} of the quantum 
Calogero-Moser system into the quantum Toda system.
\end{Remark}

\appendix

\section{A moduli space of solutions of Nahm's equations and homology
  of affine Grassmannian}\label{sec:monop-moduli-spac}

We use mainly a compact Lie group $G_c$, rather than its
complexification $G$ in this section.

The Coulomb branch $\mathcal M_C$ of the gauge theory $G_c=\SU(k)$, $\bN
= 0$ is identified with $M^0_k$, the moduli space of centered
monopoles by physical arguments (\cite{MR1490862} for $k=2$,
\cite{MR1443803} for arbitrary $k$).

Let us generalize this result to arbitrary compact Lie group $G_c$ in
our definition of $\mathcal M_C$. We shall use a moduli space of
solutions of Nahm's equations, studied (in more general setting) by
Bielawski \cite{MR1438643}.

Let $G_c^\vee$ be the Langlands dual group of $G_c$. We consider
Nahm's equations
\begin{equation*}
    \nabla_t T_\alpha + \frac12 \sum_{\beta,\gamma}
    \varepsilon_{\alpha\beta\gamma}[T_\beta,T_\gamma] = 0 \qquad (\alpha,\beta,\gamma=1,2,3)
\end{equation*}
on the interval $(-1,1)$, where $T_\alpha$ is a $\g_c^\vee$-valued
function. We require that $T_\alpha$ has at most simple poles at
$t=\pm 1$. Then its residue $\operatorname{res}_{t=\pm 1}T_\alpha$
defines a Lie algebra homomorphism $\rho_c\colon\su(2)\to\g_c^\vee$.
Then we further require that $\rho_c$ is the restriction of a
homomorphism $\rho\colon\algsl(2)\to\g^\vee$ such that $y=\rho(
\begin{smallmatrix}
    0 & 0 \\ 1 & 0
\end{smallmatrix}
)$
is a principal nilpotent element in $\g^\vee$.

We consider the moduli space of solutions, that is the quotient by the
group $\mathcal G_{G_c^\vee}$ of gauge transformations $\gamma\colon
[-1,1]\to G_c^\vee$ such that $\gamma(\pm 1)$ is the identity
element.

\begin{Theorem}\label{thm:pure}
    The Coulomb branch $\mathcal M_C$ of the pure gauge theory for
    $G_c$ \textup(i.e., $\bN = 0$\textup) is the moduli space of
    solutions of Nahm's equations for the Langlands dual group
    $G_c^\vee$.
\end{Theorem}

This result immediately follows once we combine
\cite{MR2135527,MR2422266} with \cite{MR1438643}, as reviewed below.

Since the moduli space is a hyper-K\"ahler quotient, it is natural to
expect that this remains true in physicists' (not yet mathematically
precise) definition of $\mathcal M_C$.

While preparing the manuscript, we have noticed that this statement is
mentioned in \cite[Remark~6.4]{2014arXiv1404.6305T}.

\begin{Remark}
    When $G_c = \U(k)$, the moduli space of solutions of Nahm's
    equations is isomorphic to the space of based maps from $\CP^1$ to
    itself \cite{MR769355}. This description is the $A_1$ case of
    \ref{pestun}: As for a quiver gauge theory of type $ADE$, the
    Coulomb branch $\mathcal M_C$ is the space of based maps from
    $\CP^1$ to the flag variety of the corresponding type.
\end{Remark}

\subsection{Homology of affine Grassmannian}

When $\bN=0$, our proposal~\ref{def:main} states $\mathcal M_C$ is the
spectrum of the equivariant Borel-Moore homology group
$H^{G_\cO}_*(\Gr_{G})$ of the affine Grassmannian $\Gr_{G} =
G_\cK/G_\cO$ of $G$, equipped with an algebra structure given by the
convolution. Let us use a description which naturally arises from
\cite[Th.~3]{MR2422266}.
\begin{NB}
    The result is true for $G = \GL(k)$ ?
\end{NB}%

Let $G^\vee$ be the Langlands dual group and $\g^\vee$ its Lie
algebra.
Let $\rho\colon\algsl(2)\to\g^\vee$ be a Lie algebra homomorphism
such that $y = \rho(
\begin{smallmatrix}
    0 & 0 \\ 1 & 0
\end{smallmatrix}
)$
is a principal nilpotent element in $\g^\vee$.
Let $N^\vee$ be the unipotent subgroup of $G^\vee$ whose Lie algebra
${\mathfrak n}^\vee$ is the sum of negative eigenspaces of $h = \rho(
\begin{smallmatrix}
    -1 & 0 \\ 0 & 1
\end{smallmatrix}
).$
\begin{NB}
    We have $[h,y] = 2y$, $[h,x]=-2x$.
\end{NB}%
If we regard $y$ as an element of ${\mathfrak n}^{\vee*}$ via an
invariant pairing on $\g^\vee$, it is stabilized by $N^\vee$.
\begin{NB}
    $(y, [n_1,n_2]) = 0$ for $n_1$, $n_2\in\mathfrak n$ as the
    $(-1)$-eigenspace of $h$ is $0$.
\end{NB}%

The cotangent bundle $T^* G^\vee = G^\vee\times\g^\vee$ is a
holomorphic symplectic manifold with a $G^\vee\times G^\vee$-action by
left and right multiplication. Here we identify $\g^{\vee*}$ with
$\g^\vee$ by the invariant inner product.
We have a (complex) moment map $\mu_\CC\colon
G^\vee\times\g^\vee\to\g^{\vee*}\oplus\g^{\vee*}$ given by $\mu_\CC(g,\xi)
= (\xi,-\operatorname{Ad}(g^{-1})\xi)$ for $g\in G^\vee$,
$\xi\in\g^\vee$.
Let $\overline\mu_\CC\colon G^\vee\times\g^\vee\to {\mathfrak
  n}^{\vee*}\oplus{\mathfrak n}^{\vee*}$ be the moment map for the
$N^\vee \times N^\vee$-action, that is the composite of $\mu_\CC$ and
the natural projection $\g^{\vee*}\oplus\g^{\vee*}\to{\mathfrak
  n}^{\vee*}\oplus{\mathfrak n}^{\vee*}$. Now 
\cite[Th.~3]{MR2422266} implies
\begin{equation}
    \label{eq:25}
    \Spec(H^{G_\cO}_*(\Gr_{G})) \cong
    \overline\mu_\CC^{-1}(y,y)/N^\vee\times N^\vee.
\end{equation}

\begin{Remark}
    In \cite[Th.~3]{MR2422266}, $H^{G_\cO\rtimes\CC^\times}_*(\Gr_G)$
    is described as a quantum hamiltonian reduction of the ring of
    differential operators on $G^\vee$. The above description is its
    classical limit. On the other hand, $H^{G_\cO}_*(\Gr_G)$ is
    described as $\mathfrak Z^{G^\vee}_{\g^\vee}$ in the earlier paper
    \cite[Th.~2.12]{MR2135527} as we mentioned in
    \subsecref{subsec:prev}\ref{subsub:pure}.
    Since both are $H^{G_\cO}_*(\Gr_G)$, we have an isomorphism
    $\overline\mu_\CC^{-1}(y,y)/N^\vee\times N^\vee \cong \mathfrak
    Z^{G^\vee}_{\g^\vee}$. While preparing the manuscript, we have
    learned that an explicit construction of the isomorphism is given
    in \cite[Th.~6.3]{2014arXiv1404.6305T}.
\end{Remark}

\subsection{Moduli space of solutions of Nahm's equations}

Let $M_{G_c^\vee}$ be the moduli space in \thmref{thm:pure}. By
\cite{MR1438643} it is a submanifold of $T^* G^\vee$ defined as
follows.
Let $\rho\colon\algsl(2)\to\g^\vee$ as above. We
define Kostant-Slodowy slice $S(\rho) = y + Z(x)$, where $x = \rho(
\begin{smallmatrix}
    0 & 1 \\ 0 & 0
\end{smallmatrix}
)$,
$y = \rho(
\begin{smallmatrix}
    0 & 0 \\ 1 & 0
\end{smallmatrix}
)$
and $Z(x)$ is the centralizer of $x$ in $\g^\vee$.
Now \cite[Cor.~4.1]{MR1438643} states $M_{G_c^\vee} =
\mu^{-1}_\CC(S(\rho)\times S(\rho))$.

Let us further rewrite $M_{G_c^\vee}$ as a holomorphic symplectic
quotient of $T^* G^\vee$. The result is implicit in
\cite[\S3]{MR1438643}. The point is that $S(\rho)$ equals to
$\nu^{-1}(y)/N^\vee$, where $N^\vee$ is the unipotent subgroup of
$G^\vee$ as above, and $\nu\colon\g^{\vee*}\to\mathfrak n^{\vee*}$ is the
natural projection.
\begin{NB}
   Note $\nu^{-1}(y) = y + \mathfrak b^\vee$.
\end{NB}%
Since $\overline\mu_\CC = (\nu\times\nu)\circ \mu_\CC$, we have
\(
    M_{G_c^\vee} \cong \overline\mu^{-1}_\CC(y,y)/N^\vee\times N^\vee.
\)

Combining with \eqref{eq:25}, we finish the proof of \thmref{thm:pure}.

\subsection{Centered \texorpdfstring{$\SU(2)$}{SU(2)}-monopoles}

In the remainder of this section, we explain that \thmref{thm:pure}
reproduces \cite{MR1490862,MR1443803} for $G_c= \SU(k)$.

Let $M_k$ be the framed moduli space of charge $k$, $\SU(2)$-monopoles
on $\RR^3$ (see \cite{MR934202}). By \cite{MR709461}, it is naturally
bijective to the moduli space of solutions of Nahm's equations with
values in $\mathfrak{u}(k)$ on the interval $(-1,1)$ such that
$T_\alpha$ has at most simple poles at $t=\pm 1$ and its residue
$\operatorname{res}_{t=\pm 1}T_\alpha$ gives an irreducible
$k$-dimensional representation $\rho$ of $\SU(2)$.
They are even isomorphic as hyper-K\"ahler manifolds \cite{MR1215288}.
Since the irreducible representation corresponds to the principal
nilpotent for $\gl(k)$, we conclude $M_k = M_{\U(k)}$.
We identify $M_k$ with the moduli space of solutions of Nahm's
equations for $G_c=\U(k)$ hereafter.

We need to modify this description in the case of centered monopoles.
Let $\tilde M_k$ be the space of solutions, and $\mathcal G_k =
\mathcal G_{\U(k)}$ be the group of gauge transformations, i.e., the
space of maps $\gamma\colon [-1,1]\to \U(k)$ such that $\gamma(\pm 1)
= 1$. We have $M_k = \tilde M_k/\mathcal G_k$. We introduce a larger
group $\tilde{\mathcal G}_k$ consisting of maps $\gamma$, for which we
require $\gamma(-1) = 1$ and $\gamma(1)$ commutes with the pole
$\alpha$ of $T_\alpha$ for $\alpha=1,2,3$. Since
$\operatorname{res}T_\alpha$ gives an irreducible representation,
$\gamma(1)$ must be scalar. Moreover, the action of $\tilde{\mathcal
  G}_k$ is still free by the irreducibility of $(\nabla_t,
T_\alpha)$. We have a free action of $\tilde{\mathcal G}_k/\mathcal
G_k = \U(1)$ on $M_k$.

We also have an $\RR^3$-action on $M_k$ given by $T_\alpha\mapsto
T_\alpha + ix_\alpha$ ($x_\alpha\in\RR$, $\alpha=1,2,3$). The quotient
\begin{equation*}
    M^0_k = (M_k/\U(1))/\RR^3
\end{equation*}
is called the moduli space of centered monopoles. In \cite{MR934202},
this space was introduced in terms of monopoles, one can check that it
is equivalent to the above definition. The detail is left as an
exercise for the reader.
\begin{NB}
    See ``2014-11-13 Nahm's equation and cyclic covering.xoj''.
\end{NB}%

Choose a trivialization of the vector bundle, and write $\nabla = d +
T_0$. By the gauge transformation
\begin{equation*}
    \gamma(t) \defeq \exp\left(\frac{\operatorname{id}}k
      \int_{-1}^t \tr T_0(s)ds
      \right) \in \tilde{\mathcal G}_k,
\end{equation*}
we can make $\tr T_0 = 0$. The Nahm's equation implies that
$\frac{d}{dt} \tr T_\alpha = 0$, hence we can make $\tr T_\alpha = 0$
by the $\RR^3$-action. Therefore
\begin{equation*}
    M^0_k = \{ (T_0,T_\alpha) \colon (-1,1) \to \su(k) \}/
    \tilde{\mathcal G}_{k}',
\end{equation*}
where $(d+T_0,T_\alpha)$ satisfies Nahm's equation and the condition
of the pole. The group $\tilde{\mathcal G}_{k}'$ is the subgroup of
$\tilde{\mathcal G}_k$ preserving the condition $\tr T_0 = 0$. Therefore it
consists of maps $\gamma\colon [-1,1]\to \SU(k)$ with $\gamma(-1) = 1$,
$\gamma(1) \in \ZZ_k = \U(1)\cap\SU(k)$.

Since $T_0$, $T_\alpha$ are $\su(k)$-valued, we can replace
$\tilde{\mathcal G}_{k}'$ by the space of maps $\gamma\colon [-1,1]\to
\SU(k)/\ZZ_k$. Moreover as $[-1,1]$ is contractible, such $\gamma$
lifts uniquely to $\SU(k)$ when we set $\gamma(-1) = 1$. Therefore the
group is unchanged under this replacement. Thus
\(
   \tilde{\mathcal G}_{k}' = \mathcal G_{\SU(k)/\ZZ_k},
\)
where $\mathcal G_{\SU(k)/\ZZ_k}$ is the space of maps $\gamma\colon
[-1,1]\to \SU(k)/\ZZ_k$ such that $\gamma(\pm 1) = 1$.
\begin{NB}
    We can represent the gauge group $\mathcal G^0_{G}$ for a compact
    Lie group $G$ with $\operatorname{Lie} G = \su(k)$, as
    \begin{equation*}
        \mathcal G_G = \left\{ \gamma\in\mathcal G_{\SU(k)/\ZZ_k} \middle|
          [\gamma] \in \pi_1(G) \right\},
    \end{equation*}
    where $[\gamma]$ is the class in $\pi_1(\SU(k)/\ZZ_k) = \ZZ_k$,
    which is well-defined as $\gamma(\pm 1) = 1$, and we regard
    $\pi_1(G)$ as a subgroup of $\pi_1(\SU(k)/\ZZ_k)$ via the covering
    $G\to \SU(k)/\ZZ_k$.
\end{NB}%
Thus we see that $M^0_k$ is nothing but $M_{\SU(k)/\ZZ_k}$.

The Langlands dual group of $\SU(k)$ is $\SU(k)/\ZZ_k$, hence
$\mathcal M_C$ in \thmref{thm:pure} is the moduli space $M^0_k$ of
centered $\SU(2)$-monopoles with charge $k$, as given in
\cite{MR1490862,MR1443803}.

\begin{NB}
\section{Localization and convolution}

Suppose that $M$ is smooth and have a torus action. Let $M^T$ denote
the fixed point set.

Suppose also that we have a proper $T$-equivariant morphism $M\to
X$. Let us introduce $Z = M\times_X M$ and $Z^T = M^T\times_X M^T$.
\begin{NB2}
    The latter is the $T$-fixed point set in $Z$, hence there is no
    fear of confusion.
\end{NB2}%

We have convolution products on $H^T_*(Z)$ and $H^T_*(Z^T) =
H_*(Z^T)\otimes H^*_T(\mathrm{pt})$. The purpose of this section is to
construct an algebra isomorphism
\begin{equation*}
    \br\colon H_*^T(Z)_{\text{loc}} \to 
    H_*^T(Z^T)_{\text{loc}},
\end{equation*}
where $\bullet_{\text{loc}}$ denotes the tensor product
$\bullet\otimes_{H^*_T(\mathrm{pt})}{\mathrm{Frac}(H^*_T(\mathrm{pt}))}$.

We factorize the inclusion $M^T\times M^T \to M\times M$ into two:
\begin{equation*}
    M^T\times M^T \xrightarrow{i} M\times M^T \xrightarrow{j} M\times M.
\end{equation*}
We introduce $\tilde Z \defeq j^{-1}(Z)$. So we have
\begin{equation*}
    Z^T \xrightarrow{i} \tilde Z \xrightarrow{j} Z.
\end{equation*}
We use the same letters $i$, $j$ for maps for brevity.

We define a `mixed' pull back homomorphism
\begin{equation*}
    \br\colon H_*^T(Z)_{\text{loc}} \xrightarrow{i_*^{-1} j^*}
    H_*^T(Z^T)_{\text{loc}},
\end{equation*}
where $i_*\colon H_*^T(Z^T)\to H_*^T(\tilde Z)$ becomes an isomorphism
after $\bullet_{\text{loc}}$ by the fixed point localization theorem
*****, and $i_*^{-1}$ is its inverse. Note also that the homomorphism
\(
   j^* = \bullet\cap [M\times M^T] \colon
   H_*^T(Z)\to H_*^T(\tilde Z)
\)
is the restriction with supports defined {\it relative\/} to the
ambient submanifold $M\times M^T$ of $M\times M$ (see
\cite[\S2.6.21]{CG}). Note that $j^{-1}(Z) = \tilde Z$.

\begin{Proposition}[\cite{MR3013034}]
    $\br$ commutes with the convolution product, i.e., $\br(c_{12}\ast
    c_{23}) = \br(c_{12})\ast \br(c_{23})$ for $c_{12}$, $c_{23}\in
    H_*^T(Z)_{\mathrm{loc}}$.
\end{Proposition}

This is simple to check, if we use the formula
\begin{equation*}
    \br(c_{12}) = \frac1{e(N)\boxtimes 1} i^* j^* c_{12},
\end{equation*}
where $N$ is the normal bundle of $M^T\subset M$ and $e(N)$ is its
equivariant Euler class. But the following proof avoids the explicit
usage of $e(N)$, where such a thing cannot be consider if $M^T\subset
M$ is an infinite codimension.

\begin{Remark}
    For an application we take $M = \cT$, $X = \bN_\cK$, $T =
    \CC^\times$, and the $T$-actions on $M$, $X$ are ones induced from
    multiplication on $\bN$. Hence $M^T = \Gr_G$. In this special
    case, the homomorphism $\br$ is, in fact, defined over $H_*^T(Z)$
    {\it without\/} taking $\bullet_{\text{loc}}$. Indeed, $\tilde Z =
    Z^T = \Gr_G\times \Gr_G$, hence
    \begin{equation*}
        i_* \colon H^T_*(Z^T)\to H^T_*(\tilde Z)
    \end{equation*}
    is just an identity operator. In particular, $\br$ is an embedding
    $H^T_*(Z)\to H^T_*(Z^T)$ of algebras.
\end{Remark}

This result is a direct consequence of \cite[\S5.3]{MR3118615}. Let
$c\in H^T_*(Z)$ and $c^T = \br(c)$. Then $j^*(c) = i_*(c^T)$. It was
called two classes (two objects in the derived category in the
original setting) are {\it compatible\/} when this equality holds in
\cite[\S5.3]{MR3118615}. Let $J \defeq i_*(\Delta_{M^T})$. Then we
have $j^*(c) = c\ast J$, $i_*(c^T) = J\ast c^T$, where the convolution
products are
\(
   H_*^T(Z)\otimes H_*^T(\tilde Z)\to H_*^T(\tilde Z)
\)
and
\(
   H_*^T(\tilde Z)\otimes H_*^T(Z^T)\to H_*^T(\tilde Z)
\)
respectively. Therefore $c$ and $c^T$ are compatible if and only if
$c\ast J = J\ast c^T$. Now the assertion is clear from the
associativity of the convolution product:
\begin{equation*}
    c_{12}\ast c_{23} \ast J 
    = c_{12} \ast J\ast c_{23}^T = J \ast c_{12}^T \ast c_{23}^T.
\end{equation*}

\begin{NB2}
Here is the original treatment.

\begin{proof}
    Let us assume $M = X$, and hence $M$ is compact first.

    We use the notation $M_1$, $M_2$, $M_3$ so that $c_{12}$ is a
    cohomology class on $M_1\times M_2$, etc. We write $c_{12}^T =
    \br(c_{12})$, etc.  Thus $i_{12*} c_{12}^T = j_{12}^* c_{12}$, etc.

    Consider the following commutative diagram:
    \begin{equation*}
        \begin{CD}
            M_1\times M_3 @<p_{13}<< M_1\times M_2 \times M_3
            @>p_{12}\times p_{23}>> M_1\times M_2 \times M_2 \times M_3
\\
@A{j_{13}}AA @AA{\id_{M_1}\times j_{23}}A
@AA{\id_{M_1\times M_2}\times j_{23}}A
\\
M_1\times M_3^T @<p'_{13}<< M_1\times M_2 \times M_3^T
@>(p_{12}\times p_{23})'>> M_1\times M_2 \times M_2 \times M_3^T
\\
@| @AA{\id_{M_1}\times i_{23}}A @AA{\id_{M_1\times M_2}\times i_{23}}A
\\
M_1\times M_3^T @<p''_{13}<< M_1\times M_2^T \times M_3^T
@>(p_{12}\times p_{23})''>> M_1\times M_2 \times M_2^T \times M_3^T
\\
@| @| @AA{j_{12}\times \id_{M_2^T\times M_3^T}}A
\\
M_1\times M_3^T @<p_{13}''<< M_1\times M_2^T \times M_3^T
@>(p_{12}\times p_{23})'''>> M_1\times M_2^T \times M_2^T \times M_3^T
\\
@A{i_{13}}AA @AA{i_{12}\times\id_{M_3^T}}A
@AA{i_{12}\times \id_{M_2^T\times M_3^T}}A
\\
M_1^T\times M_3^T @<p_{13}^T<< M_1^T\times M_2^T \times M_3^T
@>p_{12}^T\times p_{23}^T>> M_1^T\times M_2^T \times M_2^T \times M_3^T,
\end{CD}
\end{equation*}
where vertical arrows are products of $i_{12}$ or $j_{12}$ and the
identity operators,
all horizontal arrows are projections restricted to appropriate
spaces.

Chasing the diagram, we obtain the desired equality:
\begin{equation}\label{eq:9}
    \begin{split}
    & j_{13}^* p_{13*}(p_{12}\times p_{23})^*(c_{12}\boxtimes c_{23})
\\
=\; & p'_{13*}(\id_{M_1}\times j_{23})^*(p_{12}\times p_{23})^*(c_{12}\boxtimes c_{23})
\\
=\; & p'_{13*}(p_{12}\times p_{23})^{\prime*}
(\id_{M_1\times M_2}\times j_{23})^*(c_{12}\boxtimes c_{23})
\\
=\; & p'_{13*}(p_{12}\times p_{23})^{\prime*}
(\id_{M_1\times M_2}\times i_{23})_* (c_{12}\boxtimes c_{23}^T)
\\
=\; & p'_{13*}(\id_{M_1}\times i_{23})_* (p_{12}\times p_{23})^{\prime\prime*}
(c_{12}\boxtimes c_{23}^T)
\\
=\; & p'_{13*}(\id_{M_1}\times i_{23})_*
(p_{12}\times p_{23})^{\prime\prime\prime*}
(j_{12}\times\id_{M_2^T\times M_3^T})^*
(c_{12}\boxtimes c_{23}^T)
\\
=\; & p'_{13*}(\id_{M_1}\times i_{23})_*
(p_{12}\times p_{23})^{\prime\prime\prime*}
(i_{12}\times\id_{M_2^T\times M_3^T})_*
(c_{12}^T\boxtimes c_{23}^T)
\\
=\; & p'_{13*}(\id_{M_1}\times i_{23})_*
(i_{12}\times\id_{M_3^T})_*
(p_{12}^T\times p_{23}^T)^*
(c_{12}^T\boxtimes c_{23}^T)
\\
=\; & i_{13*}p_{13*}^T(p_{12}^T\times p_{23}^T)^*(c_{12}^T\boxtimes c_{23}^T).
\end{split}
\end{equation}
Here we have used the base change at the first, fourth, and seventh
equalities. The relevant squares are left-top, right-second, and
right-bottom respectively. We also used the assumptions $j_{23}^*
c_{23} = i_{23*} c_{23}^T$, $j_{12}^* c_{12} = i_{12*} c_{12}^T$ at
the third and sixth equalities respectively. The remaining second,
fifth and eighth equalities just follow from the commutativity.

Next consider a general case. The above diagram is replaced by
    \begin{equation*}
        \begin{CD}
            Z_{13} @<p_{13}<< (p_{12}\times p_{23})^{-1}(Z_{12}\times Z_{23})
            @>p_{12}\times p_{23}>> Z_{12}\times Z_{23}
\\
@A{j_{13}}AA @AA{\id_{M_1}\times j_{23}}A
@AA{\id_{M_1\times M_2}\times j_{23}}A
\\
\tilde Z_{13}@<p'_{13}<<(p_{12}\times p_{23})^{\prime-1}
(Z_{12}\times \tilde Z_{23})
@>(p_{12}\times p_{23})'>> Z_{12}\times \tilde Z_{23}
\\
@| @AA{\id_{M_1}\times i_{23}}A @AA{\id_{M_1\times M_2}\times i_{23}}A
\\
\tilde Z_{13} @<p''_{13}<< (p_{12}\times p_{23})^{\prime\prime-1}
(Z_{12}\times Z_{23}^T)
@>(p_{12}\times p_{23})''>> Z_{12}\times Z_{23}^T
\\
@| @| @AA{j_{12}\times \id_{M_2^T\times M_3^T}}A
\\
\tilde Z_{13} @<p_{13}''<< (p_{12}\times p_{23})^{\prime\prime\prime-1}
(\tilde Z_{12} \times Z_{23}^T)
@>(p_{12}\times p_{23})'''>> \tilde Z_{12} \times Z_{23}^T
\\
@A{i_{13}}AA @AA{i_{12}\times\id_{M_3^T}}A
@AA{i_{12}\times \id_{M_2^T\times M_3^T}}A
\\
Z_{13}^T@<p_{13}^T<< (p_{12}^T\times p_{23}^T)^{-1}(Z_{12}^T\times Z_{23}^T)
@>p_{12}^T\times p_{23}^T>> Z_{12}^T\times Z_{23}^T.
\end{CD}
\end{equation*}
We repeat the computation in \eqref{eq:9}, the only thing to be
checked is that the base change formula remains to be true even if the
pull-back homomorphism is replaced by the restriction with support.
\end{proof}
\end{NB2}%

\begin{verbatim}
\begin{equation*}
\xymatrix@!0{
& M_1 \ar@{->}[rr]\ar@{->}'[d][dd]
& & N_1 \ar@{->}[dd]
\\
X_1 \ar@{->}[ur]\ar@{->}[rr]\ar@{->}[dd]
& & Y_1 \ar@{->}[ur]\ar@{->}[dd]
\\
& M_2 \ar@{->}'[r][rr]
& & N_2
\\
X_2\ar@{->}[rr]\ar@{-}[ur]
& & Y_2 \ar@{->}[ur]
}
\end{equation*}
\end{verbatim}

\begin{equation*}
\xymatrix@!0{
& M_1 \ar@{->}[rr]\ar@{->}'[d][dd]
& & N_1 \ar@{->}[dd]
\\
X_1 \ar@{->}[ur]\ar@{->}[rr]\ar@{->}[dd]
& & Y_1 \ar@{->}[ur]\ar@{->}[dd]
\\
& M_2 \ar@{->}'[r][rr]
& & N_2
\\
X_2\ar@{->}[rr]\ar@{-}[ur]
& & Y_2 \ar@{->}[ur]
}
\end{equation*}
\end{NB}

\bibliographystyle{myamsalpha}
\bibliography{nakajima,mybib,coulomb}

\subsection*{Erratum}

\begin{enumerate}

\item In Proposition 6.2, the formula
$f[\cR_\lambda] \ast g[\cR_\mu] = a_{\lambda,\mu} fg
[\cR_{\lambda+\mu}]$ is correct for $\gr\cA$, but \emph{not} true for
$\gr \cAh$. Replace $g$ in the right hand side by
$g(\bullet+\hbar\lambda)$. This shift appears when we exchange the
order of $g$ and $r^\lambda$ in the proof.
\end{enumerate}


\end{document}